\documentclass[11pt]{article}
\usepackage{vmargin,graphicx,amsmath,amssymb,color,subfigure,enumerate,csquotes, dsfont, mathrsfs}
\definecolor{webred}{rgb}{0.75,0,0}
\definecolor{webgreen}{rgb}{0,0.75,0}
\usepackage[citecolor=webgreen,colorlinks=true,linkcolor=webred]{hyperref}

\newtheorem{theorem}{Theorem}[section]
\newtheorem{lemma}[theorem]{Lemma}
\newtheorem{notation}[theorem]{Notation}
\newtheorem{assumption}[theorem]{Assumption}
\newtheorem{proposition}[theorem]{Proposition}
\newtheorem{corollary}[theorem]{Corollary}
\newtheorem{conjecture}[theorem]{Conjecture}
\newtheorem{definition}[theorem]{Definition\rm}

\def\blacksquare{
\thinspace\nobreak \vrule width 5pt height 5pt depth 0pt}
\newtheorem{remark}[theorem]{Remark}

\newenvironment{proof}{\begin{trivlist}
                       \item[]\hspace{0cm}{\bf Proof: }
                       \hspace{0cm} }{\hfill $\blacksquare$
                     \end{trivlist}}
\newenvironment{proofof}[1]{\begin{trivlist}
                       \item[]\hspace{0cm}{\bf Proof of #1: }
                       \hspace{0cm} }{\hfill $\blacksquare$
                     \end{trivlist}}
\makeatletter

\@addtoreset{equation}{section}
\makeatother

\newcommand{\mun}{\mu}
\newcommand{\nun}{\nu}
\newcommand{\jj}{d}
\newcommand{\hs}{\sigma}
\newcommand{\htt}{\tau}
\newcommand{\hz}{{\scriptscriptstyle{\cal{Z}}}}
\newcommand{\hbfn}{\bfn}
\newcommand{\sL}{{\rm L}}

\newcommand{\re}{{\rm e}}
\newcommand{\N}{\mathbb{N}}
\newcommand{\R}{\mathbb{R}}
\newcommand{\dx}{\,\mathrm{d}}
\newcommand{\Dir}{\mathsf{Dir}}
\newcommand{\ga}{\mathfrak{a}}
\newcommand{\an}{a}
\newcommand{\bn}{b}
\newcommand{\bB}{\boldsymbol{B}}
\newcommand{\bA}{\boldsymbol{A}}
\newcommand{\cT}{\mathcal{T}}
\newcommand{\bfn}{\textbf{n}}
\newcommand{\gA}{\mathfrak{A}}
\newcommand{\dr}{\partial}
\newcommand{\Hess}{\mathsf{Hess}\,}
\newcommand{\dist}{\mathsf{dist}}
\newcommand{\spann}{\mathsf{span}}
\newcommand{\gM}{\mathsf{gM}}
\newcommand{\ess}{\mathsf{ess}}
\newcommand{\W}{\mathsf{W}}
\newcommand{\la}{\mathfrak{l}}
\newcommand{\spe}{\mathsf{sp}}
\newcommand{\eps}{\varepsilon}
\newcommand{\new}{\flat}
\newcommand{\Dom}{\mathsf{Dom}}
\newcommand{\supp}{\mathsf{supp}\,}
\newcommand{\wgt}{\mathsf{wg}}
\newcommand{\FH}{\mathsf{c}}
\newcommand{\Lc}{\mathfrak{L}^{\natural}}
\newcommand{\Mc}{\mathcal{M}}
\newcommand{\Oc}{\mathcal{O}}

\newcommand{\Vc}{\mathcal{V}}
\newcommand{\Z}{\mathsf{z}}
\newcommand{\van}{\mathsf{vf}}
\newcommand{\PR}{\mathsf{e}}

\newcommand{\sgn}{\mathsf{sgn}}

\newcommand{\range}{\mathsf{range}}
\newcommand{\Bhe}{{\mathcal B}(h^{-1/2}\varepsilon_{0})} 
\newcommand{\CBhe}{\complement{\mathcal B}(h^{-1/2}\varepsilon_{0})} 
\newcommand{\demi}{\frac 12} 

\definecolor{gr}{rgb}   {0.,   0.69,   0.23 }
\definecolor{bl}{rgb}   {0.,   0.5,   1. }
\definecolor{mg}{rgb}   {0.85,  0.,    0.85}
\definecolor{yl}{rgb}   {0.8,  0.7,   0.}
\definecolor{or}{rgb}  {0.7,0.2,0.2}

\def\bbb{{\cal B}}\def\ccc{{\cal C}}

\def\C{\mathbb C}    
  
 \def\N{\mathbb N}    
 \def\R{\mathbb R}

\def\D{\partial}
\def\eps{\varepsilon}

\def\set#1{\left\{#1\right\}}

\def\sep#1{\left(#1\right)}

\def\Re{{\mathrm {Re}\,}} \def\Im{{\mathrm {Im}\,}}

\def\max{\textrm{max}}

\begin{document}

\title{Magnetic WKB Constructions}
\author{V. Bonnaillie-No\"el\footnote{D\'epartement de Math\'ematiques et Applications (DMA - UMR 8553), PSL, CNRS, ENS Paris, 45 rue d'Ulm, F-75230 Paris cedex 05, France
\texttt{bonnaillie@math.cnrs.fr}},
F. H\'erau\footnote{LMJL - UMR6629, Universit\'e de Nantes, 2 rue de la Houssini\`ere, BP 92208, 44322 Nantes Cedex 3, France, \texttt{frederic.herau@univ-nantes.fr}}
and N. Raymond\footnote{IRMAR, Univ. Rennes 1, CNRS, Campus de Beaulieu, F-35042 Rennes cedex, France
\texttt{nicolas.raymond@univ-rennes1.fr}}}
\date{\today}
\maketitle

\begin{abstract}
This paper is devoted to the semiclassical magnetic Laplacian. Until now WKB expansions for the eigenfunctions were only established in presence of a non-zero electric potential. Here we tackle the pure magnetic case. Thanks to Feynman-Hellmann type formulas and coherent states decomposition, we develop here a magnetic Born-Oppenheimer theory. Exploiting the multiple scales of the problem, we are led to solve an effective eikonal equation in pure magnetic cases and to obtain WKB expansions. We also investigate explicit examples for which we can improve our general theorem:  global WKB expansions, quasi-optimal Agmon estimates and upper bound of the tunelling effect (in symmetric cases). We also apply our strategy to get more accurate descriptions of the eigenvalues and eigenfunctions in a wide range of situations analyzed in the last two decades.
\end{abstract}

\paragraph{Keywords.} WKB expansion, magnetic Laplacian, Born-Oppenheimer approximation, coherent states, Agmon estimates, tunnel effect.
\paragraph{MSC classification.} 35P15, 35J10, 81Q10, 81Q15.

{\small{\tableofcontents}}

\section{Motivation and main results}

\subsection{Context and motivation}
This paper is devoted to the analysis of the self-adjoint operators on $\sL^2(\R_{s}^m\times\R_{t}^n,\dx s \dx t)$ of the following type 
\begin{equation}\label{Lfh}
\mathfrak{L}_{h}=(hD_{s}+A_{1}(s,t))^2+(D_t+A_{2}(s,t))^2,
\end{equation}
where $A_{1}$ and $A_{2}$ are smooth functions (on which we will sometimes assume more), $D=-i\nabla$,  and where the space $\sL^2(\R_{s}^m\times\R_{t}^n,\dx s \dx t)$ is equipped with the standard scalar product:
$$\langle\psi_{1},\psi_{2}\rangle_{\sL^2(\R_{s}^m\times\R_{t}^n,\dx s \dx t)}=\int_{\R^m\times\R^n}\psi_{1}\overline{\psi_{2}} \dx s\dx t.$$
The corresponding quadratic form is denoted by $\mathfrak{Q}_{h}$. We would like to describe the lowest eigenpairs (eigenvalues and eigenfunctions)
of this operator in the limit $h\to 0$ under elementary confining assumptions.
\subsubsection{The Born-Oppenheimer strategy}
The problem of considering partial semiclassical problems appears for instance in the context of \cite{Martinez89, KMSW92} where the main issue is to approximate the eigenpairs of operators with electrical potentials in the form:
\begin{equation}\label{original}
-h^2\Delta_{s}-\Delta_{t}+V(s,t).
\end{equation}
The main idea, due to Born and Oppenheimer in \cite{BO27}, is to replace, for fixed $s$, the operator $-\Delta_{t}+V(s,t)$ by its eigenvalues $\mu_{k}(s)$ (by convention we omit the index for $k=1$). Then we are led  to consider for instance the reduced operator (called Born-Oppenheimer approximation)
$$-h^2\Delta_{s}+\mun(s),$$
and to apply the semiclassical techniques \textit{\`a la} Helffer-Sj\"ostrand \cite{HelSj84, HelSj85} to analyze in particular the tunnel effect when the potential $\mun$ admits symmetries. The main point is to make the reduction of dimension rigorous. Note that we have always the following lower bound
\begin{equation}\label{lbBOE}
-h^2\Delta_{s}-\Delta_{t}+V(s,t)\geq -h^2\Delta_{s}+\mun(s),
\end{equation}
which usually involves accurate Agmon estimates with respect to $s$: the eigenfunctions of the operator \eqref{original} satisfy the same decay estimates as the eigenfunctions of the one dimensional operator (see for instance \cite{Hel88}).

Our paper aims at understanding the analogy between magnetic case \eqref{Lfh} and electric case \eqref{original}. In particular even the formal dimensional reduction seems to be a little more problematic than in the electric case. Let us write the operator valued symbol of $\mathfrak{L}_{h}$. For $(x,\xi)\in\R^m \times\R^m$, we introduce the electro-magnetic Laplacian acting on $\sL^2(\R^n, \dx t)$:
\[\mathcal{M}_{x,\xi}=(D_{t} +A_{2}(x,t))^2+(\xi+A_{1}(x,t))^2.\]
Let us introduce the notation for the bottom of the spectrum of this operator.
\begin{definition}\label{def.mu}
For all $(x,\xi)\in\R^m\times\R^m$, the bottom of the spectrum of the essentially self-adjoint operator $\mathcal{M}_{x,\xi}$ is denoted by $\mun(x,\xi)$.
\end{definition}
We would like to replace $\mathfrak{L}_{h}$ by the $m$-dimensional pseudo-differential operator:
$$\mun(s,hD_{s}).$$
Under different assumptions, such reductions are considered in \cite[Theorem 2.1 and remark thereafter]{Martinez07} where it is suggested that the spectrum of $\mathfrak{L}_{h}$ could be completely determined by an effective Hamiltonian (a matrix of pseudo-differential operators) whose principal symbol can be described thanks to the spectral invariants of the operator valued symbol of $\mathfrak{L}_{h}$. For the present situation the low lying spectrum of $\mathfrak{L}_{h}$ could be described by the one of $\mun(s,hD_{s})$ modulo $\Oc(h)$ and we will see that, under generic assumptions, $\Oc(h)$ is precisely the order of the spectral gap between the first eigenvalues in the simple well case.

\subsubsection{Multiple scales induced by the fully semiclassical magnetic Laplacian}
Another important motivation to analyze partially semiclassical problems with magnetic fields comes in fact from the fully semiclassical case (\textit{i.e.} when the parameter $h$ multiplies all the derivatives). Let us now explain in which sense. The study of the discrete spectrum magnetic Laplacian $(-i\hbar\nabla+\bA)^2$ has given rise to many contributions in the last twenty years, especially in the semiclassical limit. To have an overview on the subject one may refer to the book by Fournais and Helffer \cite{FouHel10}, the survey by Helffer and Kordyukov \cite{HelKo14} and the lecture notes by Raymond \cite{Ray14}. Many papers are concerned with finding approximations of the first eigenfunctions.
Such approximations are difficult to obtain due to the geometry of a possible boundary (carrying in general a Neumann type condition) and to the possible variations of the magnetic field $\bB=\nabla\times\bA$. In dimension two the case of the disk is investigated in \cite{BPT98, BH, BeSt, PFS00} and generalized to smooth domains in \cite{HelMo01} where it is proved that
\begin{equation}\label{HM-result}
\lambda_{1}(\hbar)=\Theta_{0} \hbar-C_{1}\kappa_{max} \hbar^{3/2}+o(\hbar^{3/2}),
\end{equation}
where $\kappa_{max}$ is the maximal curvature of the boundary and where $\Theta_{0}>0$ and $C_{1}>0$ are universal constants related to a half-line model (actually we have $\Theta_{0}=\mathfrak{h}^{[0]}(\zeta^{[0]}_{0})$, see Section \ref{Sec:Mont}). An important point to notice is that, in the above mentioned papers, nothing is told about the simplicity of the first eigenvalue or about the approximation of the eigenfunctions. A reason for this is that the spirit of the analysis is essentially variational: it is based on the construction of appropriate test functions for the first Rayleigh quotient so that, even if the simplicity of the eigenfunctions were known, nothing could be deduced for the approximation of the eigenfunctions.

The paper \cite{FouHel06a} is the first one to establish, in a smooth case and under non-degeneracy assumptions, the approximation of the eigenfunctions and the simplicity of the lowest eigenvalues. The crucial idea to get such results is to understand a double scale structure due to the inhomogeneity of the pure magnetic Laplacian, which is specific to problems with smooth boundaries or without boundary, and to apply the spectral mapping theorem. In such situations it appears that the microlocalization (on possibly different scales) of the eigenfunctions plays an important role in the determination of the spectral asymptotics. In particular the papers \cite{Ray11b, DomRay12, PoRay12}, which are concerned with varying magnetic fields, establish full asymptotic expansions of the low lying eigenvalues and eigenfunctions by the reduction to the (electric) Born-Oppenheimer approximation which naturally involves different scales. The analysis of \cite{DomRay12}, related to vanishing magnetic fields, is motivated by the papers \cite{Montgomery95, HelMo96, HelKo09, HelKo12a} and solves one of their conjectures on the asymptotic simplicity of the eigenvalues. This paper will provide simple examples suggested by all the above mentioned models. Moreover, as we will see, the proof of the simplicity of the eigenvalues as well as the expansions of the eigenfunctions strongly relies on the multiscale analysis of microlocal models as it is done while studying hypoellipticity. We will see that an analysis \textit{\`a la Born-Oppenheimer} will allow us to deal with all the above mentioned situations.

\subsubsection{Magnetic WKB expansions and Agmon estimates}
In all the papers about asymptotic expansions of the magnetic eigenfunctions, one of the methods consists in using a formal power series expansion. It turns out that these constructions are never in the famous WKB form, but in a weaker and somehow more flexible one. When there is an additional electric potential providing easy confinement estimates, the WKB expansions are possible as we can see in \cite{HelSj87} and \cite{MS99}. The reason for which we would like to have a WKB description of the eigenfunctions is to get a precise estimate of the magnetic tunnel effect in the case of symmetries. Until now, such estimates are only investigated in two dimensional corner domains in \cite{BD06} and \cite{BDMV07} for the numerical counterpart. It turns out that the crucial point to get an accurate estimate of the exponentially small splitting of the eigenvalues is to establish exponential decay of Agmon type. These localization estimates are rather easy to obtain (at least to get the good scale in the exponential decay) in the corner cases due to the fact that the operator is \enquote{more elliptic} than in the regular case in the following sense: the spectral asymptotics is completely determined by the principal symbol. Nevertheless, let us notice here that, on the one hand, the numerics suggests that the eigenvalues do not seem to be simple (see for instance the case of the square in \cite{BDMV07} or of the ellipse in Figure \ref{fig.ellipse}) and, on the other hand, that establishing the optimal Agmon estimates is still an open problem. In smooth cases, due to a lack of ellipticity and to the multiple scales, the localization estimates obtained in the literature are in  general not optimal or rely on the presence of an electric potential (see \cite{N96, N99}): the principal symbol provides only a partial confinement whereas the precise localization of the eigenfunctions seems to be determined by the subprincipal terms. As far as we know, the present paper provides the first examples of WKB expansions in pure magnetic situations as well as quasi-optimal -- optimal in terms of power of $h$ but with no exhibited distance of Agmon -- Agmon estimates in model situations. In particular, we prove for a wide range of situations analyzed in the past decades that the magnetic eigenfunctions are in the WKB form under generic assumptions. This paper can be considered as the first necessary step (WKB expansions and rather accurate Agmon estimates) towards the complete comprehension of the magnetic tunnel effect.

\subsection{Main results and strategy of the proofs}\label{sec.mainresults}

\subsubsection{Spectrum of the simple magnetic wells}
In the simple well situation, we will work under the following assumptions. The first assumption states that the lowest eigenvalue of the operator symbol of $\mathfrak{L}_{h}$ admits a unique and non degenerate minimum and the second one concerns the simplicity of the spectrum of the effective harmonic oscillator.

\begin{assumption}\label{hyp-gen}
\begin{itemize}
\item[-] The function $\R^m\times\R^m\ni(x,\xi)\mapsto \mu(x,\xi)$ is continuous and admits a unique and non degenerate minimum $\mu_{0}$ at a point denoted by $(x_{0},\xi_{0})$ and such that $\liminf_{|x|+|\xi|\to+\infty}\mu(x,\xi)>\mu_{0}$.
\item[-] The family $(\mathcal{M}_{x,\xi})_{(x,\xi)\in\R^m\times\R^m}$ can be extended into a holomorphic family of type (A) in the sense of Kato \cite[Chapter VII]{Kato66} in a complex neighborhood $\mathcal{V}_{0}$ of $(x_{0},\xi_{0})$.
\item[-] For all $(x,\xi)\in\mathcal{V}_{0}\cap(\R^m\times\R^m)$, $\mu(x,\xi)$ is a simple eigenvalue. 
\end{itemize}
\end{assumption}
\begin{remark}
Let us explain an example that we have in mind when stating Assumption \ref{hyp-gen}. As we will see in Section \ref{Sec:Mont}, we will consider operator symbols (acting on $\sL^2(\R_{t})$) in the form
$$\mathcal{M}^{[k]}_{x,\xi}=D_{t}^2+V(x,\xi),\quad \mbox{ with } V(x,\xi)=\left(\xi-\gamma(x)\frac{t^{k+1}}{k+1}\right)^2,\qquad (x,\xi)\in\R^m\times\R^m\,,$$
where $k\geq 1$ and $\gamma$ is uniformly bounded from below by a positive constant. In this case, the domain of $\mathcal{M}^{[k]}_{x,\xi}$ does not depend on $(x,\xi)$ and $\mathcal{M}^{[k]}_{x,\xi}$ depends on $(x,\xi)$ analytically as soon as $\gamma$ is analytic. Therefore it is a real analytic family of type (A) and it is not difficult to extend locally this family into a holomorphic family. The operator has clearly a compact resolvent so that $\mu(x,\xi)$ is always an eigenvalue (and it is simple by a standard ODE argument). By the min-max principle, it is rather direct to see that the function $\mu$ is continuous on $\R^m\times\R^m$ and that it behaves nicely at infinity (for example with $\gamma$ as in Proposition~\ref{WKB-model}). The uniqueness of the minimum and its non degeneracy are related to more advanced considerations (see Section \ref{Sec:Mont}).
\end{remark}

\begin{remark}
Under Assumption~\ref{hyp-gen}, the function $\mu$ is analytic with respect to $(x,\xi)$ and it is 
associated with an $\sL^2$-normalized eigenfunction $u_{x,\xi}\in\mathcal{S}(\R^n)$ which also analytically depends on $(x,\xi)$.\\

In a neighborhood $\mathcal{V}_{0}$ of $(x_{0},\xi_{0})$, we still denote by $u_{x,\xi}$ and $\mun(x,\xi)$ the holomorphic extensions of $u$ and $\mun$ and we have locally:
\begin{equation}\label{u^2}
\int_{\R^n} u_{x,\xi}\overline{u_{\overline{x},\overline{\xi}}}\dx t=1.
\end{equation}
Note that the holomorphic extension of $u$ is not always $\sL^2$-normalized.
\end{remark}

\begin{assumption}\label{hyp-gen'}
Under Assumption \ref{hyp-gen}, let us denote by $\Hess\,\mun(x_{0},\xi_{0})$ the Hessian matrix of $\mun$ at $(x_{0},\xi_{0})$. We assume that the spectrum of the operator $\Hess\,\mun(x_{0},\xi_{0})(\sigma, D_{\sigma})$ is simple.
\end{assumption}
\begin{remark}\label{rem.m1}
Assumption \ref{hyp-gen'} is automatically satisfied when $m=1$.
\end{remark}
The last assumption is a spectral confinement.
\begin{assumption}\label{confining}
For $R\geq 0$, we let $\Omega_{R}=\R^{m+n}\setminus \overline{{\cal B}(0,R)}$. We denote by $\mathfrak{L}_{h}^{\Dir,\Omega_{R}}$ the Dirichlet realization on $\Omega_{R}$ of $(D_{t}+A_{2}(s,t))^2+(hD_{s}+A_{1}(s,t))^2$. 
We assume that there exist $R_{0}\geq 0$, $h_{0}>0$ and $\mu_{0}^*>\mu_{0}$ such that for all $h\in(0,h_{0})$, the bottom of the spectrum of $\mathfrak{L}_{h}^{\Dir,\Omega_{R_{0}}}$  satisfies:
$$\lambda_{1}^{\Dir,\Omega_{R_{0}}}(h)\geq \mu_{0}^*.$$
\end{assumption}
\begin{remark}\label{confining2}
In particular, due to the monotonicity of the Dirichlet realization with respect to the domain, Assumption \ref{confining} implies that there exist $R_{0}>0$ and $h_{0}>0$ such that for all $R\geq R_{0}$ and $h\in(0,h_{0})$:
$$\lambda_{1}^{\Dir,\Omega_{R}}(h)\geq \lambda_{1}^{\Dir,\Omega_{R_{0}}}(h)\geq \mu_{0}^*.$$
\end{remark}
By using the Persson's theorem (see \cite{Persson60}), we have the following proposition.
\begin{proposition}\label{essential}
Under Assumption \ref{confining}, there exists $h_{0}>0$ such that for all $h\in(0,h_{0})$:
$$\inf \spe_{\ess}(\mathfrak{L}_{h})\geq \mu_{0}^*.$$
\end{proposition}

\begin{theorem}\label{theorem-simple-well}
Under Assumptions \ref{hyp-gen}, \ref{hyp-gen'} and \ref{confining}, and assuming in addition that $A_{1}$ and $A_{2}$ are polynomials, for all $n\geq 1$, there exists $h_{0}>0$ such that for all $h\in(0,h_{0})$ the $n$-th eigenvalue of $\mathfrak{L}_{h}$ exists and satisfies
$$
\lambda_{n}(h) = \lambda_{n,0} +  \lambda_{n,1}h + o(h), $$
where
$\lambda_{n,0}=\mu_{0}$ and $\lambda_{n,1}$ is the $n$-th eigenvalue of $\frac{1}{2}\Hess\,\mun(x_{0},\xi_{0})(\sigma, D_{\sigma})$.
\end{theorem}

\begin{remark}
In fact using the double scale construction developed in the proof of the previous theorem,
it is possible to get a complete asymptotic expansion of the following type
$$\lambda_{n}(h)\underset{h\to 0}{\sim} \sum_{j\geq 0} \la_{n,j}h^{j/2},$$
where $\la_{n,0}=\mu_{0}$, $\la_{n,1}=0$ and $\la_{n,2} = \lambda_{n,1}$.
\end{remark}

\paragraph{Strategy of the proof of Theorem \ref{theorem-simple-well}.}
The proof of Theorem \ref{theorem-simple-well} is divided into two main steps. The first step is to construct quasimodes as formal series expansions and to apply the spectral theorem. In order to succeed we will establish Feynman-Hellmann formulas with multiple parameters which are consequences of the perturbation theory of Kato. The second step which is slightly more difficult is to get an accurate estimate of the spectral splitting between the eigenvalues. For that purpose, we will follow the strategy of \cite{Ray13} by using a partial coherent states decomposition with respect to the semiclassical variables $s$ and use it to establish polynomial estimates (in the spirit of \cite{Ray12} and also \cite{HelKo11}) in the phase space satisfied by the eigenfunctions. Then the Feshbach-Grushin type reduction is used to rigorously reduce the dimension and get the spectral splitting. The proof of Theorem \ref{theorem-simple-well} is the aim of Section~\ref{Sec.2}.

\subsubsection{Magnetic WKB expansions: simple well case}
We provide now WKB expansions of the lowest eigenpairs in a pure magnetic case. We reduce here our study to the case when $A_{2}=0$ for reasons motivated in Remark \ref{rem:A20}. We therefore focus now on operators of the form
 \begin{equation}\label{Lfh2}
\mathfrak{L}_{h}=D_t^2 + (hD_{s}+A_{1}(s,t))^2.
\end{equation} 
 Let us state one of the most important results of this paper.

\begin{theorem}\label{WKB-general}
We assume that $A_{2}=0$ and that $A_{1}$ is real analytic. Under Assumptions \ref{hyp-gen}, \ref{hyp-gen'} and \ref{confining}, there exist a function $\Phi=\Phi(s)$ defined in a neighborhood $\Vc$ of $x_{0}$ with  $\Re \Hess\Phi(x_0) >0$ and, for any $n\geq 1$, a sequence of  real numbers $(\lambda_{n,j})_{j\geq 0}$ such that
$$
\lambda_n(h) \underset{h\to 0}{\sim}\sum_{j\geq 0}\lambda_{n,j} h^j,
$$
in the sense of formal series, with $\lambda_{n,0}=\mu_{0}$. Besides there exists a formal
series of smooth functions on $ \Vc \times \R^n_t$
$$
 \an_{n}(. ; h)\underset{h\to 0}{\sim}\sum_{j\geq 0}\an_{n,j} h^j
$$
with $\an_{n,0} \neq 0$ such that
$$
\left(\mathfrak{L}_{h}-\lambda_{n}(h)\right)\left( \an_{n}(. ;h) \re^{-\Phi/h}\right)=\mathcal{O}\left(h^{\infty}\right)  \re^{-\Phi/h}.
$$
Furthermore the functions $t\mapsto a_{n,j}(s,t)$ belong to the Schwartz class uniformly in $s\in\Vc$. In addition, if $A_{1}$ is a polynomial function, there exists $c_0>0$ such that for all $h\in(0,h_{0})$
$$
\mathcal{B}\Big(\lambda_{n,0} + \lambda_{n,1} h,c_0h\Big)\cap \spe\left(\mathfrak{L}_{h}\right)=\{\lambda_{n}(h)\},$$
and $\lambda_{n}(h)$ is a simple eigenvalue.
\end{theorem}

In the previous theorem we used the following definition of formal series of functions.
\begin{notation} \label{formalfunctions}
Let $n\geq1$. We write $\an_{n}(s,t ;h)\underset{h\to 0}{\sim}\sum_{j\geq 0}\an_{n,j}(s,t) h^j$ when for all $J\geq 0$ and $\alpha \in \N^{m+n}$, there exist $h_{J,\alpha}>0$ and $C_{J,\alpha}>0$ such that for all $h\in(0,h_{J,\alpha})$, we have
$$\bigg|\partial^{\alpha} \Big( \an_{n}(s,t ; h)-\sum^J_{j= 0}\an_{n,j}(s,t)h^j\Big)\bigg|\leq C_{J,\alpha}h^{J+1}\quad\mbox{Êlocally in } (s,t) \in \Vc\times\R^n.$$
We also write $a = \Oc(h^\infty)$ when all the coefficients in the series are zero. The case of formal series of numbers is similar.
\end{notation}
Let us also recall that for any arbitrary sequence of smooth functions $a_j$ one can always find, by a procedure of Borel type, a smooth function $a(s,t ; h)$ (up to $\Oc(h^\infty)$) such that
$a(s,t ; h)\underset{h\to 0}{\sim}\sum_{j\geq 0}a_{j}(s,t) h^j$ (see e.g. \cite{Ma02}).
\begin{remark}
Thanks to Theorem \ref{theorem-simple-well}, we have sharp asymptotic expansions of the eigenvalues. In particular, one knows that they become simple in the semiclassical limit. Therefore, by applying the spectral theorem, we get the WKB approximation of the corresponding eigenfunctions from Theorem \ref{WKB-general}.
\end{remark}

\begin{remark}\label{rem:A20}
When $A_{2}$ is not zero, it appears that the dimensional reduction is prevented by the oscillations of the eigenfunctions of the model operator $\mathcal{M}_{x,\xi}$. The problem already appears in the case $t\in\R$: we can gauge out $A_{2}$ at the price to replace $A_{1}$ by $A_{1}+h\nabla_{s}\varphi(s,t)$ which is $h$ dependent. As a consequence of our analysis, we can check that the spectrum associated with the potential $(A_{1}+h\nabla_{s}\varphi,0)$ is shifted by a factor $\mathcal{O}(h)$ compared to the one associated with $(A_{1},0)$. In dimension one for $t$, we can even prove with our method (and a change of gauge) that the phase $\Phi$ in the WKB expansion is $(s,t)$-dependent.
\end{remark}

\paragraph{Strategy of the proof of Theorem \ref{WKB-general}.}
 The new Ansatz considered here is given by a partial WKB expansion with respect to the variable $s$. Under some analyticity assumptions, the effective eikonal equation will be solved thanks to the classical stable manifold theorem and analytic extensions of the eigenpairs of the \enquote{model} operators. The corresponding effective transport equation will be obtained as the Fredholm condition of an operator valued transport equation jointly with the Feynman-Hellmann formulas. Theorem \ref{WKB-general} will be proved in Section \ref{sec:WKB}.

\subsubsection{Generalized Montgomery operators: towards the magnetic tunnel effect}
Let us introduce a family of magnetic Laplacians in dimension two which is related to \cite{HelPer10} and the more recent result by Fournais and Persson \cite{FouPer13}.
For $k\in\N\setminus\{0\}$, we consider the so-called generalized Montgomery operator on $\sL^2(\R^2,\dx \mathsf{s}\dx \mathsf{t})$:
$$\mathcal{L}^{[k],\gM}_{\hbar}=\hbar^2 D_{\mathsf{t}}^2+\left(\hbar D_{\mathsf{s}}-\gamma(\mathsf{s})\frac{\mathsf{t}^{k+1}}{k+1}\right)^2,$$
where $\gamma$ is analytic and does not vanish. The corresponding magnetic field is 
\[B^{[k]}(\mathsf{s}, \mathsf{t})=\gamma(\mathsf{s})\mathsf{t}^k\,.\]
We call $\lambda_{n,\hbar}^{[k],\gM}$ the $n$-th eigenvalue (if exists) of this operator. In order to stick to the previous analysis, we start by the following naive but fundamental rescaling
\begin{equation}\label{eq:rescaling}
\mathsf{s}=s,\qquad\mathsf{t}=\hbar^{\frac{1}{k+2}}t.
\end{equation}
The operator becomes
$$\hbar^{\frac{2k+2}{k+2}}\left(D_{t}^2+\left(\hbar^{\frac{1}{k+2}}D_{s}-\gamma(s)\frac{t^{k+1}}{k+1}\right)^2\right).$$
The investigation is then reduced to the one of
\begin{equation} \label{defgM}
\mathfrak{L}_{h}^{[k]}=D_{t}^2+\left(hD_{s}-\gamma(s)\frac{t^{k+1}}{k+1}\right)^2,
\end{equation}
with $h=\hbar^{\frac{1}{k+2}}$. 

Under some assumptions, the operator $\mathfrak{L}_{h}^{[k]}$ is a particular case of the previous theory. We will see in Section \ref{S:verify-assumptions} (Proposition \ref{verify-assumptions}) that it satisfies Assumptions \ref{hyp-gen}, \ref{hyp-gen'} and \ref{confining}.
As a consequence, we could directly apply Theorem \ref{WKB-general}. But, at least in the case when the oscillations of the function $\gamma$ are small enough, we can prove that the first eigenfunctions are \textit{globally} in the WKB form.
\begin{proposition}\label{WKB-model}\ 
\begin{itemize}
\item If $\gamma$ is a polynomial function and admits a unique minimum $\gamma_{0}>0$ at $s_{0}=0$ which is non degenerate, then Theorem~\ref{WKB-general} applies. 
\item If $\gamma$ is an analytic function and if $\left\|1-\frac{\gamma_{0}}{\gamma}\right\|_{\infty}$ is small enough, then the conclusion of Theorem~\ref{WKB-general} is valid and we can take $x_{0}=s_{0}$, $\Vc=\R$.
\end{itemize}
\end{proposition}
\begin{remark}For the second point of Proposition~\ref{WKB-model}, 
the simplicity of the eigenvalues is established in \cite{Ray11b, DomRay12} for $k=0,1$ whereas slight adaptations have to be done to deal with the case $k\geq 2$.
\end{remark}

As a direct reformulation and using the rescaling \eqref{eq:rescaling}, we get
 the following result in the original variables.

\begin{corollary}\label{cor.convMont}
The $n$-th eigenvalue of $\mathfrak{L}^{[k],\gM}_{\hbar}$ and the corresponding WKB solution on $\R^2_{\mathsf{s},\mathsf{t}}$ are given, as $\hbar\to 0$, by
$$\lambda^{[k],\gM}_{n,\hbar} = \hbar^{\frac{2k+2}{k+2}} \lambda_{n}(\hbar^{\frac{j}{k+2}}) {\sim} \hbar^{\frac{2k+2}{k+2}}\sum_{j\geq 0} \lambda_{n,j} \hbar^{\frac{j}{k+2}},$$
and
$$u^{[k],\gM}_{n,\hbar}(\mathsf{s},\mathsf{t}){\sim} \an_{n}(\mathsf{s},\hbar^{-\frac{1}{k+2}}\mathsf{t} ; \hbar^{\frac{1}{k+2}})\re^{-\Phi(\mathsf{s})/\hbar^{\frac{1}{k+2}}}$$
where $\an_{n}$ and $\lambda_{n}$ are given by Theorem \ref{WKB-general}.
\end{corollary}

In the perspective of the analysis of the magnetic tunnelling, we will now suppose that $\gamma$, instead of having a unique non degenerate minimum, satisfies the following assumption of double well type.
\begin{proposition}\label{Agmon-puits-double}
Let us assume that the function $\gamma$ is even and has two non degenerate minima at $s_{-}<0$ and $s_{+}=-s_{-}>0$.
Let us fix $\delta\in(0, s_{+})$ and let
$$\displaystyle{\Z(s)=\chi_{\delta,-}(s)\left|\int_{s_{-}}^s \tilde\chi(s')\sqrt{\gamma(s')^{\frac{2}{k+2}}-\gamma_{0}^{\frac{2}{k+2}}}\dx s'\right|+\chi_{\delta,+}(s)\left|\int_{s_{+}}^s \tilde\chi(s')\sqrt{\gamma(s')^{\frac{2}{k+2}}-\gamma_{0}^{\frac{2}{k+2}}}\dx s'\right|},$$
where $0\leq\tilde\chi\leq 1$ is a smooth cutoff function whose support contains $s_{-}$ and $s_{+}$ and where $\chi_{\delta,-}$ and $\chi_{\delta,+}$ are smooth cutoff functions such that
$$\chi_{\delta,-}(s)=\begin{cases}
1& \mbox{for } s\leq \frac{\delta}{2},\\
0& \mbox{for } s\geq \delta,
\end{cases}
\qquad\mbox{ and }\qquad
\chi_{\delta,+}(s)=\begin{cases}
1& \mbox{for } s\geq -\frac{\delta}{2},\\
0& \mbox{for } s\leq -\delta.
\end{cases}$$
Let us consider $C_{0}>0$. There exist $\eps_{0}>0$, $C>0$ and $h_{0}>0$ such that for all eigenpairs $(\lambda,\psi)$ of $\mathfrak{L}^{[k]}_{h}$ satisfying $\lambda\leq \mu_{0}+C_{0}h$ we have, for all $h\in(0,h_{0})$,
$$\|\re^{\eps_{0}\Z/h}\psi\|\leq C\|\psi\|
\qquad\mbox{and}\qquad
\mathfrak{Q}_{h}^{[k]}(\re^{\eps_{0}\Z/h}\psi)\leq  C\|\psi\|^2.$$
\end{proposition}

We can now give a rough upper bound for the tunnel effect for the rescaled Montgomery models.
For that purpose, let us fix $\delta\in(0,s_{+})$ and define the two symmetric model wells. We assume that $\gamma$ is even and has two non degenerate minima at $s_{-}<0$ and $s_{+}=-s_{-}>0$.

We consider the Dirichlet realizations on $\sL^2((-\infty,s_{-}+\delta)\times\R)$ and $\sL^2((s_{+}-\delta,+\infty)\times\R)$ 
of $D_{t}^2+\left(hD_{s}-\gamma(s)\frac{t^{k+1}}{k+1}\right)^2$ respectively denoted by $\mathfrak{H}_{h,-}^\Dir$ and $\mathfrak{H}_{h,+}^\Dir$. These operators are isospectral by symmetry. We want to compare the spectrum of $\mathfrak{L}_{h}^{[k]}$ with the one of the direct sum $\mathfrak{H}_{h}=\mathfrak{H}_{h,-}^\Dir\oplus\mathfrak{H}_{h,+}^\Dir$.
\begin{theorem}\label{tunnelling}
Let us consider $C_{0}>0$. There exist $c>0$, $C>0$, $h_{0}>0$ such that for all $\mu\in\spe\left(\mathfrak{H}_{h}\right)$ and $\lambda\in\spe\left(\mathfrak{L}_{h}^{[k]}\right)$ with $\mu,\lambda\leq \mu_{0}+C_{0}h$, we have, for all $h\in(0,h_{0})$,
$$\range\left(\mathds{1}_{[\mu-C\re^{-c/h}, \mu+C\re^{-c/h}]}\big(\mathfrak{L}_{h}^{[k]}\big)\right)= 2$$
and
$$\dist\left(\lambda,\spe\left(\mathfrak{H}_{h}\right)\right)\leq C\re^{-c/h}.$$
\end{theorem}
The following corollary is a direct consequence of the previous theorem.
\begin{corollary}\label{rem.gap}
In terms of   operator $\mathcal{L}^{[k],\gM}_{\hbar}$  the gap between pairs of eigenvalues is given  by
$$\lambda^{[k],\gM}_{2n,\hbar }-\lambda^{[k],\gM}_{2n-1,\hbar}=
\mathcal{O}\Big(\re^{-c/\hbar^{\frac{1}{k+2}}}\Big),$$
for $\hbar$ small enough depending on $n$.
\end{corollary}

\paragraph{Strategy of the analysis of $\mathfrak{L}_{h}^{[k]}$}
In the simple well case the study of $\mathfrak{L}_{h}^{[k]}$ follows the same lines as for Theorem \ref{WKB-general} jointly with a normal form argument inspired by \cite{Ray11b, DomRay12, RVN13} which permits simultaneously to make the WKB expansions global and to get quasi-optimal Agmon estimates. Concerning the double well case, in order to get the tunneling effect, we can follow the classical procedure based on the spectral theorem and the previous Agmon estimates.

\subsection{WKB constructions: influence of the geometry}\label{SS:geom}
This section is devoted to the fully semiclassical magnetic Laplacian $(\hbar D_{\mathsf{x}}+\bA)^2$ on $\sL^2(\R^2)$. We now investigate three kinds of models for which the geometry is more intricate and for which our theorems do not directly apply. Nevertheless, our WKB strategy is robust enough and still effective: we are able to exhibit WKB expansions. Note that all the forthcoming situations have already been studied in the literature but never with the accuracy of the WKB point of view. The geometric perturbations at stake are: vanishing magnetic fields on curves or possibly singular boundary.
\subsubsection{Vanishing magnetic fields... or not}\label{intro-van}
We will work under the following assumption on $\bB$ as in \cite{HelKo09} and \cite{DomRay12} which concerns its vanishing order $k\geq 1$ (whereas the limit case $k=0$ will be described afterwards).
\begin{assumption}
We work in the two dimensional case. The zero locus of $\bB$ is a smooth closed non empty curve $\Gamma$:
$$\Gamma=\{\bB(\mathsf{x})=0\},$$
and $\bB$ vanishes exactly at the order $k\geq 1$ on $\Gamma$. Moreover we assume that the $k$-th normal derivative of $\bB$ admits a non degenerate minimum on $\Gamma$ at $\mathsf{x}_{0}$.
\end{assumption}
\begin{remark}
Here we work in dimension two, but there is no doubt that we could adapt the presentation, modulo a few technicalities, to cover the case of magnetic fields vanishing at a given order on hypersurfaces as in \cite{HelKo09}.
\end{remark}
In Section \ref{S:van} we will construct, in a neighborhood of $\mathsf{x}_{0}$ and in normal coordinates, WKB expansions which come within the study of the generalized Montgomery operator $\mathcal{L}^{[k],\gM}_{\hbar}$. Under the additional assumption that the minimum is uniquely reached at $\mathsf{x}_{0}$, the splitting between the lowest eigenvalues has been established in \cite{DomRay12} (for the case $k=1$ and the proof is completely similar for $k\geq 2$) so that these WKB expansions are local (in the sense of Theorem \ref{WKB-general}) approximations of the true eigenfunctions.

These considerations can be extended to the case $k=0$. We consider here a magnetic field which does not vanish on $\overline{\Omega}$. The curve $\Gamma$ represents the boundary of an enclosed open set $\Omega$ carrying a magnetic Neumann condition. In other words, we can perform a WKB construction, for the Neumann realization on $\Omega$ of $(\hbar D_{\mathsf{x}}+\bA)^2$, near each $\mathsf{x}_{0}$ where $\bB_{|\partial\Omega}$ is not degenerately minimal. Assuming moreover that
$$\Theta_{0}\min_{\partial\Omega} \bB<\min_{\Omega}\bB,$$
and that  $\bB_{|\partial\Omega}$ admits a unique and non degenerate minimum, we can get, by using the spectral splitting proved in \cite{Ray11b}, the local WKB approximation of the lowest eigenfunctions.\newline
All the scaling properties, stated in normal coordinates, will be addressed in Section \ref{S:van}.

\subsubsection{Varying edge}\label{intro-edge}
The strategy of this paper can also deal with more singular situations in dimension three. Such a situation is described in the paper \cite{PoRay12} where the semiclassical analysis is done when the boundary of the domain contains a varying edge. We propose to perform the WKB constructions for a simplified version of the operator introduced there. We are interested in the operator defined on $\sL^2(\mathcal{W}_{\alpha},\dx s\dx t \dx z)$ and with Neumann conditions
$$ \mathcal{L}^{\PR}_{\hbar}=\hbar^2D_{t}^2+\hbar^2D_{z}^2+(\hbar D_{s}-t)^2, $$
where
$$\mathcal{W}_{\alpha}=\left\{(s,t,z)\in\R^3, |z|\leq\cT(s)t \right\},$$
with $\cT(s)=\tan\left(\frac{\alpha(s)}{2}\right)$ and where $\alpha : \R\to \R$ is an analytic function which represents the (varying) opening of the wedge $\mathcal{W}_{\alpha}$. We will work under the following assumption.
\begin{assumption}\label{alpha-max}
The function $s\mapsto \alpha(s)$ admits a unique and non degenerate maximum $\alpha_{0}$ at $s=0$.
\end{assumption}
In Section \ref{S:edge}, we will provide local (near the point of the edge giving the maximal aperture) WKB expansions of the lowest eigenfunctions.

\subsubsection{Curvature induced magnetic bound states}\label{intro-c}
As we have seen, in many situations the spectral splitting appears in the second term of the asymptotic expansion of the eigenvalues. It turns out that we can also deal with more degenerate situations. The next lines are motivated by the initial paper \cite{HelMo01} whose main result is recalled in \eqref{HM-result}. Their fundamental result establishes that a smooth Neumann boundary can trap the lowest eigenfunctions near the points of maximal curvature. These considerations are generalized in \cite[Theorem 1.1]{FouHel06a} where the following complete asymptotic expansion of the eigenvalues is proved
\begin{equation}\label{splitting-FH}
\lambda^{\FH}_{n,\hbar}=\Theta_{0}\hbar-C_{1}\kappa_{\max}\hbar^{3/2}
+(2n-1)C_{1}\Theta_{0}^{1/4}\sqrt{\frac{3k_{2}}{2}} \hbar^{7/4}+o(\hbar^{7/4}),
\end{equation}
where $k_{2}=-\kappa''(0)$. In our paper, as in \cite{FouHel06a}, we will consider the magnetic Neumann Laplacian on a smooth domain $\Omega$ such that the algebraic curvature $\kappa$ satisfies the following assumption.
\begin{assumption}\label{kappa-max}
The function $\kappa$ is smooth and admits a unique and non degenerate maximum at $0$.
\end{assumption}
We will prove that the lowest eigenfunctions are approximated by local WKB expansions which can be made global when for instance $\partial\Omega$ is the graph of a smooth function. In particular we will recover the term $C_{1}\Theta_{0}^{1/4}\sqrt{\frac{3k_{2}}{2}}$ by a method different from the one in \cite{FouHel06a}.

\subsection{Organization of the paper}
The paper is organized as follows. Section \ref{Sec.2} is devoted to models with simple magnetic wells and to the proof of Theorem \ref{theorem-simple-well}. Section \ref{sec:WKB} is concerned with the proof of Theorem \ref{WKB-general}. In Section \ref{Sec:Mont} we establish that the generalized Montgomery operators satisfy the assumptions of Theorem \ref{WKB-general}, we prove that  the WKB expansions are global (Proposition \ref{WKB-model}) and we give an upper bound of the tunnel effect (proof of Theorem \ref{tunnelling}).
These theoretical results are illustrated by numerical simulations. Section \ref{Sec:geom} deals with the geometrical examples introduced in Section \ref{SS:geom}.

\section{Simple magnetic wells}\label{Sec.2}
This section is devoted to the proof of Theorem \ref{theorem-simple-well}. In order to perform the investigation we use the following rescaling
\begin{equation}\label{rescaling}
s=x_{0}+ h^{1/2}\sigma,\qquad t=\tau,
\end{equation}
and a gauge transform $\re^{i\xi_{0} \sigma/h^{1/2}}$, so that $\mathfrak{L}_{h}$ becomes
\begin{equation}\label{Lch}
\mathcal{L}_{h}=(D_{\tau}+A_{2}(x_{0}+h^{1/2}\sigma,\tau))^2+(\xi_{0}+h^{1/2}D_{\sigma}+A_{1}(x_{0}+h^{1/2}\sigma,\tau))^2.
\end{equation}
The corresponding quadratic form is denoted by $\mathcal{Q}_{h}$.

\subsection{Formal series and general Feynman-Hellmann formulas}
Let us start by proving the following proposition.

\begin{proposition}\label{quasimodes}
Under Assumption \ref{hyp-gen}, we assume furthermore that $A_{1}$ and $A_{2}$ are polynomials. For all $n\geq 1$,  there exist $C>0$ and $h_{0}>0$ such that, for $h\in(0,h_{0})$,
$$\dist\left(\lambda_{n,0}+\lambda_{n,1}h,\spe(\mathcal{L}_{h})\right)\leq Ch^{3/2},$$
where $\lambda_{n,0}=\mu_{0}$ and $\lambda_{n,1}$ is the $n$-th eigenvalue of $\frac{1}{2}\Hess\,\mun(x_{0},\xi_{0})(\sigma, D_{\sigma})$.
\end{proposition}
We will need the so-called Feynman-Hellmann formulas: 
\begin{proposition}\label{FH}
Let $\eta$ and $\theta$ denote one of the $x_{j}$ or $\xi_{k}$.  Then we have
\begin{equation}\label{FH1}
(\mathcal{M}_{x,\xi}-\mun(x,\xi))(\dr_{\eta} u)_{x,\xi}=(\dr_{\eta}\mun(x,\xi)-\dr_{\eta}\mathcal{M}_{x,\xi}) u_{x,\xi}
\end{equation}
and
\begin{multline}\label{FH2}
(\mathcal{M}_{x_{0},\xi_{0}}-\mu_0)(\dr_{\eta}\dr_{\theta}u)_{x_{0},\xi_{0}}\\
=\dr_{\eta}\dr_{\theta}\mun(x_{0},\xi_{0})u_{x_{0},\xi_{0}}- \dr_{\eta}\mathcal{M}_{x_{0},\xi_{0}}(\dr_{\theta} u)_{x_{0},\xi_{0}}
-\dr_{\theta}\mathcal{M}_{x_{0},\xi_{0}}(\dr_{\eta} u)_{x_{0},\xi_{0}}
-\dr_{\eta}\dr_{\theta}\mathcal{M}_{x_{0},\xi_{0}}u_{x_{0},\xi_{0}},
\end{multline}
where $(\dr_{\eta}u)_{x_{0},\xi_{0}}$ denotes $(\dr_{\eta}u_{x,\xi})_{|(x,\xi)=(x_{0},\xi_{0})}$ and similarly for the other derivatives of $u_{x,\xi}$.
In a neighborhood of $(x_{0},\xi_{0})$ in $\R^m\times\C^m$, we have,
\begin{equation}\label{FH2bis}
\dr_{\eta}\mun(x,\xi)=\int_{\R^n} (\dr_{\eta}{\mathcal{M}_{x,\xi}}\, u_{x,\xi})(\tau) \overline{u_{x,\overline{\xi}}(\tau)}\dx\tau.
\end{equation}
\end{proposition}
\begin{proof}
Feynman-Hellmann formulas \eqref{FH1}--\eqref{FH2bis} are obtained by taking the derivative of the eigenvalue equation
\begin{equation}
\mathcal{M}_{x,\xi}u_{x,\xi}=\mun(x,\xi)u_{x,\xi},
\end{equation}
with respect to $x_{j}$ and $\xi_{k}$.
\end{proof}

\begin{proofof}{Proposition \ref{quasimodes}}
Since $A_{1}$ and $A_{2}$ are polynomials, we can write, for some $M\in\N$,
$$\mathcal{L}_{h}=\sum_{j=0}^M h^{j/2}\mathcal{L}_{j},$$
with
\begin{equation}\label{eq.Lj}
\begin{gathered}
\mathcal{L}_{0}=\mathcal{M}_{x_{0},\xi_{0}},\qquad
\mathcal{L}_{1}=\sum_{k=1}^m   (\dr_{x_{k}}\mathcal{M})_{x_{0},\xi_{0}}\sigma_{k}+\sum_{k=1}^m (\dr_{\xi_{k}}\mathcal{M})_{x_{0},\xi_{0}} D_{\sigma_{k}},\\
\mathcal{L}_{2}=\frac{1}{2}\sum_{k,j=1}^m\Big( (\dr_{x_{j}}\dr_{x_{k}}\mathcal{M})_{x_{0},\xi_{0}} \sigma_{j}\sigma_{k}+(\dr_{\xi_{j}}\dr_{\xi_{k}}\mathcal{M})_{x_{0},\xi_{0}} D_{\sigma_{j}}D_{\sigma_{k}}+(\dr_{\xi_{j}}\dr_{x_{k}}\mathcal{M})_{x_{0},\xi_{0}} D_{\sigma_{j}}\sigma_{k}\\
\hfill+(\dr_{x_{k}}\dr_{\xi_{j}}\mathcal{M})_{x_{0},\xi_{0}}\sigma_{k}D_{\sigma_{j}}\Big).
\end{gathered}
\end{equation}
We look for quasimodes in the form
$$\psi=\sum_{j= 0}^2 h^{j/2}\psi_{j}\qquad\mbox{ and }\qquad
\la = \sum_{j = 0}^2 h^{j/2}\la_{j},$$
so that they solve in the sense of formal series
$$\mathcal{L}_{h}\psi = \la \psi + {\Oc}(h^{3/2}).$$ 
Let us now deal with each power of $h$.
\paragraph{Terms of order $h^0$.} By collecting the terms of order $h^0$, we get the equation
$$\mathcal{M}_{x_{0},\xi_{0}}\psi_{0}=\la_{0}\psi_{0}.$$
This leads to take
$$\la_{0}=\mu_{0}\qquad\mbox{ and }\qquad\psi_{0}(\sigma,\tau)=f_{0}(\sigma)u_{0}(\tau),$$
where $u_{0}=u_{x_{0},\xi_{0}}$ and $f_{0}$ is a function to be determined in the Schwartz class.
\paragraph{Terms of order $h^{1/2}$.}
By collecting the terms of order $h^{1/2}$, we find
$$(\mathcal{M}_{x_{0},\xi_{0}}-\mun(x_{0},\xi_{0}))\psi_{1}=(\la_{1}-\mathcal{L}_{1})\psi_{0}.$$
By using \eqref{FH1} and the Fredholm alternative (applied for $\sigma$ fixed) we get $\la_{1}=0$ and
\begin{equation}\label{psi1}
\psi_{1}(\sigma,\tau)= \sum_{k=1}^m (\dr_{x_{k}}u)_{x_{0},\xi_{0}}(\tau)\, \sigma_{k} f_{0}(\sigma)+\sum_{k=1}^m (\dr_{\xi_{k}}u)_{x_{0},\xi_{0}}(\tau)\, D_{\sigma_{k}} f_{0}(\sigma) +f_{1}(\sigma) u_{0}(\tau),
\end{equation}
where $f_{1}$ is a function to be determined in the Schwartz class.
\paragraph{Terms of order $h$.}
The equation reads
$$(\mathcal{M}_{x_{0},\xi_{0}}-\mun(x_{0},\xi_{0}))\psi_{2}=(\la_{2}-\mathcal{L}_{2})\psi_{0}-\mathcal{L}_{1}\psi_{1}.$$
The Fredholm condition gives
\begin{equation}\label{Fredholm}
\langle\mathcal{L}_{2}\psi_{0}+\mathcal{L}_{1}\psi_{1},  u_{0}\rangle_{\sL^2(\R^n,\dx \tau)}=\la_{2}f_{0}.
\end{equation}
Let us examine each term which appears  when computing the l.h.s. and recall that Proposition \ref{FH} holds (especially the Fredholm condition of \eqref{FH2}).\\
The coefficient in front of $\sigma_{j}\sigma_{k} f_{0}$ is
\begin{multline*}
\langle (\dr_{x_{j}}\mathcal{M})_{x_{0},\xi_{0}}(\dr_{x_{k}}u)_{x_{0},\xi_{0}},u_{0} \rangle_{\sL^2(\R^n,\dx \tau)}
 + \langle (\dr_{x_{k}}\mathcal{M})_{x_{0},\xi_{0}}(\dr_{x_{j}}u)_{x_{0},\xi_{0}},u_{0} \rangle_{\sL^2(\R^n,\dx \tau)}\\
+\langle (\dr_{x_{j}}\dr_{x_{k}}\mathcal{M})_{x_{0},\xi_{0}} u_{0},u_{0} \rangle_{\sL^2(\R^n,\dx \tau)}
=\dr_{x_{j}}\dr_{x_{k}}\mun(x_{0},\xi_{0}).
\end{multline*}
The coefficient in front of $D_{\sigma_{j}}D_{\sigma_{k}}$ is
\begin{multline*}
\langle(\dr_{\xi_{j}}\mathcal{M})_{x_{0},\xi_{0}}(\dr_{\xi_{k}}u)_{x_{0},\xi_{0}},u_{0}\rangle_{\sL^2(\R^n,\dx \tau)}
+\langle(\dr_{\xi_{k}}\mathcal{M})_{x_{0},\xi_{0}}(\dr_{\xi_{j}}u)_{x_{0},\xi_{0}},u_{0}\rangle_{\sL^2(\R^n,\dx \tau)}\\
+\langle (\dr_{\xi_{j}}\dr_{\xi_{k}}\mathcal{M})_{x_{0},\xi_{0}}u_{0},u_{0}\rangle_{\sL^2(\R^n,\dx \tau)}
=\dr_{\xi_{j}}\dr_{\xi_{k}}\mun(x_{0},\xi_{0}).
\end{multline*}
To deal with the coefficient in front of $D_{\sigma_{j}}\sigma_{k}+\sigma_{k}D_{\sigma_{j}}$, we use the formula
\begin{multline*}
\langle(\dr_{\xi_{j}}\mathcal{M})_{x_{0},\xi_{0}}(\dr_{x_{k}}u)_{x_{0},\xi_{0}},u_{0}\rangle_{\sL^2(\R^n,\dx \tau)}
+\langle(\dr_{x_{k}}\mathcal{M})_{x_{0},\xi_{0}}(\dr_{\xi_{j}}u)_{x_{0},\xi_{0}},u_{0}\rangle_{\sL^2(\R^n,\dx \tau)}\\
+\langle (\dr_{\xi_{j}}\dr_{x_{k}}\mathcal{M})_{x_{0},\xi_{0}}u_{0},u_{0}\rangle_{\sL^2(\R^n,\dx \tau)}
+\langle(\dr_{\xi_{k}}\mathcal{M})_{x_{0},\xi_{0}}(\dr_{x_{j}}u)_{x_{0},\xi_{0}},u_{0}\rangle_{\sL^2(\R^n,\dx \tau)}\\
+\langle(\dr_{x_{j}}\mathcal{M})_{x_{0},\xi_{0}}(\dr_{\xi_{k}}u)_{x_{0},\xi_{0}},u_{0}\rangle_{\sL^2(\R^n,\dx \tau)}
+\langle (\dr_{\xi_{k}}\dr_{x_{j}}\mathcal{M})_{x_{0},\xi_{0}}u_{0},u_{0}\rangle_{\sL^2(\R^n,\dx \tau)}
\\
=\dr_{x_{j}}\dr_{\xi_{k}}\mun(x_{0},\xi_{0})+\dr_{x_{k}}\dr_{\xi_{j}}\mun(x_{0},\xi_{0}).
\end{multline*}
Therefore the Fredholm condition \eqref{Fredholm} becomes
$$\frac{1}{2}\sum_{j,k=1}^m \left(\dr_{x_{j}}\dr_{x_{k}}\mun\sigma_{j}\sigma_{k}+\dr_{\xi_{j}}\dr_{\xi_{k}}\mun D_{\sigma_{j}}D_{\sigma_{k}}+\dr_{\xi_{j}}\dr_{x_{k}}\mun D_{\sigma_{j}}\sigma_{k}+\dr_{x_{k}}\dr_{\xi_{j}}\mun\sigma_{k}D_{\sigma_{j}} \right)f_{0}=\la_{2}f_{0},$$
where for shortness, we write $\partial_{\eta}\partial_{\theta}\mun :=(\partial_{\eta}\partial_{\theta}\mun)(x_{0},\xi_{0})$. In other words, we have
$$\tfrac{1}{2}\Hess\, \mun(x_{0},\xi_{0}) (\sigma,D_{\sigma})f_{0}=\la_{2} f_{0}.$$
We take $\la_{2}=\la_{n,2}$ the $n$-th eigenvalue of $\frac{1}{2}\Hess\, \mun(x_{0},\xi_{0}) (\sigma,D_{\sigma})$ and we choose $f_{0}$ a corresponding normalized eigenfunction. We take $f_{1}=0$.
The spectral theorem completes the proof of Proposition~\ref{quasimodes} (with $\lambda_{n,1}=\la_{n,2}$).
\end{proofof}
\begin{remark}
This construction can be continued at any order by using the same kind of double scale procedure which can be found  in \cite{Ray11b, Ray12, DomRay12} (see also \cite{FouHel06a} more in the Grushin spirit or \cite{BDPR11} in an electric case). The odd terms in the eigenvalues expansion may easily be cancelled thanks to the parity of the harmonic oscillator.
\end{remark}

\noindent We deduce from Propositions \ref{essential} and \ref{quasimodes}:
\begin{corollary}\label{rub}
For all $n\geq 1$ there exist $h_{0}>0$ and $C>0$ such that for all $h\in(0,h_{0})$ the $n$-th eigenvalue of $\mathfrak{L}_{h}$ exists and satisfies:
$$\lambda_{n}(h)\leq \mu_{0}+Ch.$$
\end{corollary}

\subsection{Semiclassical Agmon-Persson estimates}\label{micro}
This section is devoted to the rough localization and microlocalization estimates satisfied by the eigenfunctions and resulting from Assumptions  \ref{hyp-gen} and \ref{confining} and Corollary \ref{rub}.
\begin{proposition}\label{Agmon-tau}
Let $C_{0}>0$. There exist $h_{0},C, \eps_{0}>0$ such that for all eigenpairs $(\lambda,\psi)$ of $\mathfrak{L}_{h}$ with $\lambda\leq\mu_0+C_{0}h$, we have
$$\left\|\re^{\eps_{0}|t|}\psi\right\|^2\leq C\|\psi\|^2,\qquad\mathfrak{Q}_{h}\left(\re^{\eps_{0}|t|}\psi\right)\leq C\|\psi\|^2.$$
\end{proposition}

\begin{proof}
The proof is standard but we recall it for completeness. We consider a smooth cutoff function $\chi_{1}$ supported in a fixed neighborhood of $0$ and, for $\ell\geq 1$, we introduce $\chi_{\ell}(t)=\chi_{1}(\ell^{-1}t)$. We let $\Phi_{\ell}(t)=\eps_{0}\chi_{\ell}(t)|t|$ and we write the Agmon identity (see \cite{Agmon82, Agmon85})
$$\mathfrak{Q}_{h}(\re^{\Phi_{\ell}}\psi)=\lambda\|\re^{\Phi_{\ell}}\psi\|^2+\|\,|\nabla\Phi_{\ell}| \re^{\Phi_{\ell}}\psi\|^2.$$
There exists $C>0$ such that for all $\ell\geq 1$ we have
$$\|\,|\nabla\Phi_{\ell}| \re^{\Phi_{\ell}}\psi\|^2\leq C\eps_{0}^2\|\re^{\Phi_{\ell}}\psi\|^2.$$
We infer that
$$\mathfrak{Q}_{h}(\re^{\Phi_{\ell}}\psi)\leq(\mu_{0}+C_{0}h+C\eps_{0}^2)\|\re^{\Phi_{\ell}}\psi\|^2.$$
For $R>0$, we introduce a partition of unity $(\chi_{1,R},\chi_{2,R})$ in $t$-variables such that
$$\chi_{1,R}^2(t)+\chi_{2,R}^2(t)=1, \quad
|\nabla\chi_{1,R}|^2+|\nabla\chi_{2,R}|^2\leq CR^{-2}\quad\mbox{ and }\quad
\supp\chi_{2,R}\cap\bbb(0,R)=\emptyset.$$
With the so-called IMS formula (see \cite[Chapter 3]{CFKS87}), we deduce
$$\mathfrak{Q}_{h}(\chi_{1,R}\re^{\Phi_{\ell}}\psi)+\mathfrak{Q}_{h}(\chi_{2,R}\re^{\Phi_{\ell}}\psi)-CR^{-2}\|\re^{\Phi_{\ell}}\psi\|^2\leq (\mu_{0}+C_{0}h+C\eps_{0}^2)\|\re^{\Phi_{\ell}}\psi\|^2.$$
Since $\re^{\Phi_{\ell}}$ is bounded on the support of $\chi_{1,R}$, we get the existence of $C, \tilde C>0$ such that for all $\ell\geq 1$ and $h\in(0,1)$:
$$\mathfrak{Q}_{h}(\chi_{2,R}\re^{\Phi_{\ell}}\psi)-(\mu_{0}+C_{0}h+C\eps_{0}^2+CR^{-2})\|\chi_{2,R}\re^{\Phi_{\ell}}\psi\|^2\leq \tilde C\|\psi\|^2.$$
By using Assumption \ref{confining} and Remark \ref{confining2}, there exist $R_{0}>0$ and $h_{0}>0$ such that for all $R\geq R_{0}$ and $h\in(0,h_{0})$ we have
$$\mathfrak{Q}_{h}(\chi_{2,R}\re^{\Phi_{\ell}}\psi)\geq \mu_{0}^*\|\chi_{2,R}\re^{\Phi_{\ell}}\psi\|^2.$$
We infer the existence of $c>0$ such that for $h\in(0,h_{0})$
$$c\|\chi_{2,R}\re^{\Phi_{\ell}}\psi\|^2\leq \tilde C\|\psi\|^2.$$
Then there exist $C>0$ and $h_{0}>0$ such that for all $\ell\geq 1$ and $h\in(0,h_{0})$
$$\|\re^{\Phi_{\ell}}\psi\|^2\leq C\|\psi\|^2.$$
It remains to consider the limit $\ell\to+\infty$ and to use the Fatou lemma and the conclusion follows.
\end{proof}
\begin{proposition}\label{Agmon-s}
Let $C_{0}>0$. There exist $h_{0},C, \eps_{0}>0$ such that for all eigenpairs $(\lambda,\psi)$ of $\mathfrak{L}_{h}$ with $\lambda\leq\mu_0+C_{0}h$, we have
$$\left\|\re^{\eps_{0}|s|}\psi\right\|^2\leq C\|\psi\|^2,\qquad \mathfrak{Q}_{h}\left(\re^{\eps_{0}|s|}\psi\right)\leq C\|\psi\|^2.$$
\end{proposition}
\begin{proof}
The proof is the same as that of Proposition \ref{Agmon-tau} with $\Phi(s)=\eps_{0}\chi_{\ell}(s)|s|$.
\end{proof}
We get the following.
\begin{corollary}\label{poly-bound}
Let $C_{0}>0$ and $k,l,\jj\in\N$. There exist $h_{0},C, \eps_{0}>0$ such that for all eigenpairs $(\lambda,\psi)$ of $\mathfrak{L}_{h}$ with $\lambda\leq\mu_0+C_{0}h$  and all $h\in(0,h_{0})$, we have
\begin{eqnarray*}
\|t^k s^l \psi\|\leq C\|\psi\|,&& \mathfrak{Q}_{h}(t^k s^l \psi)\leq C\|\psi\|^2,\\
\|(D_{\tau})^\jj s^l t^k \psi \|\leq C\|\psi\|^2,&& \|(hD_{s})^\jj s^l t^k \psi \|\leq C\|\psi\|^2.
\end{eqnarray*}
\end{corollary}
\begin{proof}
For $\jj=1$, this is an immediate consequence of Propositions \ref{Agmon-tau} and \ref{Agmon-s}. Taking successive derivatives of the eigenvalue equation $\mathfrak{L}_{h}\psi=\lambda\psi$ we deduce the result for $\jj\geq 2$.
\end{proof}
For another purpose, we will need the following localization result which is again a consequence of Propositions \ref{Agmon-tau} and \ref{Agmon-s}.
\begin{proposition}\label{loc-space}
Let $k\in\N$. Let $\eta>0$ and $\chi$ a smooth cutoff function defined on $\R_{+}$ and being zero in a neighborhood of $0$.  There exists $h_{0}>0$ such that for all eigenpairs $(\lambda,\psi)$ of $\mathfrak{L}_{h}$ with $\lambda\leq\mu_0+C_{0}h$ and all $h\in(0,h_{0})$, we have
$$\|\chi(h^{\eta}|s|)\psi\|_{B^k(\R^{m+n})}\leq \Oc(h^{\infty})\|\psi\|, \qquad \|\chi(h^{\eta}|t|)\psi\|_{B^k(\R^{m+n})}\leq \Oc(h^{\infty})\|\psi\|,$$
where $\|\cdot\|_{B^k(\R^{n+m})}$ is the standard norm on
$$B^k(\R^{m+n})=\{\psi\in \sL^2(\R^{m+n})|  y_{j}^q \dr_{y_{l}}^p \psi\in \sL^2(\R^{n+m}),\forall j,l\in\{1,\ldots,m+n\},\, p+q\leq k\}.$$
\end{proposition}
By using a rough pseudo-differential calculus jointly with the space localization of Proposition \ref{loc-space} and standard elliptic estimates, we get
\begin{proposition}\label{loc-phase}
Let $k\in\N$. Let $\eta>0$ and $\chi$ a smooth cutoff function being zero in a neighborhood of $0$. There exists $h_{0}>0$ such that for all eigenpairs $(\lambda,\psi)$ of $\mathfrak{L}_{h}$ with $\lambda\leq\mu_0+C_{0}h$, we have
$$\|\chi(h^\eta hD_{s})\psi\|_{B^k(\R^{m+n})}\leq \Oc(h^{\infty})\|\psi\|,\qquad \|\chi(h^{\eta}D_{t})\psi\|_{B^k(\R^{m+n})}\leq \Oc(h^{\infty})\|\psi\|.$$
\end{proposition}

\subsection{Microlocalization and coherent states}\label{coherent}
In this section we follow the same philosophy as in \cite{Ray13}.
\subsubsection{Formalism and application}
Let us recall the formalism of coherent states (see for instance \cite{Fol89} and \cite{CR12}).
We define
$$g_{0}(\sigma)=\pi^{-m/4}\re^{-|\sigma|^2/2},$$
and the usual creation and annihilation operators
$$\ga_{j}=\tfrac{1}{\sqrt{2}}(\sigma_{j}+\dr_{\sigma_{j}}),\qquad \ga_{j}^*=\tfrac{1}{\sqrt{2}}(\sigma_{j}-\dr_{\sigma_{j}}),$$
which satisfy the commutator relations
$$[\ga_{j},\ga_{j}^*]=1,\quad\qquad [\ga_{j}, \ga_{k}^*]=0\quad \mbox{ if } k\neq j.$$
We notice that
\begin{equation}\label{eq.sdsaa*}
\sigma_{j}=\tfrac{1}{\sqrt{2}}(\ga_{j}+\ga_{j}^*),\qquad
\dr_{\sigma_{j}}=\tfrac{1}{\sqrt{2}}(\ga_{j}-\ga_{j}^*),\qquad
\ga_{j}\ga_{j}^*=\tfrac{1}{2}(D_{\sigma_{j}}^2+\sigma_{j}^2+1).
\end{equation}
For $(u,p)\in\R^m\times \R^m$, we introduce the coherent state
$$f_{u,p}(\sigma)=\re^{ip\cdot \sigma} g_{0}(\sigma-u),$$
and the associated projection, defined for $\psi\in \sL^2(\R^m\times\R^n)$ by
$$\Pi_{u,p}\psi=\langle\psi, f_{u,p}\rangle_{\sL^2(\R^m, \dx \sigma)} f_{u,p}=\psi_{u,p}f_{u,p}.$$
We have the identity resolution formula
\begin{equation}\label{eq.coherent-dec}
\psi=\int_{\R^{2m}}\Pi_{u,p}\psi\dx u\dx p,
\end{equation}
and the Parseval-type formula
$$\|\psi\|^2=\int_{\R^n}\int_{\R^{2m}}|\psi_{u,p}|^2\dx u\dx p\dx\tau.$$
Let us emphasize here that the coherent states decomposition \eqref{eq.coherent-dec} is performed only with respect to the variable $\sigma$. We recall that
$$\ga_{j}f_{u,p}=\frac{u_{j}+ip_{j}}{\sqrt{2}}f_{u,p}$$
and
\begin{equation}\label{eq.exact-quant}
(\ga_{j})^\ell (\ga_{k}^*)^q\psi=\int_{\R^{2m}}\left(\frac{u_{j}+ip_{j}}{\sqrt{2}}\right)^\ell\left(\frac{u_{k}-ip_{k}}{\sqrt{2}}\right)^q \Pi_{u,p}\psi\dx u\dx p.
\end{equation}
We recall that (see \eqref{Lch})
$$\mathcal{L}_{h}=\big(D_{\tau}+A_{2}(x_{0}+h^{1/2}\sigma,\tau)\big)^2+\big(\xi_{0}+h^{1/2}D_{\sigma}+A_{1}(x_{0}+h^{1/2}\sigma,\tau)\big)^2$$
and, since $A_{1}$ and $A_{2}$ are polynomials, we get an expansion in the form
$$\mathcal{L}_{h}=\mathcal{L}_{0}+h^{1/2}\mathcal{L}_{1}+h\mathcal{L}_{2}+\ldots +h^{M/2}\mathcal{L}_{M}.$$
Let us now replace $\sigma_{j}$ and $\partial_{\sigma_{j}}$ by their expressions in terms of $\ga_{j}$ and $\ga^*_{j}$ (see \eqref{eq.sdsaa*}).
Then, we write the anti-Wick ordered operator. In other words, we commute $\ga$ and $\ga^*$ to put all the $\ga$ on the left (in order to apply Formula \eqref{eq.exact-quant}) and we deduce
\begin{equation}\label{ordering}
\mathcal{L}_{h}=\underbrace{\mathcal{L}_{0}+h^{1/2}\mathcal{L}_{1}+h\mathcal{L}^\W_{2}+\ldots +(h^{1/2})^{M}\mathcal{L}^\W_{M}}_{\mathcal{L}_{h}^\W}+\underbrace{hR_{2}+\ldots+(h^{1/2})^{M}R_{M}}_{\mathcal{R}_{h}},
\end{equation}
where the $R_{\jj}$ are the remainders in the anti-Wick ordering and satisfy, for $\jj\geq 2$,
\begin{equation}\label{remainders}
h^{\jj/2}R_{\jj}=h^{\jj/2}\mathcal{O}_{\jj-2}(\sigma, D_{\sigma}),
\end{equation}
where the notation $\mathcal{O}_{\jj}(\sigma,D_{\sigma})$ stands for a polynomial operator with total degree in $(\sigma,D_{\sigma})$ less than $\jj$. We recall that
$$\mathcal{L}_{h}^\W=\int_{\R^{2m}}\mathcal{M}_{x_{0}+h^{1/2}u, \xi_{0}+h^{1/2}p}\dx u \dx p.$$
\begin{proposition}
 There exist $h_{0},C>0$ such that for all eigenpairs $(\lambda,\psi)$ of $\mathcal{L}_{h}$ with $\lambda\leq\mu_0+C_{0}h$  and all $h\in(0,h_{0})$, we have
\begin{equation}\label{rlb}
\mathcal{Q}_{h}(\psi)\geq \int_{\R^{2m}} Q_{h,u,p}(\psi_{u,p})\dx u \dx p-Ch\|\psi\|^2\geq (\mu_{0}-Ch)\|\psi\|^2,
\end{equation}
where $Q_{h,u,p}$ is the quadratic form associated with the operator $\mathcal{M}_{x_{0}+h^{1/2}u,\xi_{0}+h^{1/2}p}$.
\end{proposition}
\begin{proof}
We use \eqref{ordering}. Then the terms of $\mathcal{R}_{h}$ (see \eqref{remainders}) are in the form $h h^{\jj/2} \sigma ^l D_{\sigma}^q \tau ^\alpha D^\beta_{\tau}$ with $l+q\leq\jj$ and $\beta=0,1$. With Corollary \ref{poly-bound} and the rescaling \eqref{rescaling}, we have
$$\|h^{\jj/2} \sigma ^l D_{\sigma}^q \tau ^\alpha D^\beta_{\tau}\psi\|\leq C\|\psi\|$$
and the conclusion follows.
\end{proof}

\subsubsection{Localization in the phase space}
This section is devoted to elliptic regularity properties (both in space and frequency) satisfied by the eigenfunctions.
We will use the following generalization of the \enquote{IMS} formula the proof of which can be found in \cite{Ray13}.
\begin{lemma}[\enquote{Localization} of $P^2$ with respect to $\gA$]\label{lem.IMS-formula}
Let $\mathsf{H}$ be a Hilbert space and $P$ and $\gA $ be  two unbounded operators defined on a domain $\mathsf{D}\subset \mathsf{H}$. We assume that $P$ is symmetric and that $P(\mathsf{D})\subset \mathsf{D}$, $\gA(\mathsf{D})\subset \mathsf{D}$, $\gA ^*(\mathsf{D})\subset \mathsf{D}$.
Then, for $\psi\in \mathsf{D}$, we have
\begin{multline}\label{IMS-formula}
\Re\langle P^2\psi, \gA \gA ^*\psi \rangle=\|P(\gA ^*\psi)\|^2-\|[\gA ^*,P]\psi\|^2+\Re\langle P\psi, [[P, \gA], \gA ^*]\psi\rangle\\
+\Re\left(\langle P\psi, \gA ^*[P, \gA]\psi \rangle-\overline{\langle P\psi, \gA[P, \gA ^*]\psi \rangle}\right).
\end{multline}
\end{lemma}
The following lemma is a straightforward consequence of Assumption \ref{hyp-gen}.
\begin{lemma}\label{min-non-deg}
Under Assumption \ref{hyp-gen}, there exist $\eps_{0}>0$ and $c>0$ such that
$$\mun(x_{0}+x,\xi_{0}+\xi)-\mun(x_{0},\xi_{0})\geq c(|x|^2+|\xi|^2),\qquad \forall (x,\xi)\in{\cal B}(\varepsilon_{0}),$$
and 
$$\mun(x_{0}+x,\xi_{0}+\xi)-\mun(x_{0},\xi_{0})\geq c,\qquad \forall (x,\xi)\in\complement{\cal B}(\varepsilon_{0}),$$
where ${\cal B}(\varepsilon_{0})=\{(x,\xi): |x|+|\xi|\leq \eps_{0}\}$ and $\complement {\cal B}(\varepsilon_{0})$ is its complement.
\end{lemma}

\begin{notation}
In what follows we will denote by $\tilde\eta>0$ all the quantities which are multiples of $\eta>0$,i.e. in the form $j\eta$ for $j\in\N\setminus\{0\}$. We recall that $\eta>0$ can be chosen arbitrarily small.
\end{notation}

\begin{proposition}\label{sigma-dsigma}
There exist $h_{0},C, \eps_{0}>0$ such that for all eigenpairs $(\lambda,\psi)$ of $\mathcal{L}_{h}$ with $\lambda\leq\mu_0+C_{0}h$, we have
$$\|\sigma\psi\|^2+\|\nabla_{\sigma}\psi\|^2\leq C\|\psi\|^2.$$
\end{proposition}
\begin{proof}
We recall that \eqref{rlb} holds. We have
$$\mathcal{Q}_{h}(\psi)=\lambda\|\psi\|^2\leq (\mu_{0}+C_{0}h)\|\psi\|^2.$$
We deduce that
$$\int_{\R^{2m}} Q_{h,u,p}(\psi_{u,p})-\mu_{0}|\psi_{u,p}|^2\dx u \dx p\leq Ch\|\psi\|^2$$
and thus by the min-max principle
$$\int_{\R^{2m}} \left(\mun(x_{0}+h^{1/2}u,\xi_{0}+h^{1/2}p)-\mu_{0}\right)|\psi_{u,p}|^2\dx u \dx p\leq Ch\|\psi\|^2.$$
Let $\eps_{0}>0$ be given in Lemma \ref{min-non-deg}. We split the integral into two parts and find
\begin{eqnarray}\label{eq.decoup}
\int_{\Bhe} (|u|^2+|p|^2)|\psi_{u,p}|^2\dx u\dx p&\leq& C\|\psi\|^2,\\ 
\int_{\CBhe} |\psi_{u,p}|^2\dx u\dx p&\leq &Ch\|\psi\|^2. 
\end{eqnarray}
The first inequality is not enough to get the conclusion. We also need a control of momenta in the region $\CBhe$.
For that purpose, we write:
\begin{equation}\label{Qa*}
\mathcal{Q}_{h}(\ga_{j}^*\psi)=\int_{\R^{2m}} Q_{h,u,p}\left(\frac{u_{j}-ip_{j}}{\sqrt{2}}\psi_{u,p}\right)\dx u \dx p+\langle\mathcal{R}_{h}\ga_{j}^*\psi,\ga_{j}^*\psi\rangle.
\end{equation}
Up to lower order terms we must estimate terms in the form:
$$h\langle  h^{\jj/2} \sigma ^l D_{\sigma}^q \tau ^\alpha D^\beta_{\tau}\ga_{j}^*\psi  ,\ga_{j}^*\psi\rangle,$$
with $l+q=\jj$, $\alpha\in\N$ and $\beta=0,1$.
By using the \textit{a priori} estimates of Propositions \ref{loc-space} and \ref{loc-phase}, we have
$$\| h^{\jj/2} \sigma ^l D_{\sigma}^q  \tau ^\alpha D^\beta_{\tau} \ga_{j}^*\psi\|\leq Ch^{-\tilde\eta}\|\ga_{j}^*\psi\|.$$
The remainder is controlled by
$$|\langle\mathcal{R}_{h}\ga_{j}^*\psi,\ga_{j}^*\psi\rangle|\leq Ch^{1-\tilde \eta}(\|D_{\sigma}\psi\|^2+\|\sigma\psi\|^2).$$
Then we analyze $\mathcal{Q}_{h}(\ga_{j}^*\psi)$ by using Lemma \ref{lem.IMS-formula} with $\gA =\ga_{j}$. We need to estimate the different remainder terms. We notice that
\begin{eqnarray*}
\|[\ga_{j}^*, P_{k, r, h}]\psi\|&\leq& C h^{1/2}\|\psi\|,\\
|\langle P_{k, r, h}\psi, \ga_{j}^*[ P_{k, r, h},\ga_{j}]\psi\rangle|&\leq& \| P_{k, r , h}\psi\|\ \| \ga_{j}^*[ P_{k, r, h},\ga_{j}]\psi\|,\\
|\langle P_{k, r, h}\psi, \ga_{j}[ P_{k, r, h},\ga_{j}^*]\psi\rangle|&\leq& \| P_{k, r , h}\psi\|\ \|\ga_{j}[ P_{k, r, h},\ga_{j}^*]\psi\|,\\
|\langle P_{k, r, h}\psi, [[ P_{k, r, h},\ga_{j}],\ga_{j}^*]\psi\rangle|&\leq& \| P_{k, r , h}\psi\|\ \|[[ P_{k, r, h},\ga_{j}],\ga_{j}^*]\psi\|,
\end{eqnarray*}
where $P_{1, r, h}$ denotes  the magnetic momentum $h^{1/2}D_{\sigma_{r}}+A_{1,r}(x_{0}+h^{1/2}\sigma, \tau)$ and $P_{2, r, h}$ denotes $D_{\tau_{r}}+A_{2,r}(x_{0}+h^{1/2}\sigma, \tau)$ and $A_{k,r}$ is the $r$-th component of $A_{k}$ ($k=1,2$).
We have
$$\| P_{k, r, h}\psi\|\leq C\|\psi\|$$
and
$$\| \ga_{j}^*[ P_{k, r, h},\ga_{j}]\psi\|\leq Ch^{1/2}\|\ga^*_{j} Q(h^{1/2}\sigma,\tau)\psi\|,$$
where $Q$ is polynomial. The other terms are bounded in the same way.
We apply the estimates of Propositions \ref{loc-space} and \ref{loc-phase} to get
$$\|\ga^*_{j} Q(h^{1/2}\sigma,\tau)\psi\|\leq Ch^{-\tilde \eta}\|\ga_{j}^*\psi\|.$$
We have
$$\mathcal{Q}_{h}(\ga_{j}^*\psi)=\lambda\|\ga_{j}^*\psi\|^2+\Oc(h)\|\psi\|^2+\Oc(h^{\demi-\tilde \eta})(\|\nabla_{\sigma}\psi\|^2+\|\sigma\psi\|^2),$$
so that
$$\mathcal{Q}_{h}(\ga_{j}^*\psi)\leq \mu_{0}\|\ga_{j}^*\psi\|^2+Ch\|\psi\|^2+Ch^{\demi-\tilde \eta}(\|\nabla_{\sigma}\psi\|^2+\|\sigma\psi\|^2).$$
By using \eqref{Qa*} and splitting again the integral into two parts, it follows
\begin{eqnarray*}
\int_{\Bhe} (|u|^2+|p|^2)|(u_{j}-ip_{j})\psi_{u,p}|^2\dx u\dx p&\leq& C\|\psi\|^2+Ch^{-\demi-\tilde \eta}(\|\nabla_{\sigma}\psi\|^2+\|\sigma\psi\|^2),\\
\int_{\CBhe} |(u_{j}-ip_{j})\psi_{u,p}|^2\dx u\dx p&\leq &Ch\|\psi\|^2+Ch^{\demi-\tilde \eta}(\|\nabla_{\sigma}\psi\|^2+\|\sigma\psi\|^2).
\end{eqnarray*}
Choosing $\tilde\eta$ small enough so that $\tilde\eta<\frac{1}{2}$, combining the last inequality with the first one of \eqref{eq.decoup} and the Parseval formula we get the conclusion.
\end{proof}

\begin{proposition}\label{sigma2-dsigma2}
Let $P\in\mathbb{C}_{2}[X_{1},\ldots,X_{2m}]$. There exist $h_{0},C, \eps_{0}>0$ such that for all eigenpairs $(\lambda,\psi)$ of $\mathfrak{L}_{h}$ with $\lambda\leq\mu_0+C_{0}h$, we have
$$\|P(\sigma,D_{\sigma})\psi\|^2\leq Ch^{-\demi-\tilde \eta}\|\psi\|^2.$$
\end{proposition}
\begin{proof}
From the proof of Proposition \ref{sigma-dsigma}, we infer
\begin{eqnarray}
\label{estim2-a}
\int_{\Bhe} (|u|^2+|p|^2)|(u_{j}-ip_{j})\psi_{u,p}|^2\dx u\dx p&\leq& Ch^{-\demi-\tilde \eta}\|\psi\|^2,\\
\int_{\CBhe} |(u_{j}-ip_{j})\psi_{u,p}|^2\dx u\dx p&\leq &Ch^{\demi-\tilde \eta}\|\psi\|^2.
\end{eqnarray}
Then we use Lemma \ref{lem.IMS-formula} with $\gA =\ga_{j}\ga_{j}$. The worst remainders in \eqref{IMS-formula} are
\begin{eqnarray*}
\|[P_{k,r,h},\ga_{j}^* \ga_{j}^*]\psi\|^2 &\leq & Ch\|\psi\|^2,\\
|\langle P_{k, r, h}\psi, \ga_{j}^* \ga_{j}^*[P_{k, r, h},\ga_{j} \ga_{j}]\psi \rangle| &\leq &Ch^{\demi-\tilde \eta}\|\psi\|\|\ga_{j}^* \ga_{j}^*\psi\|.
\end{eqnarray*}
We infer
$$\mathcal{Q}_{h}(\ga_{j}^* \ga_{j}^*\psi)\leq \mu_{0}\|\ga_{j}^* \ga_{j}^*\psi\|^2+Ch^{\demi-\tilde \eta}\|\psi\|^2+Ch^{\demi-\tilde \eta}\|\ga_{j}^* \ga_{j}^*\psi\|^2.$$
We deduce
\begin{eqnarray}
&\displaystyle\int_{\Bhe} (|u|^2+|p|^2)|(u_{j}-ip_{j})^2\psi_{u,p}|^2\dx u\dx p\leq Ch^{-\frac 12-\tilde \eta}\|\psi\|^2
+Ch^{-\frac 12-\tilde \eta}\|\ga_{j}^* \ga_{j}^*\psi\|^2,\nonumber\\
&\displaystyle\int_{\CBhe} |(u_{j}-ip_{j})^2\psi_{u,p}|^2\dx u\dx p\leq Ch^{\frac 12-\tilde \eta}\|\psi\|^2+Ch^{\frac 12-\tilde \eta}\|\ga_{j}^* \ga_{j}^*\psi\|^2.
\label{estim2-b}
\end{eqnarray}
Jointly with Proposition \ref{sigma-dsigma}, estimates \eqref{estim2-a} and \eqref{estim2-b} imply that
$$\int_{\R^{2m}} (|u|^4+|p|^4)|\psi_{u,p}|^2 \dx u \dx p\leq Ch^{-\demi-\tilde \eta}\|\psi\|^2.$$
The conclusion follows from the Parseval formula and Proposition \ref{sigma-dsigma}.
\end{proof}

\subsection{Spectral gap}\label{projection}
We introduce the projection
$$\Psi_{0}=\Pi_{0}\psi=\langle\psi,u_{x_{0},\xi_{0}}\rangle_{\sL^2(\R^n,\dx \tau)} u_{x_{0},\xi_{0}}$$
and, inspired by \eqref{psi1} where $f_{0}$ is replaced by $\langle\psi,u_{x_{0},\xi_{0}}\rangle_{\sL^2(\R^n,\dx \tau)}$ and $f_{1}$ by $0$,
\begin{equation}
\Psi_{1}=\sum_{j=1}^m (\dr_{x_{j}}u)_{x_{0},\xi_{0}}\, \sigma_{j} \langle\psi,u_{x_{0},\xi_{0}}\rangle_{\sL^2(\R^n,\dx \tau)}+\sum_{j=1}^m (\dr_{\xi_{j}}u)_{x_{0},\xi_{0}}\, D_{\sigma_{j}} \langle\psi,u_{x_{0},\xi_{0}}\rangle_{\sL^2(\R^n,\dx \tau)}.
\end{equation}
This leads to define the corrected Feshbach projection
\begin{equation}
\Pi_{h}\psi=\Psi_{0}+h^{1/2}\Psi_{1}
\end{equation}
and
$$R_{h}=\psi-\Pi_{h}\psi$$
Note that the functions $\Psi_{0}$ and $\Psi_{1}$ will be {\it a priori} $h$-dependent. By the $\sL^2$-normalization of $u_{x,\xi}$ (when $\xi\in\R^m$), $\Psi_{1}$ and $R_{h}$ are orthogonal (with respect to the $\tau$-variable) to $u_{0}$ (and $\Psi_{0}$). Let us recall that the $\mathcal{L}_{j}$ were defined in \eqref{eq.Lj}. Furthermore, we have by construction and Proposition \ref{FH}
\begin{equation}\label{L1Psi0}
(\mathcal{L}_{0}-\mu_{0})\Psi_{1}=-\mathcal{L}_{1}\Psi_{0}
\end{equation}
and, by the Fredholm alternative,
$$\langle\mathcal{L}_{1}\Psi_{0},\Psi_{0}\rangle_{\sL^2(\R^n,\dx \tau)}=0.$$

\subsubsection{Approximation results}\label{subsubsec.approx}

We can prove a first approximation.
\begin{proposition}
There exist $h_{0},C>0$ such that for all eigenpairs $(\lambda,\psi)$ of $\mathcal{L}_{h}$ with $\lambda\leq\mu_0+C_{0}h$, we have
$$\|\psi-\Pi_{0}\psi\|\leq Ch^{\frac12-\tilde\eta}\|\psi\|.$$
\end{proposition}
\begin{proof}
We can write
$$(\mathcal{L}_{0}-\mu_{0})\psi=(\lambda-\mu_{0})\psi-h^{1/2}\mathcal{L}_{1}\psi-h\mathcal{L}_{2}\psi-\ldots-h^{M/2}\mathcal{L}_{M}\psi.$$
By using the rough microlocalization given in Propositions \ref{loc-space} and \ref{loc-phase} and Proposition \ref{sigma2-dsigma2}, we infer that for $\jj\geq 2$
\begin{equation}\label{control-reste}
h^{\jj/2}\|\tau^\alpha D_{\tau} ^\beta \sigma^l D_{\sigma}^q \psi\|
\leq C h^{\frac \jj2-\frac{\jj-2}2-\frac14-\tilde \eta}\|\psi\|=Ch^{\frac{3}{4}-\tilde\eta}\|\psi\|,
\end{equation}
and thanks to Proposition \ref{sigma-dsigma}
$$\|\mathcal{L}_{1}\psi\|\leq Ch^{-\tilde\eta}\|\psi\|,$$
so that
$$\|(\mathcal{L}_{0}-\mu_{0})\psi\|\leq Ch^{\demi-\tilde\eta}\|\psi\|,$$
and the conclusion follows.
\end{proof}

\begin{corollary}\label{cor.approx0}
There exist $h_{0},C>0$ such that for all eigenpairs $(\lambda,\psi)$ of $\mathcal{L}_{h}$ with $\lambda\leq\mu_0+C_{0}h$, we have
$$\|\sigma(\psi-\Pi_{0}\psi)\|\leq Ch^{\frac14-\tilde\eta}\|\psi\|,\qquad
\|D_{\sigma}(\psi-\Pi_{0}\psi)\|\leq Ch^{\frac14-\tilde\eta}\|\psi\|.$$
\end{corollary}
We can now estimate $\psi-\Pi_{h}\psi$.

\begin{proposition}\label{control-Rh}
There exist $h_{0},C>0$ such that for all eigenpairs $(\lambda,\psi)$ of $\mathcal{L}_{h}$ with $\lambda\leq\mu_0+C_{0}h$, we have
$$\|R_{h}\psi\|=\|\psi-\Pi_{h}\psi\|\leq Ch^{\frac34-\tilde \eta}\|\psi\|.$$
\end{proposition}

\begin{proof}
Let us write
$$\mathcal{L}_{h}\psi=\lambda\psi.$$
We have
$$(\mathcal{L}_{0}+h^{1/2}\mathcal{L}_{1})\psi=(\mu_{0}+\Oc(h))\psi-h\mathcal{L}_{2}\psi-\ldots-h^{M/2}\mathcal{L}_{M}\psi.$$
Let us notice that, as in \eqref{control-reste}, we have, for $\jj\geq 2$,
$$h^{\jj/2}\|\mathcal{L}_{\jj}\psi\|\leq C h^{\frac 3 4-\tilde\eta}\|\psi\|.$$
We get
$$(\mathcal{L}_{0}-\mu_{0})R_{h}=-h^{1/2}\mathcal{L}_{1}(\psi-\Psi_{0})+\Oc(h)\psi-h\mathcal{L}_{2}\psi-\ldots-h^{M/2}\mathcal{L}_{M}\psi$$
It remains to apply Corollary \ref{cor.approx0} and obtain
$$h^{1/2}\|\mathcal{L}_{1}(\psi-\Psi_{0})\|\leq  \tilde C h^{\frac34-\tilde\eta}\|\psi\|.$$
\end{proof}

\subsubsection{Proof of Theorem \ref{theorem-simple-well}}

Let us introduce a subspace of dimension $P\geq 1$. For $j\in\{1,\ldots,P\}$ we can consider an $\sL^2$-normalized eigenfunction of $\mathcal{L}_{h}$ denoted by $\psi_{j,h}$ and so that the family $(\psi_{j,h})_{j\in\{1,\ldots,P\}}$ is orthogonal. We let
$$\mathfrak{E}_{P}(h)=\underset{j\in\{1,\ldots,P\}}\spann\psi_{j,h}.$$
\begin{remark}
We can extend the estimates of Propositions \ref{sigma-dsigma} and \ref{sigma2-dsigma2} as well as the approximations proved in Section \ref{subsubsec.approx} to $\psi\in\mathfrak{E}_{P}(h)$.
\end{remark}
We can now prove the following proposition.
\begin{proposition}\label{prop.lbQ}
For all $n\geq 1$, there exists $h_{0}>0$ such that, for all $h\in(0, h_{0})$, we have
$$\lambda_{n}(h)\geq \mu_{0}+\lambda_{n, 1}h+o(h),$$
where we recall that $\lambda_{n, 1}$ is the $n$-th eigenvalue of $\frac{1}{2}\Hess\,\mun(x_{0},\xi_{0})(\sigma, D_{\sigma})$.
\end{proposition}
\begin{proof}
Since we want to establish a lower bound for the eigenvalues, let us prove a lower bound for the quadratic form on $\mathfrak{E}_{P}(h)$, for $P\geq 1$. By definition of the $\mathcal{L}_{j}$ in \eqref{eq.Lj}, we have, for $\psi\in\mathfrak{E}_{P}(h)$,
$$\mathcal{Q}_{h}(\psi)=\langle\mathcal{L}_{0}\psi,\psi\rangle+h^{1/2}\langle\mathcal{L}_{1}\psi,\psi\rangle+h\langle\mathcal{L}_{2}\psi,\psi\rangle+\ldots+h^{p/2}\langle\mathcal{L}_{p}\psi,\psi\rangle+\ldots+h^{M/2}\langle\mathcal{L}_{M}\psi,\psi\rangle.$$
Using Propositions \ref{sigma-dsigma}, \ref{sigma2-dsigma2}, \ref{loc-space} and \ref{loc-phase}, we have, for $\jj\geq 3$
$$|h^{\jj/2}\langle\mathcal{L}_{\jj}\psi,\psi\rangle|\leq Ch^{\frac \jj2-\frac{\jj-3}2-\tilde\eta-\frac14}\|\psi\|^2=Ch^{\frac 5 4-\tilde\eta}\|\psi\|^2.$$
We infer
$$\mathcal{Q}_{h}(\psi)\geq\langle\mathcal{L}_{0}\psi,\psi\rangle+h^{1/2}\langle\mathcal{L}_{1}\psi,\psi\rangle+h\langle\mathcal{L}_{2}\psi,\psi\rangle-Ch^{\frac 5 4-\tilde\eta}\|\psi\|^2.$$
It remains to analyze the different terms.
We have
$$\langle\mathcal{L}_{0}\psi,\psi\rangle=\langle\mathcal{L}_{0}(\Psi_{0}+h^{1/2}\Psi_{1}+R_{h}),\Psi_{0}+h^{1/2}\Psi_{1}+R_{h}\rangle.$$
The orthogonality (with respect to $\tau$) cancels the terms $\langle\mathcal{L}_{0}\Psi_{1},\Psi_{0}\rangle$ and $\langle R_{h},\Psi_{0}\rangle$. Moreover, we have, with Propositions \ref{loc-space} and \ref{loc-phase},
$$h^{1/2}|\langle\mathcal{L}_{0}R_{h},\Psi_{1}\rangle|\leq h^{1/2-\tilde\eta}\|R_{h}\|\|\Psi_{1}\|,$$
and we use Proposition \ref{sigma-dsigma} to get
$$\|\Psi_{1}\|\leq C\|\psi\|,$$
so that, with Proposition \ref{control-Rh},
$$\langle\mathcal{L}_{0}\psi,\psi\rangle=\mu_{0}\|\Psi_{0}\|^2+h\langle\mathcal{L}_{0}\Psi_{1},\Psi_{1}\rangle+\Oc(h^{\frac 5 4-\tilde\eta})\|\psi\|^2.$$
We have
$$\langle\mathcal{L}_{1}\psi,\psi\rangle=\langle\mathcal{L}_{1}\Psi_{0},\Psi_{0}\rangle+2h^{1/2}\langle\mathcal{L}_{1}\Psi_{0},\Psi_{1}\rangle+h\langle\mathcal{L}_{1}\Psi_{1},\Psi_{1}\rangle+2\langle\mathcal{L}_{1}\psi,R_{h}\rangle.$$
Then, a Feynman-Hellmann formula provides $\langle\mathcal{L}_{1}\Psi_{0},\Psi_{0}\rangle=0$. Using again Propositions \ref{loc-space}, \ref{loc-phase}, \ref{sigma-dsigma}, \ref{sigma2-dsigma2} and \ref{control-Rh}, we notice that
$$\langle\mathcal{L}_{1}\psi,\psi\rangle=2h^{1/2}\langle\mathcal{L}_{1}\Psi_{0},\Psi_{1}\rangle+\Oc(h^{\frac 34-\tilde\eta})\|\psi\|^2.$$
We notice
$$\langle\mathcal{L}_{2}\psi,\psi\rangle=\langle\mathcal{L}_{2}\Psi_{0},\Psi_{0}\rangle+\langle\mathcal{L}_{2}(\psi-\Psi_{0}),\psi\rangle+\langle\mathcal{L}_{2}\psi,\psi-\Psi_{0}\rangle.$$
Writing $\psi-\Psi_{0}=h^{1/2} \Psi_{1}+R_{h}$, we have the estimate
$$|\langle\mathcal{L}_{2}(\psi-\Psi_{0}),\psi\rangle+\langle\mathcal{L}_{2}\psi,\psi-\Psi_{0}\rangle|\leq Ch^{-\frac14-\tilde\eta}h^{\demi-\tilde\eta}\|\psi\|^2.$$
We infer the lower bound
$$\mathcal{Q}_{h}(\psi)\geq \mu_{0}\|\Psi_{0}\|^2+h\langle\mathcal{L}_{0}\Psi_{1},\Psi_{1}\rangle+h\langle\mathcal{L}_{1}\Psi_{0},\Psi_{1}\rangle+h\langle\mathcal{L}_{1}\Psi_{1},\Psi_{0}\rangle+h\langle\mathcal{L}_{2}\Psi_{0},\Psi_{0}\rangle-Ch^{\frac 5 4-\tilde\eta}\|\psi\|^2.$$
Using \eqref{L1Psi0}, we get
$$h\langle\mathcal{L}_{0}\Psi_{1},\Psi_{1}\rangle+h\langle\mathcal{L}_{1}\Psi_{0},\Psi_{1}\rangle=h\mu_{0}\|\Psi_{1}\|^2,$$
so that, by orthogonality,
$$\mathcal{Q}_{h}(\psi)\geq \mu_{0}\|\Psi_{0}+h^{1/2}\Psi_{1}\|^2+h\langle\mathcal{L}_{1}\Psi_{1},\Psi_{0}\rangle+h\langle\mathcal{L}_{2}\Psi_{0},\Psi_{0}\rangle-Ch^{\frac 5 4-\tilde\eta}\|\psi\|^2.$$
Since $\langle R_{h},\Psi_{0}\rangle=0$ we deduce that
$$\|\Psi_{0}+h^{1/2}\Psi_{1}\|^2=\|\Psi_{0}+h^{1/2}\Psi_{1}+R_{h}\|^2+\Oc(h^{\frac 5 4-\tilde\eta})\|\psi\|^2.$$
It follows that
$$\mathcal{Q}_{h}(\psi)- \mu_{0}\|\psi\|^2\geq h\langle\mathcal{L}_{1}\Psi_{1},\Psi_{0}\rangle+h\langle\mathcal{L}_{2}\Psi_{0},\Psi_{0}\rangle+\Oc(h^{\frac 5 4-\tilde\eta})\|\psi\|^2,$$
and, since $\mathcal{Q}_{h}(\psi)\leq \lambda_{P}(h)\|\psi\|^2$, we have
$$(\lambda_{P}(h)-\mu_{0})\|\psi\|^2\geq h\langle\mathcal{L}_{1}\Psi_{1},\Psi_{0}\rangle+h\langle\mathcal{L}_{2}\Psi_{0},\Psi_{0}\rangle+\Oc(h^{\frac 5 4-\tilde\eta})\|\psi\|^2.$$
Thus we get
$$(\lambda_{P}(h)-\mu_{0})\|\Psi_{0}\|^2\geq h\langle\mathcal{L}_{1}\Psi_{1},\Psi_{0}\rangle+h\langle\mathcal{L}_{2}\Psi_{0},\Psi_{0}\rangle+\Oc(h^{\frac 5 4-\tilde\eta})\|\psi\|^2.$$
We recall that (see \eqref{Fredholm} and below)
\begin{multline*}
\langle\mathcal{L}_{1}\Psi_{1},\Psi_{0}\rangle+\langle\mathcal{L}_{2}\Psi_{0},\Psi_{0}\rangle\\
=\left\langle\tfrac{1}{2}\Hess\, \mu(x_{0},\xi_{0}) (\sigma,D_{\sigma})(\langle\psi, u_{0}\rangle_{\sL^2(\R^n,\dx \tau)}),\langle\psi, u_{0}\rangle_{\sL^2(\R^n,\dx \tau)}\right\rangle_{\sL^2(\R^m,\dx \sigma)}.
\end{multline*}
Finally we apply the min-max principle to the $P$-dimensional space $\left\langle\mathfrak{E}_{P}(h),u_{0}\right\rangle_{\sL^2(\R^n,\dx \tau)}$ to get the wished lower bound.
\end{proof}
Theorem \ref{theorem-simple-well} is a consequence of Propositions \ref{prop.lbQ} and \ref{quasimodes}.

\section{Magnetic WKB constructions}\label{sec:WKB}

\subsection{Stable manifold and eikonal equation}\label{Sec.eikonale}

In this section we study the construction of WKB solutions in the general case
$$
\mathfrak{L}_{h} = (h D_s + A_1(s,t))^2 + (D_t + A_2(s,t))^2\qquad\mbox{ with }\quad D = -i \nabla.
$$
As mentioned in the introduction,
for $(x,\xi)\in\R^m \times\R^m$, we  are interested in
the following electro-magnetic Laplacian acting on $\sL^2(\R^n, \dx t)$, when looking at the partial semiclassical  symbol of $\mathfrak{L}_{h}$  in variable $s$,
\begin{equation} \label{mxxi}
\mathcal{M}_{x,\xi}=(D_t+A_{2}(x,t))^2+(\xi+A_{1}(x,t))^2.
\end{equation}
Denoting by $\mun(x,\xi)$ its lowest eigenvalue, we would like to replace (in spirit) $\mathfrak{L}_{h}$ by the $m$-dimensional pseudo-differential operator $\mun(x,hD_x)$.
In order to complete this program,  the main assumption on the operator $\mathcal{M}_{x,\xi}$ in variable $t$ concerns its lowest eigenvalue $\mun$ and is stated in Assumption \ref{hyp-gen} in the introduction.

In order to build suitable quasimodes for $\mathfrak{L}_{h}$, it will be of great use
to first study   the following eikonal equation
\begin{equation} \label{eikpsi}
\mun(x_0+x, \xi_0+ i\nabla\Phi(x)) = \mu_0,
\end{equation}
with unknown $\Phi$, where we recall that $(x_0,\xi_0)$ is the point where the minimum $\mu_0$ of $\mun$,  is attained (as a real function of $(x,\xi) \in \R^{2m}$). Although certainly well-known (see e.g. \cite{sjo74}), in particular in the context of Sj\"ostrand's theory of FIO with complex phases,  we recall in the next subsections this construction with elementary tools.
In order to simplify the notation, we denote in the following
$$
 p(x, \xi) =\mun^\natural(x,\xi) -\mu_0\qquad\mbox{ with }\quad \mun^\natural(x,\xi)= \mun(x_0+ x,\xi_0+ \xi).
$$
With these notations, we deal with a real analytic symbol $p$ defined at least in a neighborhood of $(0,0)$ in the complex plane, and such that
$$
p(0,0) = 0, \quad  \nabla p(0,0) = 0, \quad \textrm{ and }\quad\Hess p(0,0) \textrm{ is (real) positive definite}.
$$
The point $\rho_0 = (0,0)$ is then a so-called doubly characteristic point for $p$, and the  eikonal equation now reads
\begin{equation*}
p(x,i \nabla \Phi(x)) = 0.
\end{equation*}
In the next subsection, we introduce our general framework.

\subsubsection{Framework}
In order to stick to the standard theory (see e.g. \cite{DiSj99}), we introduce
$$
 q(x,\xi) = -p(x, i\xi)
$$
and the eikonal equation reads
\begin{equation} \label{eikphi}
q(x,\nabla \Phi(x)) = 0.
\end{equation}
We focus from now on \eqref{eikphi}. 
In general, the symbol $q$ is not necessarily real, so that we cannot  expect to get a real phase $\Phi$.
Anyway the classical construction in $\C^{2m}$ of the phase remains true as we shall shortly recall below (both in quadratic and general cases). We can look at the Hamiltonian vector field $H_q$, in a small neighborhood of $\rho_{0}=(0,0)$,
$$
H_q(x,\xi) = \frac{\partial q}{\partial\xi}(x, \xi)\cdot \nabla_x - \frac{\partial q}{\partial x}(x,\xi)\cdot \nabla_\xi,
$$
and its linearization $F_q $, at $\rho_{0}$, called the fundamental matrix, is
$$
F_q =J\Hess q(\rho_{0})= \sep{ \begin{matrix}
\frac{\partial^2 q}{\partial x\partial \xi}(\rho_0) & \frac{\partial^2 q}{\partial \xi^2}(\rho_0) \\[5pt]
                          -\frac{\partial^2 q}{\partial x^2}(\rho_0) & - \frac{\partial^2 q}{\partial x\partial \xi}(\rho_0)
                          \end{matrix} },\quad J=\left(\begin{matrix}
0&I\\
-I&0
\end{matrix}\right).
$$
Since $p$ is real, the Hessian of $p$ is real, and of the form
$$
M_p = \Hess\ p(\rho_0) = \sep{ \begin{matrix} A & B \\ B & C \end{matrix} },
$$
with $A$, $B$ and $C$ real symmetric matrices. If $M_{q}$ denotes the Hessian of $q$, we therefore get that
$$
M_q= \sep{ \begin{matrix} -A & -iB \\ -iB & C \end{matrix} }.
$$
Since $\rho_0$ is a non degenerate minimum of $p$, the Hessian of $p$ at $\rho_0$ is positive definite.
We directly check that $F_q= J M_q$ and we have
$$
F_q = \sep{ \begin{matrix} -iB & C \\ A & iB \end{matrix} }.
$$
Now using that $F_p= J M_p$, we also have
$$
F_p = F_0=  \sep{ \begin{matrix} B & C \\ -A & -B \end{matrix} },
$$
so that
\begin{equation} \label{relationFpq}
F_q= -i \left( \begin{matrix}I&0\\ 0&-i I \end{matrix} \right) F_p \left( \begin{matrix}I&0\\ 0&i I \end{matrix} \right).
\end{equation}
We now deal with the eikonal equation in the quadratic case. 
\subsubsection{The quadratic case}
In this subsection we recall basic facts from \cite{sjo74}  about the quadratic case  and also 
about symplectic geometry. Recall that
$T^*\R^{m}$ is endowed with the canonical symplectic 2-form which can be written as  $\omega = \sum_j \dx\xi_j \wedge \dx x_j$ in coordinates, and that this form naturally extends to a symplectic
2-form in $T^*\C^{m}$ with the same expression 
\begin{equation}\label{eq.oxxi}
\omega\big((x,\xi), (y,\eta)\big) = \xi \cdot y - x \cdot \eta, \qquad (x,\xi), (y,\eta) \in T^*\C^{m},
\end{equation}
where we denote $x\cdot y=\sum_{j=1}^m x_{j} y_{j}$. Note that
\begin{equation} \label{xxbar}
\frac{1}{i} \omega( X,\overline{X} ) = \frac{1}{i} (\xi \cdot \overline{x} - x\cdot\overline{\xi}) = 2 \Im (\xi \cdot \overline{x}) \in \R.
\end{equation}
  We recall that an endomorphism $\kappa^0$ of $T^*\R^{m}$ is said to be symplectic if $\omega(\kappa^0., \kappa^0.) = \omega(., .)$. In coordinates, $\omega$ is represented by the matrix $J$ so that $\omega(X,Y)={}^t Y J X$.
Let $X=(x,\xi) \in T^*\C^m$ and consider a quadratic form $p_{0}$ associated with the bilinear form $b_{0}$,
$$
X \mapsto b_0(X,X)=p_{0}(X)\quad\mbox{ on }T^*\C^{m},
$$
which is real positive definite when restricted to $T^*\R^m$. Note that this implies that
$$
b_0(X,\overline{X}) > 0,\qquad \forall X \in T^*\C^m\setminus\set{0}.
 $$
We first recall that the fundamental matrix $F_0$ of $p_0$ may be defined through the symplectic form via the following formula
$$
2 b_0(X,Y) = \omega(X, F_0 Y), \qquad\forall X,Y \in T^*\C^m.
$$
Let us recall some properties of $F_{0}$ (see \cite{DiSj99,Ma02}). Let $M$ be the matrix of the quadratic form $p_{0}$. We have $F_{0}=JM$. Since $M$ is positive, we may consider $M^{1/2}$ and $M^{1/2}F_{0}M^{-1/2}=M^{1/2}JM^{1/2}$ is a real antisymmetric matrix. Its eigenvalues are purely imaginary and conjugate (the matrices are real). We denote them by $\pm i\vartheta_{j}$ with $\vartheta_{j}>0$. Following \cite{HelSj84, DiSj99}, we introduce
$$
\Lambda_+^0 = \bigoplus_{j} E_{i\vartheta_j}, \qquad\Lambda_-^0 = \bigoplus_{j} E_{-i\vartheta_j},
$$
and note that $\Lambda_+^0$ and $\Lambda_-^0$ are Lagrangian vector spaces of $T^*\C^m$ (we recall that both families $(E_{i\vartheta_j})_{j}$, $(E_{-i\vartheta_j})_{j}$ are $\omega$-orthogonal).

The next step is to show a transversality property.

\begin{lemma} \label{lambdainterreel}
We have 
$\Lambda_+^0 \cap
T^* \R^m = \set{0}$.
\end{lemma}

\begin{proof}
Let us take $X \in \Lambda_+^0 \cap T^* \R^m $. Then by stability of $\Lambda_+^0$, we also have $F_0 X \in \Lambda_+^0 $ and we get
$$
2 b_0(X,X) = \omega(X, F_0X) = 0,
$$
since $\Lambda_+^0 $ is Lagrangian. In addition we know that $p_0$ is positive definite, and this implies $X=0$.
\end{proof}
We now show that
\begin{lemma}\label{lem.oxx0}
For all $ X \in \Lambda_+^0\setminus\{0\}$,
we have $\frac{1}{i} \omega( X,\overline{X} ) >0$.
\end{lemma}

\begin{proof} This is done by a perturbation argument. For $X=(x,\xi) \in T^*\C^m$, we denote $p_1(X) =x^2+ \xi^2$ the harmonic oscillator, and we introduce
$$
p_t = (1-t)p_0 + t p_1, \qquad \forall t\in [0,1],
$$
and we denote by $F_{t}$ the fundamental matrix of $p_{t}$. The eigenvalues of $F_{1}=J$ are $\pm i$ and the Lagrangian subspaces are given by
$$
\Lambda_+^1 = \set{ \xi = ix} , \qquad \Lambda_-^1 = \set{ \xi = -ix}.
$$
In particular we have from (\ref{xxbar}) that
for all $X = (x,ix) \in \Lambda_+^1\setminus\{0\}$,
\begin{equation}\label{eq.oxx1}
\frac{1}{i} \omega( X,\overline{X} ) = \frac{1}{i} \omega( (x,ix) ,\overline{(x,ix )}) =  2 \Im ( ix \cdot \overline{x}) = 2 |x|^2 >0.
\end{equation}
Now we can consider
$$
\Lambda_+^t = \bigoplus_{j} E_{i\vartheta_j^t}\, ,
$$
where $i\vartheta_j^t$ are the eigenvalues of $F_t$. 

For $R>0$, let $\gamma_{R}$ be the contour made of a segment $[-R, R]$ and the semi-circle $\ccc(0,R) \cap \set{\Im z >0}$. Since $\vartheta_j^t >0$ for all $t \in [0,1]$ and $j\geq 1$, the open semi-disk surrounded by $\gamma_{R}$ contains all eigenvalues $i\vartheta_j^t$ as soon as $R$ is sufficiently large. We get
$$
\Lambda_+^t = \range \frac{1}{2i\pi} \int_{\gamma_{R}} (z-F_t)^{-1} \dx z 
$$
so that the application $t\mapsto\Lambda_+^t$ is a continuous family of subspaces.

Now we want to show that 
\begin{equation}\label{eq.oxxt}
\frac{1}{i} \omega( X,\overline{X} ) >0, \qquad \forall X \in \Lambda_+^t\setminus\{0\},\quad \forall t\in[0,1].
\end{equation}
This is already known for $t=1$ by \eqref{eq.oxx1}. Let us prove \eqref{eq.oxxt} by contradiction. Let us consider the largest $t_{0}\in[0,1]$ so that there exists $X_0 \in \Lambda^{t_0}_+\setminus\set{0}$ with
$$
\frac{1}{i} \omega( X_0,\overline{X_0} )=0.
$$
By definition of $t_{0}$ and continuity this implies that $(X,Y) \longrightarrow \frac{1}{i} \omega( X,\overline{Y} )$ is a non negative Hermitian form  on $\Lambda_+^{t_0}$.
By the Cauchy-Schwarz inequality, we get
$$
\Big|\frac{1}{i} \omega( X,\overline{Y} )\Big|^2 \leq \frac{1}{i} \omega( X,\overline{X} ) \frac{1}{i} \omega( Y,\overline{Y} ),\qquad  \forall X, Y \in \Lambda_+^{t_0}.
$$
Applying this to $Y=X_0$, we have
$$
  \omega( X,\overline{X_0} ) = 0,\quad \forall X \in \Lambda_+^{t_0},
$$
which implies $\overline{X_0} \in (\Lambda_+^{t_0})^{\perp_\omega} = \Lambda_+^{t_0}$, since $\Lambda_+^{t_0}$ is Lagrangian.
As a consequence we get that
$$
 \Re X_0 = \frac 12(X_0 + \overline{X_0}) \in \Lambda_+^{t_0} \cap T^* \R^m, \qquad  \Im X_0 = \frac{1}{2i}(X_0 - \overline{X_0}) \in \Lambda_+^{t_0} \cap T^* \R^m.
$$
From Lemma \ref{lambdainterreel}, we get that $X_0 = 0$, which gives a contradiction and proves \eqref{eq.oxxt}. Taking $t=0$ in \eqref{eq.oxxt} gives the lemma.
 \end{proof}
Let us now prove that $\Lambda_+^0$ is a graph. For this we establish first a transversality lemma.
\begin{lemma} \label{transp}
$\Lambda_+^0 \cap \set{ x=0} = \set{ 0}$.
\end{lemma}

\begin{proof} 
Let us consider $(0,\xi_0) \in \Lambda_+^0$. Applying \eqref{xxbar} gives $\frac{1}{i} \omega( (0, \xi_0),(0,\overline{\xi_0}) )=0$. Using Lemma \ref{lem.oxx0} with $X=(0,\xi_{0})$, we get $\xi_{0}=0$.
\end{proof}
Then we describe a parametrization of the graph.
\begin{lemma}\label{lem.B}
There exists a unique matrix $L \in M_m(\C)$ such that $\Lambda_+^0 = \set{ (x,Lx), \ x \in \C^m}$ and it satisfies ${}^t L = L$ and $\Im L >0$.
\end{lemma}

\begin{proof} By Lemma \ref{transp}, we know that there exists a matrix $L $ such that $ \Lambda_+^0 = \{ (x,Lx): x \in \C^m\}$. We have to check that $L$ satisfies the required properties. Since $ \Lambda_+^0$  is Lagrangian, we have, with \eqref{eq.oxxi},
$$
0 = \omega( (x, Lx), ( y, L y)) = Lx \cdot y - Ly \cdot x = x ( {}^tL - L) y,\qquad \forall x,y \in \C^m.
$$
This implies ${}^tL = L$. From \eqref{xxbar}, we also have for all $x \in \C^m\setminus\{0\}$,
\begin{equation}
0  < \tfrac{1}{i} \omega( (x, Lx),(\overline{x},\overline{Lx}) )
 = \tfrac{1}{i} \sep{ Lx \cdot \overline{x} - \overline{L}\overline{x} \cdot x}
 = \tfrac{1}{i}(L- {}^t\overline{L}) x \cdot \overline{x}
 = 2 \Im L  x \cdot \overline{x}.
\end{equation}
This implies $\Im L >0$ which is the desired result.
\end{proof}
We can now solve the so-called eikonal equation in the quadratic case.
\begin{proposition}\label{prop.quad}
Let $L$ given by Lemma \ref{lem.B} and define $\Phi_0(x) = \frac{1}{2i}Lx\cdot x$.\\
Then $\Phi_{0}$ is a quadratic homogeneous polynomial with $\Re \Hess \Phi_0 >0$ and such that
$$
\Lambda_+^0 = \set{ \xi = i\nabla \Phi_0(x)},\qquad \Lambda_-^0 = \set{ \xi = -i\overline{\nabla \Phi_0(\overline{x})}}.
$$
Furthermore this function satisfies the eikonal equations 
$$ p_0(x, i\nabla \Phi_0(x)) = 0,\qquad p_0(x, -i\overline{\nabla \Phi_0(\overline{x})}) = 0, \qquad \forall x \in \R^m.$$
\end{proposition}

\begin{proof} 
Since ${}^t L = L$, we have $\nabla \Phi_0(x) = \tfrac{1}{i} Lx$ and $\Re\Hess \Phi_{0}=\Im L>0$. For the eikonal equation, we check that
for $X=(x,Lx) \in \Lambda_+^0$,
$$
p_0(x, i\nabla \Phi_0(x)) = b_0(X,X) = \omega(X,F_{0}X) = 0,
$$
since $X$ and $F_{0}X \in \Lambda_+^0$, and $\Lambda_+^0$ is Lagrangian. We can deal similarly with $\Lambda_{-}^0$. This concludes the proof.
\end{proof}
Let us now explain the relation with the transport equation. We will later need some information about the transport operator
$$\mathscr{T}_{0}(x)=\tfrac12\left(\nabla_{\xi}q(x,\nabla\Phi_{0}(x))\cdot \nabla_{x}+\nabla_{x}\cdot\nabla_{\xi}q(x,\nabla\Phi_{0}(x))\right).$$
We let
$$
{\mathsf v}_{0}(x)
= \tfrac{1}{2i} \nabla_x\cdot \nabla_\xi q(x,i\nabla\Phi_{0}(x)) =
\tfrac{1}{2} \nabla_x\cdot \nabla_\xi q (x,\nabla \Phi_{0}(x)).
$$
Then the transport $\mathscr{T}_{0}$ is exactly the  projection on $T \R^m_x$ of the Hamiltonian vector field of $q(x,\xi)$ at the point $(x,\nabla \Phi_{0}(x)) \in \Lambda^0_+$,
$$
\mathscr{T}_{0}(x) = \pi_x (H_q (x,\nabla\Phi_{0}(x)))+\mathsf{v}_{0}(x),
$$
which reads in suitable coordinates
$$
\mathscr{T}_{0}(x)= \sum_{j=1}^m \vartheta_j y_j \nabla_{y_j} + {\mathsf v}_{0}(\kappa^{-1}(y)),
$$
where $(y,\eta) = (\kappa(x), ^t\dx\kappa(x)^{-1}\xi)$ is the corresponding symplectic change of variable (the existence of $\kappa$ is justified by the transversality of $\Lambda_+^0$ with $\{x=0\}$ stated in Lemma \ref{transp2}).
\begin{remark}\label{rem.transp}
We recall that the spectrum of the operator $\mathscr{T}_{0}$ acting on $\sL^2(\re^{-\Phi/h}\dx x)$ is nothing but the one of ${\mathsf{Op}}_{h}^{w}(p)$.
\end{remark}

\subsubsection{General case} 
Let us now deal with the general case. By Proposition \ref{prop.quad}, we have the following lemma.
 \begin{lemma}
The matrix $F_q$ is antisymmetric with respect to $\omega$. The eigenvalues of $F_q$ are of the form $(-\vartheta_j, +\vartheta_j)$, $j\in\{1,\ldots, m\}$, where the $\vartheta_j>0$ are counted with multiplicity. In addition, for $q_0$ the quadratic approximation of $q$ at $0$,  we have
 $$
 q_0(x, \nabla\Phi_0(x)) = 0,
 $$
 where $\Phi_0$ is defined in Proposition \ref{prop.quad} with the $p_{0}$ associated with $q_{0}$.
 \end{lemma}

Here we recall that $\Lambda_+^0 = \set{ (x,Lx) :  x\in\C^m } $ is the Lagrangian subspace associated with the eigenvalues $i\vartheta_j$ of $F_{p}$ ($\vartheta_j >0$). The transformation
$$
U : (x,\xi)  \mapsto (x, -i\xi)
$$
and \eqref{relationFpq} give directly that $\Lambda^{0,q}_+=U\Lambda_+^0 = \set{ (x,-iLx) : x\in\C^m }$ is also a Lagrangian subspace
associated to eigenvalues of positive real part $\set{\vartheta_j, 1\leq j\leq m}$.
We have another interpretation of the set $\Lambda^{0,q}_+$: For this we study the linearized flow at $\rho_0$ given by
$$
Z'(t)=F_q(Z)\cdot\nabla Z.
$$
This is clear that $\Lambda_+^{0,q}$ and $\Lambda_-^{0,q}$ are respectively the unstable and stable
manifolds  associated with the vector field $F_q(Z)\cdot\nabla$.
By this we mean that
\begin{equation}\label{limZ}
\forall Z_0 \in \Lambda_\mp^{0,q}, \qquad  \lim_{t \to \pm \infty } Z(t) = \rho_0.
\end{equation}
As already noted, these two spaces $\Lambda_\pm^{0,q}$ are Lagrangian. Now we show that the stable and unstable manifolds $\Lambda_{\pm}^q$ associated with the vector field $H_q$ are also Lagrangian.
Knowing the spectrum of the linearization of $H_q$, we just have to  apply the (complex) stable manifold theorem and we directly get that there exists one unstable holomorphic manifold  $\Lambda_+^q$ and one stable holomorphic manifold $\Lambda_-^q$, for which we have
\begin{equation} \label{tangent}
T_{\rho_0} \Lambda_\pm^q = \Lambda_\pm^{0,q}.
\end{equation}
Since $\Lambda_\pm^{0,q}$ are $m$ dimensional, so are $\Lambda_\pm^q$. Our next result is the following
\begin{lemma}\label{transp2}
The manifolds $\Lambda_\pm^q$ are transverse to $\set{ x=0}$ and $\set{\xi = 0}$.
\end{lemma}
\begin{proof}
The first result is a direct consequence of Lemma \ref{transp} for $p_0$ and \eqref{tangent}.
\end{proof}
Now we can solve the eikonal equation, near $q_{0}$, in the general case.
\begin{proposition} There exists a holomorphic function $\Phi$ such that
$$\Lambda_+^q = \set{\xi = \nabla \Phi(x)}.$$ 
In addition $\Phi$ solves (locally) the eikonal equation $q(x, \nabla \Phi(x)) = 0$ and $\Lambda_-^q= \set{\xi = - \overline{\nabla\Phi (\overline{x})}}$, and $\Re \Hess \Phi(\rho_0)$  is positive definite.
\end{proposition}

\begin{proof} The existence of $\Phi$ is a consequence of the fact that $\Lambda_+^q \cap \set{x=0} = \set{0}$ and that $\Lambda_+^q$ is Lagrangian. 
We first recall that at the linearization level (see Proposition \ref{prop.quad}),
we have
$$
\Lambda_+^{0,q} = \big\{ \xi =  \nabla \Phi_0(x)\big\},
$$
where
 $ \nabla \Phi_0(x) = -i Lx $ is exactly the linear part of the expansion of $\nabla \Phi(x)$ at $\rho_0$ and $\Re \Hess \Phi(\rho_0)$ is a real positive definite quadratic form.

Now for $x$ in a neighborhood of $\rho_0$ in $ \C^m$, and looking at the characteristic $X(t)$  of $H_q$ starting at  $X_0 =(x, \nabla\Phi(x))$, we get that
$$
q(x, \nabla \Phi(x)) = q(X_0) = \lim_{t\to -\infty} q(X(t)) = q(\rho_0) = 0,
$$
where we have used that $q$ is constant along the characteristics of $H_q$, \eqref{limZ} and the stable manifold theorem. The phase $\Phi$ therefore solves the eikonal equation. For the description of $\Lambda_-^q$, we first  check that for all $x$ in a neighboorhood  of $\rho_0$ in $\C^m$, $\overline{x}$ also belongs to  a  neighborhood of $\rho_0$ and therefore $q(\overline{x}, \nabla \Phi(\overline{x}))= 0$. Taking the complex conjugate and using that $p$ is real analytic we get
\begin{equation}\label{eq.qbar}
0  = \overline{ q(\overline{x}, \nabla \Phi(\overline{x}))} = \overline{ -p(\overline{x}, i\nabla \Phi(\overline{x}))} =
-p(x, -i \overline{\nabla \Phi(\overline{x})}) = q(x, - \overline{\nabla \Phi(\overline{x})})
\end{equation}
which gives that $q$ is zero on $\big\{ \xi = - \overline{\nabla \Phi(\overline{x})}\big\}$. This holomorphic manifold is of dimension $m$, is clearly Lagrangian and we only have to check that it coincides with $\Lambda_-^q$, with tangent space $\Lambda_{-}^{0,q}$. For this,  it  is sufficient to check that $(x,- \overline{\nabla \Phi(\overline{x})}) \in \Lambda_-^q$. If we look at the solution of $
X'(t) = H_q(X(t))$ with initial condition $X_0 = (x,- \overline{\nabla \Phi(\overline{x})})$, we get, by the stable manifold theorem, that $\lim_{t \to +\infty} X(t)= \rho_0$. From dimensional considerations, we get that
$\big\{  \xi = - \overline{\nabla \Phi(\overline{x})}\big\} = \Lambda_-^q$.
\end{proof}

\subsection{WKB expansions}
We assume that $A_{2}=0$ in \eqref{Lfh} (note that this was not the case in the previous subsection).
Before starting our magnetic WKB analysis we center the operator $\mathfrak{L}_{h}$ at $x_{0}$ and perform a change of gauge in order to center the phase at $\xi_{0}$. This means that we rather consider
$$\Lc_{h} =D_t^2 +(hD_{s}+A^\natural)^2,\qquad A^\natural(s,t)=\xi_{0}+A_{1}(x_{0}+s,t).$$
In order to lighten the notation, we introduce
$$\Mc^\natural_{x,\xi}=\Mc_{x+x_{0},\xi+\xi_{0}},\qquad u^\natural_{x,\xi}=u_{x+x_{0},\xi+\xi_{0}},\qquad \mu^\natural(x,\xi)=\mun(x+x_{0},\xi+\xi_{0}).$$
We always have
$$\left(\Mc^\natural_{x,\xi}\right)^*=\Mc^\natural_{x,\overline{\xi}},\qquad \forall (x,\xi)\in\R^m\times\C^m$$
and the assumption $A_{2}=0$ implies the fundamental property:
\begin{equation}\label{barbar}
\overline{u^\natural_{x,\xi}}=u^\natural_{x,\overline{\xi}},\qquad \forall (x,\xi)\in\R^m\times\C^m.
\end{equation}
We conjugate $\Lc_{h}$ via a weight function $\Phi=\Phi(s)$ and define
\begin{eqnarray*}
\Lc_{\Phi}&=& \re^{\Phi(s)/h}\ \Lc_{h}\ \re^{-\Phi(s)/h}\\
&=&D^2_{t}+(hD_{s}+i\nabla\Phi+A^\natural)^2\\
&=& \Lc_{0}+h \Lc_{1}+h^2 \Lc_{2},
\end{eqnarray*}
with
\begin{eqnarray*}
\hspace{-0.7cm}&&\Lc_{0}=D_{t}^2+(i\nabla\Phi+A^\natural)^2=\Mc^\natural_{s,i\nabla\Phi(s)},\\
\hspace{-0.7cm}&&\Lc_{1}=D_{s}\cdot(i\nabla\Phi+A^\natural)+(i\nabla\Phi+A^\natural)\cdot D_{s}
=\tfrac  1 2\Big(D_{s}\cdot(\nabla_{\xi}\Mc^\natural)_{s,i\nabla\Phi(s)}+(\nabla_{\xi}\Mc^\natural)_{s,i\nabla\Phi(s)}\cdot D_{s}\Big),\\
\hspace{-0.7cm}&&\Lc_{2}=D_{s}^2\Phi,
\end{eqnarray*}
where we have used
\begin{equation}\label{dxiM}
(\nabla_{\xi}\Mc^\natural)_{s,i\nabla\Phi(s)}=2(i\nabla\Phi(s)+A^\natural).
\end{equation}
We now look for a formal solution on the form
$$\lambda\sim \sum_{j\geq 0}\lambda_{j} h^{j}, \qquad
\an \sim \sum_{j\geq 0}\an_{j} h^j$$
such that $\Lc_{\Phi}\an = \lambda\an$.
We cancel each power of $h$ step by step.

\subsubsection{Term in $h^0$: Solving the operator valued eikonal equation}
We have to find $(\lambda_{0},\an_{0})$ such that
$$\Lc_{0}\an_{0}=\lambda_{0}\an_{0}.$$
According to Theorem \ref{theorem-simple-well}, we must choose
$$\lambda_{0}=\mu_{0}.$$
Thus we have to find $\an_{0}$ such that
\begin{equation}\label{h0}
\Lc_{0}\an_{0}=\mu_{0}\an_{0},
\end{equation}
that is to say
$$\Mc^\natural_{s,i\nabla\Phi(s)}\an_{0}=\mu_{0}\an_{0}.$$
To solve \eqref{h0}, we choose $\an_{0}$ in the form
\begin{equation}\label{eq.a0}
\an_{0}(s,t)=u^\natural_{s,i\nabla\Phi(s)}(t) \bn_{0}(s),
\end{equation}
where $\bn_{0}$ has to be determined and $\Phi$ is the solution constructed in Section \ref{Sec.eikonale} of the following  eikonal equation
$$\mun^\natural(s,i\nabla_{s}\Phi)=\mu_{0}.$$

\subsubsection{Term in $h^1$}
Collecting the terms in $h^1$, we obtain the first transport equation
$$(\Lc_{0}-\mu_{0})\an_{1}=-(\Lc_{1}-\lambda_{1})\an_{0}.$$
Pointwise in $s$, the Fredholm compatibility condition can be written as
\begin{equation}\label{eq.opvaltrans}
(\lambda_{1}-\Lc_{1})\an_{0}\in ({\rm Ker}(\Lc_{0}-\mu_{0})^*)^\perp= ({\rm Ker}({\Lc_{0}}^*-\mu_{0}))^\perp.
\end{equation}
\begin{lemma}\label{lem.transporteq}
The condition \eqref{eq.opvaltrans} is equivalent to 
$$
\mathscr{T} b_0 = \lambda_1 b_0,
$$
where 
$$
\mathscr{T} = \tfrac{1}{2}\left(\nabla_{\xi}\mun^\natural\cdot D_{s}+D_{s}\cdot\nabla_{\xi}\mun^\natural\right).
$$
\end{lemma}
\begin{proof}
From Assumption \ref{hyp-gen}, we have ${\rm Ker}({\Lc_{0}}^*-\mu_{0})=\spann(u^\natural_{s,-i\nabla\overline{\Phi}(s)})$, so that, with \eqref{eq.a0}, the compatibility condition is equivalent to
\begin{equation*}
\lambda_{1}\left\langle u^\natural_{s,i\nabla\Phi(s)}\bn_{0}(s),u^\natural_{s,-i\nabla\overline{\Phi}(s)}\right\rangle_{\sL^2(\R^m,\dx t)}=
\left\langle\Lc_{1}u^\natural_{s,i\nabla\Phi(s)}\bn_{0}(s),u^\natural_{s,-i\nabla \overline\Phi(s)}\right\rangle_{\sL^2(\R^m,\dx t)}, \ \forall s\in\R^m.
\end{equation*}
By \eqref{u^2}, we have
\begin{equation}
\int_{\R^n} u^\natural_{s,i\nabla\Phi(s)}(t)\overline{u^\natural_{s,-i\nabla\overline{\Phi}(s)}(t)}dt=1,\qquad \forall s\in\R^m.
\end{equation}
Let us rewrite the first term:
\begin{eqnarray} \label{simplb1}
\lambda_{1}\left\langle u^\natural_{s,i\nabla\Phi(s)}\bn_{0}(s),u^\natural_{s,-i\nabla \overline\Phi(s)}\right\rangle_{\sL^2(\R^m,\dx t)}
&=&\lambda_{1}\bn_{0}(s).
\end{eqnarray}
Let us deal with the second term.
Differentiating the relation $\Mc^\natural_{x,\xi}u^\natural_{x,\xi}=\mun^\natural(x,\xi) u^\natural_{x,\xi}$ with respect to (the complex variable) $\xi$ leads to (see Proposition \ref{FH})
\begin{equation}\label{eq.derMxi}
\Big(\Mc^\natural_{x,\xi}-\mun^\natural(x,\xi)\Big) \nabla_{\xi} u^\natural_{x,\xi}=\Big(\nabla_{\xi}\mun^\natural(x,\xi)-\nabla_{\xi}\Mc^\natural_{x,\xi}\Big)u^\natural_{x,\xi}.
\end{equation}
The Fredholm condition associated with \eqref{eq.derMxi} can be written as
$$\langle(\nabla_{\xi}\Mc^\natural_{x,\xi}-\nabla_{\xi}\mun^\natural(x,\xi))u^\natural_{x,\xi},u^\natural_{x,\overline\xi}\rangle_{\sL^2(\R^n,\dx t)}=0.$$
Consequently, taking $\xi=i\nabla\Phi(s)$ and using \eqref{barbar}, we get
\begin{equation}\label{dximu1}
\nabla_{\xi}\mun^\natural(s,i\nabla\Phi(s))=\int_{\R^n}(\nabla_{\xi}\Mc^\natural)_{s,i\nabla\Phi(s)}\ u^\natural_{s,i\nabla\Phi(s)}(t)u^\natural_{s,i\nabla\Phi(s)}(t)\dx t.
\end{equation}
Multiplying by $b_{0}$ and differentiating with respect to $s$, we infer
\begin{align*}
D_{s}\cdot(b_{0}(s)\nabla_{\xi}\mun^\natural(s,i\nabla\Phi(s)))=
&\int_{\R^n} D_{s}\cdot\left((\nabla_{\xi}\Mc^\natural)_{s,i\nabla\Phi(s)} b_{0}(s) u^\natural_{s,i\nabla\Phi(s)}(t)\right)u^\natural_{s,i\nabla\Phi(s)}(t)\dx t\\
&+\int_{\R^n} (\nabla_{\xi}\Mc^\natural)_{s,i\nabla\Phi(s)} b_{0}(s) u^\natural_{s,i\nabla\Phi(s)}(t)\cdot D_{s}u^\natural_{s,i\nabla\Phi(s)}(t)\dx t\\
=&\int_{\R^n} D_{s}\cdot\left((\nabla_{\xi}\Mc^\natural)_{s,i\nabla\Phi(s)} b_{0}(s) u^\natural_{s,i\nabla\Phi(s)}(t)\right)u^\natural_{s,i\nabla\Phi(s)}(t)\dx t\\
&+\int_{\R^n} (\nabla_{\xi}\Mc^\natural)_{s,i\nabla\Phi(s)}  u^\natural_{s,i\nabla\Phi(s)}(t) \cdot D_{s} \big(b_{0}(s)u^\natural_{s,i\nabla\Phi(s)}(t)\big)\dx t\\
&-D_{s}b_{0}(s)\cdot\int_{\R^n} (\nabla_{\xi}\Mc^\natural)_{s,i\nabla\Phi(s)}  u^\natural_{s,i\nabla\Phi(s)}(t) u^\natural_{s,i\nabla\Phi(s)}(t)\dx t,
\end{align*}
where we have used $b_{0}D_{s}u=D_{s}(b_{0}u)-u D_{s}b_{0}$. We recognize the expression of $\Lc_{1}$ and using \eqref{dximu1}, this yields
\begin{multline*}
\langle\Lc_{1}u^\natural_{s,i\nabla\Phi(s)}\bn_{0}(s),u^\natural_{s,-i\nabla\overline \Phi(s)}\rangle_{\sL^2(\R^m,\dx t)}\\
 =\tfrac{1}{2}\Big(\nabla_{\xi}\mun^\natural(s,i\nabla\Phi(s))\cdot D_{s}b_{0}(s)+D_{s}\cdot\Big(\nabla_{\xi}\mun^\natural(s,i\nabla\Phi(s))b_{0}(s)\Big)\Big).
\end{multline*}
Combining this last relation with \eqref{simplb1}, the compatibility condition gives the effective transport equation
$$\tfrac{1}{2}\left(\nabla_{\xi}\mun^\natural\cdot D_{s}+D_{s}\cdot\nabla_{\xi}\mun^\natural\right)b_{0}=\lambda_{1}b_{0},$$
where $\mun^\natural$ stands for $\mun^\natural(s,i\nabla \Phi(s))$ for short.
\end{proof}

\subsubsection{Solving the effective transport equation}
We recall now some notations from Section \ref{Sec.eikonale} in order to solve the standard transport equation given in Lemma \ref{eq.opvaltrans}.
We let
$$
q(s,\xi)=\mu_{0} -\mun^\natural(s,i\xi)
$$
so that with this notation
\begin{equation} \label{maintransport}
\mathscr{T}(s) = \tfrac{1}{2}\left(\nabla_{\xi}q (s,\nabla \Phi(s)) \cdot \nabla_{s}+\nabla_{s}\cdot\nabla_{\xi} q (s,\nabla \Phi(s)) \right).
\end{equation}
We let
$$
{\mathsf v}(s)
= \tfrac{1}{2i} \nabla_s\cdot (\nabla_\xi \mun^\natural{(s,i\nabla\Phi(x))}) =
\tfrac{1}{2} \nabla_s\cdot (\nabla_\xi q (s,\nabla \Phi(s)).
$$
Then the transport $\mathscr{T}$ is exactly the  projection on $T \C^m_s$ of the Hamiltonian of $q(s,\xi)$ at the point $(s,\nabla \Phi(s)) \in \Lambda^q_+$,
$$
\mathscr{T}(s) = \pi_s (H_q (s,\nabla\Phi(s)))+\mathsf{v}(s),
$$
which reads in suitable coordinates
$$
\mathscr{T}= \sum_{j=1}^m \vartheta_j y_j \nabla_{y_j} + \Oc(y^2)\nabla_{y_j} + {\mathsf v},
$$
where $(y,\eta) = (\kappa(s), ^t\dx\kappa(s)^{-1}\xi)$ is the corresponding symplectic change of variable (the existence of $\kappa$ is justified by the transversality of $\Lambda_+^q$ with $\{x=0\}$ stated in Lemma \ref{transp2}).

Using again this change of coordinates, we also directly get that
$$
{\mathsf v}(s) = \sum_j \tfrac{1}{2} \vartheta_j +  \Oc(s).
$$
We are  reduced to the study of a standard transport equation (see e.g. \cite{DiSj99})
and we can at least formally solve the equation $\mathscr{T} b_0 = \lambda_1 b_0$ in the space of formal series first, provided that
$$
\lambda_1 \in \Big\{ \sum_{j=1}^m (\tfrac{1}{2}+ \alpha_j)\vartheta_{j}, \quad \alpha \in \N^m\Big\}.
$$
By Remark~\ref{rem.transp}, these values are exactly the eigenvalues of $\frac{1}{2}\Hess\,\mun(x_{0},\xi_{0})(\sigma, D_{\sigma})$, with in particular
$\tfrac{1}{2} \sum_j \vartheta_j$ for the smallest one and recall that they were supposed to be simple by
Assumption \ref{hyp-gen'}. This is of course coherent with the statement of Theorem \ref{theorem-simple-well}. Following again e.g. \cite{DiSj99}, we can also solve this first transport equation in $\ccc^\infty $ in a neighborhood of $0$ and so determine $b_{0}$.

\subsubsection{Determination of $a_{1}$}
Let us now take $\alpha \in \N^m$ fixed and let $b_{0}$ be the corresponding solution of the transport equation. Let us come back to
$$(\Lc_{0}-\mu_{0})\an_{1}=-(\Lc_{1}-\lambda_{1})\an_{0},$$
where $\an_{0}$ is given in \eqref{eq.a0}. We look at a function $\an_{1}$ not necessary given in a tensor form, but with an
orthogonal decomposition in the form
\begin{equation} \label{splita1}
\an_{1}{(s,t)} = u^\natural_{s,i\nabla\Phi(s)}(t) \bn_{1}(s) + \tilde{\an}_{1} (s,t),
\end{equation}
with
$$
\tilde{\an}_1 \in ({\rm Ker}(\Lc_{0}-\mu_{0}))^\perp.
$$
From the transport equation, we see that we can directly find
$$
\tilde{\an}_{1} = -(\Lc_{0}-\mu_{0})^{-1} (\Lc_{1}-\lambda_{1})\an_{0},
$$
where by operator $(\Lc_{0}-\mu_{0})^{-1}$ we mean the inverse of the operator
$$
(\Lc_{0}-\mu_{0}) : \Dom(\Lc_{0}-\mu_{0}) \cap({\rm Ker}(\Lc_{0}-\mu_{0}))^\perp \longrightarrow ({\rm Ker}({\Lc_{0}}^*-\mu_{0}))^\perp
$$
where $\Dom(\Lc_{0}-\mu_{0})$ is the domain in $\sL^2(\R^m,\dx t)$ of operator $\Lc_{0}-\mu_{0}$, pointwise in $s$. Note that $a_0(s, .)$ belongs to $\Dom(\Lc_{0}-\mu_{0})$ because of the properties of $u^\natural$, and that we indeed get that $\tilde{\an}_1$ is smooth by elliptic regularity.

To summarize, at this point, and whatever $b_1$ is, we have been able to find
a  function
$$
\an^{[1]}(s,t) = \an_0(s,t) + h \an_1(s,t)
$$
such that
$$
(\Lc_h - \mu_0 -h\lambda_{1}) \an^{[1]} \re^{-\Phi/h} = r^{[1]}   \re^{-\Phi/h}.
$$
Furthermore, there exist a neighborhood $\mathcal{V}$ of $0\in\R^m$ and for all $\alpha\in\N^m, \beta \in \N^{n}$, a constant $C>0$ such that
$$|\D_{s}^\alpha\D_{t}^\beta r^{[1]}(s,t)|\leq Ch^2,\quad\forall s\in\mathcal{V}, \forall t\in\R^n.$$

\subsubsection{Full asymptotic expansion}
The first step is to find
the function $\bn_1$ built in the previous section. For this we look at the next
transport equation, which reads
$$
(\Lc_{0}-\mu_{0})\an_{2}=-(\Lc_{1}-\lambda_{1})\an_{1} - (\Lc_{2}- \lambda_2) \an_0.$$
We look at the compatibility Fredholm condition
which gives, pointwise in $s$,
$$
(\Lc_{1}-\lambda_{1})\an_{1} + (\Lc_2 -\lambda_2) \an_0\in ({\rm Ker}(\Lc_{0}-\mu_{0})^*)^\perp.
$$
This condition is equivalent to
\begin{equation*}
\begin{split}
& \left\langle (\Lc_1-\lambda_{1}) \an_1(s),u^\natural_{s,-i\nabla\overline{\Phi}(s)}
\right\rangle_{\sL^2(\R^m,\dx t)} \\ & =
- \left\langle \Lc_2 a_0(s) ,u^\natural_{s,-i\nabla\overline{\Phi}(s)}\right\rangle_{\sL^2(\R^m,\dx t)}  +
\lambda_2 \left\langle  a_0(s) ,u^\natural_{s,-i\nabla\overline{\Phi}(s)}
\right\rangle_{\sL^2(\R^m,\dx t)} \\
& =  - \left\langle \Lc_2 a_0(s) ,u^\natural_{s,-i\nabla\overline{\Phi}(s)}\right\rangle_{\sL^2(\R^m,\dx t)}  +
\lambda_2 b_0(s), \qquad\qquad \forall s\in\R^m,
\end{split}
\end{equation*}
where we used (\ref{simplb1}). Using the splitting \eqref{splita1} and the expression for $\an_0$ in \eqref{eq.a0}
we get that for all $s \in \R^m$,
 \begin{equation*}
\begin{split}
& \left\langle (\Lc_1-\lambda_{1}) u^\natural_{s,i\nabla\Phi(s)}\bn_{1}(s),u^\natural_{s,-i\nabla\overline{\Phi}(s)}
\right\rangle_{\sL^2(\R^m,\dx t)} \\
& = - \left\langle \Lc_2 a_0(s) ,u^\natural_{s,-i\nabla\overline{\Phi}(s)}\right\rangle_{\sL^2(\R^m,\dx t)}  +
\lambda_2 \left\langle  a_0(s) ,u^\natural_{s,-i\nabla\overline{\Phi}(s)}
\right\rangle_{\sL^2(\R^m,\dx t)} \\
 & \qquad \qquad \qquad \qquad \qquad \qquad \qquad \qquad \qquad + \left\langle (\Lc_1-\lambda_{1}) \tilde{\an}_1(s),u^\natural_{s,-i\nabla\overline{\Phi}(s)}\right\rangle_{\sL^2(\R^m,\dx t)} \\
& =  - \left\langle \Lc_2 a_0(s) ,u^\natural_{s,-i\nabla\overline{\Phi}(s)}\right\rangle_{\sL^2(\R^m,\dx t)}  
+ \lambda_2 b_0(s) 
+ \left\langle (\Lc_1-\lambda_{1}) \tilde{\an}_1(s),u^\natural_{s,-i\nabla\overline{\Phi}(s)}
\right\rangle_{\sL^2(\R^m,\dx t)}.
\end{split}
\end{equation*}
 Using the definition of the reduced transport introduced in (\ref{maintransport}), this equation reads
 $$
 \mathscr{T} b_1 = \lambda_{1} b_1 + \lambda_2 b_0+ R_1,
 $$
 where $\lambda_2$ has to be determined and $R_1$ is already known and defined by 
 $$
 R_1(s) = - \left\langle \Lc_2 a_0(s) ,u^\natural_{s,-i\nabla\overline{\Phi}(s)}\right\rangle_{\sL^2(\R^m,\dx t)} 
 + \left\langle (\Lc_1-\lambda_{1}) \tilde{\an}_1(s),u^\natural_{s,-i\nabla\overline{\Phi}(s)}\right\rangle_{\sL^2(\R^m,\dx t)}.
$$
Using again the theory of formal series, this completely determines the value of $\lambda_2$ as well as the Taylor expansion of $b_{1}$. Now, exactly as we did for $b_0$ in the previous section, we can also solve in $\ccc^\infty$ the transport equation on $b_1$.

We can look for the next  function $\an_{2}$ in the form
\begin{equation} \label{splita1bis}
\an_{2}(s,t) = u^\natural_{s,i\nabla\Phi(s)}(t) \bn_{2}(s) + \tilde{\an}_{2} (s,t)
\end{equation}
with
$$
\tilde{\an}_2 \in ({\rm Ker}(\Lc_{0}-\mu_{0}))^\perp.
$$
From the transport equation, we see that we can directly find
$$
\tilde{\an}_{2} = -(\Lc_{0}-\mu_{0})^{-1} (\Lc_{1}-\lambda_{1})\an_{1}.
$$
At this stage, whatever $b_2$ is, we have found a function
$$
\an^{[2]}(s,t) = \an_0(s,t) + h \an_1(s,t) + h^2\an_2(s,t)
$$
such that
$$
(\Lc_h - \mu_0 -h\lambda_{1}-h^2 \lambda_2) \an^{[2]} \re^{-\Phi/h} = r^{[2]}   \re^{-\Phi/h}
$$
where, for all $\alpha\in\N^m, \beta \in \N^{n}$,
$$|\D^\alpha_{s}\D^\beta_{t} r^{[2]}| = \Oc(h^3), \quad\mbox{ locally in } s.$$
The same procedure can be continued at any order and we are able to find a full family of functions $\an_{j }$ and $\lambda_{j}$ such that
the formal series
$$
\an \sim \sum_{j\geq 0}\an_{j} h^j \qquad \lambda\sim \sum_{j\geq 0}\lambda_{j} h^{j},
$$
solve the equation
$$(\Lc_h -\lambda)\an \re^{-\Phi/h} = \Oc( h^{\infty}) \re^{-\Phi/h}.$$
Then we consider a truncated (with respect to $s$) $\mathcal{C}^\infty$ realization of $\an$ and we apply the spectral theorem. Using Theorem \ref{theorem-simple-well}, this completes the proof of Theorem \ref{WKB-general}.

\section{Generalized Montgomery operators}\label{Sec:Mont}
We focus on the operator
$$\mathfrak{L}_{h}^{[k]}=D_{t}^2+\left(hD_{s}-\gamma(s)\frac{t^{k+1}}{k+1}\right)^2$$
obtained in (\ref{defgM}) after the rescaling described there. We denote by
$\lambda_n^{[k]}(h)$ its $n$-th eigenvalue if it exists. 

\subsection{Verifying assumptions}\label{S:verify-assumptions}
The aim of this section is to prove that the operator $\mathfrak{L}_{h}^{[k]}$
 fulfills the assumptions mentioned in Section~\ref{sec.mainresults}.
\begin{definition}
If $k\geq 0$ is an integer, we let $\Omega_{k}=\R$ for $k\geq 1$ and $\Omega_{k}=\R_{+}$ for $k=0$. Let us introduce the generalized Montgomery operators $\mathfrak{h}^{[k]}_{\zeta}$ as the self-adjoint realization, on $\sL^2(\Omega_{k},\dx \tau)$, of
\begin{equation}\label{Montgomery}
D_{\tau}^2+\left(\zeta-\frac{\tau^{k+1}}{k+1}\right)^2,
\end{equation}
with Neumann condition in the case $k=0$. Let $u^{[k]}_{\zeta}$ be an $\sL^2$-normalized eigenfunction depending analytically on $\zeta$ and associated with the first eigenvalue $\nun^{[k]}(\zeta)$. 

For $k=0$, $\mathfrak{h}^{[k]}_{\zeta}$ is nothing but the de Gennes operator (see \cite{DauHel} or \cite[Chapter 3]{FouHel10}) and in this case we omit the superscript $[0]$.
 \end{definition}

\begin{proposition}\label{verify-assumptions}
Let us assume that either $\gamma$ is polynomial and admits a unique minimum $\gamma_{0}>0$ at $s_{0}=0$ which is non degenerate, or $\gamma$ is analytic with  $\liminf_{x\to\pm\infty}\gamma=\gamma_{\infty}\in(\gamma_{0},+\infty)$.
For $k\in\N\setminus\{0\}$, the operator $\mathfrak{L}_{h}^{[k]}$ satisfies Assumptions \ref{hyp-gen}, \ref{hyp-gen'} and \ref{confining}. 
\end{proposition}
This proposition is proved in the following two sections.

\subsubsection{Uniqueness and non-degeneracy}
The symbol of $\mathfrak{L}_{h}^{[k]}$ with respect to $s$ is
$$\mathcal{M}^{[k]}_{x,\xi}=D_{t}^2+V(x,\xi),\quad \mbox{ with } V(x,\xi)=\left(\xi-\gamma(x)\frac{t^{k+1}}{k+1}\right)^2.$$
The family $(\mathcal{M}^{[k]}_{x,\xi})_{(x,\xi)\in\R^{2m}}$ is clearly analytic of type (A). Given a point $(x_{1},\xi_{1})\in\R^{2m}$, there exist $c,C>0$ and a $\C$-neighborhood $\mathcal{V}$ of $(x_{1},\xi_{1})$ such that,
$$\Re V(x,\xi)\geq -C +cV(x_{1},\xi_{1})\geq -C,\qquad \forall (x,\xi)\in\mathcal{V}.$$
Therefore, for $(x,\xi)\in\mathcal{V}$, the operator $\mathcal{M}_{x,\xi}$ is well defined thanks to the Lax-Milgram theorem. This extension is holomorphic of type $(A)$.
The lowest eigenvalue of $\mathcal{M}^{[k]}_{x,\xi}$, denoted by $\mun^{[k]}(x,\xi)$, satisfies
$$\mun^{[k]}(x,\xi)=(\gamma(x))^{\frac{2}{k+2}}\nun^{[k]}\left((\gamma(x))^{-\frac{1}{k+2}}\xi\right).$$
It is proved in \cite[Theorem 1.3]{FouPer13} that $\R\ni\zeta\mapsto\nun^{[k]}(\zeta)$ admits a unique and non degenerate minimum.
Therefore Assumption \ref{hyp-gen} is satisfied. Note that Assumption~\ref{hyp-gen'} is satisfied since $m=1$ according to Remark~\ref{rem.m1}.
\begin{notation}
We denote by $\zeta_{0}^{[k]}$ the point $\zeta$ where the minimum of $\nun^{[k]}$ is reached.
\end{notation}

\subsubsection{Confinement}\label{ss-verify-conf}
This verification is a little more delicate. It is based on a normal form procedure. Let us denote by $\mathfrak{Q}_{h}^{[k]}$ the quadratic form associated with $\mathfrak{L}_{h}^{[k]}$. For $\psi\in\Dom(\mathfrak{Q}_{h}^{[k]})$, we have
$$\mathfrak{Q}_{h}^{[k]}(\psi)=\int_{\R^2} |D_{t}\psi|^2+\left|\left(hD_{s}-\gamma(s)\frac{t^{k+1}}{k+1}\right)\psi\right|^2\dx t \dx s.$$
We would like to prove a lower bound for $\mathfrak{Q}_{h}^{[k]}(\psi)$ when $\psi\in\mathcal{C}_{0}^\infty(\R^2)$ is supported away from the square $[-R_{0},R_{0}]^2$.
\paragraph{Magnetic confinement for the variable $t$.}
If $\psi$ is supported in $\{(s,t)\in\R^2: |t|>R\}$, we can use the standard inequality (see for instance \cite[Lemma 1.4.1]{FouHel10}):
\begin{equation}\label{min-loin-t}
\mathfrak{Q}_{h}^{[k]}(\psi)\geq\int_{\R^2}\gamma(s)|t|^k |\psi|^2\dx s \dx t\geq R^k\gamma_{0}\|\psi\|^2.
\end{equation}
\paragraph{Spectral confinement for the variable $s$.}
If $\psi$ is supported in the set $\{(s,t)\in\R^2: |t|<R,\ |s|\geq R\}$, we use the canonical transformation associated with the change of variables
\begin{equation}\label{can}
t=(\gamma(\sigma))^{-\frac{1}{k+2}}\tau,\quad s=\sigma,
\end{equation}
we deduce that  $\mathfrak{L}_{h}^{[k]}$ is unitarily equivalent to the operator on $\sL^2(\R^2,\dx \sigma\dx \tau)$
$$\mathfrak{L}_{h}^{[k],\new}=\gamma(\sigma)^{\frac{2}{k+2}}D_{\tau}^2+\left(hD_{\sigma}-\gamma(\sigma)^{\frac{1}{k+2}}\frac{\tau^{k+1}}{k+1}+\frac{h}{2(k+2)}\frac{\gamma'(\sigma)}{\gamma(\sigma)}(\tau D_{\tau}+D_{\tau}\tau)\right)^2.$$
Let us denote by $\psi^\new$ the function $\psi$ transported by the canonical transformation. In terms of the quadratic from, we have
$$\mathfrak{Q}_{h}^{[k]}(\psi)=\mathfrak{Q}_{h}^{[k],\new}(\psi^\new).$$
Let us notice that $\psi^\new$ is supported in $\left\{(\sigma,\tau)\in\R^2: |\sigma|\geq R, \ |\tau|\leq R \gamma(\sigma)^{\frac{1}{k+2}}\right\}$. We can write, for all $\eps\in(0,1)$
\begin{multline}\label{lb-normal}
\mathfrak{Q}_{h}^{[k],\new}(\psi^\new)\geq
(1-\eps)\int_{\R^2}\gamma(\sigma)^{\frac{2}{k+2}}\left(|D_{\tau}\psi^\new|^2+\left|\left(h\gamma(\sigma)^{-\frac{1}{k+2}}D_{\sigma}-\frac{\tau^{k+1}}{k+1}\right)\psi^\new\right|^2\right)\dx \sigma \dx \tau\\
-\frac{\eps^{-1}h^2}{(2(k+2))^2}\int_{\R^2}\left|\frac{\gamma'(\sigma)}{\gamma(\sigma)}(\tau D_{\tau}+D_{\tau}\tau)\psi^\new\right|^2 \dx \sigma\dx \tau.
\end{multline}
In the analytical case for $\gamma$,
there exists $\eta_{0}>0$ and for all $\eta\in(0,\eta_{0})$, $R_{0}>0$ such that,  on $\{|\sigma|\geq R_{0}\}$
\begin{equation}\label{min-gamma}
\gamma(\sigma)^{\frac{2}{k+2}}\min_{\zeta\in\R} \nun^{[k]}(\zeta)
\geq \left\{\gamma_{\infty}^{\frac{2}{k+2}}-\eta\right\}\min_{\zeta\in\R} \nun^{[k]}(\zeta)>\mun^{[k]}(x_{0},\xi_{0}).
\end{equation}
Note that in the polynomial case for $\gamma$, we can replace $\gamma_{\infty}$ by any positive constant larger than $\gamma_{0}$. 
Up to choosing $R_{0}$ larger, we may also assume that
\begin{equation}\label{min-gamma'}
\left(\frac{R_{0}}{2}\right)^k\gamma_{0}\geq
\left\{\gamma_{\infty}^{\frac{2}{k+2}}-\eta\right\}\min_{\zeta\in\R} \nun^{[k]}(\zeta).
\end{equation}
Moreover, we have
\begin{multline}\label{E:g'g}
\int_{\R^2}\left|\frac{\gamma'(\sigma)}{\gamma(\sigma)}(\tau D_{\tau}+D_{\tau}\tau)\psi^\new\right|^2 \dx \sigma\dx \tau\leq C\int_{\R^2}|\psi^\new|^2 \dx \sigma \dx\tau\\+4\int_{\R^2}\left(\frac{\gamma'(\sigma)}{\gamma(\sigma)}\right)^2|\tau D_{\tau}\psi^\new|^2 \dx \sigma \dx\tau.
\end{multline}
Using support considerations, we get
\begin{multline*}
\int_{\R^2}\left|\frac{\gamma'(\sigma)}{\gamma(\sigma)}(\tau D_{\tau}+D_{\tau}\tau)\psi^\new\right|^2 \dx \sigma\dx \tau\\
\leq C\int_{\R^2}|\psi^\new|^2 \dx \sigma \dx\tau+4R^2\int_{\R^2}\gamma(\sigma)^{\frac{2}{k+2}-2}\gamma'(\sigma)^2|D_{\tau}\psi^\new|^2 \dx \sigma \dx\tau\\
\leq C\int_{\R^2}|\psi^\new|^2 \dx \sigma \dx\tau+4CR^2\int_{\R^2}|D_{\tau}\psi^\new|^2 \dx \sigma \dx\tau.
\end{multline*}
We deduce
\begin{multline*}
\mathfrak{Q}_{h}^{[k],\new}(\psi^\new)\geq
(1-\eps)\left(\gamma_{\infty}^{\frac{2}{k+2}}-\eta\right) \int_{\R^2} \Big(|D_{\tau}\psi^\new|^2+\Big|\Big(h\gamma(\sigma)^{-\frac{1}{k+2}}D_{\sigma}-\frac{\tau^{k+1}}{k+1}\Big)\psi^\new\Big|^2\Big)\dx \sigma \dx \tau\\
-C\eps^{-1}h^2\int_{\R^2}\left|\psi^\new\right|^2 \dx \sigma\dx \tau-4CR^2\eps^{-1}h^2\int_{\R^2}\left|D_{\tau}\psi^\new\right|^2 \dx \sigma\dx \tau.
\end{multline*}
We choose $\eps=h$, we infer that
\begin{equation}\label{min-Q-q}
\mathfrak{Q}_{h}^{[k],\new}(\psi^\new)\geq
\left\{(1-h)\left(\gamma_{\infty}^{\frac{2}{k+2}}-\eta\right)-4CR^2 h \right\}\mathfrak{q}_{h}^{[k],\new}(\psi^\new)\\
-Ch\int_{\R^2}\left|\psi^\new\right|^2 \dx \sigma\dx \tau,
\end{equation}
where
$$\mathfrak{q}_{h}^{[k],\new}(\psi^\new)=\int_{\R^2}\left(|D_{\tau}\psi^\new|^2+\left|\left(h\gamma(\sigma)^{-\frac{1}{k+2}}D_{\sigma}-\frac{\tau^{k+1}}{k+1}\right)\psi^\new\right|^2\right)\dx \sigma \dx \tau.$$
We have
\begin{multline}\label{before-self}
\mathfrak{q}_{h}^{[k],\new}(\psi^\new)=\\
\int_{\R^2}\left(|D_{\tau}\psi^\new|^2+\left|\left(h\Xi(\sigma,D_{\sigma})-\frac{\tau^{k+1}}{k+1}-\frac{i h}{2k+4}\gamma'(\sigma)\gamma(\sigma)^{-\frac{k+3}{k+2}}\right)\psi^\new\right|^2\right)\dx \sigma \dx \tau,
\end{multline}
with $\Xi(\sigma,D_{\sigma})=\gamma(\sigma)^{-\frac{1}{2k+4}}D_{\sigma}\gamma(\sigma)^{-\frac{1}{2k+4}}$. We get the lower bound:
\begin{multline*}
\mathfrak{q}_{h}^{[k],\new}(\psi^\new)\geq \int_{\R^2}\left(|D_{\tau}\psi^\new|^2+\left|\left(h\Xi(\sigma,D_{\sigma})-\frac{\tau^{k+1}}{k+1}\right)\psi^\new\right|^2\right)\dx \sigma \dx \tau\\
+\frac{2h}{2k+4}\Re\int_{\R^2} i\gamma'(\sigma)\gamma(\sigma)^{-\frac{k+3}{k+2}}\left(h\Xi(\sigma,D_{\sigma})-\frac{\tau^{k+1}}{k+1}\right)\psi^\new\,\overline{\psi^\new} \dx\sigma \dx \tau.
\end{multline*}
This becomes
\begin{multline*}
\mathfrak{q}_{h}^{[k],\new}(\psi^\new)\geq \int_{\R^2}\left(|D_{\tau}\psi^\new|^2+\left|\left(h\Xi(\sigma,D_{\sigma})-\frac{\tau^{k+1}}{k+1}\right)\psi^\new\right|^2\right)\dx \sigma \dx \tau\\
+\frac{2h^2}{2k+4}\Re\int_{\R^2} i\gamma'(\sigma)\gamma(\sigma)^{-\frac{k+3}{k+2}}\Xi(\sigma,D_{\sigma})\psi^\new\,\overline{\psi^\new} \dx\sigma \dx \tau.
\end{multline*}
By using in particular that $2\Re\left(\dr_{\sigma}\psi^\new\, \overline{\psi^\new}\right)=\dr_{\sigma}|\psi^\new|^2$ and by integrating by parts, we infer
$$\left|\frac{2h^2}{2k+4}\Re\int_{\R^2} i\gamma'(\sigma)\gamma(\sigma)^{-\frac{k+3}{k+2}}\Xi(\sigma,D_{\sigma})\psi^\new\,\overline{\psi^\new} \dx\sigma \dx \tau\right|\leq Ch^2\int_{\R^2}|\psi^\new|^2 \dx\sigma\dx\tau.$$
By using the functional calculus, we get
\begin{equation}\label{min-q}
\mathfrak{q}_{h}^{[k],\new}(\psi^\new)\geq \left(\min_{\zeta\in\R} \nun^{[k]}(\zeta) -Ch^2\right)\int_{\R^2} |\psi^\new|^2\dx\sigma\dx\tau.
\end{equation}
Fixing $R=R_{0}$ and $\eta_{0}$ defined in \eqref{min-gamma} and \eqref{min-gamma'} and combining \eqref{min-Q-q} and \eqref{min-q}, we infer the existence of $h_{0}>0$ and $C>0$ such that for $h\in(0,h_{0})$ and all $\eta\in(0,\eta_{0})$
$$\mathfrak{Q}_{h}^{[k],\new}(\psi^\new)\geq\left(\gamma_{\infty}^{\frac{2}{k+2}}\min_{\zeta\in\R} \nun^{[k]}(\zeta)-\eta    -Ch \right)\int_{\R^2}|\psi^\new|^2\dx \sigma\dx\tau$$
or equivalently
$$\mathfrak{Q}_{h}^{[k]}(\psi)\geq \left(\gamma_{\infty}^{\frac{2}{k+2}}\min_{\zeta\in\R} \nun^{[k]}(\zeta)-\eta    -Ch \right) \int_{\R^2}|\psi|^2\dx s\dx t,$$
for all $\psi\in\mathcal{C}^\infty_{0}(\R^2)$ supported in $\{(s,t)\in\R^2: |t|<R_{0},\ |s|\geq R_{0}\}$.
\paragraph{Gluing the lower bounds.} Using \eqref{min-loin-t} (with $R=\frac{R_{0}}{2}$) and \eqref{min-gamma'}, we deduce
$$\mathfrak{Q}_{h}^{[k]}(\psi)\geq \left(\gamma_{\infty}^{\frac{2}{k+2}}\min_{\zeta\in\R} \nun^{[k]}(\zeta)-\eta    -Ch \right) \int_{\R^2}|\psi|^2\dx s\dx t,$$
for all $\psi\in\mathcal{C}^\infty_{0}(\R^2)$ supported either in $\{(s,t)\in\R^2: |t|<R_{0},\ |s|\geq R_{0}\}$ or in $\{(s,t)\in\R^2 : |t|>\frac{R_{0}}{2}\}$. We now use a standard partition of unity with respect to $t$ such that
$$\chi_{1,R_{0}}^2+\chi_{2,R_{0}}^2=1,\qquad
\chi_{1,R_{0}}=\begin{cases}1 \mbox{ for }|t|\leq\frac{R_{0}}{2},\\ 0\mbox{ for }|t|\geq R_{0},\end{cases}\quad\mbox{ and }
\left(\chi'_{1,R_{0}}\right)^2+\left(\chi'_{2,R_{0}}\right)^2\leq CR_{0}^{-2}.$$
The \enquote{IMS} formula provides, for all $\psi\in\mathcal{C}_{0}^{\infty}\left(\R^2\setminus [-R_{0},R_{0}]^2\right)$:
$$\mathfrak{Q}_{h}^{[k]}(\psi)\geq \mathfrak{Q}_{h}^{[k]}(\chi_{1,R_{0}}\psi)+\mathfrak{Q}_{h}^{[k]}(\chi_{2,R_{0}}\psi)-CR_{0}^{-2}\|\psi\|^2.$$
By supports considerations we have
$$\supp (\chi_{1,R_{0}}\psi)\subset \{(s,t)\in\R^2 : |t|<R_{0},\ |s|\geq R_{0}\},\quad \supp (\chi_{2,R_{0}}\psi)\subset \{(s,t)\in\R^2 : |t|\geq \tfrac{R_{0}}{2}\},$$
so that, for all $\psi\in\mathcal{C}_{0}^{\infty}\left(\R^2\setminus [-R_{0},R_{0}]^2\right)$, there holds
$$\mathfrak{Q}_{h}^{[k]}(\psi)\geq  \left(\gamma_{\infty}^{\frac{2}{k+2}}\min_{\zeta\in\R} \nun^{[k]}(\zeta)-\eta    -Ch-CR_{0}^{-2} \right) \int_{\R^2}|\psi|^2\dx s\dx t.$$
Therefore, for all $\mu_{0}^*\in\left(\mu_{0},\gamma_{\infty}^{\frac{2}{k+2}}\min \nun^{[k]}\right)$, up to choosing $\eta$, $h$ small enough and possibly $R_{0}$ larger, we have
$$\mathfrak{Q}_{h}^{[k]}(\psi)\geq \mu_{0}^* \int_{\R^2}|\psi|^2\dx s\dx t,\quad \forall\psi\in\mathcal{C}_{0}^{\infty}\left(\R^2\setminus [-R_{0},R_{0}]^2\right).$$

\subsection{Explicit WKB expansions, weak magnetic barrier}

\subsubsection{Renormalization}
The key point to perform the spectral analysis of $\mathfrak{L}_{h}^{[k]}$ is the normal form procedure introduced in \cite{Ray11b, DomRay12, PoRay12}, see also \cite{RVN13} which is pervaded by this spirit. Let us explain this basic idea. We again use the canonical transformation associated with the change of variables
$$t=(\gamma(\sigma))^{-\frac{1}{k+2}}\tau,\qquad s=\sigma.$$
We deduce that  $\mathfrak{L}_{h}^{[k]}$ is unitarily equivalent to the operator on $\sL^2(\dx \sigma\dx \tau)$
$$\mathfrak{L}_{h}^{[k],\new}=\gamma(\sigma)^{\frac{2}{k+2}}D_{\tau}^2+\left(hD_{\sigma}-\gamma(\sigma)^{\frac{1}{k+2}}\frac{\tau^{k+1}}{k+1}+\frac{h}{2(k+2)}\frac{\gamma'(\sigma)}{\gamma(\sigma)}(\tau D_{\tau}+D_{\tau}\tau)\right)^2.$$
In order to estimate the low lying eigenvalues, we may write the heuristic approximation
\begin{eqnarray*}
\mathfrak{L}_{h}^{[k],\new}&\approx&\gamma(\sigma)^{\frac{2}{k+2}}\left(D_{\tau}^2+\left(h\gamma(\sigma)^{-\frac{1}{k+2}}D_{\sigma}-\frac{\tau^{k+1}}{k+1}\right)^2\right)\\
                                               &\approx&\gamma(\sigma)^{\frac{2}{k+2}}\nun^{[k]}\left(h\gamma(\sigma)^{-\frac{1}{k+1}}D_{\sigma}\right).
\end{eqnarray*}
Let now make this approximation more rigorous. We may change the gauge
\begin{align*}
&\re^{-ig(\sigma)/h}\ \mathfrak{L}_{h}^{[k],\new}\ \re^{ig(\sigma)/h}\\
&=\gamma(\sigma)^{\frac{2}{k+2}}D_{\tau}^2+\left(hD_{\sigma}+\zeta^{[k]}_{0}\gamma(\sigma)^{\frac{1}{k+2}}-\gamma(\sigma)^{\frac{1}{k+2}}\frac{\tau^{k+1}}{k+1}+\frac{h}{2(k+2)}\frac{\gamma'(\sigma)}{\gamma(\sigma)}(\tau D_{\tau}+D_{\tau}\tau)\right)^2,
\end{align*}
with
$$g(\sigma)=\zeta_{0}^{[k]}\int_{0}^{\sigma}\gamma(\tilde\sigma)^{\frac{1}{k+2}}\dx\tilde\sigma.$$
For some function $\Phi=\Phi(\sigma)$ to be determined, we consider
$$\mathfrak{L}^{[k],\wgt}_{h}=\re^{\Phi/h}\re^{-ig(\sigma)/h} \mathfrak{L}_{h}^{[k],\new}\ \re^{ig(\sigma)/h}\re^{-\Phi/h}=\mathfrak{L}^{[k],\wgt,0}+h\mathfrak{L}^{[k],\wgt,1}+h^2\mathfrak{L}^{[k],\wgt,2},$$
with
\begin{align*}
&\mathfrak{L}^{[k],\wgt,0}=\gamma^{\frac{2}{k+2}}\left( D_{\tau}^2+\left(V_{\zeta^{[k]}_{0}}+i\gamma^{-\frac{1}{k+2}}\Phi'\right)^2\right),\\
&\mathfrak{L}^{[k],\wgt,1}=\left(\gamma^{\frac{1}{k+2}}V_{\zeta^{[k]}_{0}}+i\Phi'\right)D_{\sigma}+D_{\sigma}\left(\gamma^{\frac{1}{k+2}}V_{\zeta^{[k]}_{0}}+i\Phi'\right)+\mathfrak{R}_{1}(\sigma, \tau; D_{\tau}),\\
&\mathfrak{L}^{[k],\wgt,2}=D_{\sigma}^2+\mathfrak{R}_{2}(\sigma,\tau; D_{\sigma}, D_{\tau}),
\end{align*}
where
$$V_{\zeta}(\tau)=\zeta-\frac{\tau^{k+1}}{k+1},$$
and
where $\mathfrak{R}_{1}(\sigma, \tau; D_{\tau})$ is of order zero in $D_{\sigma}$ and cancels for $\sigma=0$ whereas $\mathfrak{R}_{2}(\sigma,\tau; D_{\sigma}, D_{\tau})$ is an operator of order one with respect to the variable $\sigma$.

Now, let us solve, as in the previous section, the eigenvalue equation
$$\mathfrak{L}^{[k],\wgt}_{h}\an=\lambda\an$$
in the sense of formal series in $h$,
$$\an\sim\sum_{j\geq 0}h^j\an_{j},\qquad \lambda\sim\sum_{j\geq 0}h^j\lambda_{j}.$$

\subsubsection{Solving the operator valued eikonal equation}
The first equation is
$$\mathfrak{L}^{[k],\wgt,0}\an_{0}=\lambda_{0}\an_{0}.$$
We must choose
$$\lambda_{0}=\gamma_{0}^{\frac{2}{k+2}}{\nun^{[k]}}\big(\zeta_{0}^{[k]}\big),$$
and we are led to take
\begin{equation}\label{psi0}
\an_{0}(\sigma,\tau)=f_{0}(\sigma)u^{[k]}_{w(\sigma)}(\tau),\qquad
\mbox{Êwith }\quad w(\sigma)=\zeta^{[k]}_{0}+i\gamma(\sigma)^{-\frac{1}{k+2}}\Phi'(\sigma).
\end{equation}
Recall that $u^{[k]}_{\zeta}$ (with $\zeta\in\R$) denotes an $\sL^2$-normalized eigenfunction associated with $\nun^{[k]}(\zeta)$ for the generalized Montgomery operator \eqref{Montgomery} and that it admits a holomorphic extension near $\zeta_{0}$. Then the equation becomes
$$\gamma(\sigma)^{\frac{2}{k+2}}\nun^{[k]}\left(\zeta^{[k]}_{0}+i\gamma(\sigma)^{-\frac{1}{k+2}}\Phi'(\sigma)\right)
=\gamma_{0}^{\frac{2}{k+2}}\nun^{[k]}\Big(\zeta_{0}^{[k]}\Big),$$
and this can be written in the form
$$\nun^{[k]}\left(\zeta^{[k]}_{0}+i\gamma(\sigma)^{-\frac{1}{k+2}}\Phi'(\sigma)\right)-\nun^{[k]}\Big(\zeta_{0}^{[k]}\Big)
=\left(\gamma_{0}^{\frac{2}{k+2}}\gamma(\sigma)^{-\frac{2}{k+2}}-1\right)\nun^{[k]}\Big(\zeta_{0}^{[k]}\Big).$$
Therefore we are in the framework of the following elementary lemma.
\begin{lemma}\label{lem.Morse}
For $r>0$, let us consider a holomorphic function $\nu : \mathcal B(0,r)\to\mathbb C$ such that $\nu(0)=\nu'(0)=0$ and $\nu''(0)\in\R_{+}$. Let us also introduce a smooth and real-valued function $F$ defined in a real neighborhood of $\sigma=0$ such that $\sigma=0$ is a non degenerate maximum. Then, there exists a neighborhood of $\sigma=0$ such that the equation
\begin{equation}\label{eikonal}
\nu(i\varphi(\sigma))=F(\sigma)
\end{equation}
admits a smooth solution $\varphi$  such that $\varphi(0)=0$ and $\varphi'(0)>0$.
\end{lemma}
\begin{proof}
We can apply the Morse lemma to deduce that \eqref{eikonal} is equivalent to
$$\widetilde\nu(i\varphi(\sigma))^2=-f(\sigma)^2,$$
where $f$ is a non negative function such that $f'(0)=\sqrt{-\frac{F''(0)}{2}}$ and $F(\sigma)=-f(\sigma)^2$ and $\widetilde\nu$ is a holomorphic function in a neighborhood of $0$ such that $\widetilde\nu^2=\nu$ and $\widetilde\nu'(0)=\sqrt{\frac{\nu''(0)}{2}}$. This provides the equations
$$\widetilde\nu(i\varphi(\sigma))=if(\sigma),\quad \mbox{ or }\quad \widetilde\nu(i\varphi(\sigma))=-if(\sigma).$$
Since $\widetilde\nu$ is a local biholomorphism and $f(0)=0$, we can write the equivalent equations
$$\varphi(\sigma)=-i\widetilde\nu^{-1}(if(\sigma)),\quad\mbox{ or }\quad \varphi(\sigma)=-i\widetilde\nu^{-1}(-if(\sigma)).$$
The function $\varphi=-i\widetilde\nu^{-1}(if)$ satisfies our requirements since $\varphi'(0)=\sqrt{-\frac{F''(0)}{\nu''(0)}}$.
\end{proof}
We use Lemma \ref{lem.Morse} with $F(\sigma)=\left(\gamma_{0}^{\frac{2}{k+2}}\gamma(\sigma)^{-\frac{2}{k+2}}-1\right)\nun^{[k]}\big(\zeta_{0}^{[k]}\big)$ and, for the function $\varphi$ given by the lemma, we have (up to a translation by $\zeta_{0}^{[k]}$):
$$\Phi'(\sigma)=\gamma(\sigma)^{\frac{1}{k+2}}\varphi(\sigma)$$
and we take
$$\Phi(\sigma)=\int_{0}^\sigma \gamma(\tilde\sigma)^{\frac{1}{k+2}}\varphi(\tilde\sigma)\dx \tilde\sigma,$$
which is defined in a fixed neighborhood of $0$ and satisfies $\Phi(0)=\Phi'(0)=0$ and
\begin{equation}\label{Phi''}
\Phi''(0)=\gamma_{0}^{\frac{1}{k+2}}\sqrt{\frac{2}{k+2}\frac{\gamma''(0)\ \nun^{[k]}\big(\zeta_{0}^{[k]}\big)}{\gamma(0)\ \left(\nun^{[k]}\right)''\big(\zeta_{0}^{[k]}\big) }}>0.
\end{equation}
Therefore \eqref{psi0} is well defined in a neighborhood of $\sigma=0$.

\subsubsection{Solving the transport equation}
We can now deal with the operator valued transport equation
$$(\mathfrak{L}^{[k],\wgt,0}-\lambda_{0})\an_{1}=(\lambda_{1}-\mathfrak{L}^{[k],\wgt,1})\an_{0}.$$
For each $\sigma$ the Fredholm condition is
$$\left\langle(\lambda_{1}-\mathfrak{L}^{[k],\wgt,1})\an_{0}, \overline{u^{[k]}_{w(\sigma)}}\right\rangle_{\sL^2(\R,\dx\tau)}=0.$$
Using \eqref{u^2}, \eqref{barbar} and a Feynman-Hellmann formula (as in Section \ref{sec:WKB}), we get the transport equation
\begin{multline*}
\lambda_{1}f_{0}=\left\langle\mathfrak{L}^{[k],\wgt,1}\an_{0}, \overline{u^{[k]}_{w(\sigma)}}\right\rangle_{\sL^2(\R,\dx\tau)}\\
=\tfrac 1 2\left\{\gamma(\sigma)^{\frac{1}{k+2}}\left(\nun^{[k]}\right)'(w(\sigma))D_{\sigma}+D_{\sigma}\gamma(\sigma)^{\frac{1}{k+2}}\left(\nun^{[k]}\right)'(w(\sigma))\right\}f_{0}\\
+\left\langle\mathfrak{R}_{1}u^{[k]}_{w(\sigma)}, \overline{u^{[k]}_{w(\sigma)}}\right\rangle_{\sL^2(\R,\dx\tau)}f_{0}.
\end{multline*}
The term
$$\left\langle\mathfrak{R}_{1}u^{[k]}_{w(\sigma)}, \overline{u^{[k]}_{w(\sigma)}}\right\rangle_{\sL^2(\R,\dx\tau)}$$
is just a smooth function which cancels in $\sigma=0$ so that we have just to consider the linearization of the first part of the equation. The linearized operator is
$$\frac{1}{2}\left(\nun^{[k]}\right)''\big(\zeta_{0}^{[k]}\big)\Phi''(0)(\sigma\dr_{\sigma}+\dr_{\sigma}\sigma).$$
The eigenvalues of this operator in the corresponding weighted space are
\begin{equation}\label{evset}
\left\{\left(\nun^{[k]}\right)''\big(\zeta_{0}^{[k]}\big)\Phi''(0)\left(j+\tfrac{1}{2}\right),\quad j\in\N\right\}.
\end{equation}
Let us notice that
$$
\left(\nun^{[k]}\right)''\big(\zeta_{0}^{[k]}\big)\,\Phi''(0)=\gamma_{0}^{\frac{1}{k+2}}\sqrt{\frac{2}{k+2}\frac{\gamma''(0)\nun^{[k]}(\zeta_{0}^{[k]})\left(\nun^{[k]}\right)''\big(\zeta_{0}^{[k]}\big)}{\gamma(0)}}.
$$

\subsection{Estimates of Agmon in the normal form spirit}

\subsubsection{Weighted semiclassical elliptic estimates}
\begin{proposition}\label{weighted-Agmon}
Let us assume Assumption \ref{confining}. Let $C_{0}>0$ and $\Z:\R^m \to \R_{+}$ be a Lipschitzian function. There exist $\eps_{0}>0$, $\eps_{1}>0$, $h_{0}>0$ and $C>0$ such that for all eigenpairs $(\lambda,\psi)$ of $\mathfrak{L}_{h}$ satisfying $\lambda\leq \mu_{0}+C_{0}h$ we have, for all $p\geq 1$:
$$\left\| \re^{\eps_{1}|t|+\eps_{0}h^{-1}\chi_{p}\Z}\psi\right\|^2\leq C\|\re^{\eps_{0}h^{-1} \chi_{p} \Z}\psi\|^2,$$
$$\mathfrak{Q}_{h}\left(\re^{\eps_{1}|t|+\eps_{0}h^{-1}\chi_{p}\Z}\psi\right)\leq C\|\re^{\eps_{0}h^{-1} \chi_{p} \Z}\psi\|^2,$$
where $\chi_{p}(s)=\chi_{1}(p^{-1}s)$, with $0\leq \chi_{1}\leq 1$ a smooth cutoff function supported near $0$.
\end{proposition}
\begin{proof}
Let us consider an eigenpair $(\lambda,\psi)$ of $\mathfrak{L}_{h}$ such that $\lambda\leq \mu_{0}+C_{0}h$. We have the Agmon type formula
$$\mathfrak{Q}_{h}\left(\re^{\Phi}\psi\right)=\lambda\|\re^{\Phi}\psi\|^2+\|\dr_{t}\Phi \re^{\Phi}\psi\|^2+\|h\dr_{s}\Phi \re^{\Phi}\psi\|^2,$$
where $\Phi$ is bounded and Lipschitzian. We look for $\Phi$ in the form
$$\Phi(s,t)=\Phi_{p}(s,t)=\eps_{1}|t|+\eps_{0}h^{-1} \chi_{p}(s) \Z(s).$$
We get
$$\mathfrak{Q}_{h}\left(\re^{\Phi_{p}}\psi\right)\leq (\mu_{0}+C_{0}h+\eps_{1}^2)\|\re^{\Phi_{p}}\psi\|^2+2\eps_{0}^2\|\Z'\chi_{p}\re^{\Phi_{p}}\psi\|^2+2\eps_{0}^2\|\Z\chi_{p}' \re^{\Phi_{p}}\psi\|^2.$$
Since $\Z$ is Lipschitzian, there exists $K\geq 0$ such that for all $s\in\R^m$
$$|\Z(s)|\leq |c(0)|+K|s|.$$
Therefore there exists $\tilde C>0$ such that for all $p\geq 1$, $\eps_{0}>0$ and $\eps_{1}>0$ we have
$$\mathfrak{Q}_{h}\left(\re^{\Phi_{p}}\psi\right)\leq (\mu_{0}+C_{0}h+\eps_{1}^2+\tilde C\eps_{0}^2)\|\re^{\Phi_{p}}\psi\|^2,$$
where we have used that $s\mapsto \chi'_{p}(s) \Z(s)$ is uniformly bounded with respect to $p$.
We introduce the partition of the unity
$$\chi_{1,R}^2(t)+\chi_{2,R}^2(t)=1,$$
where $\chi_{2,R}$ is supported in $\{|t|\geq R\}$. We may assume that there exists $C>0$ such that for all $R>0$
$$\chi_{1,R}'^2+\chi_{1,R}'^2\leq CR^{-2}.$$
The so-called \enquote{IMS} formula implies that
$$\mathfrak{Q}_{h}(\chi_{1,R} \re^{\Phi_{p}}\psi)+\mathfrak{Q}_{h}(\chi_{2,R} \re^{\Phi_{p}}\psi)-CR^{-2}\|\re^{\Phi_{p}}\psi\|^2\leq (\mu_{0}+C_{0}h+\eps_{1}^2+C\eps_{0}^2)\|\re^{\Phi_{p}}\psi\|^2.$$
We choose $R$ large enough, $\eps_{0}>0$, $\eps_{1}>0$, $h_{0}>0$ small enough such that we have for $h\in(0,h_{0})$ and $p\geq 1$
$$\mathfrak{Q}_{h}(\chi_{2,R} \re^{\Phi_{p}}\psi)\geq \mu_{0}^*\|\chi_{2,R} \re^{\Phi_{p}}\psi\|^2$$
and
$$C_{0}h+\eps_{1}^2+\tilde C\eps_{0}^2+CR^{-2}<\frac{\mu_{0}^*-\mu_{0}}{2}.$$
For these choices of $R$, $\eps_{0}$, $\eps_{1}$ and $h_{0}$, we find $\hat c, \hat C>0$ such that for all $h\in(0,h_{0})$ and $p\geq 1$, we have
$$\hat c\|\chi_{2,R}\re^{\Phi_{p}}\psi\|^2\leq \hat C\|\re^{\eps_{0}\chi_{p} h^{-1} \Z}\psi\|^2,$$
and the conclusion easily follows.
\end{proof}

\subsubsection{Agmon estimates}\label{SS:Agmon}
\begin{proposition}\label{Agmon-puits-simple}
Let us assume that $\gamma$ is analytic with $\liminf_{\pm\infty}\gamma\in(\gamma_{0},+\infty)$ and a unique minimum which is non degenerate. 
We let
$$\displaystyle{\Z(s)=\left|\int_{s_{0}}^s \chi(s')\sqrt{\gamma(s')^{\frac{2}{k+2}}-\gamma_{0}^{\frac{2}{k+2}}}\dx s'\right|},$$
with $0\leq\chi\leq 1$ a smooth cutoff function whose support contains $s_{0}$. Let us consider $C_{0}>0$.
There exist $\eps_{0}>0$, $C>0$ and $h_{0}>0$ such that for all eigenpairs $(\lambda,\psi)$ of $\mathfrak{L}^{[k]}_{h}$ satisfying $\lambda\leq \mu_{0}+C_{0}h$ we have
$$\|\re^{\eps_{0}h^{-1}\Z}\psi\|\leq C\|\psi\|$$
and
$$\mathfrak{Q}_{h}^{[k]}(\re^{\eps_{0}h^{-1}\Z}\psi)\leq  C\|\psi\|^2.$$
\end{proposition}
\begin{proof}
Let us consider an eigenpair $(\lambda,\psi)$ of $\mathfrak{L}^{[k]}_{h}$ such that $\lambda\leq \mu_{0}+C_{0}h$. We first use the Agmon estimate
\begin{multline}
\mathfrak{Q}_{h}^{[k]}(\re^{\eps_{0}\chi_{p} h^{-1}\Z}\psi)\leq\lambda\|\re^{\eps_{0}\chi_{p} h^{-1}\Z}\psi\|^2+\tilde C\eps_{0}^2\|\Z'\re^{\eps_{0}\chi_{p} h^{-1}\Z}\psi\|^2
+\tilde C\eps_{0}^2\|\chi_{p}' \Z\re^{\eps_{0}\chi_{p} h^{-1}\Z}\psi\|^2.
\end{multline}
In order to simplify the notation, we consider the weighted $\psi$:
$$\psi^{\wgt}=\re^{\eps_{0}\chi_{p} h^{-1}\Z}\psi.$$
Now, we shall establish a very fine lower bound of $\mathfrak{Q}_{h}^{[k]}(\psi^{\wgt})$. For that purpose we use the normal form already introduced in Section \ref{ss-verify-conf} and associated with the change of variables \eqref{can}. If we denote by $\psi^{\wgt,\new}$ the function $\psi^{\wgt}$ transported by the canonical transform, we get
$$\mathfrak{Q}_{h}^{[k]}(\psi^\wgt)=\mathfrak{Q}_{h}^{[k],\new}(\psi^{\wgt, \new}).$$
Using again \eqref{lb-normal} with $\eps=h$, we deduce with \eqref{E:g'g}
\begin{multline}\label{E:Qhknewmin}
\mathfrak{Q}_{h}^{[k],\new}(\psi^{\wgt, \new})\geq\\
(1-h)\int_{\R^2}\gamma(\sigma)^{\frac{2}{k+2}}\left(|D_{\tau}\psi^{\wgt, \new}|^2+\left|\left(h\gamma(\sigma)^{-\frac{1}{k+2}}D_{\sigma}-\frac{\tau^{k+1}}{k+1}\right)\psi^{\wgt, \new}\right|^2\right)\dx \sigma \dx \tau\\
-Ch\int_{\R^2}\left|\tau D_{\tau}\psi^{\wgt, \new}\right|^2 \dx \sigma\dx \tau-Ch\|\psi^{\wgt,\new}\|^2.
\end{multline}
But we have
$$\int_{\R^2}\left|\tau D_{\tau}\psi^{\wgt, \new}\right|^2 \dx \sigma\dx \tau=\int_{\R^2}\left|t D_{t}\psi^{\wgt}\right|^2 \dx s \dx t$$
and we apply Proposition \ref{weighted-Agmon} to get
$$\int_{\R^2}\left|t D_{t}\psi^{\wgt}\right|^2 \dx s \dx t\leq C\|\psi^{\wgt,\new}\|^2.$$
With \eqref{E:Qhknewmin}, we infer
\begin{multline}
\mathfrak{Q}_{h}^{[k],\new}(\psi^{\wgt, \new})\geq\\
(1-h)\int_{\R^2}\gamma(\sigma)^{\frac{2}{k+2}}\left(|D_{\tau}\psi^{\wgt, \new}|^2+\left|\left(h\gamma(\sigma)^{-\frac{1}{k+2}}D_{\sigma}-\frac{\tau^{k+1}}{k+1}\right)\psi^{\wgt, \new}\right|^2\right)\dx \sigma \dx \tau\\
-Ch\|\psi^{\wgt,\new}\|^2.
\end{multline}
We can write
\begin{align*}
&\int_{\R^2}\gamma(\sigma)^{\frac{2}{k+2}}\left(|D_{\tau}\psi^{\wgt, \new}|^2+\left|\left(h\gamma(\sigma)^{-\frac{1}{k+2}}D_{\sigma}-\frac{\tau^{k+1}}{k+1}\right)\psi^{\wgt, \new}\right|^2\right)\dx \sigma \dx \tau\\
=&\int_{\R^2}\left(|D_{\tau}\phi^{\wgt, \new}|^2+\left|\left(hD_{\sigma}\gamma(\sigma)^{-\frac{1}{k+2}}-\frac{\tau^{k+1}}{k+1}\right)\phi^{\wgt, \new}\right|^2\right)\dx \sigma \dx \tau\\
=&\int_{\R^2}\left(|D_{\tau}\phi^{\wgt, \new}|^2+\left|\left(h\Xi(\sigma,D_{\sigma})-\frac{\tau^{k+1}}{k+1}+\frac{i h}{2k+4}\gamma'(\sigma)\gamma(\sigma)^{-\frac{k+3}{k+2}}\right)\phi^{\wgt, \new}\right|^2\right)\dx \sigma \dx \tau,
\end{align*}
with
$$\phi^{\wgt, \new}(\sigma,\tau)=\gamma(\sigma)^{\frac{1}{k+2}}\psi^{\wgt, \new}(\sigma,\tau).$$
We are reduced to the same analysis as after \eqref{before-self}. We deduce that
$$\mathfrak{Q}_{h}^{[k],\new}(\psi^{\wgt, \new})\geq(1-h)\mathfrak{q}_{h}^{[k],\new}(\gamma^{\frac{1}{k+2}}\psi^{\wgt,\new})-Ch\|\psi^{\wgt,\new}\|^2.$$
We may use the functional calculus for the self-adjoint operator $\gamma^{-\frac{1}{2k+4}}D_{\sigma}\gamma^{-\frac{1}{2k+4}}$ and it follows that
$$\mathfrak{Q}_{h}^{[k],\new}(\psi^{\wgt, \new})\geq(1-h)\min_{\zeta\in\R} \nun^{[k]}(\zeta)\|\gamma^{\frac{1}{k+2}}\psi^{\wgt,\new}\|^2-Ch\|\psi^{\wgt,\new}\|^2.$$
We now come back in the variables $(s,t)$ and we have proved
$$\min_{\zeta\in\R} \nun^{[k]}(\zeta)\|\gamma^{\frac{1}{k+2}}\psi^{\wgt}\|^2-Ch\|\psi^{\wgt}\|^2 \leq\lambda\|\psi^{\wgt}\|^2+\tilde C\eps_{0}^2\|\Z'\psi^{\wgt}\|^2+\tilde C\eps_{0}^2\|\chi_{p}' \Z\psi^{\wgt}\|^2.$$
We deduce that
\begin{multline}
\min_{\zeta\in\R} \nun^{[k]}(\zeta)\int_{\R^2} \Big(\gamma^{\frac{2}{k+2}}-\gamma_{0}^{\frac{2}{k+2}}\Big)|\psi^{\wgt}|^2\dx s \dx t-\tilde Ch\|\psi^{\wgt}\|^2\\
\leq \tilde C\eps_{0}^2\|\Z'\psi^{\wgt}\|^2+\tilde C\eps_{0}^2\|\chi_{p}' \Z\psi^{\wgt}\|^2.
\end{multline}
We infer that there exist $c>0$, $\tilde C>0$, $h_{0}>0$ and $\eps_{0}>0$ such that for all $h\in(0,h_{0})$ and $p\geq 1$ we have:
\begin{equation}\label{energy-inequality}
c\int_{\R^2} \Big(\gamma^{\frac{2}{k+2}}-\gamma_{0}^{\frac{2}{k+2}}\Big)|\psi^{\wgt}|^2\dx s \dx t-\tilde Ch\|\psi^{\wgt}\|^2\leq\tilde C\eps_{0}^2\|\chi_{p}' \Z\psi^{\wgt}\|^2.
\end{equation}
We deduce by standard arguments that  there exist $\hat C>0$, $h_{0}>0$ and $\eps_{0}>0$ such that for all $h\in(0,h_{0})$ and $p\geq 1$ we have:
$$\|\psi^{\wgt}\|^2\leq \hat C\|\psi\|^2+h^{-1}\hat C\eps_{0}^2\|\chi_{p}' \Z\psi^{\wgt}\|^2.$$
Since $c$ is bounded and $|\chi'_{p}|\leq Cp^{-1}$, it remains to take the $\displaystyle{\liminf_{p\to+\infty}}$ and to apply the Fatou lemma.
\end{proof}
In fact, when $\gamma$ admits two non degenerate minima, we also have Agmon estimates. Let us now prove Proposition \ref{Agmon-puits-double}.
\begin{proof}
The proof is essentially the same as for Proposition \ref{Agmon-puits-simple}. The inequality \eqref{energy-inequality} becomes in this case
\begin{multline*}
c\int_{\R^2} \Big(\gamma^{\frac{2}{k+2}}-\gamma_{0}^{\frac{2}{k+2}}\Big)|\psi^{\wgt}|^2\dx s \dx t-\tilde Ch\|\psi^{\wgt}\|^2\\
\leq\tilde C\eps_{0}^2\|\chi_{p}' \Z\psi^{\wgt}\|^2+C\eps_{0}^2\|f_{-}\psi^{\wgt}\|^2+C\eps_{0}^2\|f_{+}\psi^{\wgt}\|^2,
\end{multline*}
with
$$f_{-}(s)=\chi'_{-,d}(s)\left|\int_{s_{-}}^s \tilde\chi(s')\sqrt{\gamma(s')^{\frac{2}{k+2}}-\gamma_{0}^{\frac{2}{k+2}}}\dx s'\right|,$$
$$f_{+}(s)=\chi'_{+,d}(s)\left|\int_{s_{-}}^s \tilde\chi(s')\sqrt{\gamma(s')^{\frac{2}{k+2}}-\gamma_{0}^{\frac{2}{k+2}}}\dx s'\right|.$$
We can deal with the last two terms by using the positivity of $\gamma^{\frac{2}{k+2}}-\gamma_{0}^{\frac{2}{k+2}}$ away from the minima and the conclusion standardly follows from an electric type case.
\end{proof}

\subsection{Tunnelling estimates}
Let us first state an elementary result.
\begin{proposition}
The asymptotic expansions (modulo $h^{\infty}$) of the $n$-th eigenvalue of $\mathfrak{H}_{h,-}^\Dir$ are the same as for $\mathfrak{L}^{[k]}_{h}$. In particular, the spectral gap between two consecutive eigenvalues is of order $h$. Moreover the eigenfunctions of $\mathfrak{H}_{h,-}^\Dir$ satisfy the same kind of Agmon estimates as in Proposition \ref{Agmon-puits-simple}.
\end{proposition}
\begin{proof}
The construction of quasimodes is the same as for Proposition \ref{quasimodes}. We have just to add a suitable cutoff function and to use the exponential decay of the explicit quasimodes. In order to estimate the spectral gap between the lowest eigenvalues of $\mathfrak{H}_{h,-}^\Dir$, by the min-max principle, we have just to notice that $\mathfrak{H}_{h,-}^\Dir$ is bounded from below by the realization on $\sL^2(\R^2)$ of $D_{t}^2+\left(hD_{s}-\tilde\gamma(s)\frac{t^{k+1}}{k+1}\right)^2$ where the smooth function $\tilde\gamma$ coincides with $\gamma$ on $(-\infty,s_{-}+\delta)$ and admits a unique minimum at $s_{-}$. This lower bound is enough to deduce the Agmon estimates.
\end{proof}
Let us now prove that the low lying eigenvalues of $\mathfrak{L}_{h}^{[k]}$ are exponentially close to the eigenvalues of $\mathfrak{H}_{h}$.
\begin{proposition}\label{spec-loc}
Let us consider $C_{0}>0$. There exist $c>0$, $C>0$, $h_{0}>0$ such that for all $\mu\in\spe\left(\mathfrak{L}_{h}^{[k]}\right)$ with $\mu\leq \mu_{0}+C_{0}h$, we have, for all $h\in(0,h_{0})$,
$$\dist\left(\mu,\spe\left(\mathfrak{H}_{h}\right)\right)\leq C\re^{-c/h}.$$
\end{proposition}
\begin{proof}
The proof is based on the introduction of suitable quasimodes for the operator $\mathfrak{H}_{h}$ and on the Agmon estimates satisfied by the eigenfunctions of $\mathfrak{L}_{h}^{[k]}$. Let us consider an eigenpair $(\lambda,\psi)$ of $\mathfrak{L}_{h}^{[k]}$ such that $\lambda\leq \mu_{0}+C_{0}h$. One knows that $\psi$ satisfies the estimates of Proposition \ref{Agmon-puits-double}. In particular, we deduce that
$$\mathfrak{H}_{h,-}^{\Dir}(\chi_{-,d'}\psi)=\lambda (\chi_{-,d'}\psi)+\Oc(\re^{-c/h})\|\psi\|^2,\quad \mathfrak{H}_{h,+}^{\Dir}(\chi_{+,d'}\psi)=\lambda (\chi_{+,d'}\psi)+\Oc(\re^{-c/h})\|\psi\|^2.$$
The spectral theorem provides the conclusion.
\end{proof}
In fact, exponentially close to each eigenvalue of $\mathfrak{H}_{h}$ there are at least two eigenvalues of $\mathfrak{L}_{h}$.
\begin{proposition}\label{r>2}
Let us consider $C_{0}>0$. There exist $c>0$, $C>0$, $h_{0}>0$ such that for all $\mu\in\spe\left(\mathfrak{H}_{h}\right)$ satisfying $\mu\leq \mu_{0}+C_{0}h$, we have, for all $h\in(0,h_{0})$,
$$\range\left(\mathds{1}_{[\mu-C\re^{-c/h}, \mu+C\re^{-c/h}]}\big(\mathfrak{L}_{h}^{[k]}\big)\right)\geq 2.$$
\end{proposition}
\begin{proof}
Let us consider $\mu\in\spe\left(\mathfrak{H}_{h}\right)$ such that $\mu\leq \mu_{0}+C_{0}h$. We can find a corresponding normalized eigenfunction $(\varphi_{-},\varphi_{+})$ of $\mathfrak{H}_{h}$. The Agmon estimates imply that
$$\left(\mathfrak{L}_{h}^{[k]}-\mu\right)(\chi_{-,d'}\varphi_{-})=\Oc(\re^{-c/h}),\quad \left(\mathfrak{L}_{h}^{[k]}-\mu\right)(\chi_{+,d'}\varphi_{+})=\Oc(\re^{-c/h}),$$
and where $\chi_{-,d'}\varphi_{-}$ and $\chi_{+,d'}\varphi_{+}$ are orthogonal. The proof follows again from the spectral theorem.
\end{proof}
We have now all the elements for our tunnelling result. It remains to prove that, in Proposition \ref{r>2}, the range of the spectral projection is exactly $2$. By contradiction, if this range were at least $3$, then we could consider three eigenfunctions $(u_{j,h}^{[k]})$ of $\mathfrak{L}_{h}^{[k]}$ mutually orthogonal associated with an eigenvalue $\lambda$ such that there exists $\mu\in\spe\left(\mathfrak{H}_{h}\right)$ satisfying $|\mu-\lambda|\leq C\re^{-c/h}$. Then one could apply the same argument as in the proof of Proposition \ref{spec-loc} with $\psi=\psi_{j}$. The spectral theorem would imply that the multiplicity of $\mu$ is at least $3$ but this is impossible since the multiplicity of the lowest eigenvalues of $\mathfrak{H}_{h}$ is $2$.

\subsection{Numerical simulations}

\subsubsection{Method}\label{sec.methnum}
Let us now deal with numerical simulations for the model operator $\mathfrak L_{h}^{[k]}$ on $\R^2$ if $k=1$ and on the half-plane $\R^2_{+}$ if $k=0$.
In each case, we propose simulations for the simple and double well models. We analyze the convergence of the eigenvalues as $h\to0$ and give the first eigenfunctions for small $h$. This illustrates the localization of the modulus and the behavior of the phase.

To approximate the eigenpairs $(\lambda_{n}^{[k]}(h),u_{n,h}^{[k]})$ of the operator $\mathfrak L_{h}^{[k]}$ on $\Omega_{k}$ (with $\Omega_{0}=\R^2_{+}$ and $\Omega_{1}=\R^2$), we use the Finite Element Library {\sc M\'elina++} \cite{Melina++}. Since the domain $\Omega_{k}$ is unbounded, we use an artificial domain ${\cal R}_{k,a,b}=\Omega_{k}\cap (-a,a)\times (-b,b)$. We compute the eigenpairs $(\lambda_{n}^{[k]}(h,a,b),u_{n,h}^{[k]}(a,b))$ on ${\cal R}_{k,a,b}$ and impose Dirichlet conditions on the artificial boundaries $\{|x|=a\}\cup\{|y|=b\}$. We use quadrangular elements and polynomial approximation $\mathbb Q_{p}$ and a mesh $n_{x}\times n_{y}$. 

By this way, we obtain upper bounds for $\lambda_{n}^{[k]}(h)$ and we know that
$$\lambda_{n}^{[k]}(h,a,b)\to \lambda_{n}^{[k]}(h)\quad\mbox{ as }\quad\min(a,b)\to +\infty.$$
We compute the eigenpairs for several sets of parameters $(a,b)$ with several combinations of degree of approximation $p$ and of size of the meshes $n_{x}\times n_{y}$ until convergence is found.
We normalize the computed eigenfunctions so that $\|u_{n,h}^{[k]}(a,b))\|_{\infty}=1$.

\subsubsection{Simple well models}
For the numerical simulations, we take
$$\gamma(s)=1+4s^2,\qquad s\in\R.$$
Parameters used for the numerical simulations are given in Table~\ref{table1}.
Among the computations, we take, for each value of $h$, the smallest numerical eigenvalues for the different choices of $a$, $b$, $p$, $n_{x}$, $n_{y}$.

\begin{table}[h!t]
\begin{center}
\begin{tabular}{ccccccccc}
$a$ & $n_{x}$ & $b$ & $n_{y}$ & $p$ & $1/h$ \\
1 & 5 & 20 & 20 & 16 & $1:1:100$, $110:10:1000$\\
1 & 5 & 40 & 20 & 14 & $10:10:1000$\\
1.5 & 5 & 30 & 20 & 14 & $1:1:100$\\
2 & 5 & 40 & 20 & 14 & $1:1:100$\\
\end{tabular}
\caption{Simple well: Parameters of the numerical simulations for $\mathfrak L_{h}^{[k]}$, $k=0,1$.}\label{table1}
\end{center}
\end{table}
In Figure~\ref{fig.VP1puitsR2}, we analyze the convergence of the eigenvalues $\lambda_{n}^{[k]}(h)$ as $h\to 0$ for $1\leq n\leq 12$.
Figure~\ref{fig.1puitsVPh} gives an approximation of the first twelve eigenvalues $\lambda_{n}^{[k]}(h)$, $1/h\in [1,100]$ and corroborates the convergence 
$$\lambda_{n}^{[k]}(h)\to \underline\nu^{[k]}=\nun^{[k]}\big(\zeta_{0}^{[k]}\big)\qquad\mbox{ as }h\to 0.$$
Using \cite{BR13,Bon12}, we know that
$$\underline\nu^{[1]}\simeq 0.5698,\qquad
\underline\nu^{[0]}=\Theta_{0}\simeq 0.5901.
$$
\begin{figure}[h!t]
\begin{center}
\subfigure[$\lambda_{n}^{[k]}(h)$ vs. $\frac 1 h$, $\frac1h\in\{1,\ldots,100\}$.\label{fig.1puitsVPh}]{\begin{tabular}{cc}
\includegraphics[width=4.7cm]{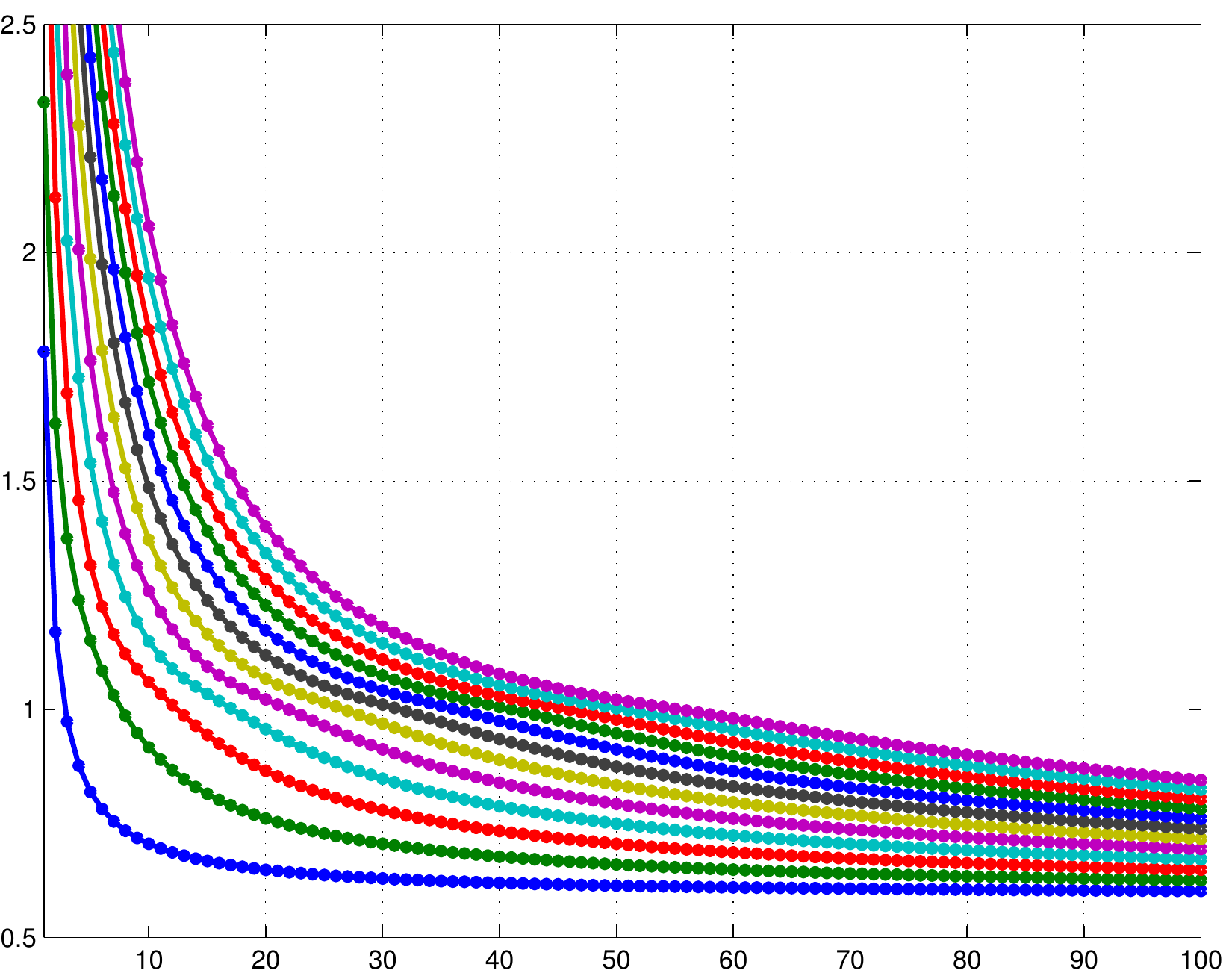}
&\includegraphics[width=4.7cm]{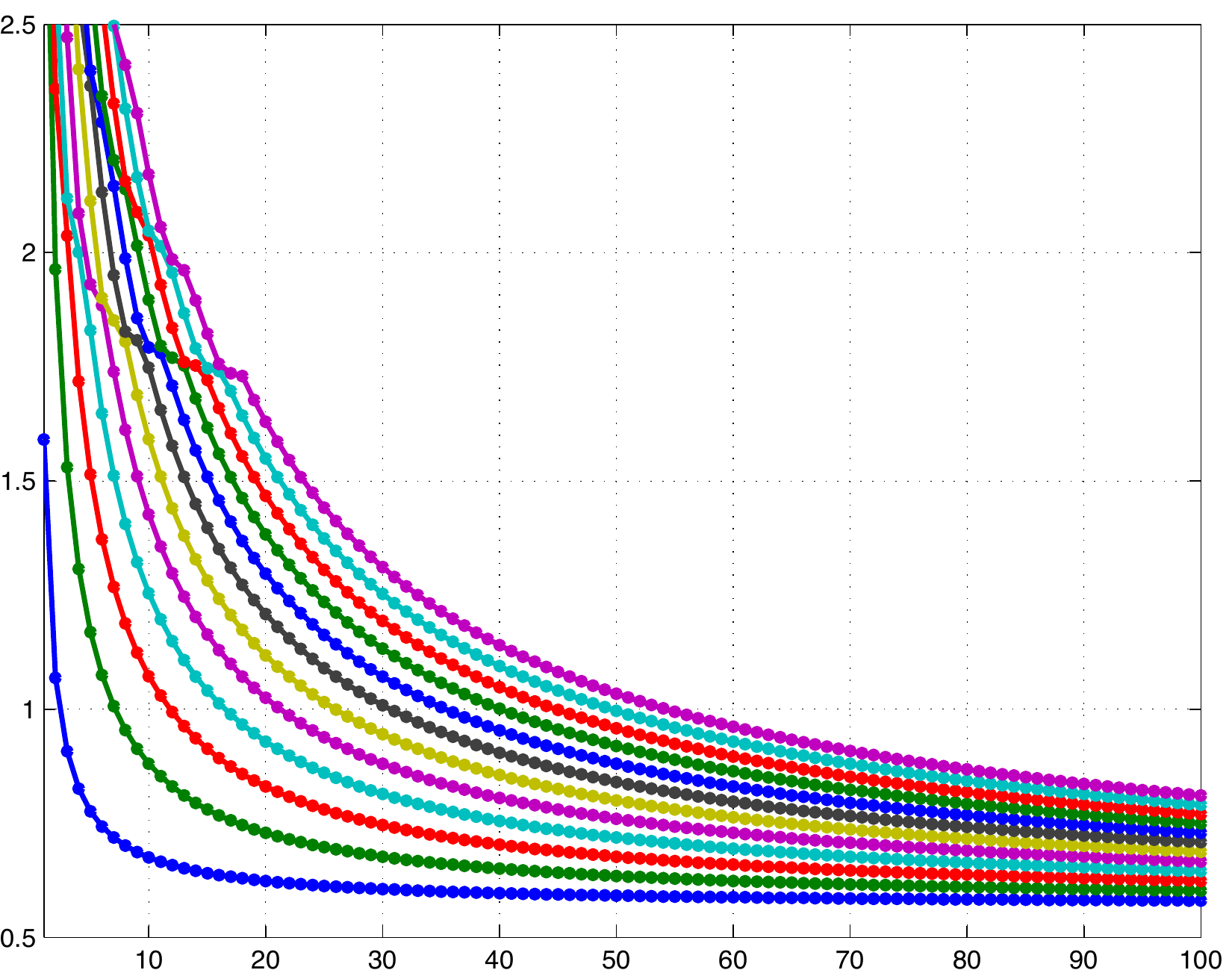}\\
$k=0$ & $k=1$
\end{tabular}}
\subfigure[$\ln\Big(\lambda_{n}^{[k]}(h)-\underline\nu^{[k]}\Big)$ vs. $\ln \frac1h$, $\frac1h\in\{1,\ldots,1000\}$.\label{fig.1puitsVPhlog}]{\begin{tabular}{cc}
\includegraphics[width=4.7cm]{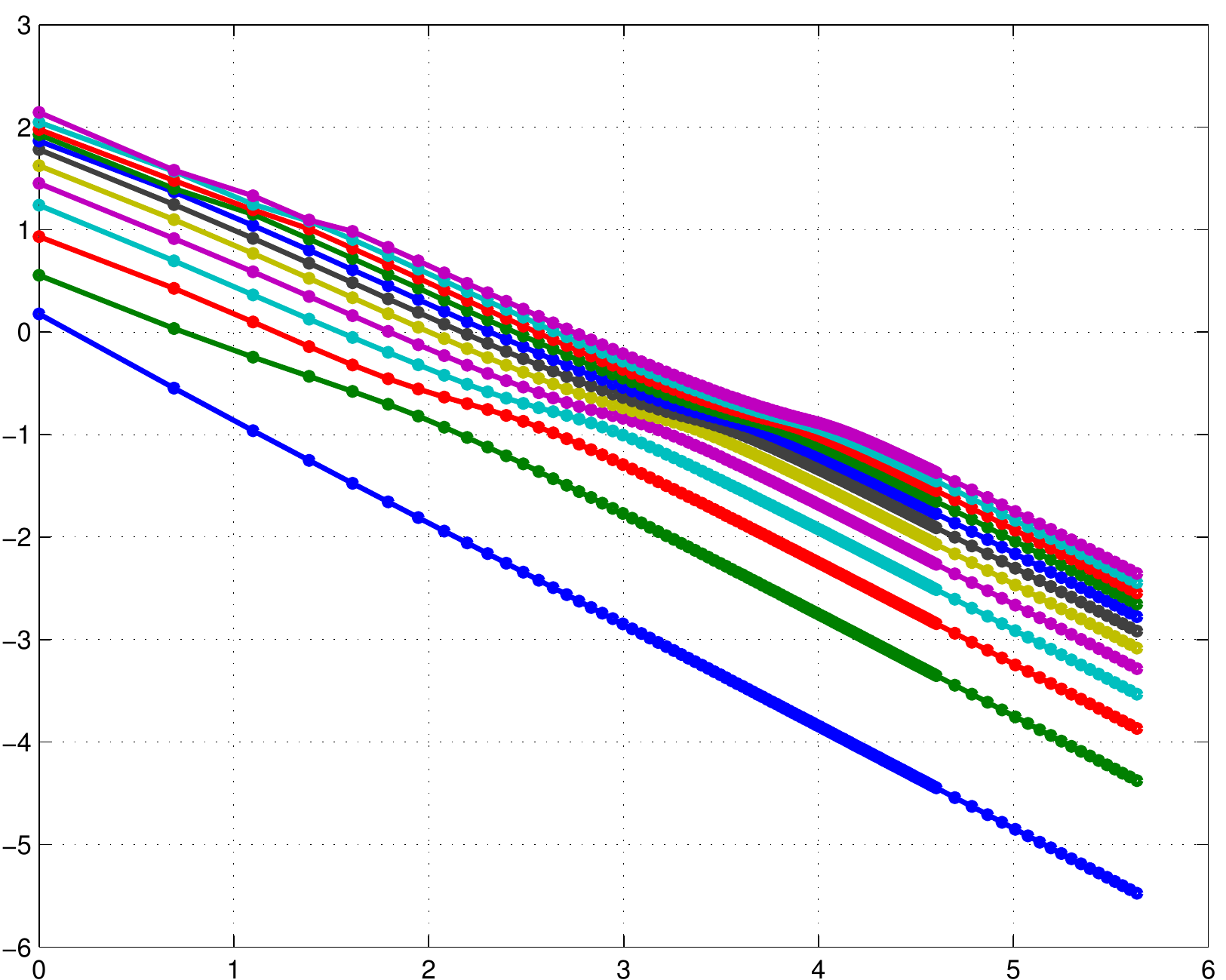}
&\includegraphics[width=4.7cm]{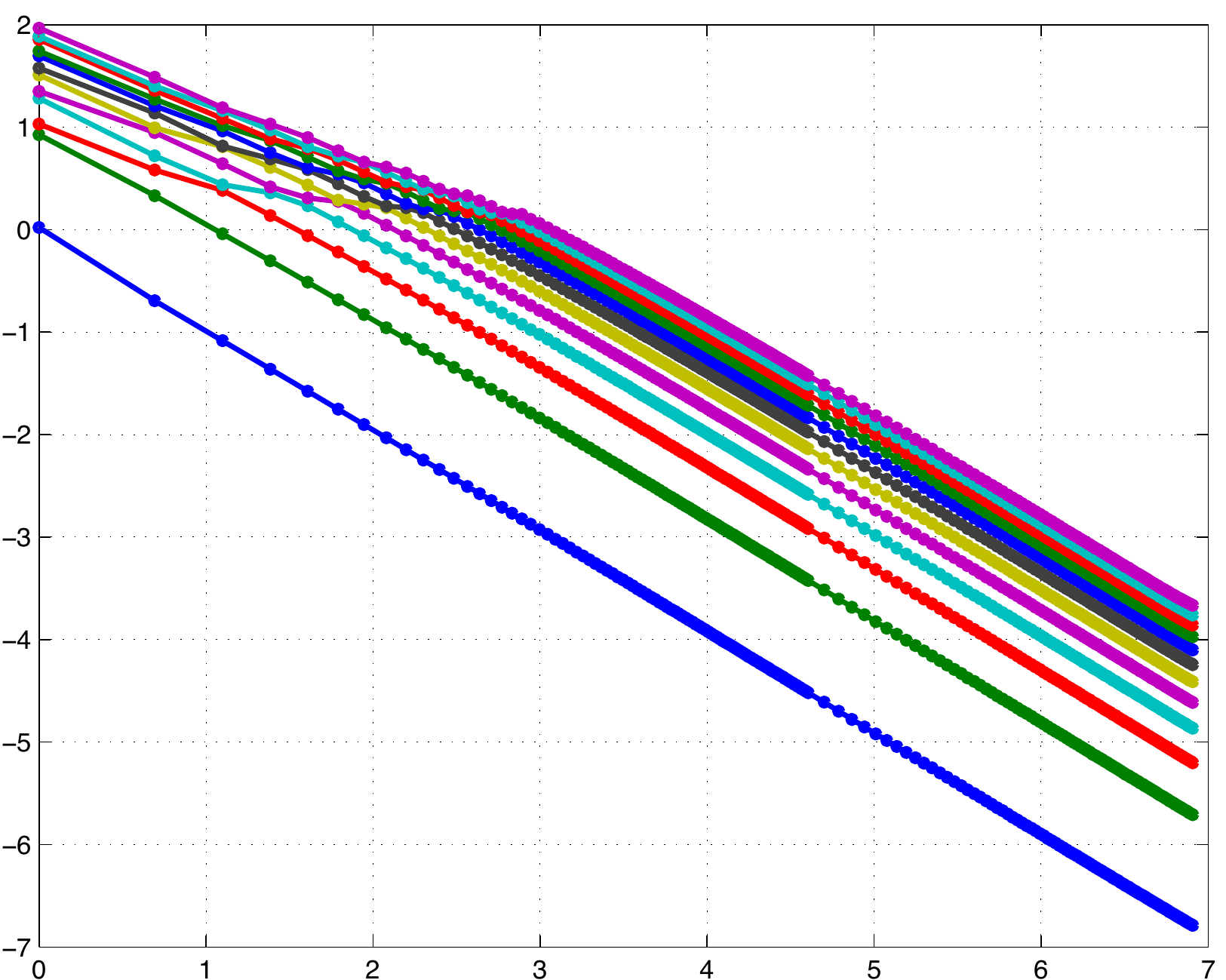}\\
$k=0$ & $k=1$
\end{tabular}}
\caption{Convergence of the eigenvalues $\lambda_{n}^{[k]}(h)$. \label{fig.VP1puitsR2}}
\end{center}
\end{figure}
\begin{figure}[h!t]
\begin{center}
\subfigure[First eigenfunction $u_{1,h}^{[0]}$, $h=\frac1{15}$]{\includegraphics[width=2.3cm]{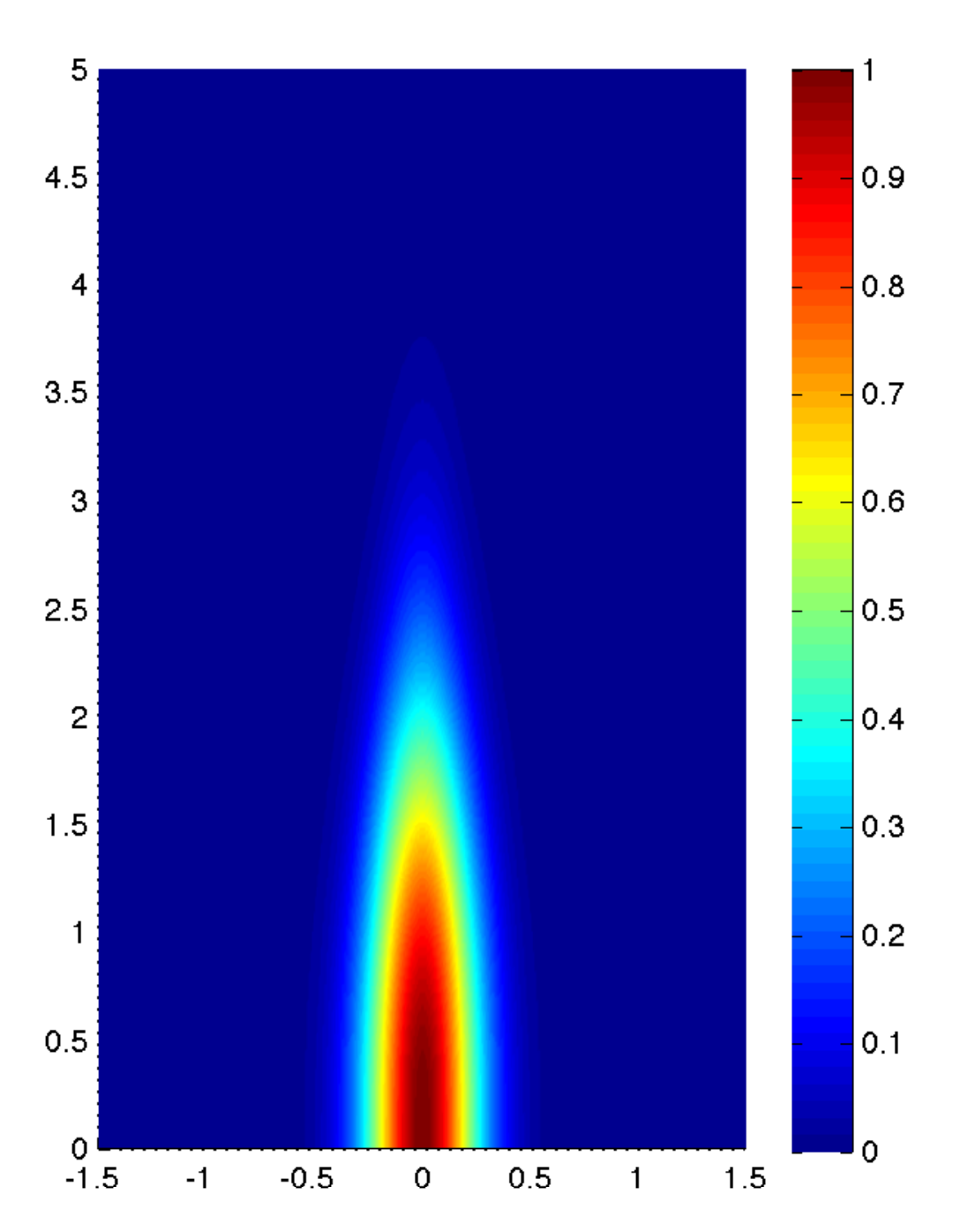}
\includegraphics[width=2.3cm]{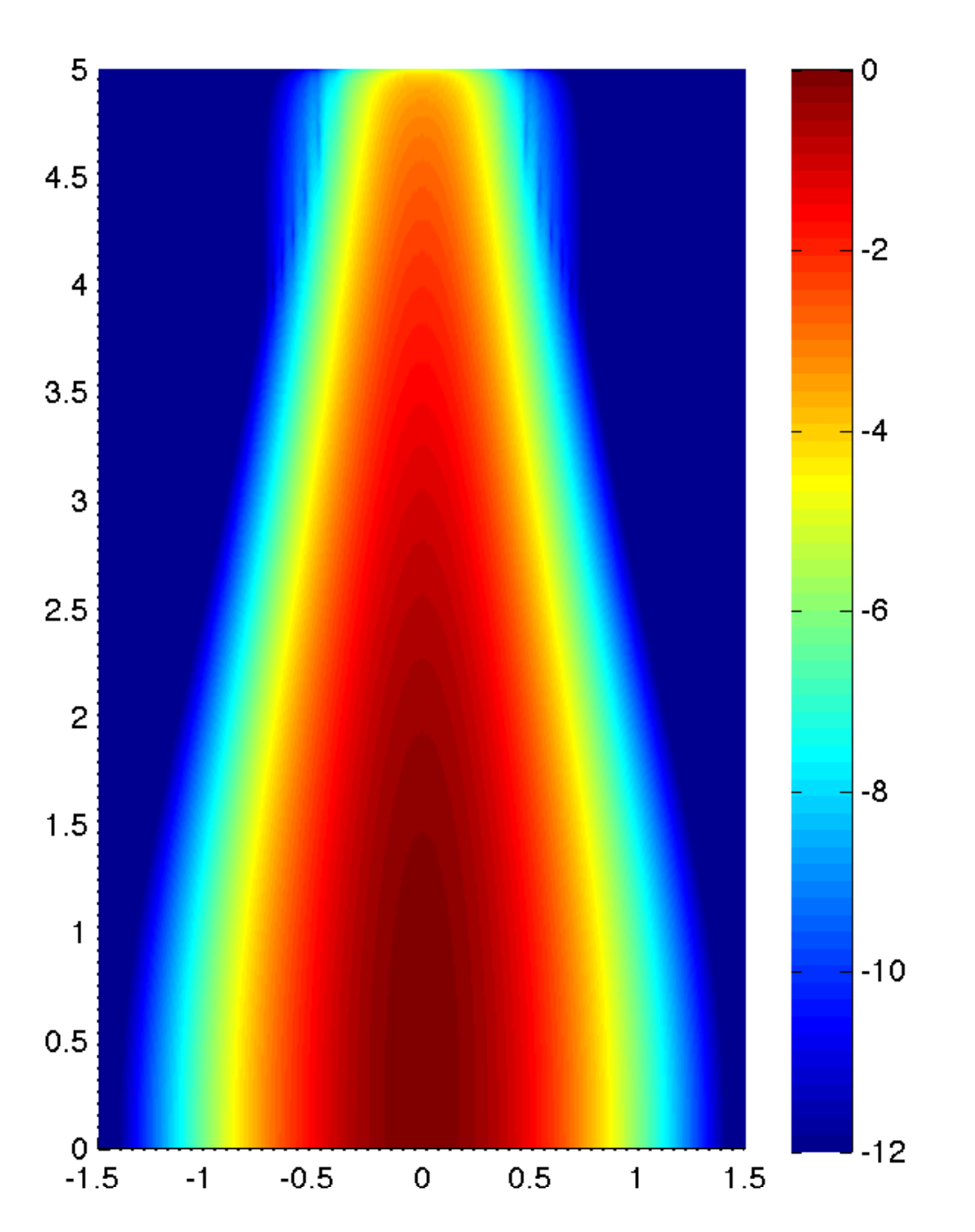}
\includegraphics[width=2.3cm]{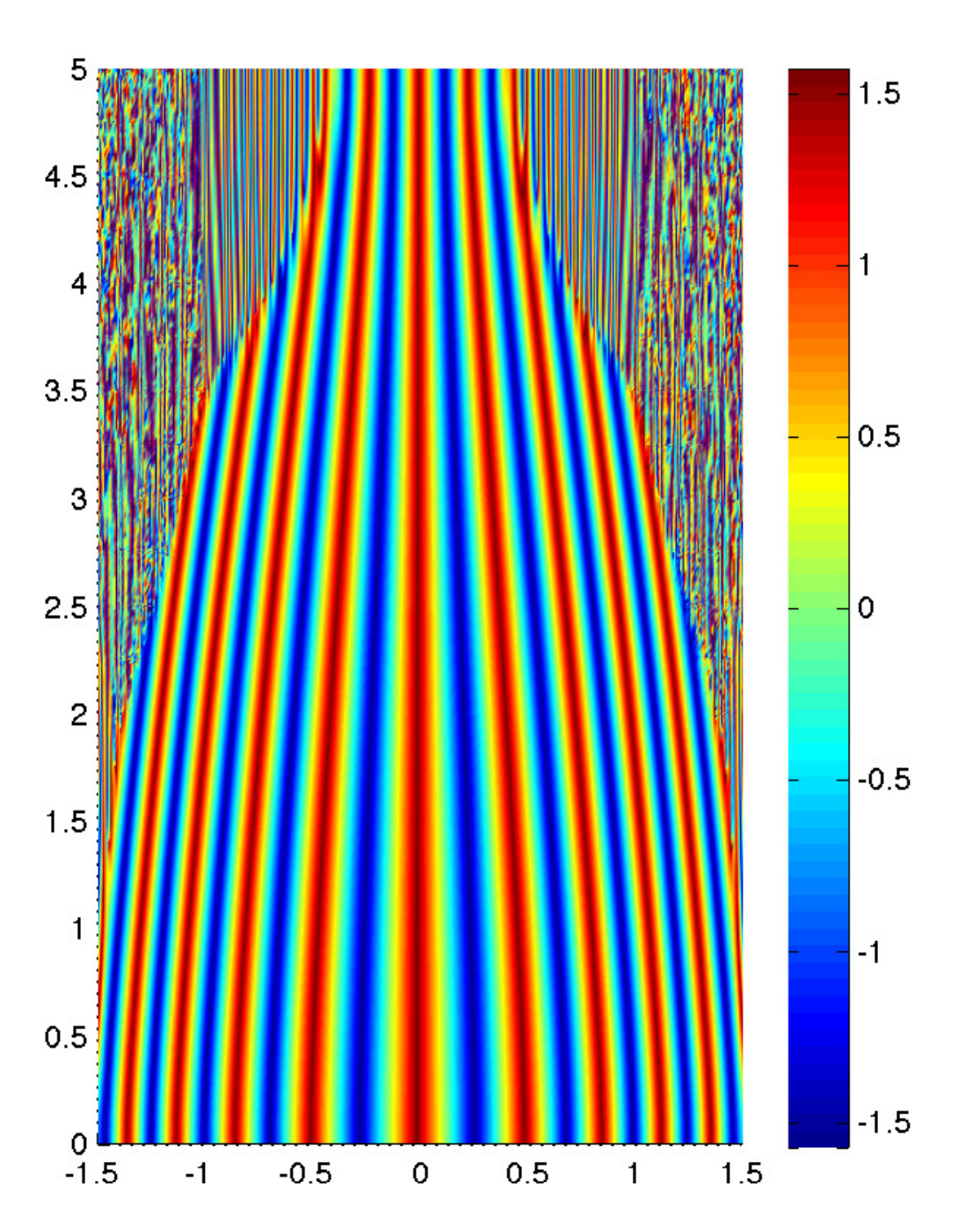}}
\subfigure[Second eigenfunction $u_{2,h}^{[0]}$, $h=\frac1{15}$]{\includegraphics[width=2.3cm]{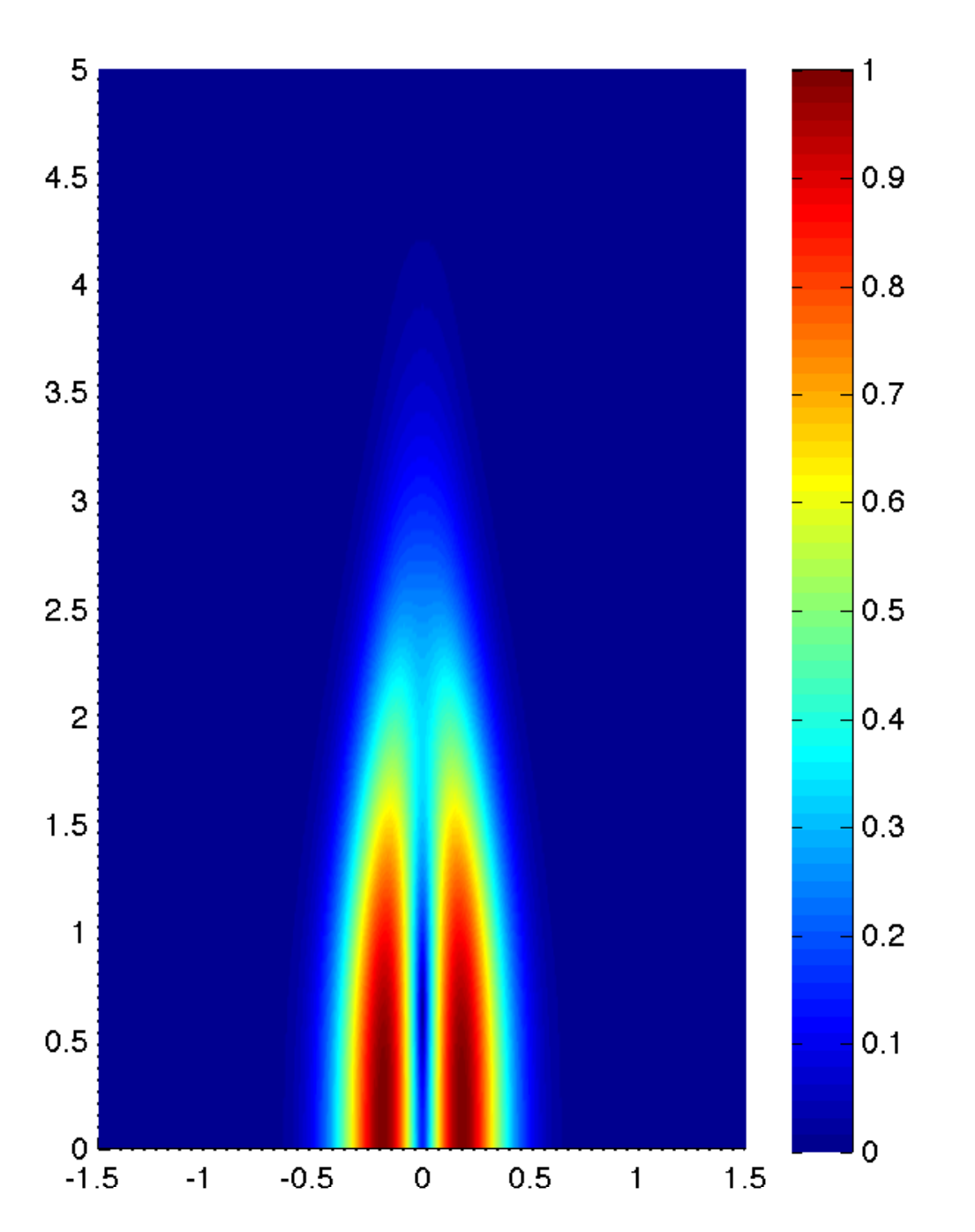}
\includegraphics[width=2.3cm]{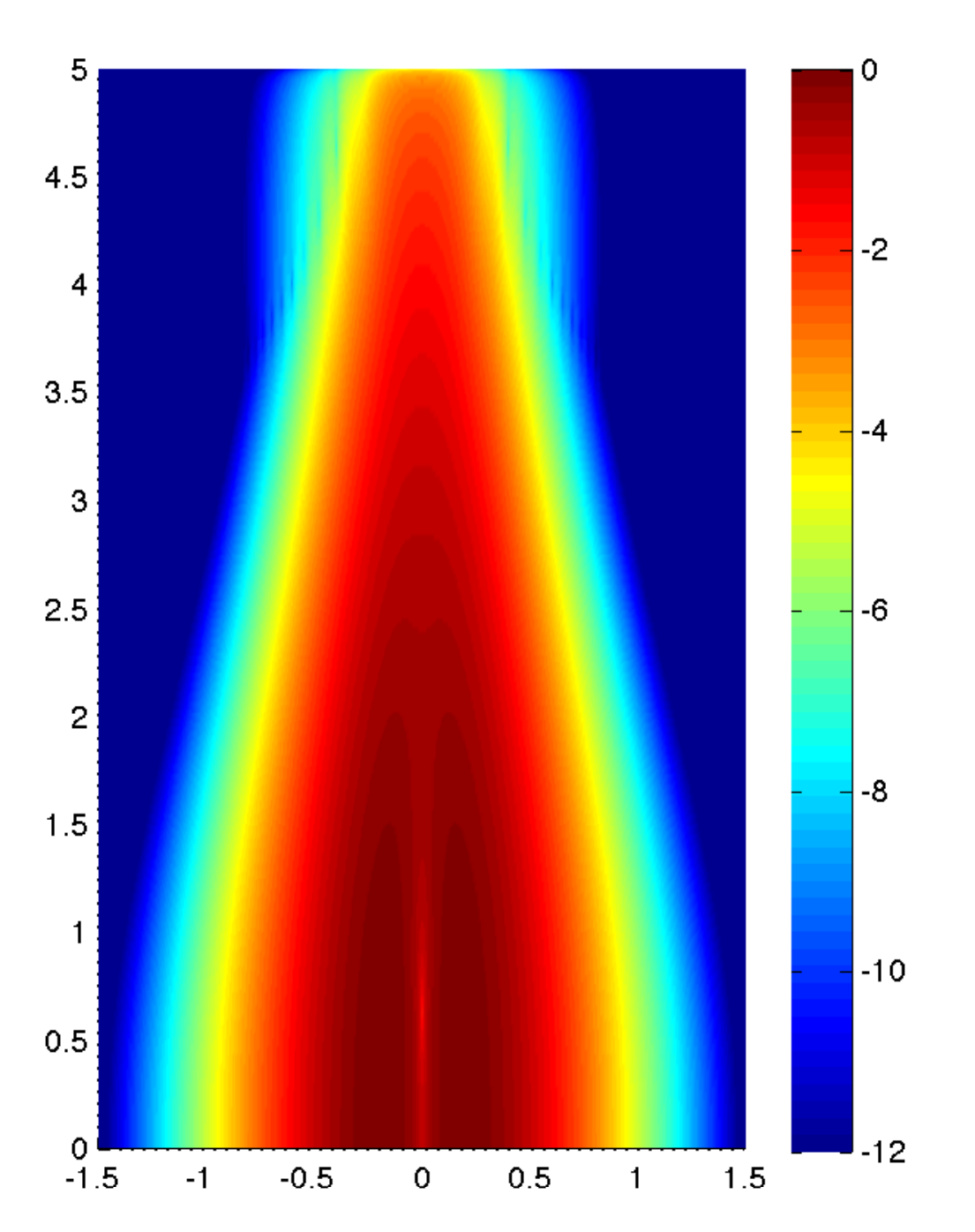}
\includegraphics[width=2.3cm]{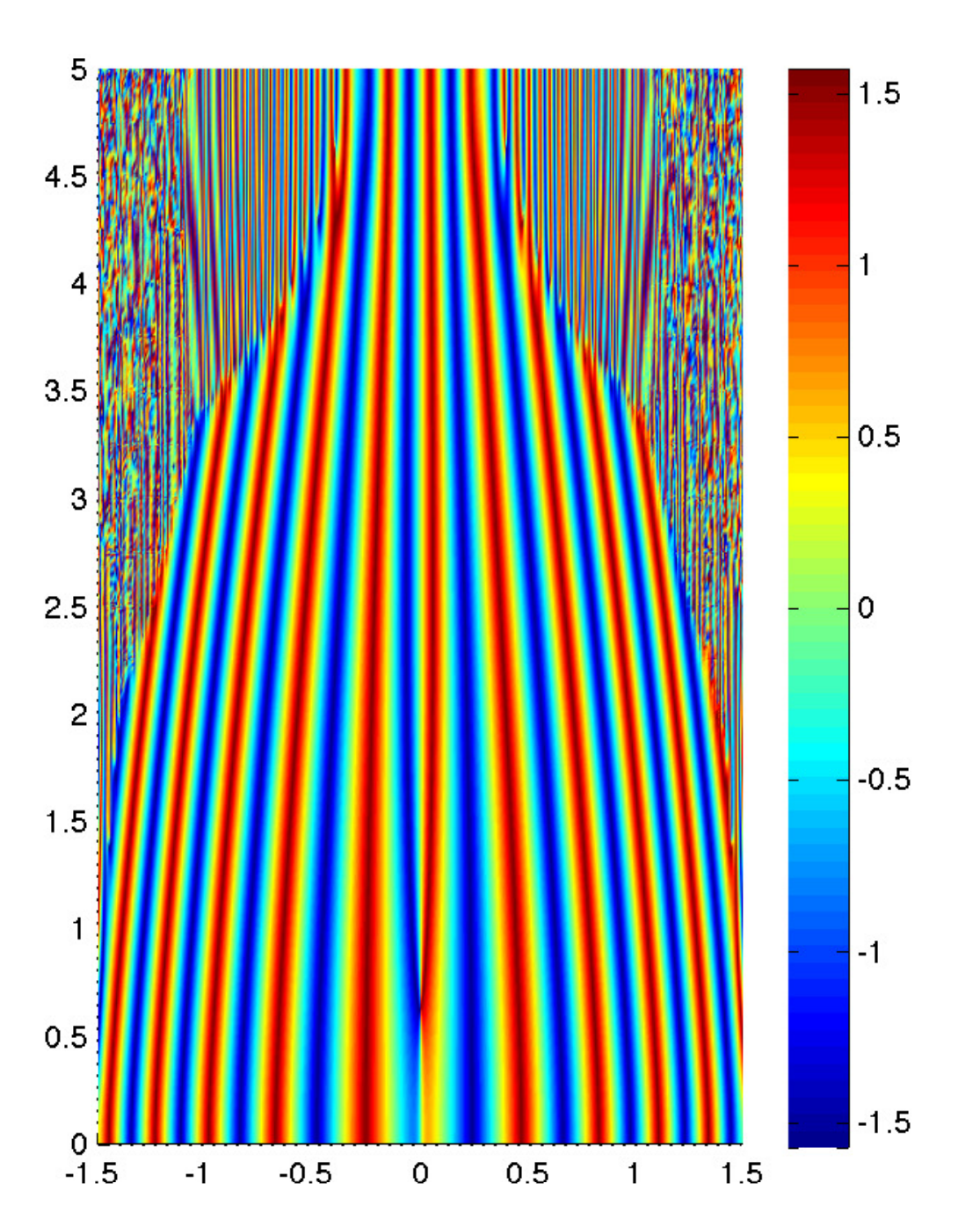}}
\subfigure[First eigenfunction $u_{1,h}^{[1]}$, $h=\frac1{20}$]{\includegraphics[width=2.3cm]{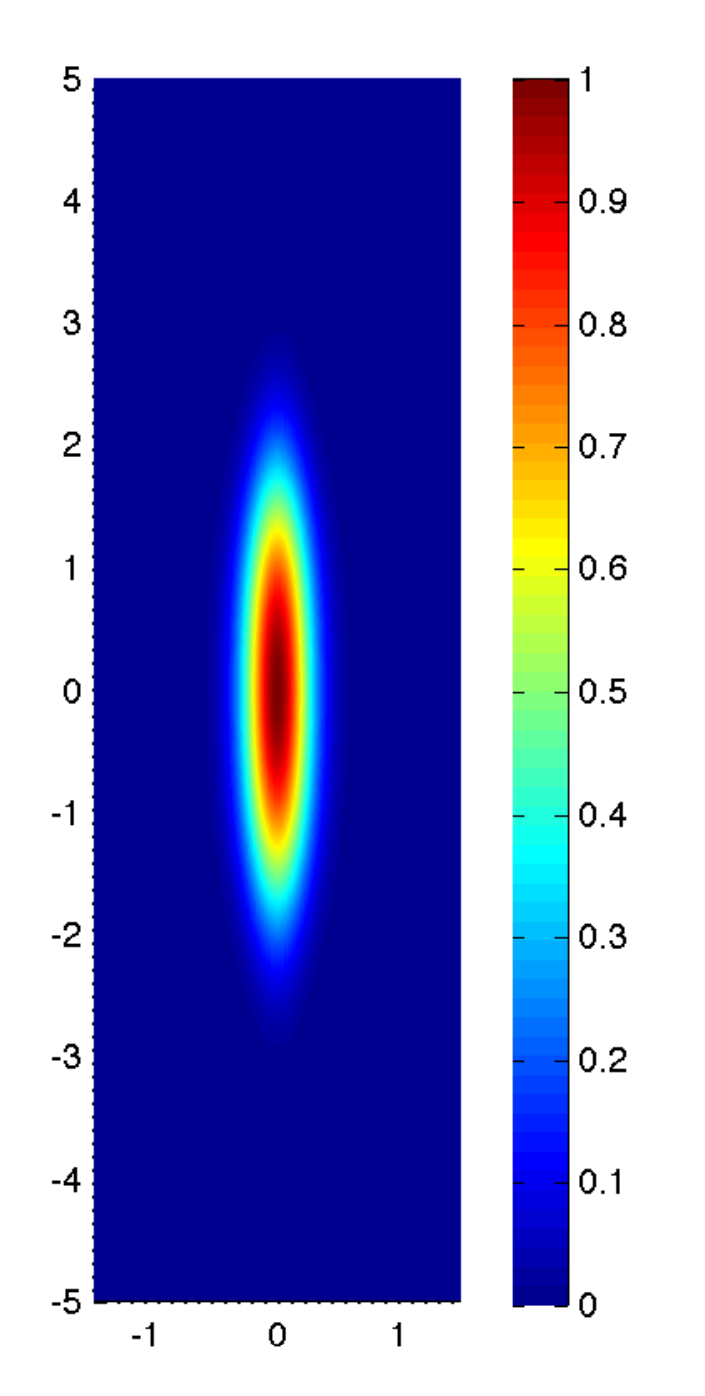}
\includegraphics[width=2.3cm]{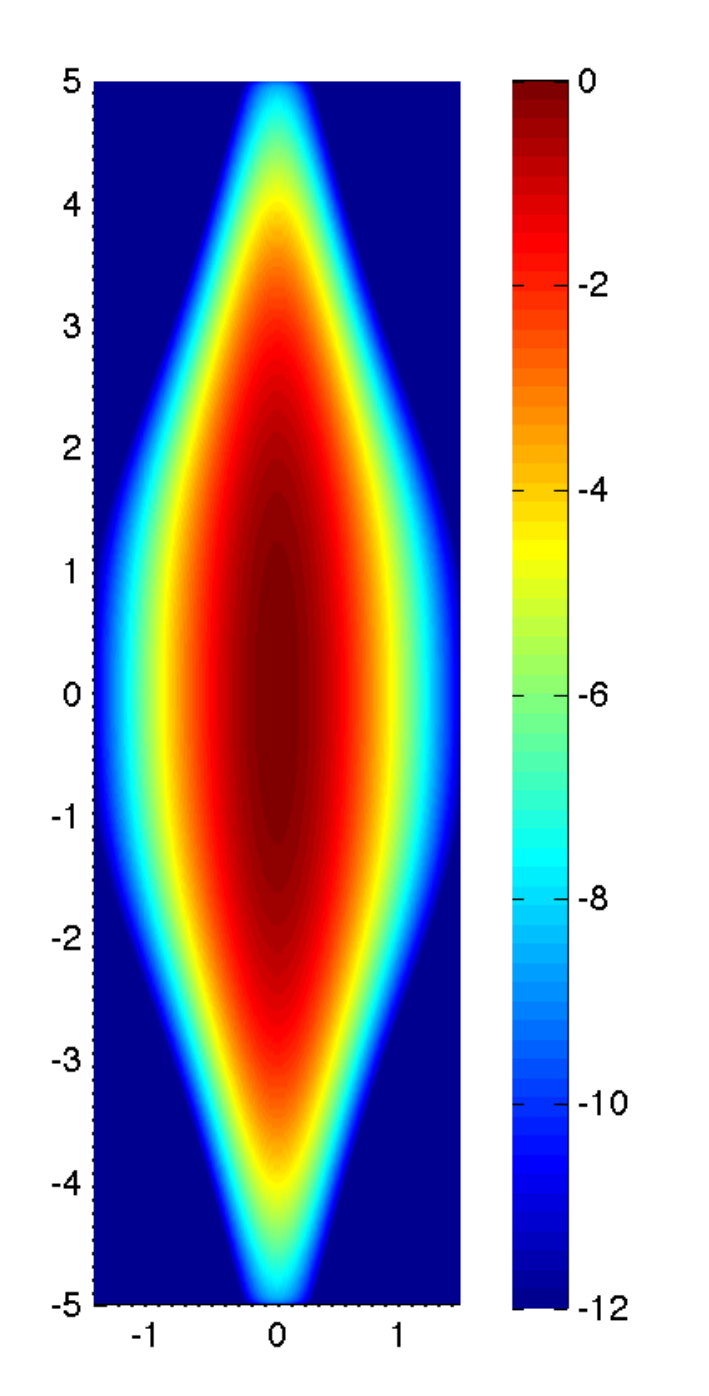}
\includegraphics[width=2.3cm]{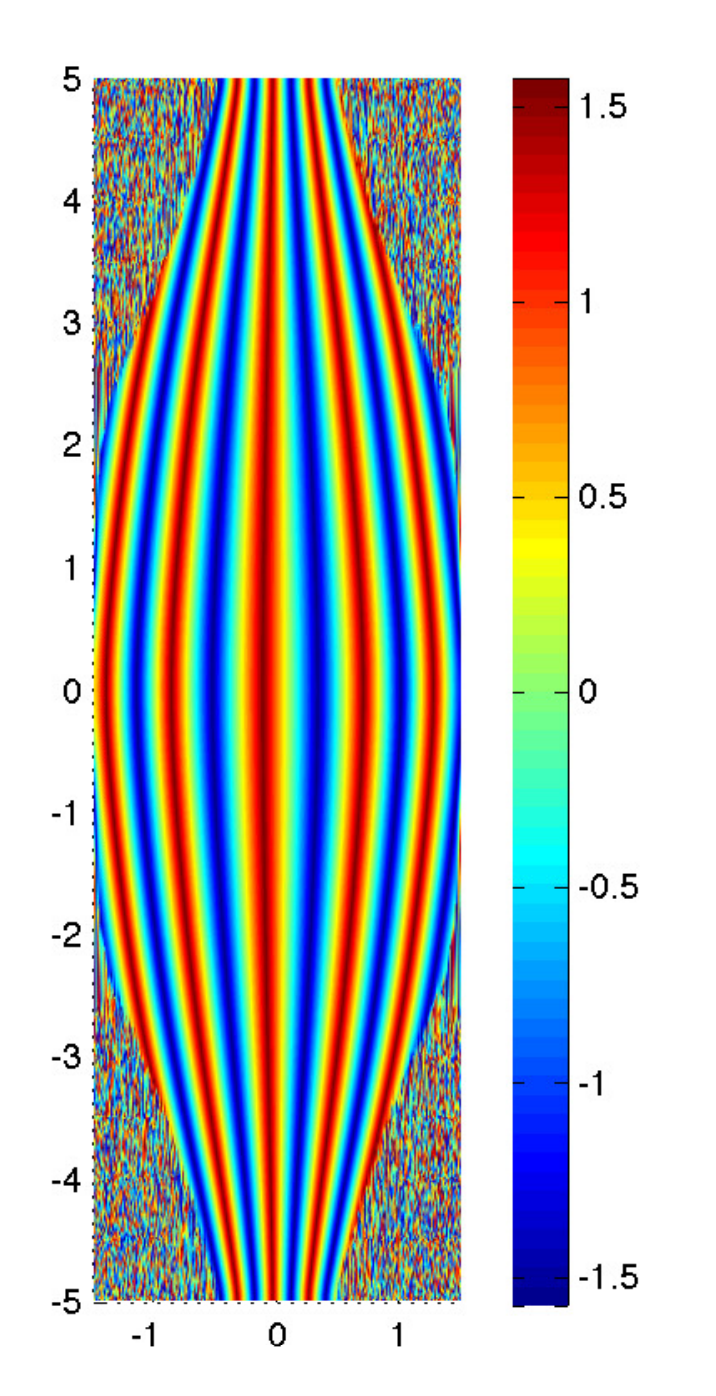}}
\subfigure[Second eigenfunction $u_{2,h}^{[1]}$, $h=\frac1{20}$]{\includegraphics[width=2.3cm]{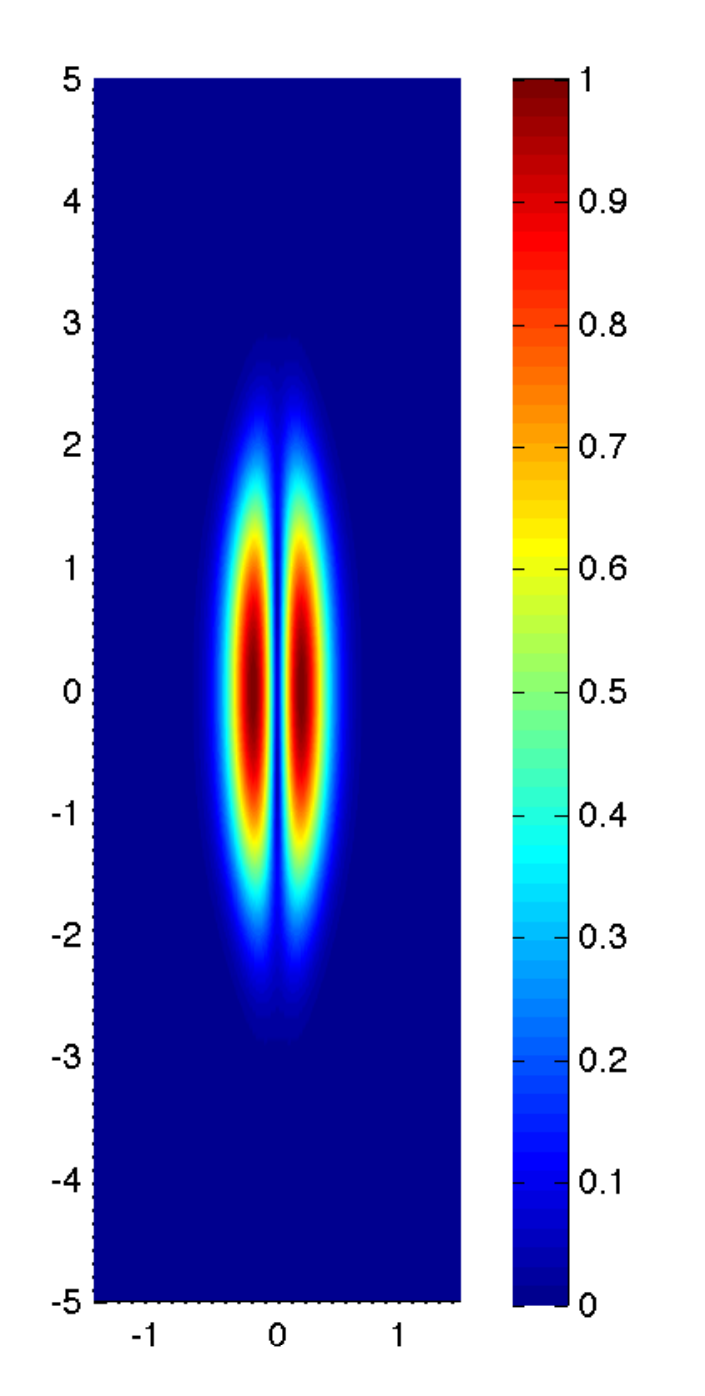}
\includegraphics[width=2.3cm]{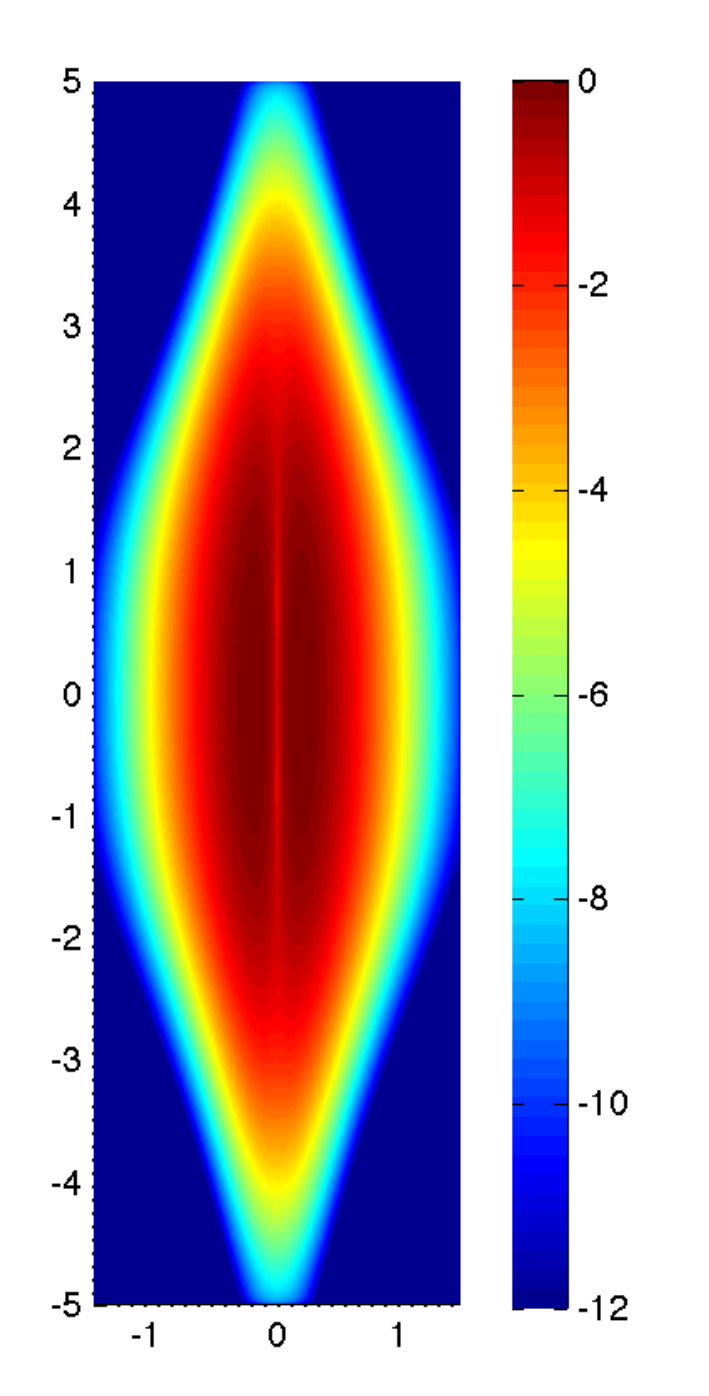}
\includegraphics[width=2.3cm]{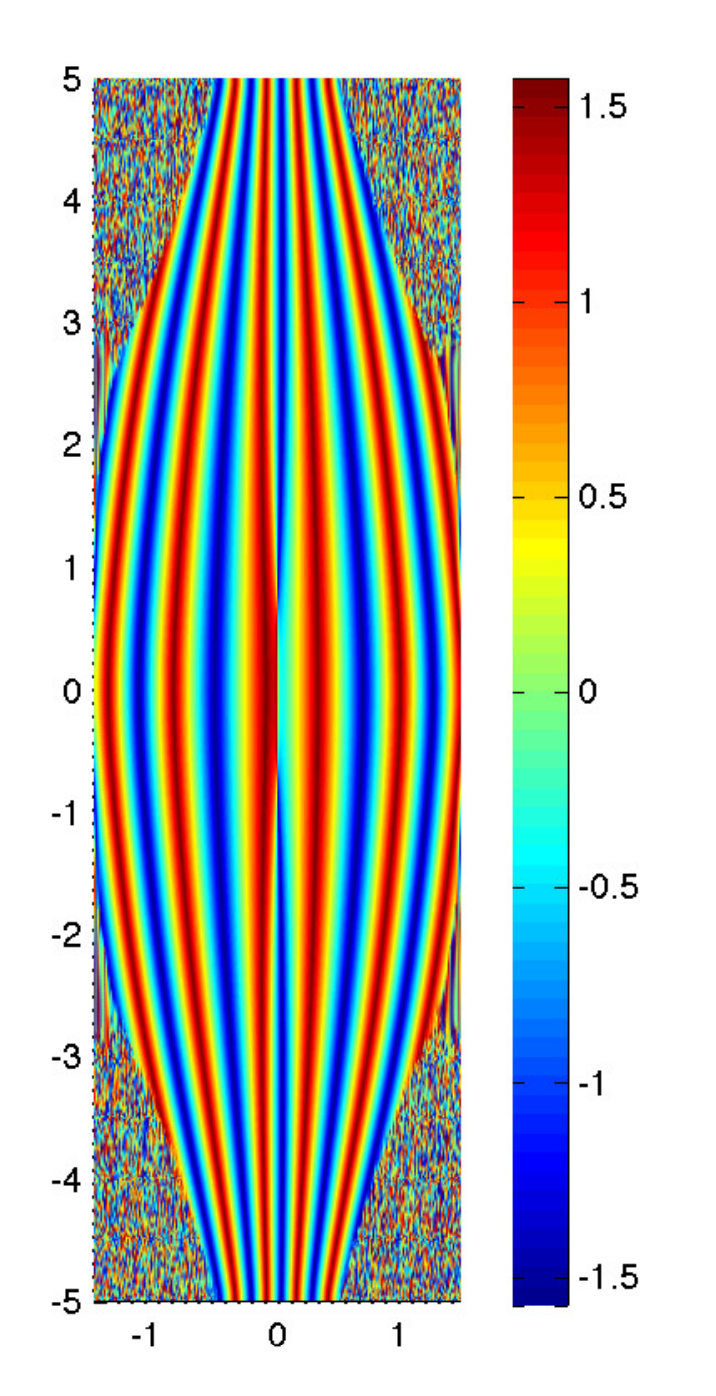}}
\caption{Moduli, log$_{10}$(moduli) and phases of the first two eigenfunctions, $u_{n,h}^{[k]}$. \label{fig.VecPR21puits}} 
\end{center}
\end{figure}
To catch the next term in the expansion of the eigenvalues, we plot in Figure~\ref{fig.1puitsVPhlog}
$$\ln\frac1h\mapsto \ln\frac{\lambda_{n}^{[k]}(h)-\underline\nu^{[k]}}{h}.$$
We observe a linear convergence: the slope $r$ illustrates the behavior
$$\lambda_{n}^{[k]}(h)= \underline\nu^{[k]}+ C_{n} h^{r}+ o (h^r).$$

In Figure~\ref{fig.VecPR21puits}, we give the approximation of the first two eigenfunctions $u_{n,h}^{[k]}$ for $h=1/20$ if $k=1$ and $h=1/15$ if $k=0$. We draw the modulus, the logarithm of the modulus and the phase.

\subsubsection{Double well models}
Let us now consider the double well model and take
$$\gamma(s)=1+(s^2-1)^2,\qquad s\in\R.$$
Parameters used for the numerical simulations are given in Table~\ref{table2}.
\begin{table}[h!t]
\begin{center}
\begin{tabular}{ccccccccc}
$a$ & $n_{x}$ & $b$ & $n_{y}$ & $p$ & $1/h$ \\
2 & 5 & 40 & 20 & 12 & $1:0.1:500$\\
2 & 5 & 40 & 20 & 16 & $10:1:1000$\\
3 & 5 & 60 & 20 & 14 & $10 :1: 300$\\
4 & 5 & 40 & 20 & 14 & $10 :1: 300$\\
5 & 5 & 40 & 20 & 14 & $10 :1: 200$\\
10 & 10 & 50 & 50 & 10 & $10 :1: 200$
\end{tabular}
\caption{Double well: Parameters of the numerical simulations for $\mathfrak L_{h}^{[k]}$, $k=0,1$.}\label{table2}
\end{center}
\end{table}
\begin{figure}[h!t]
\begin{center}
\begin{tabular}{p{.7cm} c p{2.5cm} c p{1cm} cc}
&$\lambda_{n}^{[k]}(h)$ &
& $\lambda_{2n}^{[k]}(h)-\lambda_{2n-1}^{[k]}(h)$&
& $-h\ln(\lambda_{2}^{[k]}(h)-\lambda_{1}^{[k]}(h))$
\end{tabular}\\
\vspace{-.2cm}
\subfigure[$k=0$] {\includegraphics[width=5cm]{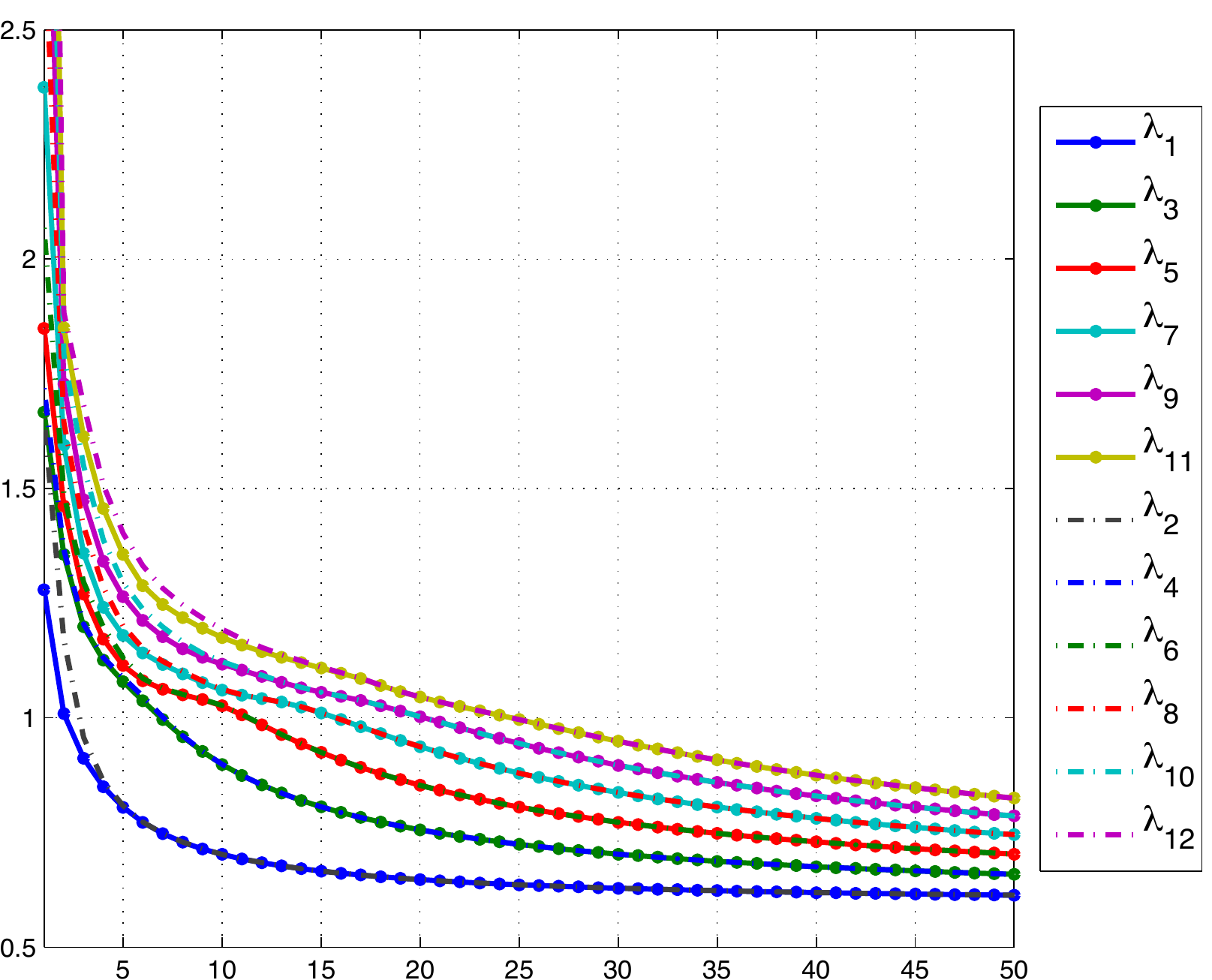}
\includegraphics[width=5cm]{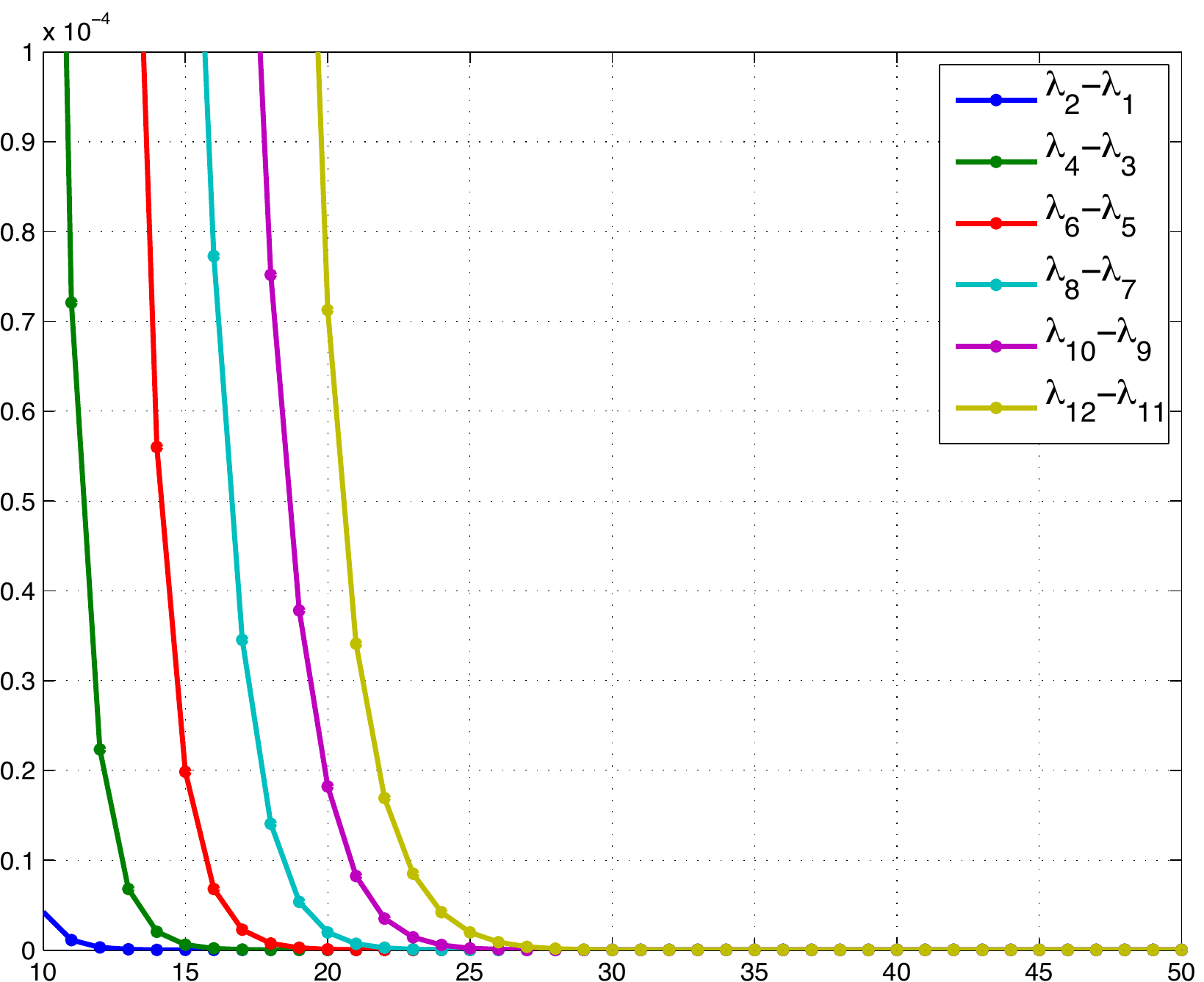}
\includegraphics[width=5cm]{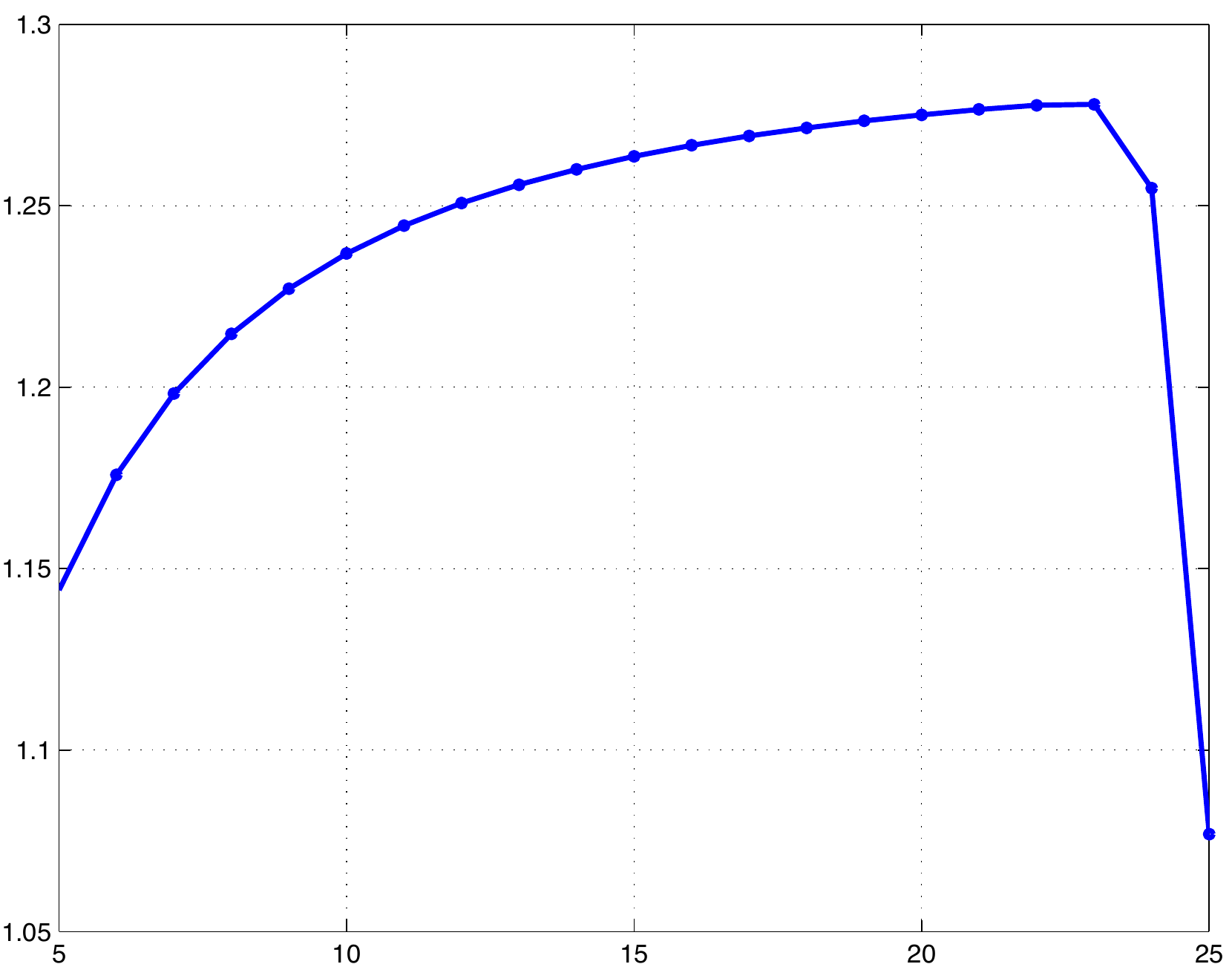}}
\subfigure[$k=1$] {\includegraphics[width=5cm]{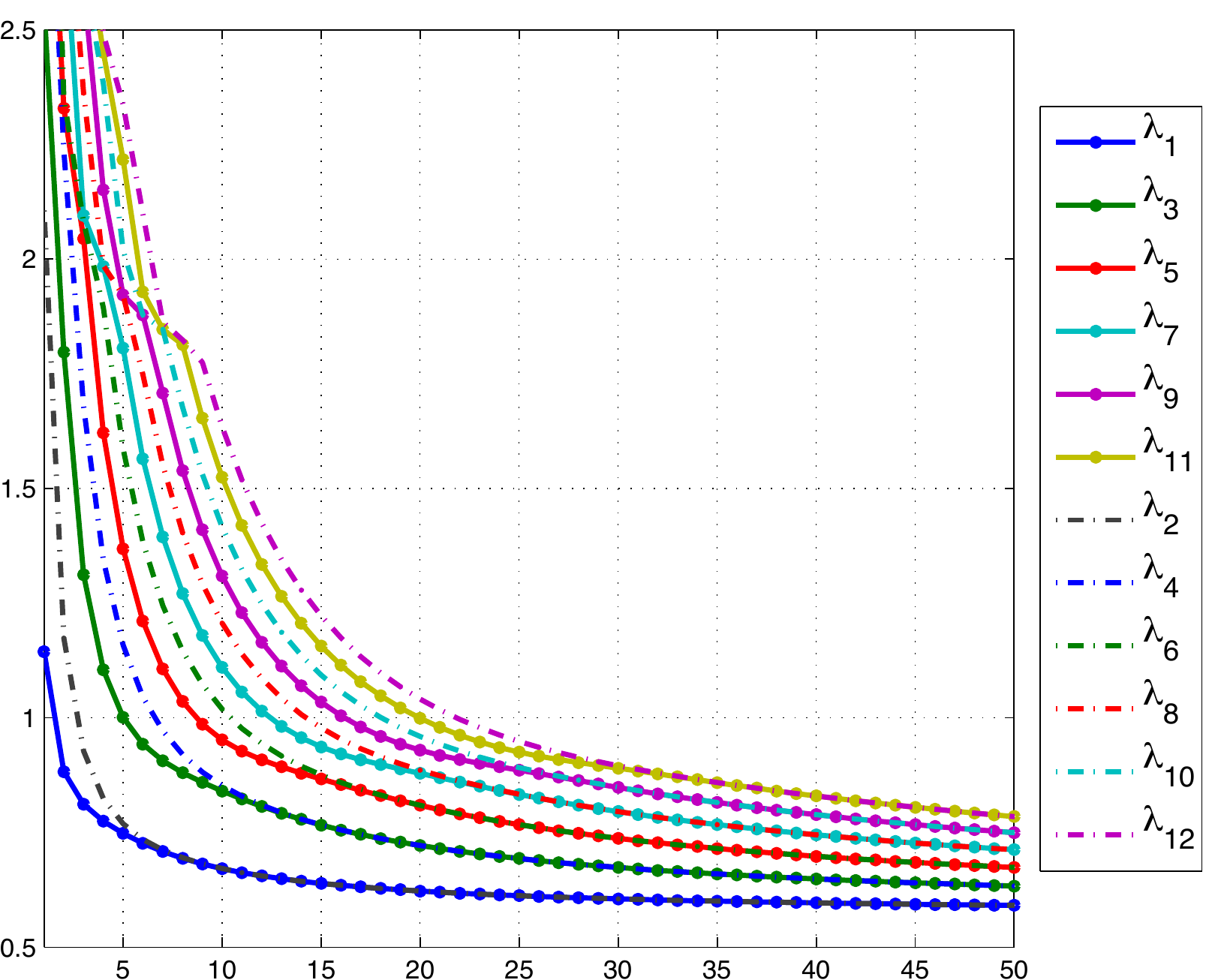}
\includegraphics[width=5cm]{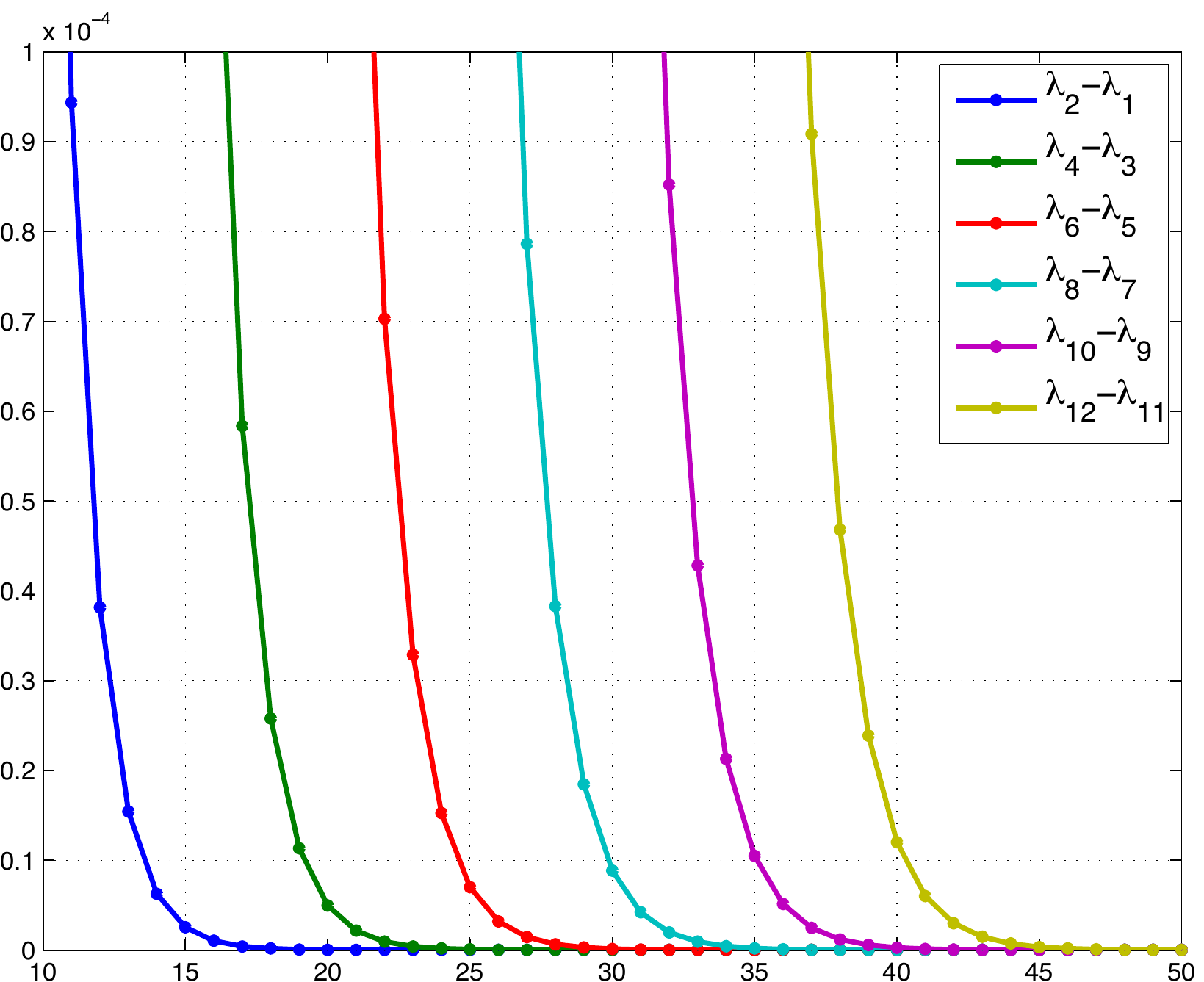}
\includegraphics[width=5cm]{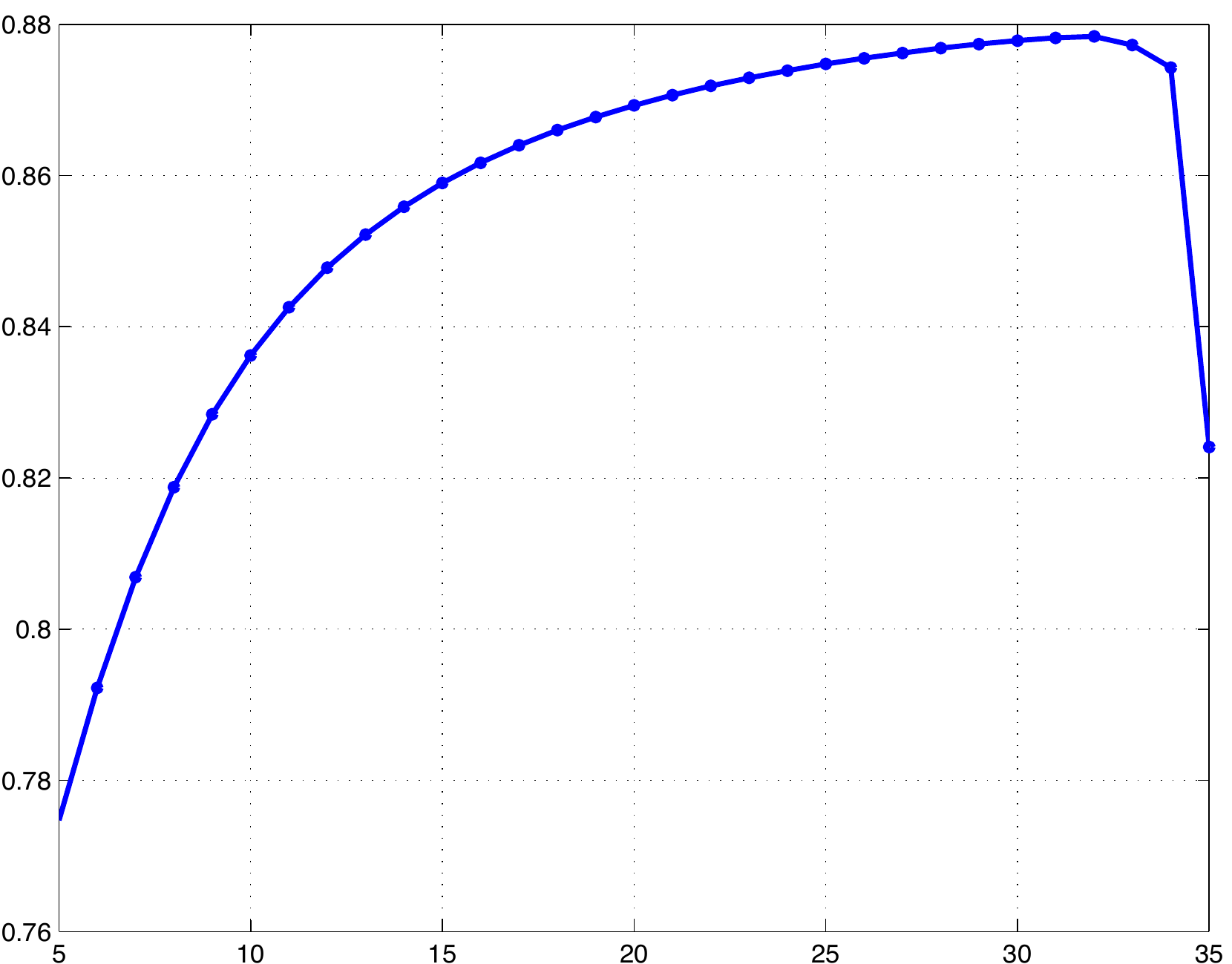}}
\caption{First eigenvalues $\lambda_{n}^{[k]}(h)$ vs. $1/h$, $k=0,1$. \label{fig.VPR2R2p}}
\end{center}
\end{figure}

Figure~\ref{fig.VPR2R2p} illustrates Corollary~\ref{cor.convMont}, Theorem \ref{tunnelling} and Remark~\ref{rem.gap} in the scale $h$ instead of $\hbar$. The first line concerns the low eigenvalues $\lambda_{n}^{[k]}(h)$ with $k=0$ (on the half-plane) and the second line with $k=1$ (on the plane). The first column illustrates the convergence
$$
\lambda_{n}^{[k]}(h) \to \underline\nu^{[k]}\qquad\mbox{ as }h\to 0.
$$
The second column represents the splitting $\lambda_{2n}^{[k]}(h)-\lambda_{2n-1}^{[k]}(h)$ according to $1/h\in\{1,\ldots, 50\}$. We recover the exponential decay of Remark~\ref{rem.gap}. This decay is faster when $k=0$.
In the last column, we aim at catching the exponential decay rate and we plot
$$1/h\mapsto -h\ln\Big(\lambda_{2}^{[k]}(h)-\lambda_{1}^{[k]}(h)\Big).$$
Let us discuss this last column a little more. We observe a break of the curve when $1/h$ becomes too large ($1/h\geq24$ for $k=0$ and $1/h\geq 33$ for $k=1$).
For smaller $h$, the gap between the first two eigenvalues is very small: $\lambda_{2}^{[k]}(h)-\lambda_{1}^{[k]}(h)< 3\ 10^{-12}$, which is the accuracy of the computations. So the gap is no more significant when $h$ becomes too small: the error due to the computations and the splitting is at the same order $\simeq 10^{-12}$. We try here to catch two scales: a polynomial scale for the convergence of the eigenvalues as $h\to 0$ and an exponential scale for the splitting. Thus the range of $h$ to have the two convergences is small. Let us recall that, from Remark~\ref{rem.gap}, we have
$$\lambda_{2}^{[k]}(h)-\lambda_{1}^{[k]}(h) ={\Oc}(\re^{-c_{k}/h}).$$
The third column in Figure~\ref{fig.VPR2R2p} gives the following estimates
$$1.25\leq c_{0}\leq 1.3,\quad 0.86\leq c_{1}\leq 0.9.$$
\begin{figure}[h!t]
\begin{center}
\subfigure[First eigenfunction $h=\frac 1{15}$]{\includegraphics[width=2.38cm]{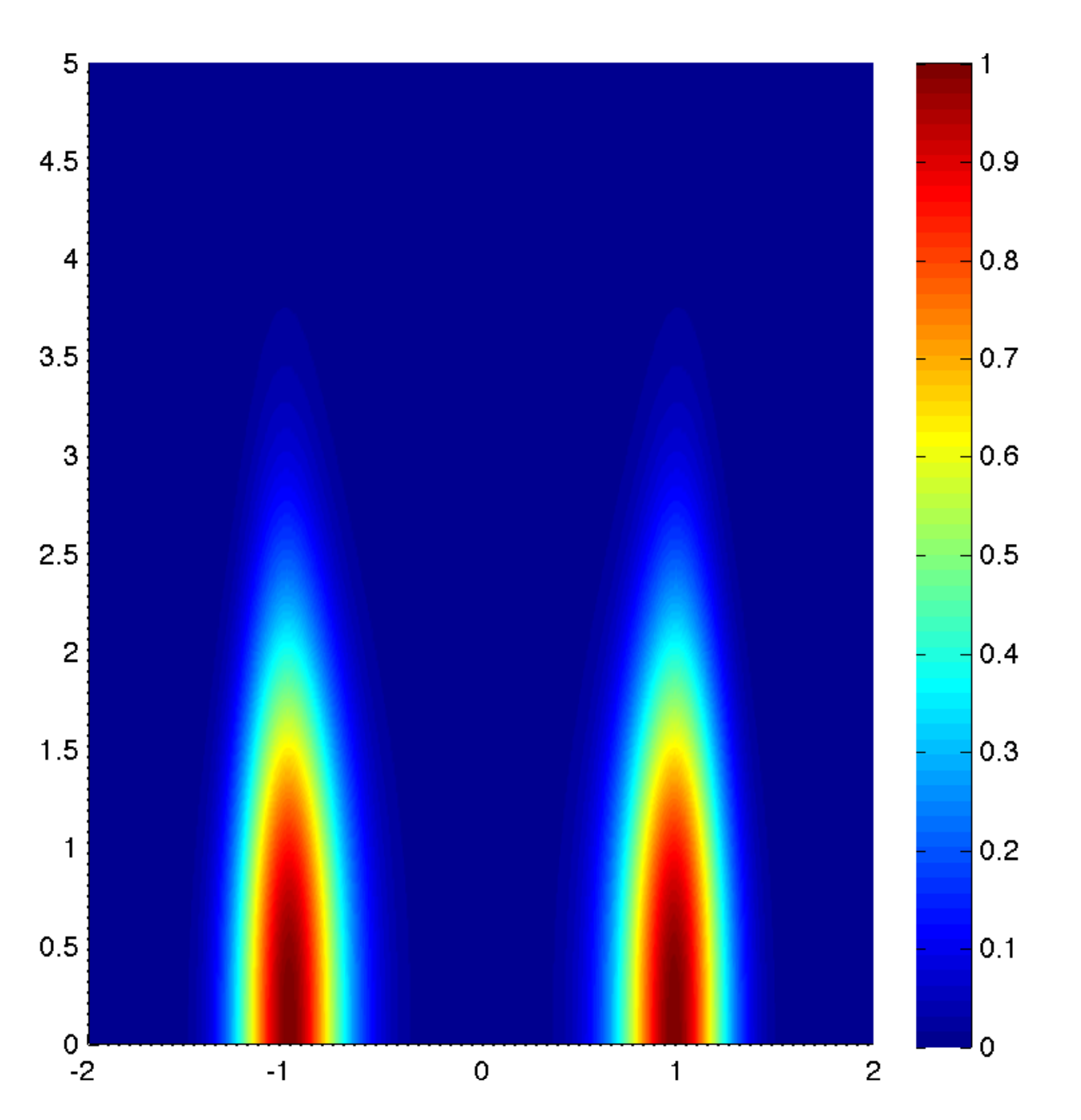}
\includegraphics[width=2.38cm]{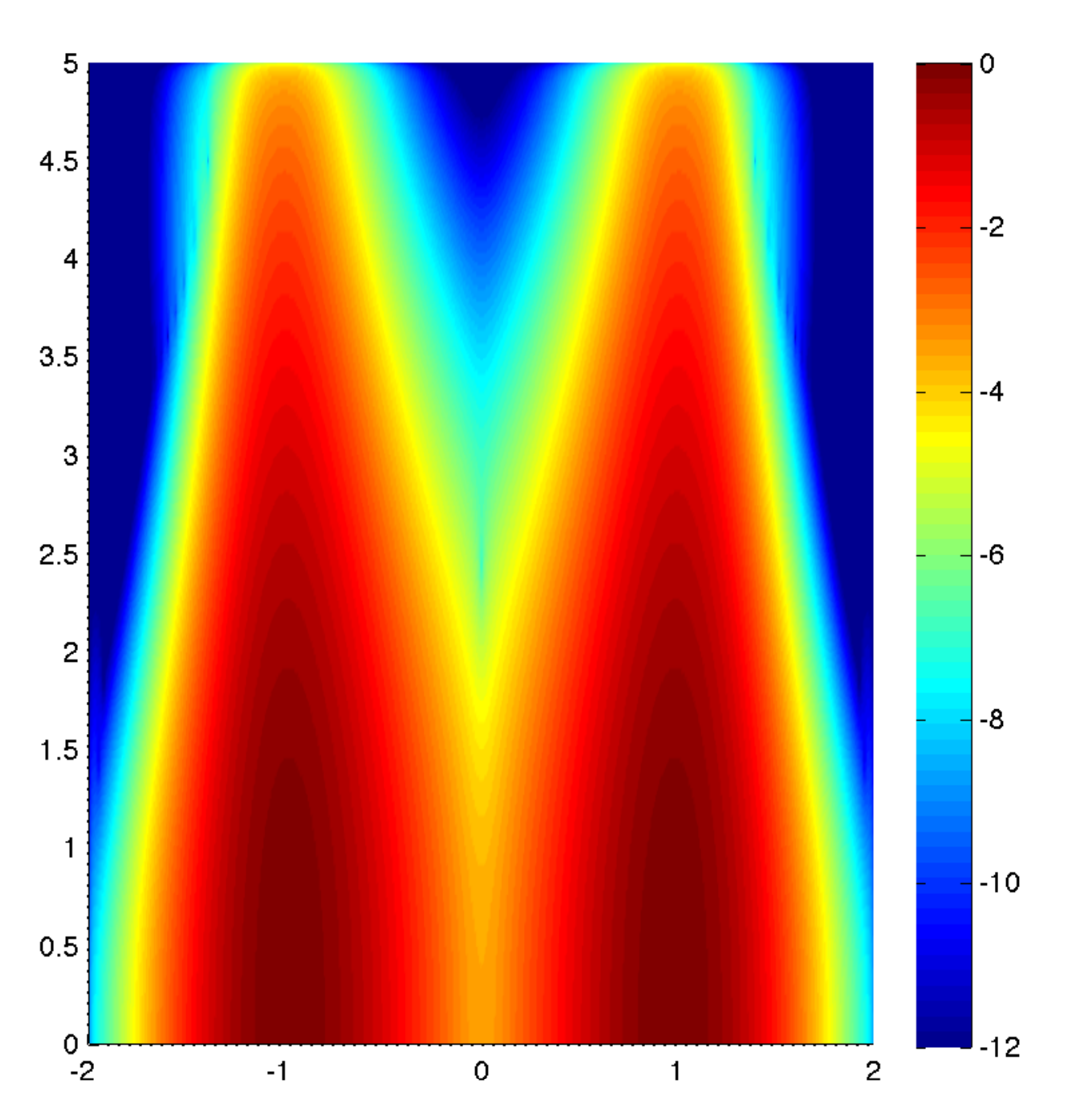}
\includegraphics[width=2.38cm]{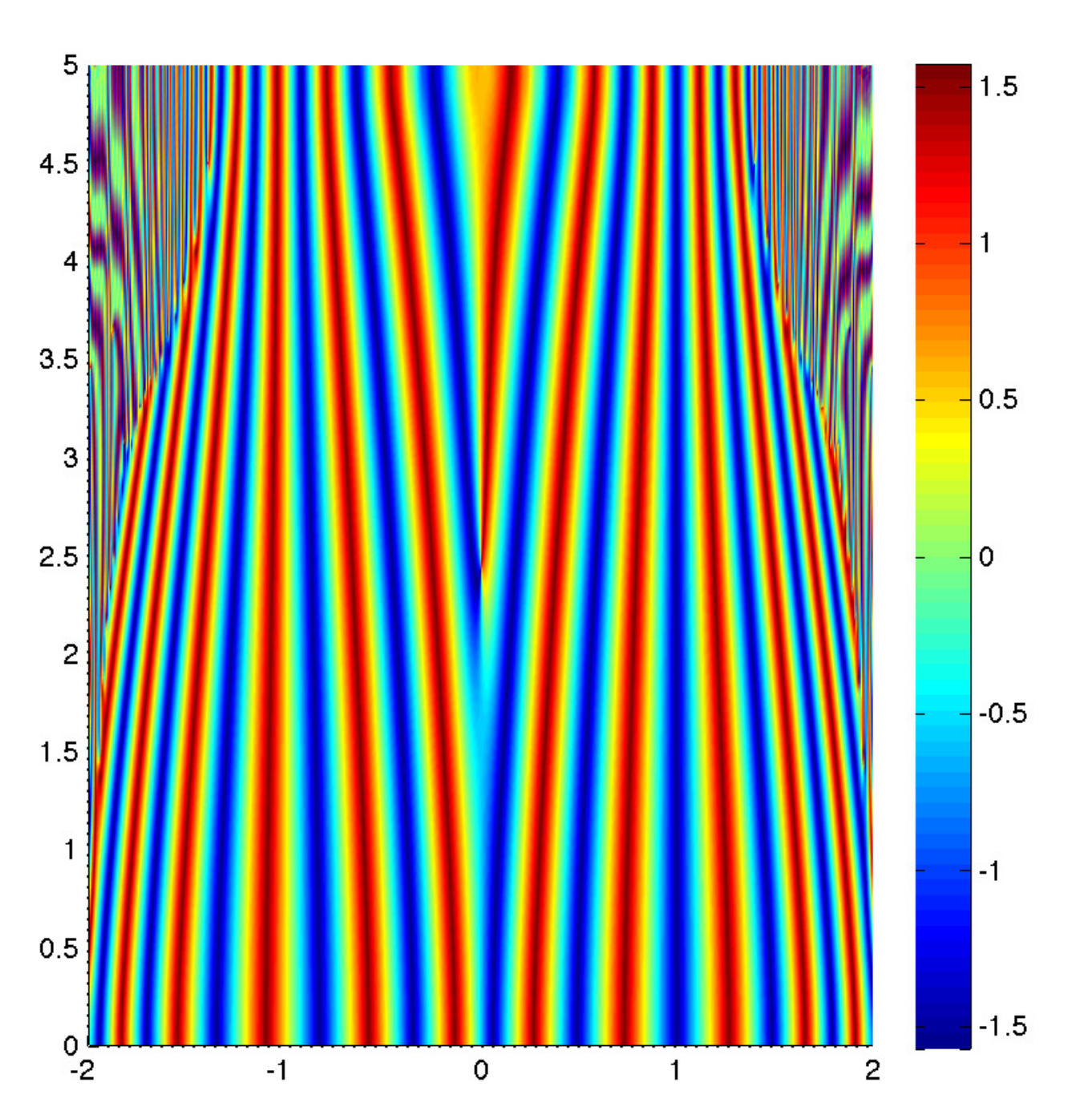}}
\hfill
\subfigure[Second eigenfunction $h=\frac 1{15}$]{\includegraphics[width=2.38cm]{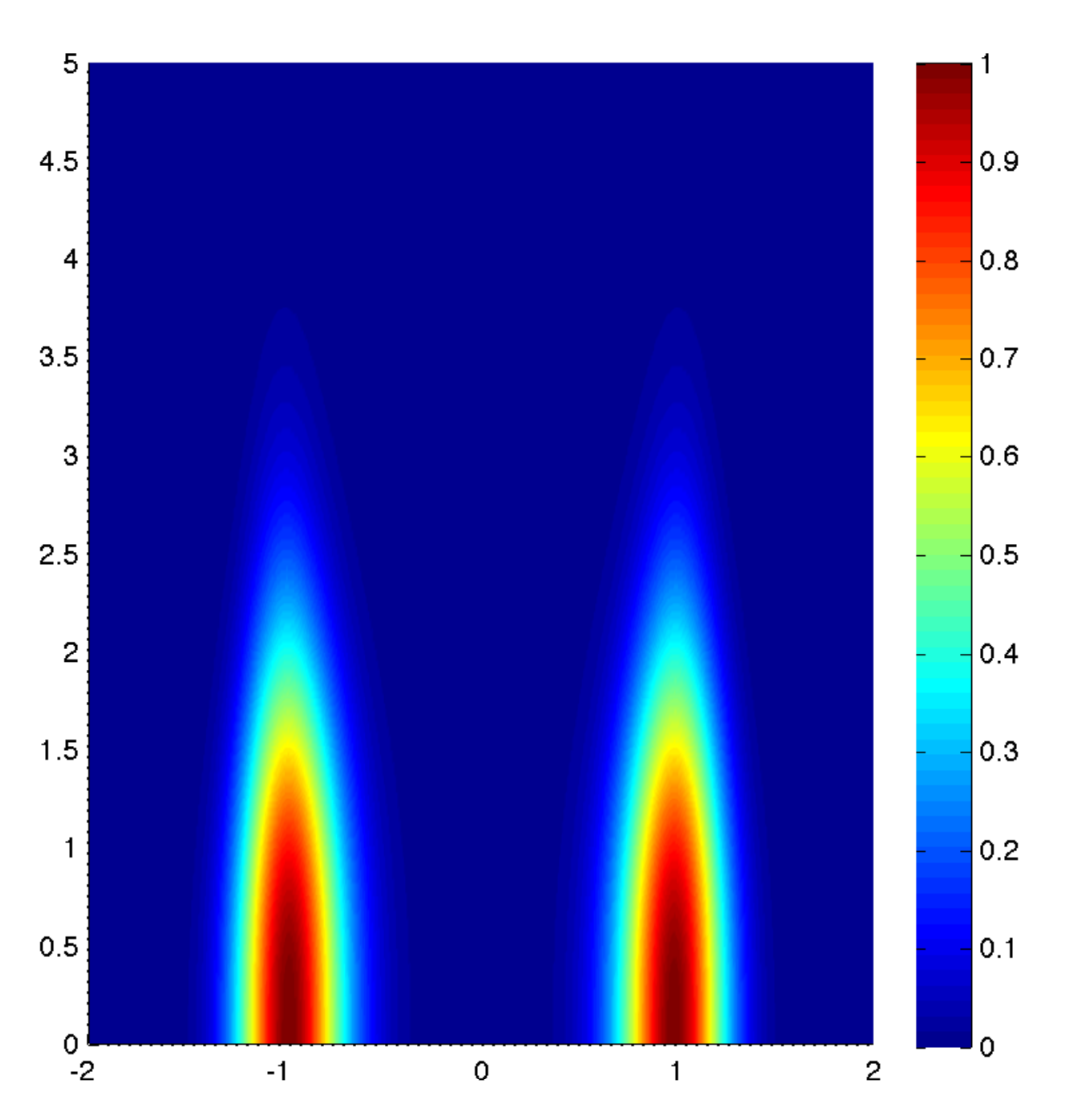}
\includegraphics[width=2.38cm]{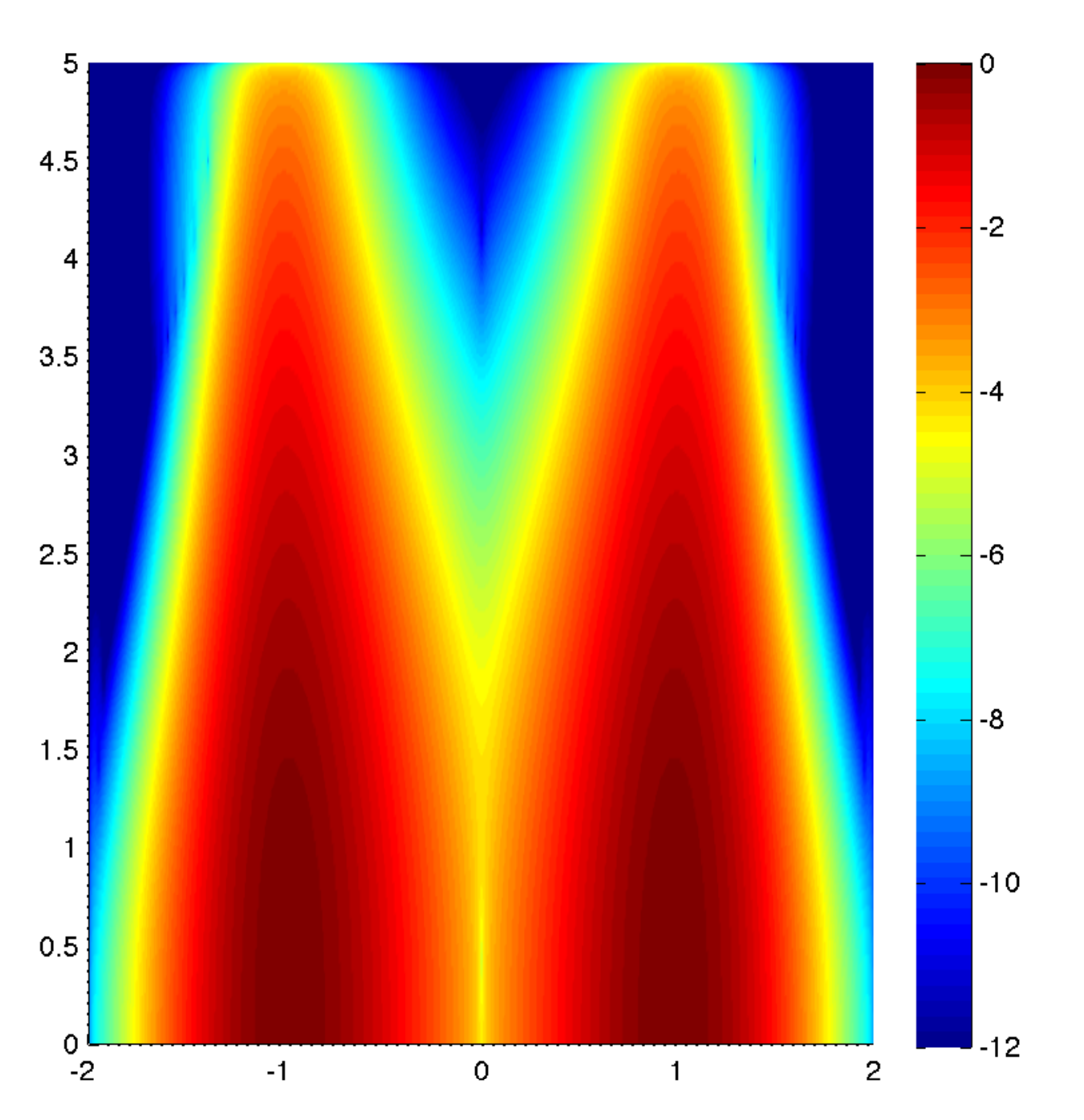}
\includegraphics[width=2.38cm]{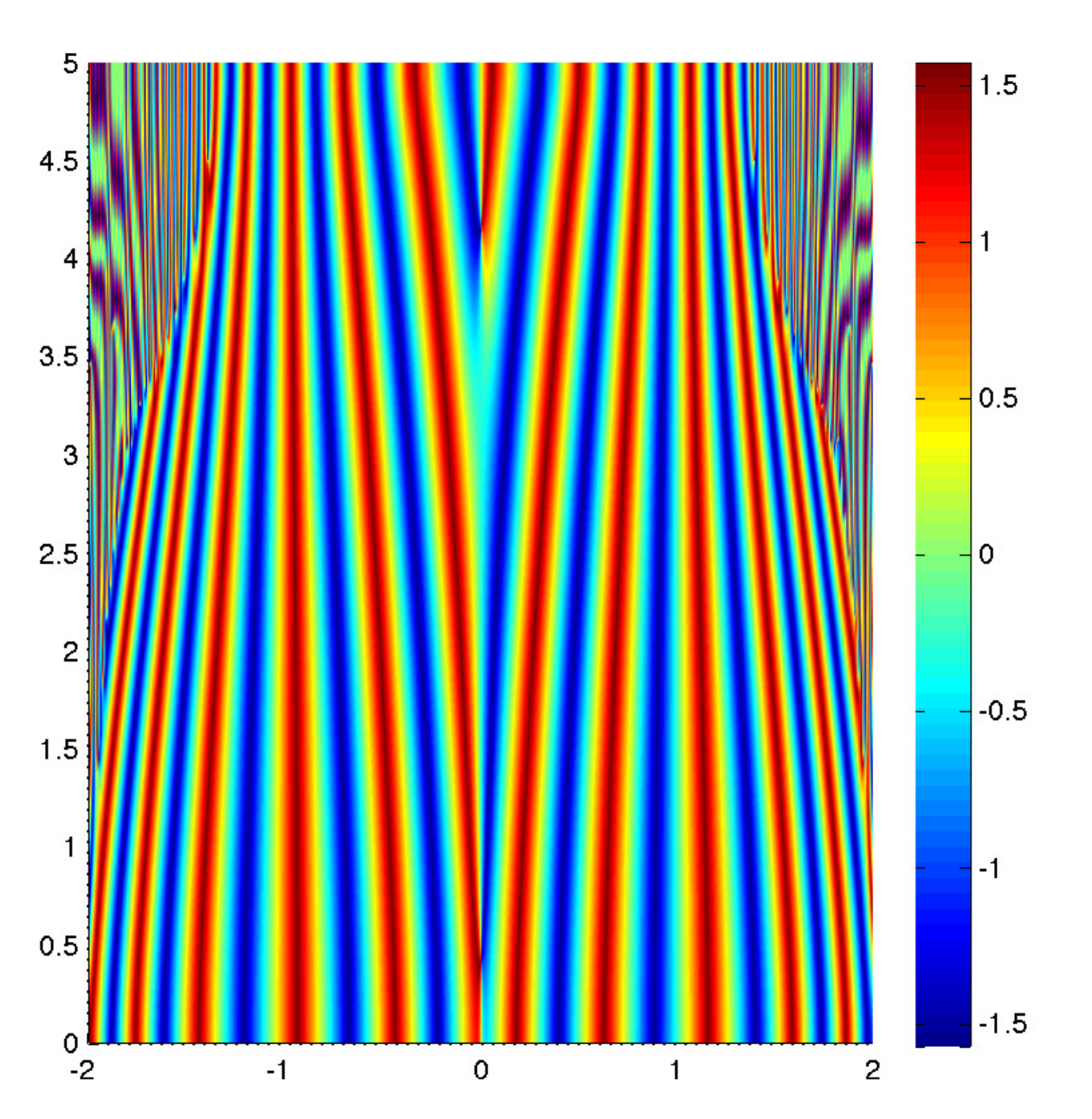}}
\subfigure[First eigenfunction $h=\frac 1{20}$]{\includegraphics[width=2.38cm]{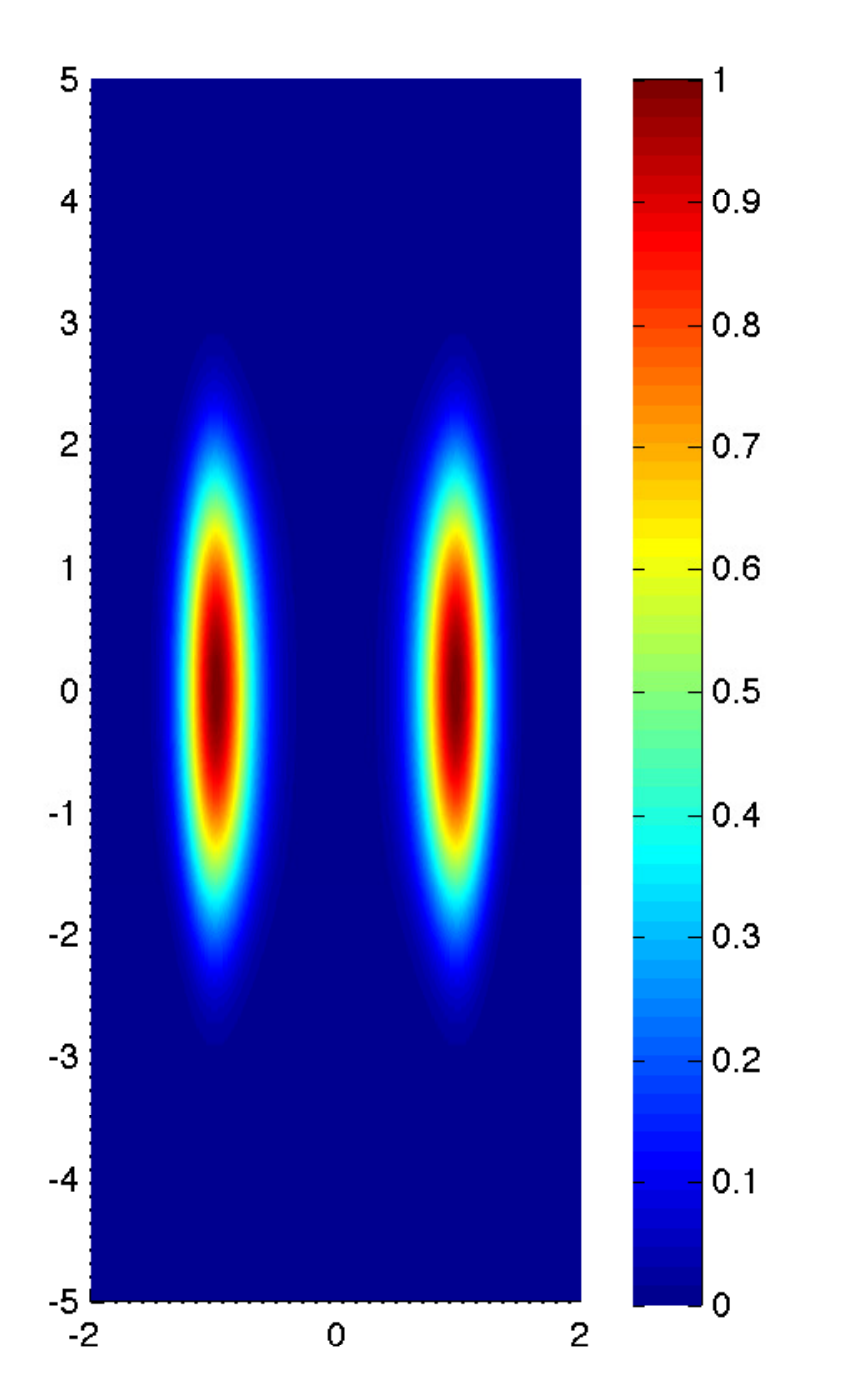}
\includegraphics[width=2.38cm]{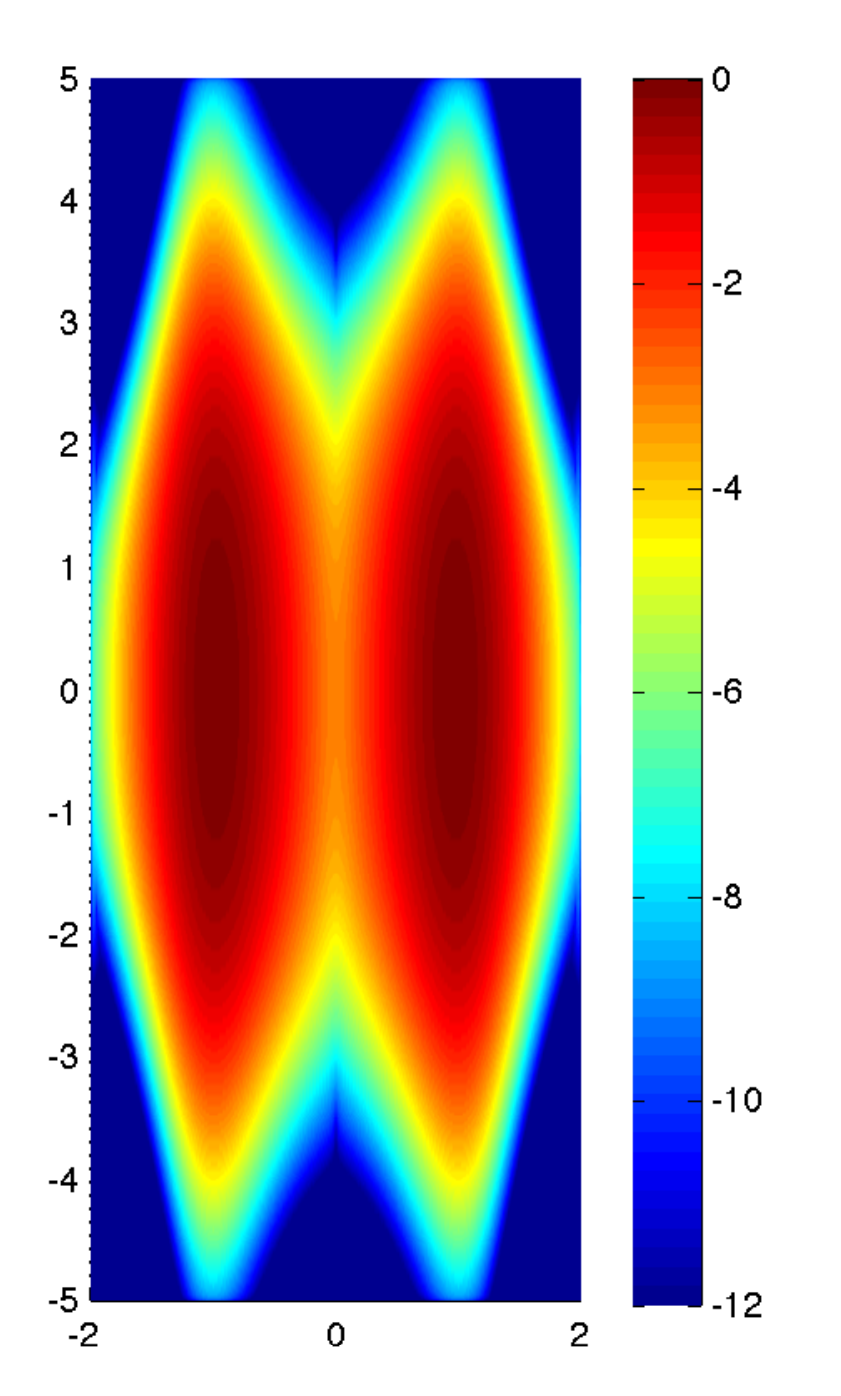}
\includegraphics[width=2.38cm]{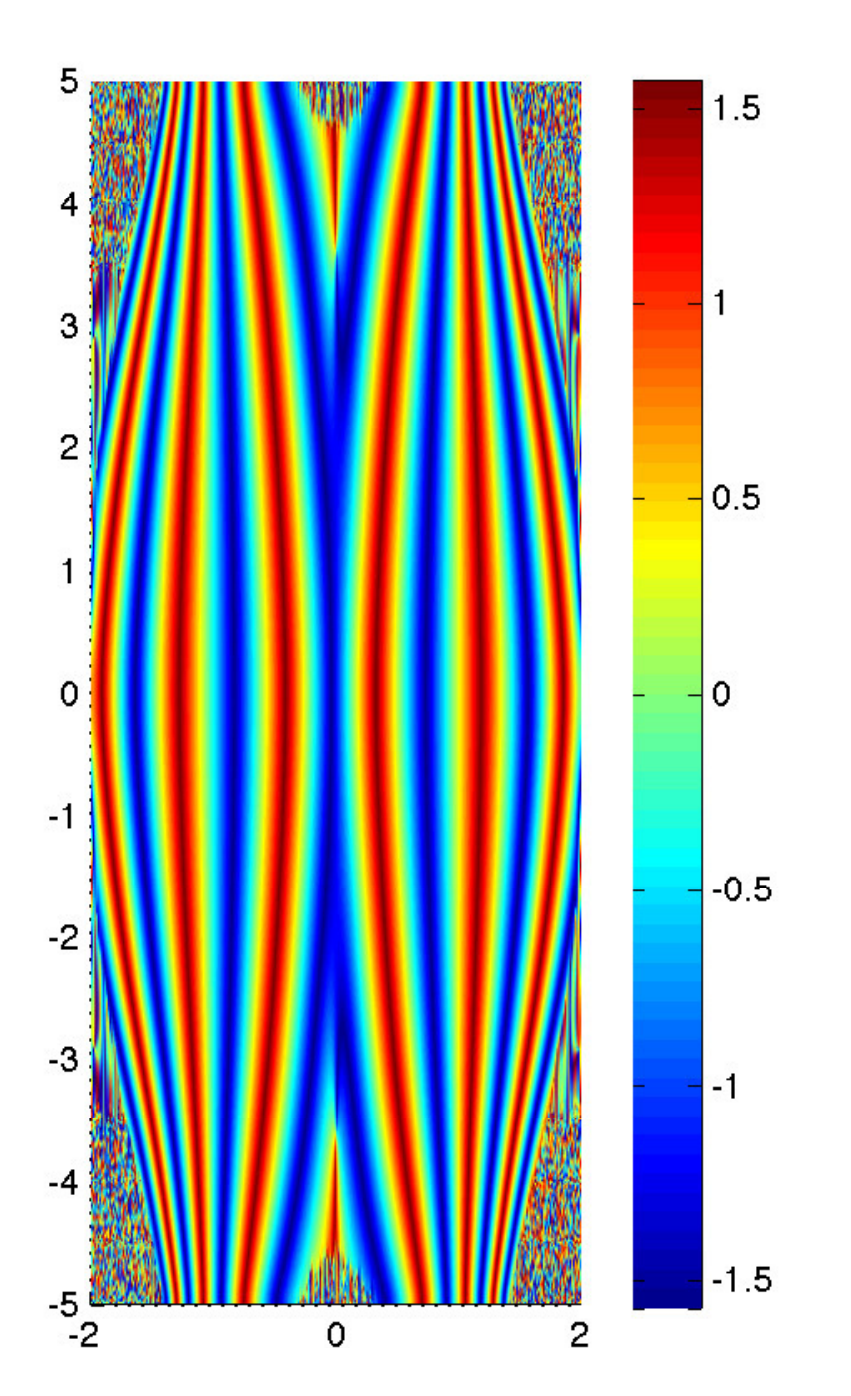}}
\hfill
\subfigure[Second eigenfunction $h=\frac 1{20}$]{\includegraphics[width=2.38cm]{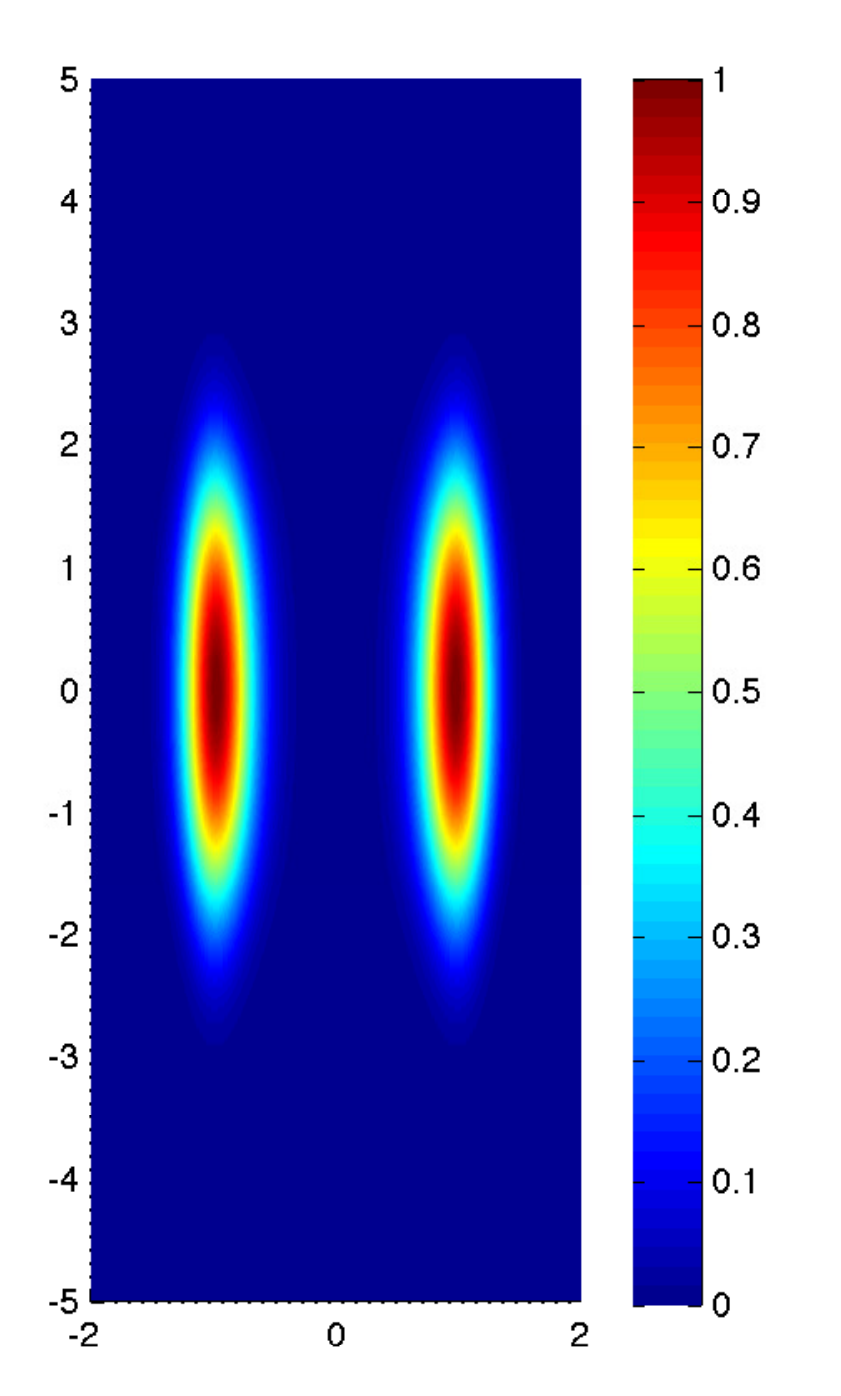}
\includegraphics[width=2.38cm]{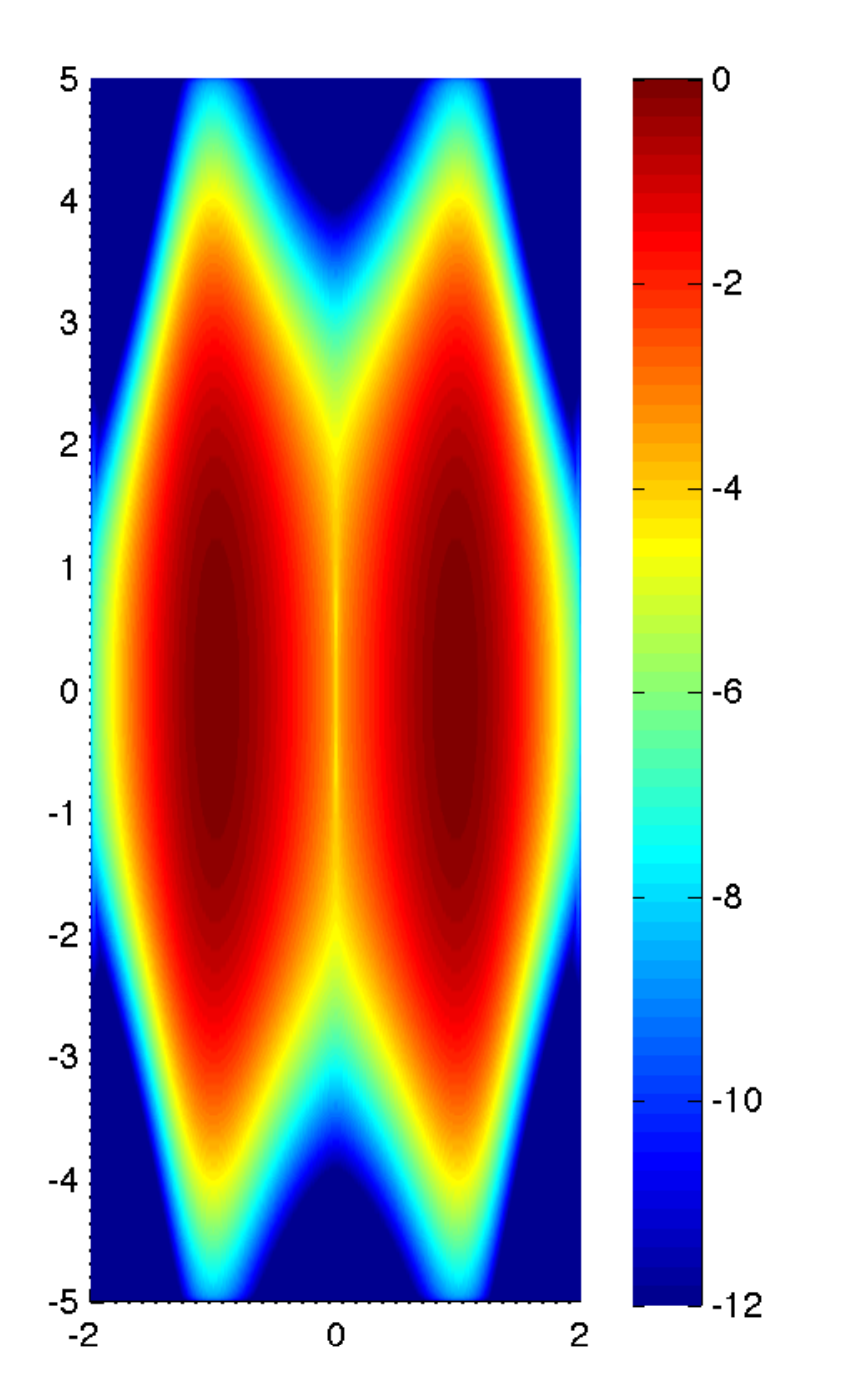}
\includegraphics[width=2.38cm]{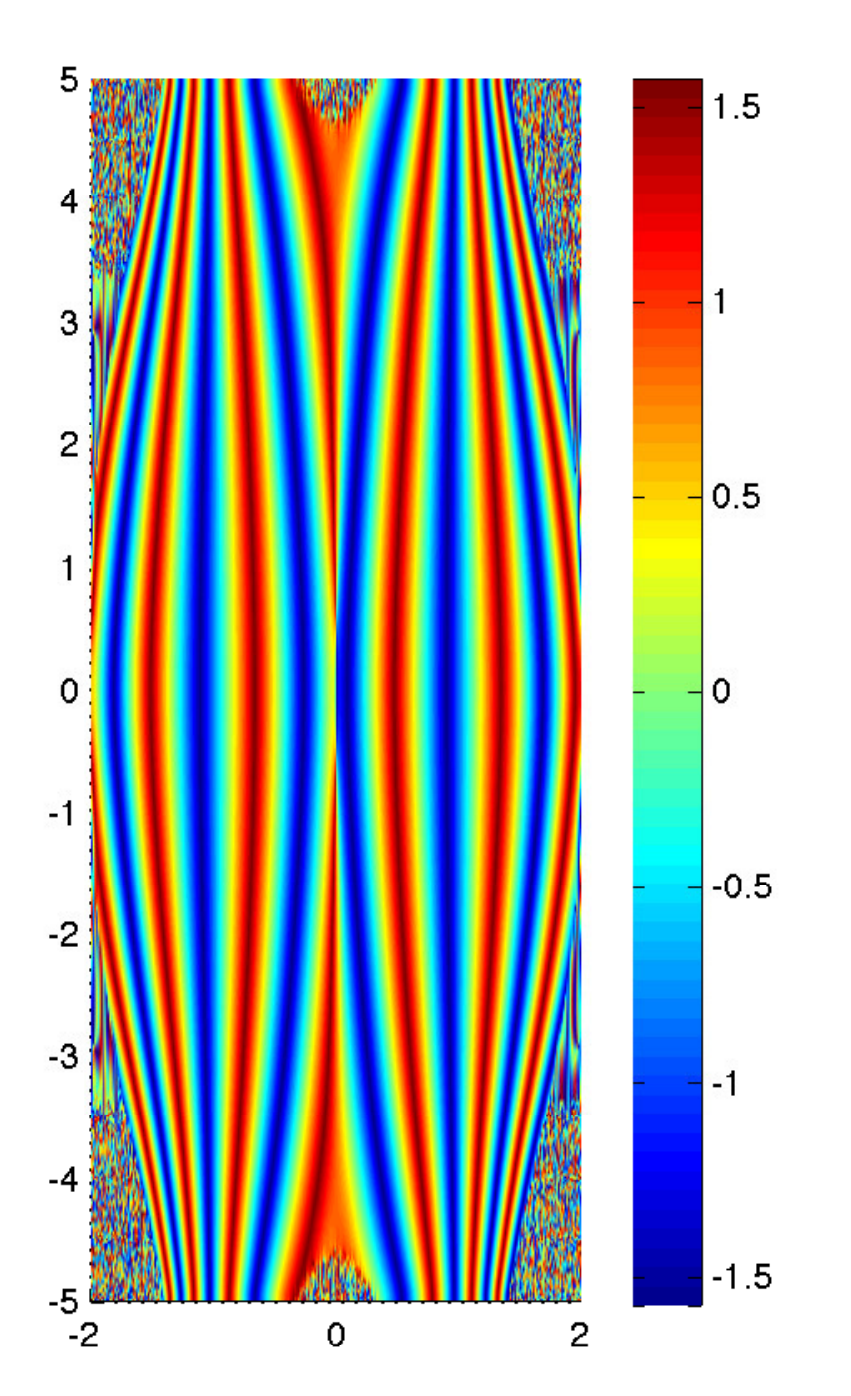}}
\caption{Moduli, log$_{10}$(moduli) and phases of the first two eigenfunctions. \label{fig.VecP}}
\end{center}
\end{figure}
In Figure~\ref{fig.VecP}, we give the first two eigenfunctions (modulus, logarithm of the modulus and phase) of $\mathfrak L_{h}^{[k]}$ for $h=1/15$ if $k=0$ and $h=1/20$ if $k=1$. To compute them, we use a $\mathbb Q_{10}$ approximation and ${\mathcal R}_{k,2,5}$ as artificial domain for computations. We observe the change of symmetry between the first two eigenfunctions: the first eigenfunction satisfies the Neumann condition along the symmetry axis $x=0$ whereas the second one is antisymmetric as it can be seen on the phase or on the  logarithm of the modulus.
If we take too small values for $h$ ($h\leq 1/24$ if $k=0$ and $h\leq 1/33$ if $k=1$), then the accuracy of our computations is no more sufficient to catch the tunneling effect and the modulus the first two eigenfunctions is no more symmetric: The first computed eigenfunction is essentially localized in one well whereas the second one is localized in the other well.

\section{Geometric models: an application of the strategy}\label{Sec:geom}
In this section, we use the same notation as previously but add an exponent $\sharp=\van, \PR,\FH$ to distinguish between our three geometric operators.

\subsection{Vanishing magnetic fields}\label{S:van}
This section is concerned with the result announced in Section \ref{intro-van}.
\subsubsection{Description of the operator in curvilinear coordinates}
If $k\geq 0$ is an integer, we let $\Omega_{k}=\R$ if $k\geq 1$ and $\Omega_{k}=\R_{+}$ if $k=0$. By using the standard tubular coordinates near $\Gamma$ (see \cite[Appendix F]{FouHel10}), we are reduced to analyze the following operator, depending on the integer $k\geq 0$ and acting on $\sL^2(\R\times\Omega_{k}, (1-t\kappa(s))\dx s\dx t)$ and with Neumann condition on $t=0$ if $k=0$:
\begin{multline*}
\mathcal{L}^{\van,[k]}_{\hbar}=(1-t\kappa(s))^{-1}\hbar D_{t}(1-t\kappa(s))\hbar D_{t}\\
+(1-t\kappa(s))^{-1}(\hbar D_{s}-A^{\van,[k]}(s,t))(1-t\kappa(s))^{-1}(\hbar D_{s}-A^{\van,[k]}(s,t)),
\end{multline*}
with
$$A^{\van,[k]}(s,t)=\int_{0}^t(1-t' \kappa(s))B^{\van,[k]}(s,t')\dx t'$$
and where $B^{\van,[k]}$ is a magnetic field which satisfies
$$B^{\van,[k]}(s,t)=\gamma(s)t^{k}+\delta(s)t^{k+1}+\Oc(t^{k+2})$$
so that
$$A^{\van,[k]}(s,t)=\gamma(s)\frac{t^{k+1}}{k+1}+\tilde\delta(s)\frac{t^{k+2}}{k+2}+\Oc(t^{k+3})$$
where
$$\tilde\delta(s)=\delta(s)-\gamma(s)\kappa(s).$$
The function $\kappa$ is nothing but the curvature function of the zero line of the magnetic field (if $k\geq 1$) or of the boundary (if $k=0$).
We will work under the following assumption.
\begin{assumption}\label{van}
The functions $\kappa$ and $\bB$ (or equivalently $B^{\van,[k]}$) are smooth, $\gamma$ is analytic and admits a positive and non degenerate minimum at $s=0$.
\end{assumption}
Let us perform the rescaling
$$h=\hbar^{\frac1{k+2}},\qquad s=\sigma,\qquad t=h\tau.$$
We denote by $\mathfrak{L}^{\van,[k]}_{h}$ the rescaled operator divided by $h^{2k+2}$:
\begin{multline*}
\mathfrak{L}^{\van,[k]}_{h}=(1-h\tau\kappa(\sigma))^{-1}D_{\tau}(1-h\tau\kappa(\sigma))D_{\tau}\\
+(1-h\tau\kappa(\sigma))^{-1}(hD_{\sigma}-A_{h}^{\van,[k]}(\sigma,\tau))(1-h\tau\kappa(\sigma))^{-1}(hD_{\sigma}-A_{h}^{\van,[k]}(\sigma,\tau)),
\end{multline*}
with
$$A_{h}^{\van,[k]}(\sigma,\tau)=h^{-(k+1)}A^{\van,[k]}(\sigma,h\tau).$$

\begin{theorem}\label{WKB-van}
Under Assumption \ref{van}, there exist a function $\Phi=\Phi(\sigma)$ defined in a neighborhood $\Vc$ of $(0,0)$ with  $\Re \Hess\Phi(0) >0$ and, for any $n\geq 1$, a sequence of real numbers $(\lambda_{n,j}^{\van,[k]})_{j\geq 0}$ such that
$$
\lambda_n^{\van,[k]}(h) \underset{h\to 0}{\sim}\sum_{j\geq 0}\lambda_{n,j}^{\van,[k]} h^j,
$$
in the sense of formal series. Besides there exists a formal
series of smooth functions on $ \Vc $
$$
 \an_{n}^{\van,[k]}(. ; h)\underset{h\to 0}{\sim}\sum_{j\geq 0}\an_{n,j}^{\van,[k]} h^j,
$$
with $\an_{n,0}^{\van,[k]}(0,0) \neq 0$ such that
$$
\left(\mathfrak{L}_{h}^{\van,[k]}-\lambda_{n}^{\van,[k]}(h)\right)\left( \an_{n}(. ; h) \re^{-\Phi/h}\right)=\mathcal{O}\left(h^{\infty}\right)  \re^{-\Phi/h}.
$$
We also have that  $\lambda^{{\van,[k]}}_{n,0}=\gamma(0)^{\frac{2}{k+2}}\nun^{[k]}(\zeta_{0}^{[k]})$ and that $\lambda^{{\van,[k]}}_{n,1}$ is the $n$-th eigenvalue of the operator
\begin{equation}\label{op-van}
\tfrac 1 2 \Hess\,\mun^{[k]}(0,\zeta^{[k]}_{0})(\sigma, D_{\sigma})+R^{\van,[k]}(0),
\end{equation}
with
\begin{multline}\label{Rvan}
R^{\van,[k]}(0)=2\gamma(0)\left(\delta(0)+\frac{\kappa(0)\gamma(0)}{k+1}\right)\int_{\Omega_{k}} \frac{\tau^{2k+3}}{(k+1)(k+2)}\left(u_{0,i\Phi'(0)}^{[k]}(\tau)\right)^2   \dx \tau\\
+\kappa(0)\int_{\Omega_{k}} \dr_\tau u_{0,i\Phi'(0)}^{[k]}(\tau) u_{0,i\Phi'(0)}^{[k]}(\tau) \dx\tau.
\end{multline}
The main term in the Ansatz is
$$\an^{\van,[k]}_{n,0}(\sigma, \tau)=f^{\van,[k]}_{n,0}(\sigma)u_{\sigma,i\Phi'(\sigma)}^{[k]}(\tau),$$
where $f^{\van,[k]}_{n,0}(\sigma)$ is the $n$-th normalized eigenfunction of \eqref{op-van}. Moreover, for all $n\geq 1$, there exist $h_{0}>0$, $c>0$ such that for all $h\in(0,h_{0})$, we have
$$\mathcal{B}\big(\lambda^{{\van,[k]}}_{n,0}+\lambda^{{\van,[k]}}_{n,1}h,ch\big)\cap \spe\left(\mathfrak{L}^{\van,[k]}_{h}\right)=\{\lambda^{{\van,[k]}}_{n}(h)\},$$
and $\lambda^{{\van,[k]}}_{n}(h)$ is a simple eigenvalue. 
\end{theorem}

\subsubsection{WKB expansion}
Let us now prove Theorem \ref{WKB-van}.
We have the expansion (in powers of $h$): 
$$A_{h}^{\van,[k]}(\sigma,\tau)=\gamma(\sigma)\frac{\tau^{k+1}}{k+1}+\tilde\delta(\sigma) h \frac{\tau^{k+2}}{k+2}+\Oc(h^2\tau^{k+3}).$$
We can write the following formal series expansion
$$\re^{\Phi(\sigma)/h}\ \mathfrak{L}^{\van,[k]}_{h}\ \re^{-\Phi(\sigma)/h}\sim\sum_{j\geq 0} h^j\mathfrak{L}_{j},$$
where we have
$$\mathfrak{L}_{0}=D_{\tau}^2+\left(i\Phi'(\sigma)-\gamma(\sigma)\frac{\tau^{k+1}}{k+1}\right)^2$$
and
$$\mathfrak{L}_{1}=-\frac{\tau^{k+1}}{k+1}(D_{\sigma}\gamma(\sigma)+\gamma(\sigma) D_{\sigma})+\frac{2\gamma(\sigma)\tilde\delta(\sigma)\tau^{2k+3}}{(k+2)(k+1)}+\kappa(\sigma)\dr_{\tau}+2\tau\kappa(\sigma)\left(\gamma(\sigma)\frac{\tau^{k+1}}{k+1}\right)^2.$$
Our Ans\"atze are again in the form
$$\an\sim\sum_{j\geq 0} h^j\an_{j},\qquad \lambda\sim\sum_{j\geq 0} h^j\lambda_{j}.$$
The first equation is given by
$$\mathfrak{L}_{0}\an_{0}=\lambda_{0}\an_{0}.$$
This leads to the choice
$$\an_{0}(\sigma,\tau)=f_{0}(\sigma)u_{\sigma,i\Phi'(\sigma)}^{[k]}(\tau)$$
and $\Phi$ must be such that
$$\mun^{[k]}(\sigma,i\Phi'(\sigma))=\lambda_{0}$$
and $\lambda_{0}=\gamma(0)^{\frac{2}{k+2}}\nun^{[k]}(\zeta_{0}^{[k]})$.
The next equation to solve is
$$(\mathfrak{L}_{0}-\mu_{0})\an_{1}=(\lambda_{1}-\mathfrak{L}_{1})\an_{0}$$
and the associated Fredholm condition is given by
$$\left\langle\mathfrak{L}_{1}\an_{0},u_{\sigma,-i \overline{\Phi}'(\sigma)}^{[k]}\right\rangle_{\sL^2(\Omega_{k},\dx \tau)}=\lambda_{1}f_{0}(\sigma).$$
We get the transport equation
$$\left(D_{\sigma}\dr_{\zeta}\mun^{[k]}(\sigma,i\Phi'(\sigma))+\dr_{\zeta}\mun^{[k]}(\sigma,i\Phi'(\sigma))D_{\sigma}\right)f_{0}+R^\van(\sigma)f_{0}=\lambda_{1}f_{0},$$
where $R^\van$ is an explicit smooth function. Considering the linearized equation near $\sigma=0$, we are led to choose $\lambda_{1}$ in the set
$$\spe\left(\tfrac 1 2 \Hess\,\mun^{[k]}(0,\zeta^{[k]}_{0})(\sigma, D_{\sigma})\right)+R^\van(0).$$
We recognize the set which appears in \cite[$\gamma_{n,2}$ in Theorem 1.3]{Ray11b} (for $k=0$), \cite[$\theta_{2}^n$ in Theorem 1.6]{DomRay12} (for $k=1$) and in the conjecture of \cite[$A$ for our $R^\van(0)$ in (4.5)]{HelKo09} (for $k\geq 1$). In particular the simplicity of the eigenvalues is established in \cite{Ray11b, DomRay12} for $k=0,1$ whereas slight adaptations have to be done to deal with the case $k\geq 2$.

\subsection{Along a varying edge in dimension three}\label{S:edge}
Let us now deal with the situation described in Section \ref{intro-edge}. Let us recall that the bottom of the spectrum of the magnetic Neumann Laplacian on the wedge $\mathcal{W}_{\alpha}$ with constant aperture $\alpha$ with a magnetic field normal to the symmetry plane is a non increasing function with respect to $\alpha$ (see \cite{Popoff}). The wedge is so that $\mathcal{W}_{\alpha}=\R\times\mathcal{S}_{\alpha}$ where $\mathcal{S}_{\alpha}$ is the angular sector in $\R^2$.

\subsubsection{Framework}
We will need the Neumann realization of the operator defined on $\sL^2(\mathcal{S}_{\alpha_{0}},\dx t \dx z)$ by
$$\mathcal{M}^{\PR}_{s,\eta}=D_{t}^2+\cT(s)^{-2}\cT(0)^2D_{z}^2+(\eta-t)^2,$$
whose form domain is
$$\Dom(\mathcal{Q}^{\PR}_{s,\eta})=\left\{\psi\in\sL^2(\mathcal{S}_{\alpha_{0}}) : D_{t}\psi\in\sL^2(\mathcal{S}_{\alpha_{0}}), D_{z}\psi\in\sL^2(\mathcal{S}_{\alpha_{0}}), t\psi\in \sL^2(\mathcal{S}_{\alpha_{0}}) \right\}$$
and with operator domain
$$\Dom(\mathcal{M}^{\PR}_{s,\eta})=\left\{\psi\in\Dom(\mathcal{Q}^{\PR}_{s,\eta}) :  \mathcal{M}^{\PR}_{s,\eta}\psi\in\sL^2(\mathcal{S}_{\alpha_{0}}), \mathfrak{T}(s)\psi=0\right\},$$
where
$$\mathfrak{T}(s)=-\sgn(z) D_{t}+\cT(s)^{-2}\cT(0) D_{z}.$$
The family $(\mathcal{M}^{\PR}_{s,\eta})_{(s,\eta)\in\R^2}$ is analytic of type (A). Note that, near each point $(s_{1},\eta_{1})\in\R^2$, this family can be holomorphically extended. The lowest eigenvalue of $\mathcal{M}^{\PR}_{s,\eta}$ is denoted by $\mun^\PR(s,\eta)$. As in \cite{PoRay12}, we will also investigate the consequences of the following conjecture (see \cite[Remark 1.8]{PoRay12}).
\begin{conjecture}\label{conj}
For all $\alpha_{0}\in(0,\pi)$, the function  $\eta\mapsto\mun^\PR(0,\eta)$ admits a unique critical point $\eta_{0}$ which is a non degenerate minimum.
\end{conjecture}

\begin{proposition}
Under Assumption \ref{alpha-max} and if Conjecture \ref{conj} is true, the function $\mun^\PR$ admits a local non degenerate minimum at $(0,\eta_{0})$. Moreover the Hessian at $(0,\eta_{0})$ is given by
\begin{equation}\label{symbol-PR}
4\kappa\mathcal{T}(0)^{-1}\|D_{z}u^{\PR}_{0,\eta_{0}}\| s^2+\left(\partial^2_{\eta}\mun^\PR\right)_{0,\eta_{0}}\eta^2,
\end{equation}
where $\kappa=-\frac{\mathcal{T}''(0)}{2}>0$.
\end{proposition}
\begin{proof}
The proof follows from the perturbation theory. We have the eigenvalue equation
$$\mathcal{M}^{\PR}_{s,\eta} u^{\PR}_{s,\eta}=\mun^\PR(s,\eta)u^{\PR}_{s,\eta},	\qquad \mathfrak{T}(s)u^{\PR}_{s,\eta}=0.$$
We notice that  $\mathfrak{T}'(0)=0$ and $\mathfrak{T}''(0)=4\kappa\mathcal{T}(0)^{-2}D_{z}$. Let us analyze the derivative with respect to $s$. We have
$$(\mathcal{M}^{\PR}_{s,\eta}-\mun^\PR(s,\eta))\left(\partial_{s}u^{\PR}\right)_{s,\eta}=\partial_{s}\mun^\PR(s,\eta)u^{\PR}_{s,\eta}-\left(\partial_{s}\mathcal{M}^\PR_{s,\eta}\right)u^{\PR}_{s,\eta}.$$
We notice that $\partial_{s}\mathcal{M}^\PR_{0,\eta}=0$ and $\mathfrak{T}(0)\left(\partial_{s}u^{\PR}\right)_{0,\eta}=0$. This implies that $\partial_{s}\mun^\PR(0,\eta)=0$ by the Fredholm condition. Therefore $(0,\eta_{0})$ is a critical point of $\mun^\PR$. Let us now consider the derivative with respect to $s$ and $\eta$. We have $\partial_{s}\partial_{\eta}\mathcal{M}_{s,\eta}^\PR=0$ and by the Feynman-Hellmann formula
$$\partial_{s}\mun^\PR(0,\eta)=\int_{\mathcal{S}_{\alpha_{0}}} \partial_{s}\mathcal{M}_{0,\eta}^\PR u^{\PR}_{0,\eta} u^{\PR}_{0,\eta}\dx z\dx t,$$
we get
$$\partial_{s}\partial_{\eta}\mun^\PR(0,\eta_{0})=0.$$
We shall now analyze the second order derivative with respect to $s$:
$$(\mathcal{M}^{\PR}_{0,\eta_{0}}-\mun^\PR(0,\eta_{0}))\left(\partial_{s}^2u^{\PR}\right)_{0,\eta_{0}}=\partial^2_{s}\mun^\PR(0,\eta_{0})u^{\PR}_{0,\eta_{0}}-\partial^2_{s}\mathcal{M}^{\PR}_{0,\eta_{0}}u^{\PR}_{0,\eta_{0}},$$
with boundary condition $\mathfrak{T}(0)\left(\partial^2_{s}u^{\PR}\right)_{0,\eta_{0}}=-\mathfrak{T}''(0)u^{\PR}_{0,\eta_{0}}$. We have $\partial^2_{s}\mathcal{M}^{\PR}_{0,\eta_{0}}=4\kappa\mathcal{T}(0)^{-1}D_{z}^2$. With the Fredholm condition, we get $\partial^2_{s}\mun^\PR(0,\eta_{0})=4\kappa \mathcal{T}(0)^{-1}\|D_{z}u^{\PR}_{0,\eta_{0}}\|^2$.
\end{proof}
The function \eqref{symbol-PR} is the symbol of the effective harmonic oscillator introduced in \cite{PoRay12}. We will see that our WKB analysis succeeds as soon as we work near a local and non degenerate minimum of $\mun^\PR$. The goal of Assumption \ref{alpha-max} and Conjecture \ref{conj} is to provide sufficient conditions to have such a critical point as well as to get the spectral splitting as in \cite{PoRay12}.

\subsubsection{Normal form}

We introduce the change of variables
$$\check{s}= s,\qquad \check{t}=t,\qquad \check{z}=\cT( s)^{-1}\cT(0) z$$
and we let
$$\check\nabla_{\hbar}=\begin{pmatrix}
\hbar D_{\check s}\\
\hbar D_{\check t}\\
\hbar\cT(\check s)^{-1}\cT(0)D_{\check z}
\end{pmatrix}+
\begin{pmatrix}
-\check t-\hbar\frac{\cT'}{2\cT}(\check z D_{\check z}+D_{\check z}\check z)\\
0\\
0
\end{pmatrix}.$$
The operator $\mathfrak{L}^{\PR}_{\hbar}$ is unitarily equivalent to the operator $\check{\mathfrak{L}}^{\PR}_{\hbar}$ on $\sL^2(\mathcal{W}_{\alpha_{0}},\dx \check s\dx\check t\dx\check z)$ defined by
$$\check{\mathcal{L}}^{\PR}_{\hbar}=\left(\hbar D_{\check s}-\check t-\hbar\frac{\cT'(\check s)}{2\cT(\check s)}(\check z D_{\check z}+D_{\check z}\check z)\right)^2+\hbar^2D_{\check t}^2+\hbar^2\cT(\check s)^{-2}\cT(0)^2 D_{\check z}^2.$$
The boundary condition becomes, on $\dr\mathcal{W}_{\alpha_{0}}$,
\begin{equation}\label{n-check}
\check\nabla_{\hbar}\check\psi\cdot\check \bfn=0,\qquad\mbox{ with }\qquad
\check\bfn=\begin{pmatrix}
-\cT'(\check s)\check t
 \\
-\cT(\check s)
\\
\pm 1
\end{pmatrix}.
\end{equation}
We now perform the scaling which preserves $\mathcal{W}_{\alpha_{0}}$:
$$h=\hbar^{1/2},\qquad \check s=\hs,\qquad \check t=h^{1/2}\htt,\qquad \check z=h^{1/2}\hz.$$
The operator $\hbar^{-1}\check{\mathcal{L}}^{\PR}_{\hbar}$ becomes
$\mathfrak{L}^{\PR}_{h}$:
$$\mathfrak{L}^{\PR}_{h}
=\left(h D_{\hs}-\htt-h \frac{\cT'(\hs)}{2\cT(\hs)}(\hz D_{\hz}+D_{\hz}\hz)\right)^2+D_{\htt}^2+\cT(\hs)^{-2}\cT(0)^2 D_{\hz}^2.$$
Now, the boundary condition is, on $\dr\mathcal{W}_{\alpha_{0}}$,
\begin{equation}\label{n-checkbis}
\widehat{\nabla}_{h}\psi\cdot\hat \bfn=0 \qquad\mbox{ with }\qquad
\hat \bfn=\hbfn_{0}+h^{1/2}\bfn_{1}
\end{equation}
where
\begin{equation*}
\hbfn_{0}=\begin{pmatrix}
0
 \\
-\cT(\hs)
\\
\pm 1
\end{pmatrix}, \qquad \hbfn_{1}=\begin{pmatrix}
-\cT'(\hs)\htt
 \\
0
\\
0
\end{pmatrix}
\end{equation*}
and
$$\widehat{\nabla}_{h}=\begin{pmatrix}
h D_{\hs}\\
D_{\htt}\\
\cT(\hs)^{-1}\cT(0)D_{\hz}
\end{pmatrix}+
\begin{pmatrix}
-\htt-h \frac{\cT'(\hs)}{2\cT(\hs)}(\hz D_{\hz}+D_{\hz}\hz)\\
0\\
0
\end{pmatrix}.$$

\begin{theorem}\label{WKB-edge}
Under Assumption \ref{alpha-max} and Conjecture \ref{conj}, there exist 
 a function $\Phi=\Phi(\sigma)$ defined in a neighborhood $\Vc$ of $0$ such that $\Re \Hess\Phi(0) >0$ on $\Vc$ and sequence  of real numbers $(\lambda^{\PR}_{n,j})_{j\geq 0}$ such that 
 $$
 \lambda_{n}^{\PR}(h)\underset{h\to 0}{\sim}\sum_{j\geq 0} \lambda^{\PR}_{n,j} h^j.
 $$
 in the sense of formal series.
 Besides there exists a formal series of smooth functions $(\an^{\PR}_{n,j}(\sigma, \tau,\hz))$ defined for $(\sigma, \tau,\hz)\in \Vc\times\mathcal{S}_{\alpha_{0}}$,
 $$\an_{n}^{\PR}\underset{h\to 0}{\sim} \sum_{j \geq 0 } \an^{\PR}_{n,j} h^{j} ,$$
such that
$$
\left(\mathfrak{L}^{\PR}_{h}-\lambda_{n}^{\PR}(h)\right)\left( \an_{n}^{\PR} \re^{-\Phi/h}\right)=\mathcal{O}\left(h^{\infty} \right)  \re^{-\Phi/h}.$$
We also have that  $\lambda^{\PR}_{n,0}=\mun^\PR(0,\eta_{0})$ and that $\lambda^{\PR}_{n,1}$ is the $n$-th eigenvalue of the operator
\begin{equation}\label{op-PR}
\tfrac 1 2 \Hess\,\mun^{\PR}(0,\eta_{0})(\sigma, D_{\sigma}).
\end{equation}
The main term in the Ansatz is in the form $\an^{\PR}_{n,0}(\sigma, \tau,\hz)=f^\PR_{n,0}(\sigma)u_{\sigma,i\Phi'(\sigma)}^{\PR}(\tau,\hz)$.
Moreover, for all $n\geq 1$, there exist $h_{0}>0$, $c>0$ such that for all $h\in(0,h_{0})$,  we have
$$\mathcal{B}(\lambda^{\PR}_{n,0}+\lambda^{\PR}_{n,1}h,ch)\cap \spe\left(\mathfrak{L}^\PR_{h}\right)=\{\lambda^{\PR}_{n}(h)\},$$
and $\lambda^{\PR}_{n}(h)$ is a simple eigenvalue. 
\end{theorem}

\subsubsection{WKB expansion for the normal form}
Let us consider the conjugate operator
$$\re^{\Phi(\hs)/h}\mathfrak{L}^{\PR}_{h}
\re^{-\Phi(\hs)/h}=\left(h D_{\hs}+i\Phi'(\hs)-\htt-h\frac{\cT'(\sigma)}{2\cT(\sigma)}(\hz D_{\hz}+D_{\hz}\hz)\right)^2 +D_{\htt}^2+\frac{\cT(0)^2}{\cT(\hs)^{2}} D_{\hz}^2,$$
with the corresponding boundary conditions. We can write the formal power series expansion:
$$\re^{\Phi(\hs)/h}\mathfrak{L}^{\PR}_{h}
\re^{-\Phi(\hs)/h}\sim \sum_{j\geq 0} h^{j}\mathfrak{L}_{j}$$
with
$$\mathfrak{L}_{0}=D_{\htt}^2+\cT(\hs)^{-2}\cT(0)^2 D_{\hz}^2+(i\Phi'(\hs)-\htt)^2,$$
$$\mathfrak{L}_{1}=(i\Phi'(\hs)-\htt)D_{\hs}+D_{\hs}(i\Phi'(\hs)-\htt)-(i\Phi'(\hs)-\htt)\frac{\cT'(\sigma)}{\cT(\sigma)}(\hz D_{\hz}+D_{\hz}\hz).$$
Our Ans\"atze are in the form:
$$\an\sim \sum_{j\geq 0} h^{j}\an_{j},\qquad\lambda\sim\sum_{j\geq 0} h^{j}\lambda_{j}.$$
The first equation is given by
$$\mathfrak{L}_{0}\an_{0}=\lambda_{0}\an_{0},$$
with boundary condition (which is in fact a Neumann condition)
$$\begin{pmatrix}
-\htt\\
D_{\htt}\\
\cT(\hs)^{-1}\cT(0)D_{\hz}
\end{pmatrix}\an_{0}\cdot\bfn_{0}=0.$$
We take
$$\an_{0}(\hs,\htt,\hz)=f_{0}(\hs)u^\PR_{\hs,i\Phi'(\hs)}(\htt,\hz)$$
and $\lambda_{0}=\mun^\PR(0,\eta_{0})$. The equation becomes
$$\mun^{\PR}(\hs,i\Phi'(\hs))=\lambda_{0}.$$
The second equation is
$$(\mathfrak{L}_{0}-\lambda_{0})\an_{1}=(\lambda_{1}-\mathfrak{L}_{1})\an_{0}$$
with boundary condition
$$\begin{pmatrix}
-\htt\\
D_{\htt}\\
\cT(\hs)^{-1}\cT(0)D_{\hz}
\end{pmatrix}\an_{0}\cdot\bfn_{1}+\begin{pmatrix}
-\htt\\
D_{\htt}\\
\cT(\hs)^{-1}\cT(0)D_{\hz}
\end{pmatrix}\an_{1}\cdot\bfn_{0}=0.$$
The Fredholm condition can be rewritten in the form
$$\left\{\frac{1}{2}\left(\dr_{\eta}\mun^{\PR}(\hs,i\Phi'(\hs)) D_{\hs}+D_{\hs}\partial_{\eta}\mun^{\PR}(\hs,i\Phi'(\hs)) \right)+R^\PR(\hs)\right\}f_{0}=\lambda_{1}f_{0},$$
where the smooth function $\hs\mapsto R^\PR(\hs)$ vanishes at $\hs=0$ since $\cT'(0)=0$. The conclusion follows by iteration. The simplicity of the lowest eigenvalues follows from \cite[Theorem 1.14]{PoRay12}.

\subsection{Curvature induced magnetic bound states}
This section is devoted to the analysis of the result announced in Section \ref{intro-c}.
\subsubsection{A higher order degeneracy}
Let us consider the following Neumann realization on $\sL^2(\R^2_{+}, m(s,t)\dx s\dx t)$,
\begin{multline}
\mathcal{L}^{\FH}_{\hbar}=m(s,t)^{-1}\hbar D_{t}m(s,t)\hbar D_{t}\\
+m(s,t)^{-1}\left(\hbar D_{s}+\zeta_{0}\hbar^{\frac 1 2}-t+\kappa(s)\frac{t^2}{2}\right)m(s,t)^{-1}\left(\hbar D_{s}+\zeta_{0}\hbar^{\frac 1 2}-t+\kappa(s)\frac{t^2}{2}\right),
\end{multline}
 where $m(s,t)=1-t\kappa(s)$. Thanks to the rescaling
 $$ h=\hbar^{1/2},\qquad t=h\tau,\qquad s=\sigma,$$
 and after division by $h^2$ the operator $\mathcal{L}^{\FH}_{\hbar}$ becomes
\begin{multline}
\mathfrak{L}^{\FH}_{h}=m(\sigma,h\tau)^{-1}D_{\tau}m(\sigma,h\tau)D_{\tau}\\
+m(\sigma,h\tau)^{-1}\left(hD_{\sigma}+\zeta_{0}-\tau+h\kappa(\sigma)\frac{\tau^2}{2}\right)m(\sigma,h\tau)^{-1}\left(hD_{\sigma}+\zeta_{0}-\tau+h\kappa(\sigma)\frac{\tau^2}{2}\right),
\end{multline}
on the space $\sL^2(m(\sigma,h\tau)\dx \sigma\dx\tau)$.

\begin{theorem}\label{WKB-FH}
Under Assumption \ref{kappa-max}, there exist a function 
$$\Phi=\Phi(\sigma)=\left(\frac{2C_{1}}{\nu''(\zeta_{0})}\right)^{1/2}\left|\int_{0}^\sigma (\kappa(0)-\kappa(\varsigma))^{1/2}\dx \varsigma\right|$$
defined in a neighborhood $\Vc$ of $(0,0)$ such that $\Re\Phi''(0)>0$, and a sequence of real numbers $(\lambda^{\FH}_{n,j})_{j\geq 0}$ such that 
$$
  \lambda_{n}^{\FH}(h)\underset{h\to 0}{\sim} \sum_{j\geq 0} \lambda^{\FH}_{n,j} h^{\frac j 2}.
$$
Besides there exists a formal series of smooth functions on $ \Vc$,
$$
 \an_{n}^{\FH} \underset{h\to 0}{\sim} \sum_{j\geq 0} \an^{\FH}_{n,j} h^{\frac j 2}
$$ 
such that 
$$
\left(\mathfrak{L}^{\FH}_{h}-\lambda_{n}^{\FH}(h)\right)\left(\an_{n}^{\FH} \re^{-\Phi/h^{\frac12}}\right)=\mathcal{O}\left(h^{\infty} \right) \re^{-\Phi/h^{\frac12}}.
$$
We also have that  $\lambda^{\FH}_{n,0}=\Theta_{0}$, $\lambda^{\FH}_{n,1}=0$, $\lambda^{\FH}_{n,2}=-C_{1}\kappa_{\max}$ and  $\lambda^{\FH}_{n,3}=(2n-1)C_{1}\Theta_{0}^{1/4}\sqrt{\frac{3k_{2}}{2}}$. The main term in the Ansatz is in the form
$$\an^{\FH}_{n,0}(\sigma,\tau)=f^\FH_{n,0}(\sigma)u_{\zeta_{0}}(\tau).$$
Moreover, for all $n\geq 1$, there exist $h_{0}>0$, $c>0$ such that for all $h\in(0,h_{0})$,  we have
$$\mathcal{B}\Big(\lambda^{\FH}_{n,0}+ \lambda^{\FH}_{n,2}h + \lambda^{\FH}_{n,3}h^{\frac{ 3}{ 2}}
, ch^{\frac 3 2}\Big)\cap \spe\left(\mathfrak{L}^\FH_{h}\right)=\{\lambda^{\FH}_{n}(h)\},$$
and $\lambda^{\FH}_{n}(h)$ is a simple eigenvalue. 
\end{theorem}
\begin{remark}
In particular, Theorem \ref{WKB-FH} proves that there are no odd powers of $\hbar^{\frac{1}{8}}$ in the expansion of the eigenvalues (compare with \cite[Theorem 1.1]{FouHel06a}).
\end{remark}

\subsubsection{WKB expansion}
Let us introduce a phase function $\Phi=\Phi(\sigma)$ defined in a neighborhood of $\sigma=0$ the unique and non degenerate maximum of the curvature $\kappa$. We consider the conjugate operator
$$\mathfrak{L}^{\FH,\wgt}_{h}=\re^{\Phi(\sigma)/h^{\frac{1}{2}}}\mathfrak{L}^{\FH}_{h}\re^{-\Phi(\sigma)/h^{\frac{1}{2}}}.$$
As usual, we look for
$$\an\sim \sum_{j\geq 0} h^{\frac j 2}\an_{j},\qquad \lambda\sim\sum_{j\geq 0} \lambda_{j}h^{\frac j 2}$$
such that, in the sense of formal series we have
$$\mathfrak{L}^{\FH,\wgt}_{h}\an \sim \lambda\an.$$
We may write
$$\mathfrak{L}^{\FH,\wgt}_{h}\sim \mathfrak{L}_{0}+ h^{\frac 1 2}\mathfrak{L}_{1}+ h\mathfrak{L}_{2}+ h^{\frac 3 2}\mathfrak{L}_{3}+\ldots,$$
where
\begin{align*}
&\mathfrak{L}_{0}=D_{\tau}^2+(\zeta_{0}-\tau)^2, \\
&\mathfrak{L}_{1}=2(\zeta_{0}-\tau)i\Phi'(\sigma),\\
&\mathfrak{L}_{2}=\kappa(\sigma)\dr_{\tau}+2\left(D_{\sigma}+\kappa(\sigma)\frac{\tau^2}{2}\right)(\zeta_{0}-\tau)-\Phi'(\sigma)^2+2\kappa(\sigma)(\zeta_{0}-\tau)^2\tau,\\
&\mathfrak{L}_{3}=\left(D_{\sigma}+\kappa(\sigma)\frac{\tau^2}{2}\right)(i\Phi'(\sigma))+(i\Phi'(\sigma))\left(D_{\sigma}+\kappa(\sigma)\frac{\tau^2}{2}\right)+4i\Phi'(\sigma)\tau \kappa(\sigma)(\zeta_{0}-\tau).
\end{align*}
Let us now solve the formal system. The first equation is
$$\mathfrak{L}_{0}\an_{0}=\lambda_{0}\an_{0}$$
and leads to take
$$\lambda_{0}=\Theta_{0},\qquad \an_{0}(\sigma,\tau)=f_{0}(\sigma)u_{\zeta_{0}}(\tau),$$
where $f_{0}$ has to be determined. The second equation is
$$(\mathfrak{L}_{0}-\lambda_{0})\an_{1}=(\lambda_{1}-\mathfrak{L}_{1})\an_{0}=(\lambda_{1}-2(\zeta_{0}-\tau) i\Phi'(\sigma))u_{\zeta_{0}}(\tau) f_{0}(\sigma)$$
and, due to the Fredholm alternative, we must take $\lambda_{1}=0$ and
$$\an_{1}(\sigma,\tau)=i\Phi'(\sigma) f_{0}(\sigma)\left(\dr_{\zeta}u\right)_{\zeta_{0}}(\tau)+f_{1}(\sigma)u_{\zeta_{0}}(\tau),$$
where $f_{1}$ is to be determined in a next step. Then the third equation is
$$(\mathfrak{L}_{0}-\lambda_{0})\an_{2}=(\lambda_{2}-\mathfrak{L}_{2})\an_{0}-\mathfrak{L}_{1}\an_{1}.$$
Let us explicitly write the r.h.s. It equals
\begin{multline*}
\lambda_{2}u_{\zeta_{0}} f_{0}+\Phi'^2(u_{\zeta_{0}}+2(\zeta_{0}-\tau)(\dr_{\zeta}u)_{\zeta_{0}})f_{0}-2(\zeta_{0}-\tau)u_{\zeta_{0}}(i\Phi' f_{1}-i\dr_{\sigma}f_{0})\\
+\kappa(\sigma)f_{0}(\dr_{\tau}u_{\zeta_{0}}-2(\zeta_{0}-\tau)^2\tau u_{\zeta_{0}}-\tau^2(\zeta_{0}-\tau)u_{\zeta_{0}}).
\end{multline*}
Therefore the equation becomes
$$(\mathfrak{L}_{0}-\lambda_{0})\tilde\an_{2}=\lambda_{2}u_{\zeta_{0}} f_{0}+\frac{\nu''(\zeta_{0})}{2}\Phi'^2 u_{\zeta_{0}}f_{0}+\kappa f_{0}(-\dr_{\tau}u_{\zeta_{0}}-2(\zeta_{0}-\tau)^2\tau u_{\zeta_{0}}-\tau^2(\zeta_{0}-\tau)u_{\zeta_{0}}),$$
where $$\tilde\an_{2}=\an_{2}-(\dr_{\zeta}u)_{\zeta_{0}}(i\Phi' f_{1}-i\dr_{\sigma}f_{0})+\tfrac{1}{2}(\dr_{\zeta}^2 u)_{\zeta_{0}}\Phi'^2 f_{0}.$$
Let us now use the Fredholm alternative (with respect to $\tau$). We will need the following lemma the proof of which relies on Feynman-Hellmann formulas (like in Proposition \ref{FH}) and on \cite[p. 19]{FouHel06a} (for the last one).
\begin{lemma}
We have:
\begin{equation*}
\begin{gathered}
\int_{\R_{+}} (\zeta_{0}-\tau)u^2_{\zeta_{0}}(\tau)\dx \tau=0,\qquad \int_{\R_{+}} (\dr_{\zeta} u)_{\zeta_{0}}(\tau) u_{\zeta_{0}}(\tau)\dx \tau=0,\\
2\int_{\R_{+}} (\zeta_{0}-\tau)(\dr_{\zeta} u)_{\zeta_{0}}(\tau) u_{\zeta_{0}}(\tau)\dx \tau=\frac{\nu''(\zeta_{0})}{2}-1,\\
\int_{\R_{+}} \left(2\tau(\zeta_{0}-\tau)^2+\tau^2(\zeta_{0}-\tau)\right)u^2_{\zeta_{0}}+u_{\zeta_{0}}\dr_{\tau}u_{\zeta_{0}} \dx\tau=-C_{1}.
\end{gathered}
\end{equation*}
\end{lemma}
We get the equation
$$\lambda_{2}+\frac{\nu''(\zeta_{0})}{2}\Phi'^2(\sigma)+C_{1}\kappa(\sigma)=0,\qquad C_{1}=\frac{u^2_{\zeta_{0}}(0)}{3}.$$
This eikonal equation is the one of a pure electric problem in dimension one whose potential is given by the curvature. Thus we take
$$\lambda_{2}=-C_{1}\kappa(0),$$
and
$$\Phi(\sigma)=\left(\frac{2C_{1}}{\nu''(\zeta_{0})}\right)^{1/2}\left|\int_{0}^\sigma (\kappa(0)-\kappa(\varsigma))^{1/2}\dx \varsigma\right|.$$
In particular we have:
$$\Phi''(0)=\left(\frac{k_{2}C_{1}}{\nu''(\zeta_{0})}\right)^{1/2},$$
where $k_{2}=-\kappa''(0)>0$.\\
This leads to take
$$\an_{2}=f_{0}\hat\an_{2}+(\dr_{\zeta}u)_{\zeta_{0}}(i\Phi' f_{1}-i\dr_{\sigma}f_{0})-\tfrac{1}{2}(\dr^2_{\eta}u)_{\zeta_{0}}\Phi'^2 f_{0}+f_{2}u_{\zeta_{0}},$$
where $\hat\an_{2}$ is the unique solution, orthogonal to $u_{\zeta_{0}}$ for all $\sigma$, of
$$(\mathfrak{L}_{0}-\nu_{0})\hat\an_{2}=\nu_{2}u_{\zeta_{0}} +\frac{\nu''(\zeta_{0})}{2}\Phi'^2 u_{\zeta_{0}}+\kappa\left(-\dr_{\tau}u_{\zeta_{0}}-2(\zeta_{0}-\tau)^2\tau u_{\zeta_{0}}-\tau^2(\zeta_{0}-\tau)u_{\zeta_{0}}\right),$$
and $f_{2}$ has to be determined.\\
Finally we must solve the fourth equation given by
$$(\mathfrak{L}_{0}-\lambda_{0})\an_{3}=(\lambda_{3}-\mathfrak{L}_{3})\an_{0}+(\lambda_{2}-\mathfrak{L}_{2})\an_{1}-\mathfrak{L}_{1}\an_{2}.$$
The Fredholm condition provides the following equation in the variable $\sigma$:
$$\langle \mathfrak{L}_{3}\an_{0}+(\mathfrak{L}_{2}-\lambda_{2})\an_{1}+\mathfrak{L}_{1}\an_{2} ,u_{\zeta_{0}}\rangle_{\sL^2(\R_{+},\dx\tau)}=\lambda_{3}f_{0}.$$
Using the previous steps of the construction, it is not very difficult to see that this equation does not involve $f_{1}$ and $f_{2}$ (due to the choice of $\Phi$ and $\lambda_{2}$ and Feynman-Hellmann formulas). Using the same formulas, we may write it in the form
\begin{equation}\label{transport-FH}
\frac{\nu''(\zeta_{0})}{2}\left(\Phi'(\sigma)\dr_{\sigma}+\dr_{\sigma}\Phi'(\sigma)\right)f_{0}+F(\sigma)f_{0}=\lambda_{3}f_{0},
\end{equation}
where $F$ is a smooth function which vanishes at $\sigma=0$. Therefore the linearized equation at $\sigma=0$ is given by
$$\Phi''(0)\frac{\nu''(\zeta_{0})}{2}\left(\sigma\dr_{\sigma}+\dr_{\sigma}\sigma\right)f_{0}=\lambda_{3}f_{0}.$$
We recall that
$$\frac{\nu''(\zeta_{0})}{2}=3C_{1}\Theta_{0}^{1/2}$$
so that the linearized equation becomes
$$C_{1}\Theta_{0}^{1/4}\sqrt{\frac{3k_{2}}{2}}\left(\sigma\dr_{\sigma}+\dr_{\sigma}\sigma\right)f_{0}=\lambda_{3}f_{0}.$$
We have to choose $\lambda_{3}$ in the spectrum of this transport equation, which is given by the set
$$\left\{(2n-1)C_{1}\Theta_{0}^{1/4}\sqrt{\frac{3k_{2}}{2}},\quad n\geq 1\right\}.$$
If $\lambda_{3}$ belongs to this set, we may solve locally the transport equation \eqref{transport-FH} and thus find $f_{0}$. This procedure can be continued at any order.

\subsubsection{Numerical estimates of the magnetic camel}
For the numerical computations, we consider the magnetic potential $\bA=(-x_{2},0)$ and we denote by $(\lambda_{n}^{\FH}(h),u_{n,h}^{\FH})$ the $n$-th eigenpair of the magnetic Laplacian
$(hD_{x_{1}}-x_{2})^2+h^2D_{x_{2}}^2$ on $\Omega$.

\paragraph{Camel with one bump}
Let us first consider the case where $\Omega$ is an unbounded domain with a unique point with maximal curvature. We consider
$$\Omega = \{(x_{1},x_{2})\in\R^2, x_{2}< -4x_{1}^2\}.$$
For the numerical computations, we proceed as explained in Section~\ref{sec.methnum}: we bound the domain and impose Dirichlet condition on the artificial boundary. Let us define the truncated domain
$$\Omega_{H} = \{(x_{1},x_{2})\in\Omega, x_{2}>-H\}.$$
We consider triangular elements of degree $\mathbb P_{6}$. For the numerical computations, we take $H=2.5, 3, 4$ and a mesh with approximately 3000, 3600, 4800 triangular elements and $1/h\in\{1:0.1:1000\}$.

Figure~\ref{fig.VecPDroma} illustrates the asymptotic expansion \eqref{HM-result} for the first eigenvalue:
\begin{equation}
\frac{\lambda_{1}^{\FH}(h)}h=\Theta_{0}-C_{1}\kappa_{max} h^{1/2}+o(h^{1/2}),\qquad \mbox{ with }C_{1}=\frac{u^2_{\zeta_{0}}(0)}{3}.
\end{equation}
In our example, we have $\kappa_{max}=8$.
Using \cite{Bon12}, we have
$$\Theta_{0}\simeq 0.59010\qquad\mbox{ and }\qquad C_{1}\simeq 0.873043.$$
Figure~\ref{fig.drom1} shows the convergence to $\Theta_{0}$, which is quite slow because only in ${\Oc}(h^{1/2})$.
In Figures~\ref{fig.drom2}--\ref{fig.drom3}, we aim at recovering numerically the power appearing in the expansion.
For this, we plot, according to $\ln\frac1h$ the quantities
$$\ln\Big(\Theta_{0}-\frac{\lambda^{\FH}_{1}(h)}h\Big)\qquad\mbox{ and }\qquad
\ln\frac{\lambda^{\FH}_{1}(h)}h-\Theta_{0}+C_{1}\kappa_{max}h^{1/2}.$$

\begin{figure}[h!t]
\begin{center}
\subfigure[$\frac{\lambda^{\FH}_{1}(h)}h$ vs. $\frac1h$\label{fig.drom1}]{\includegraphics[height=3.8cm]{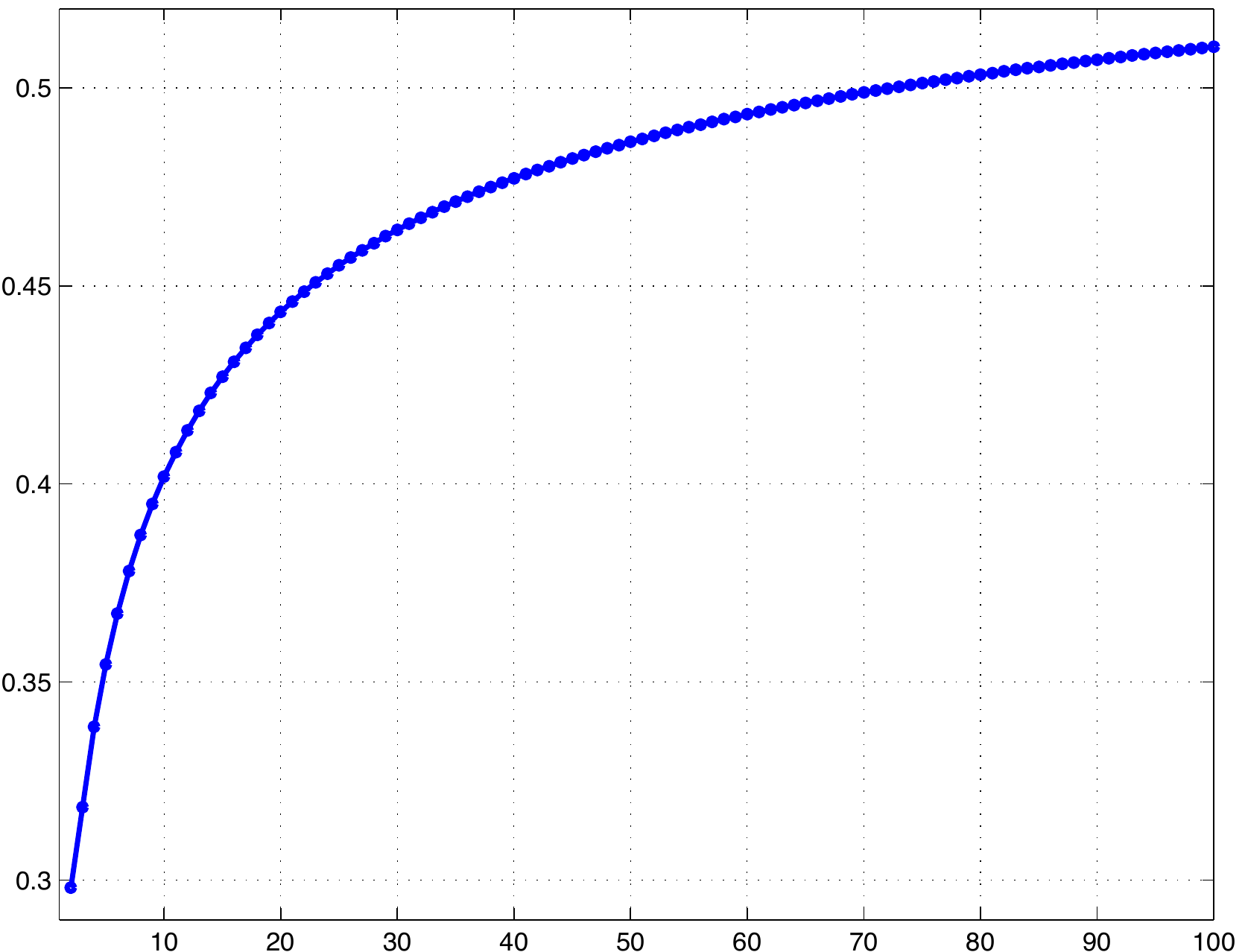}}
\subfigure[$\ln\Big(\Theta_{0}-\frac{\lambda^{\FH}_{1}(h)}h\Big)$ vs. $\ln\frac1h$\label{fig.drom2}]{\includegraphics[height=3.8cm]{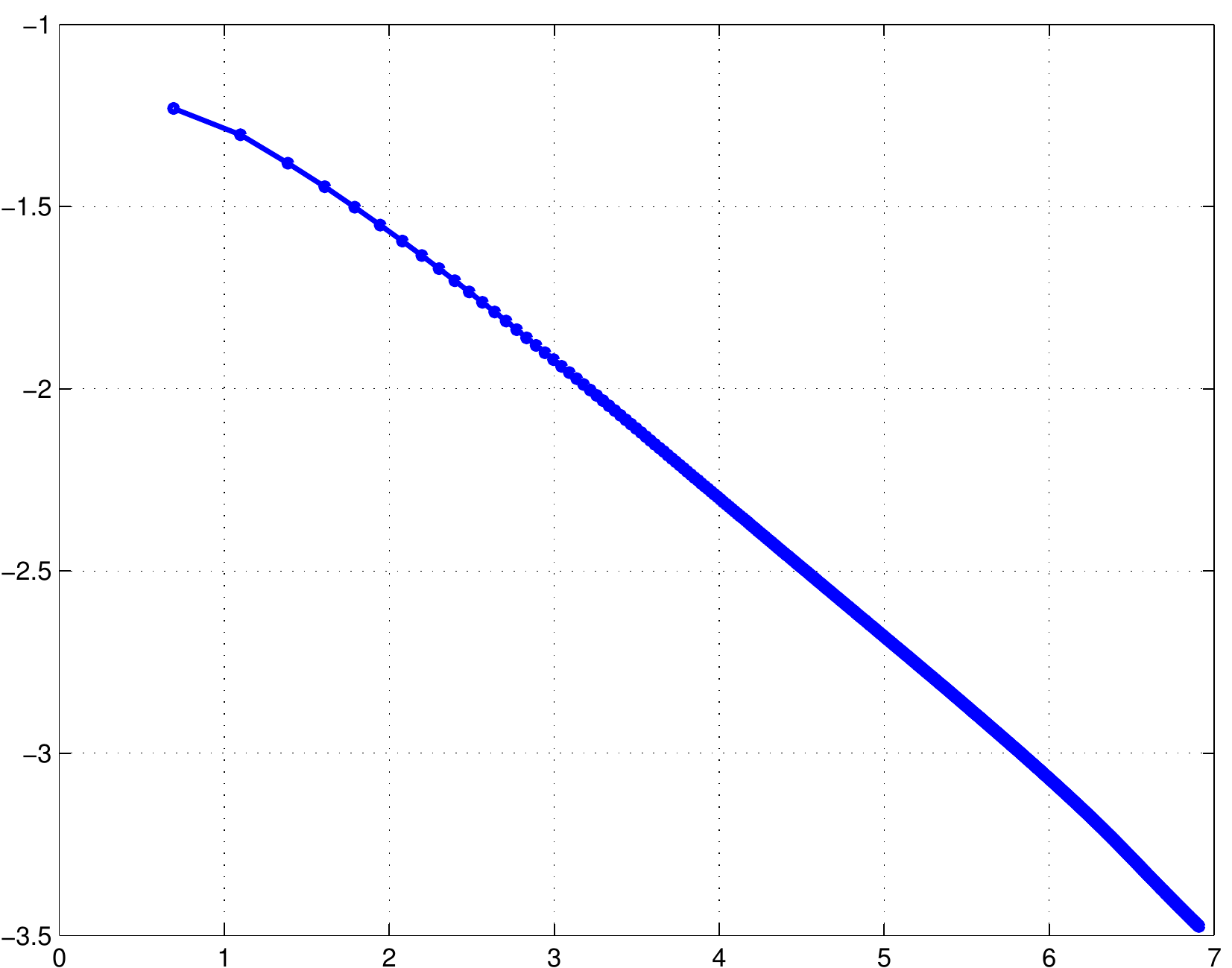}}
\subfigure[$\ln\frac{\lambda^{\FH}_{1}(h)}h-\Theta_{0}+C_{1}\kappa_{max}h^{\frac 1 2}$ vs. $\ln\frac 1h$\label{fig.drom3}]{\includegraphics[height=3.8cm]{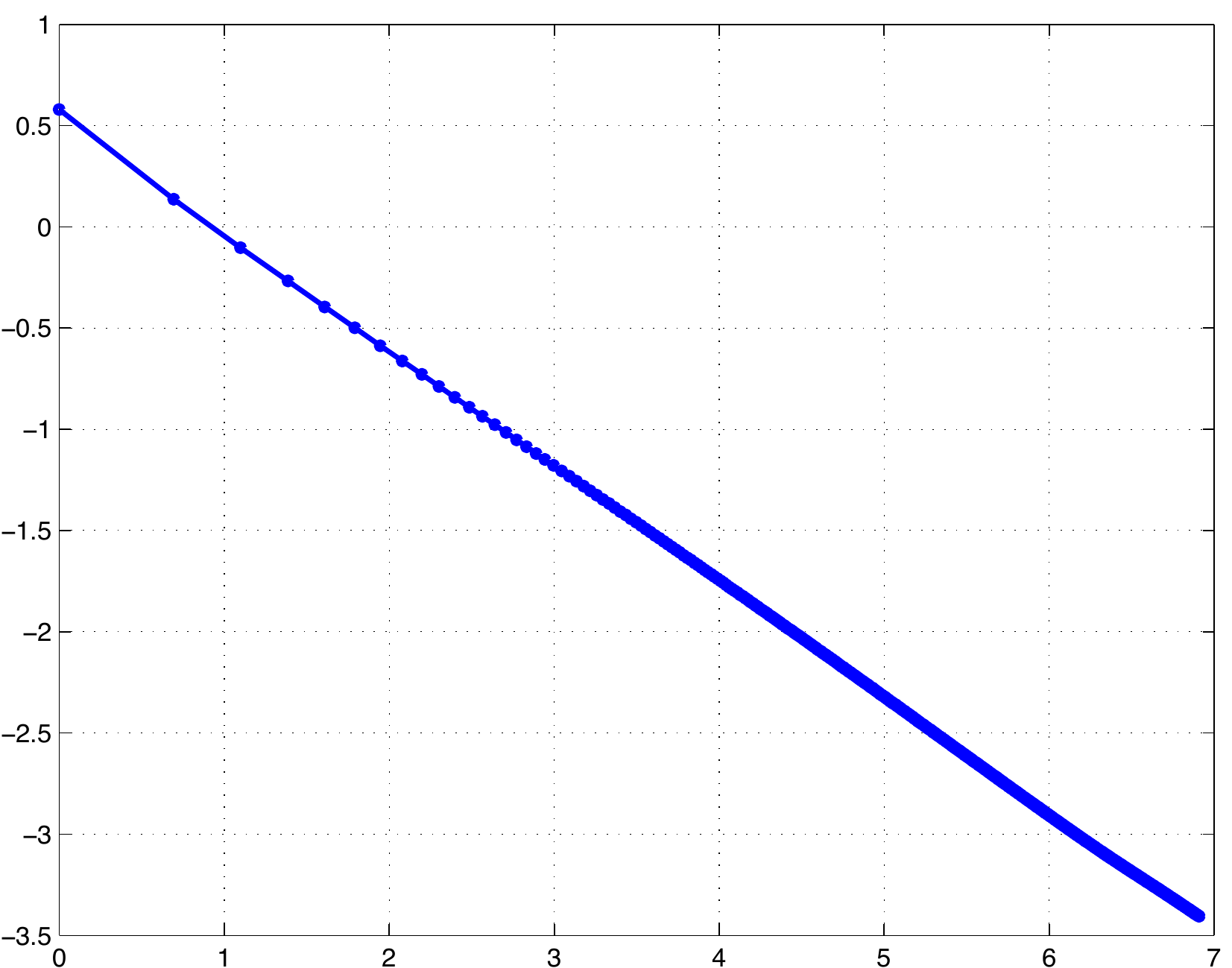}}
\caption{Convergence of the eigenvalues $\lambda^{\FH}_{1}(h)$. \label{fig.VPDroma}}
\end{center}
\end{figure}

In Figure~\ref{fig.VecPDoublepuitsbis}, are represented the modulus, the logarithm of the modulus and the phase of the first eigenfunction for $h=1/20$.
\begin{figure}[h!t]
\begin{center}
\includegraphics[height=4.1cm]{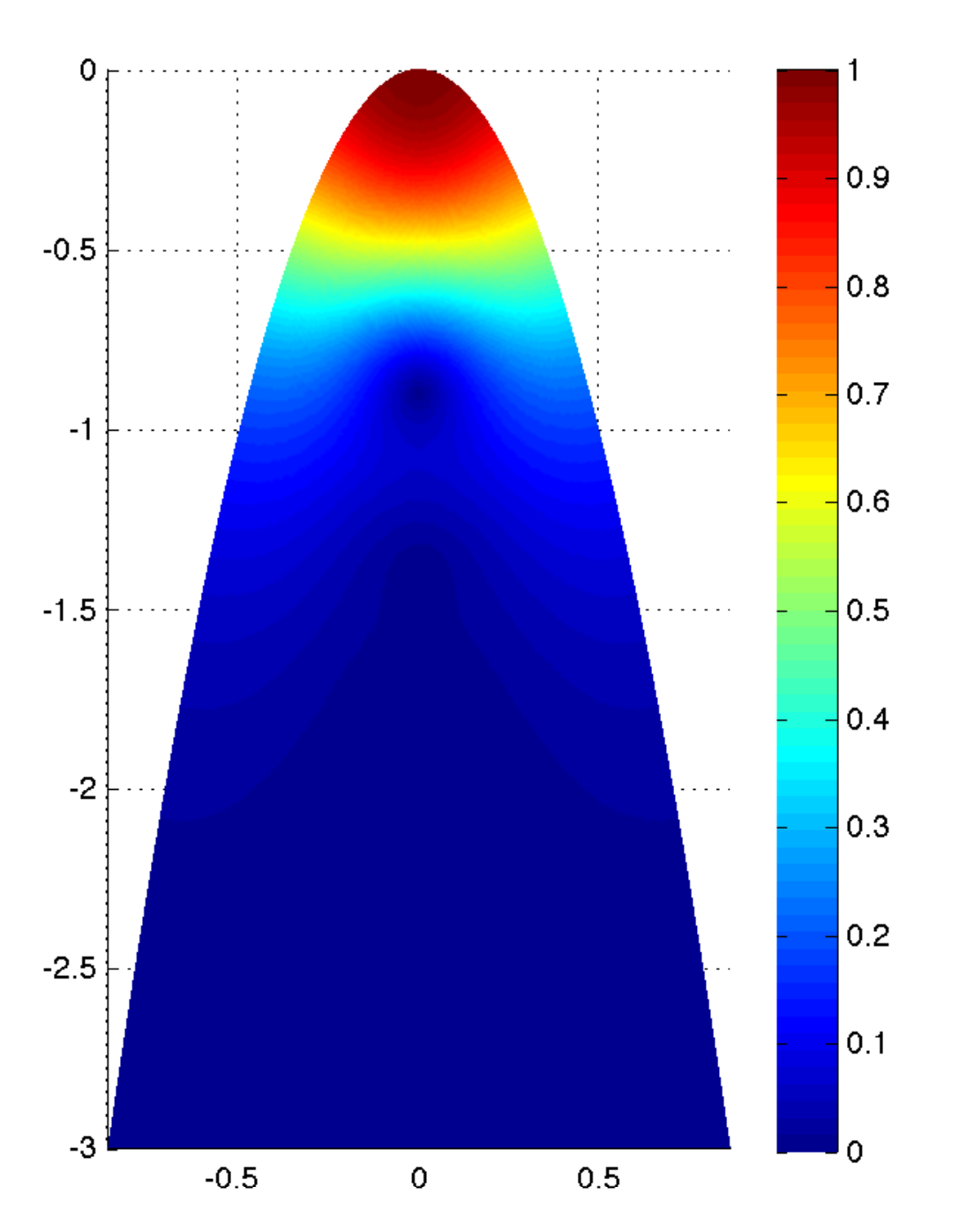}
\includegraphics[height=4.1cm]{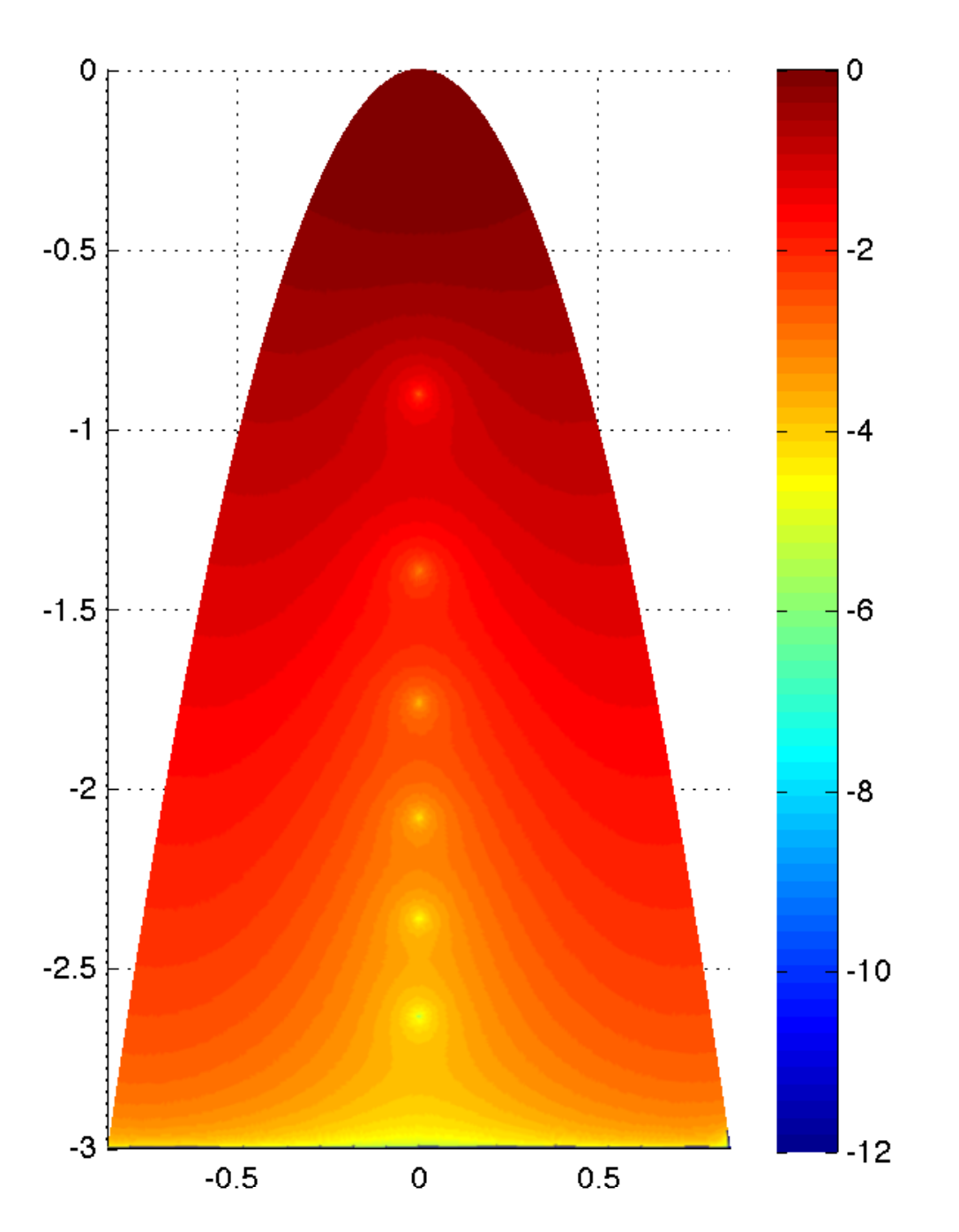}
\includegraphics[height=4.1cm]{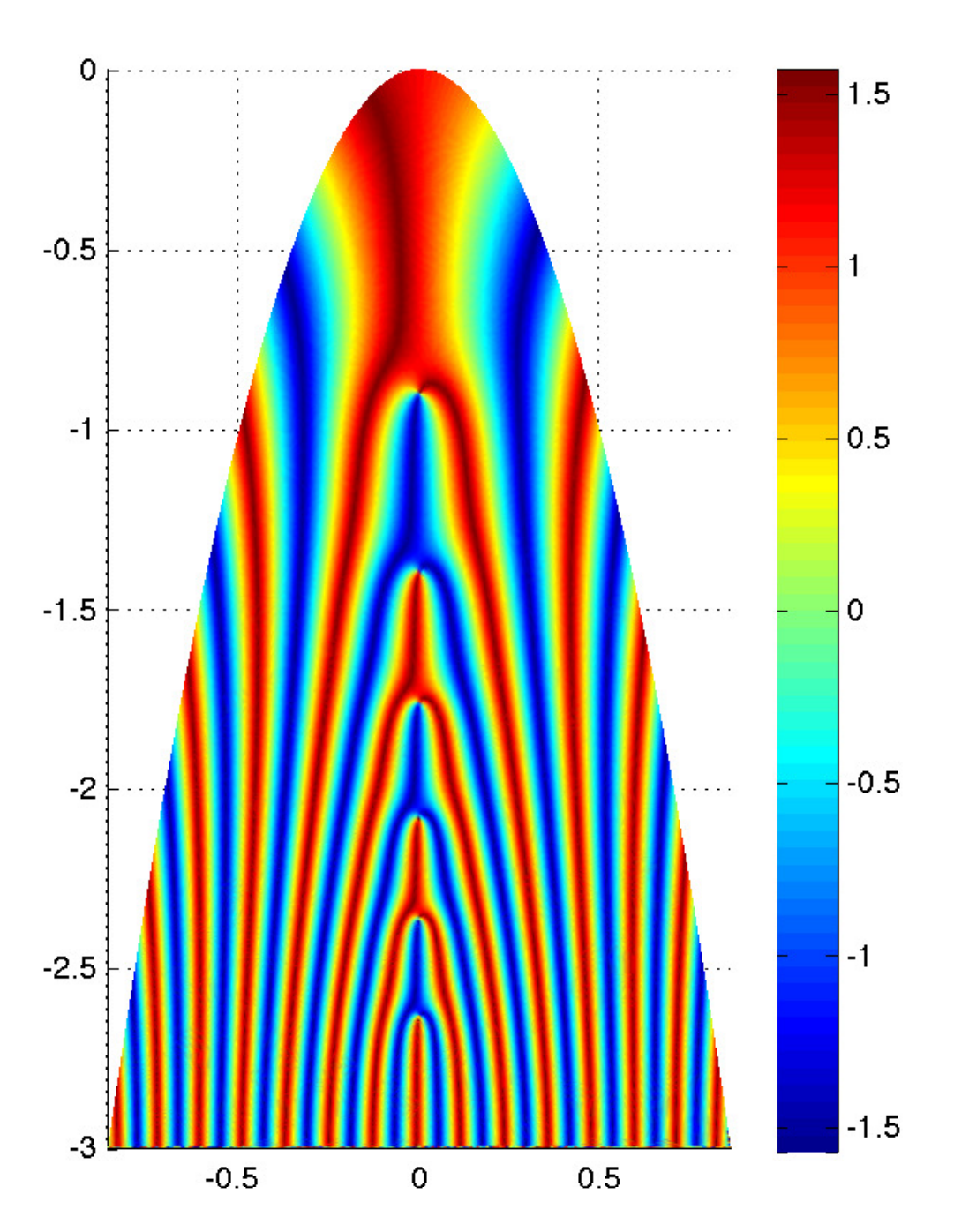}
\caption{Moduli, log$_{10}$(moduli) and phases of the first two eigenfunctions, $h=\frac{1}{20}$. \label{fig.VecPDroma}}
\end{center}
\end{figure}

\paragraph{Camel with two bumps}
Let us now deal with the case of a double well on the geometry. For this, we consider
$$\Omega = \{(x_{1},x_{2})\in\R^2, x_{2}<- (1-x_{1}^2)^2\}.$$
Let us look at the behavior of the first two eigenpairs. Figure~\ref{fig.chameau1} illustrates the convergence of the first two eigenvalues $\lambda_{n}^{\FH}(h)$ to $\Theta_{0}$ as $h\to 0$. We represent
$$\frac 1h\mapsto \frac{\lambda_{n}^{\FH}(h)}h,\qquad n=1,2.$$
To analyze the splitting between the first two eigenvalues, we plot in Figures~\ref{fig.chameau2}--\ref{fig.chameau3}, according to $\ln1/h$
$$ \frac{\lambda_{2}^{\FH}(h)-\lambda_{1}^{\FH}(h)}h\qquad\mbox{Êand }\qquad
-h^{1/4}\ln\frac{\lambda_{2}^{\FH}(h)-\lambda_{1}^{\FH}(h)}h.$$
For the last figure, we take $h\geq 1/70$ otherwise the splitting computed numerically is of the same order as the accuracy of our computation and the numerics is no more relevant when $h<1/70$. These computations suggest that
$$\frac{\lambda_{2}^{\FH}(h)-\lambda_{1}^{\FH}(h)}h={\Oc}(\re^{-C/h^{1/4}})\qquad\mbox{ with }\quad 5.2\leq C\leq 5.4.$$
\begin{figure}[h!t]
\begin{center}
\subfigure[$\frac{\lambda_{n}^{\FH}(h)}h$, $n=1,2$ vs. $\frac1h$\label{fig.chameau1}]{\includegraphics[height=3.8cm]{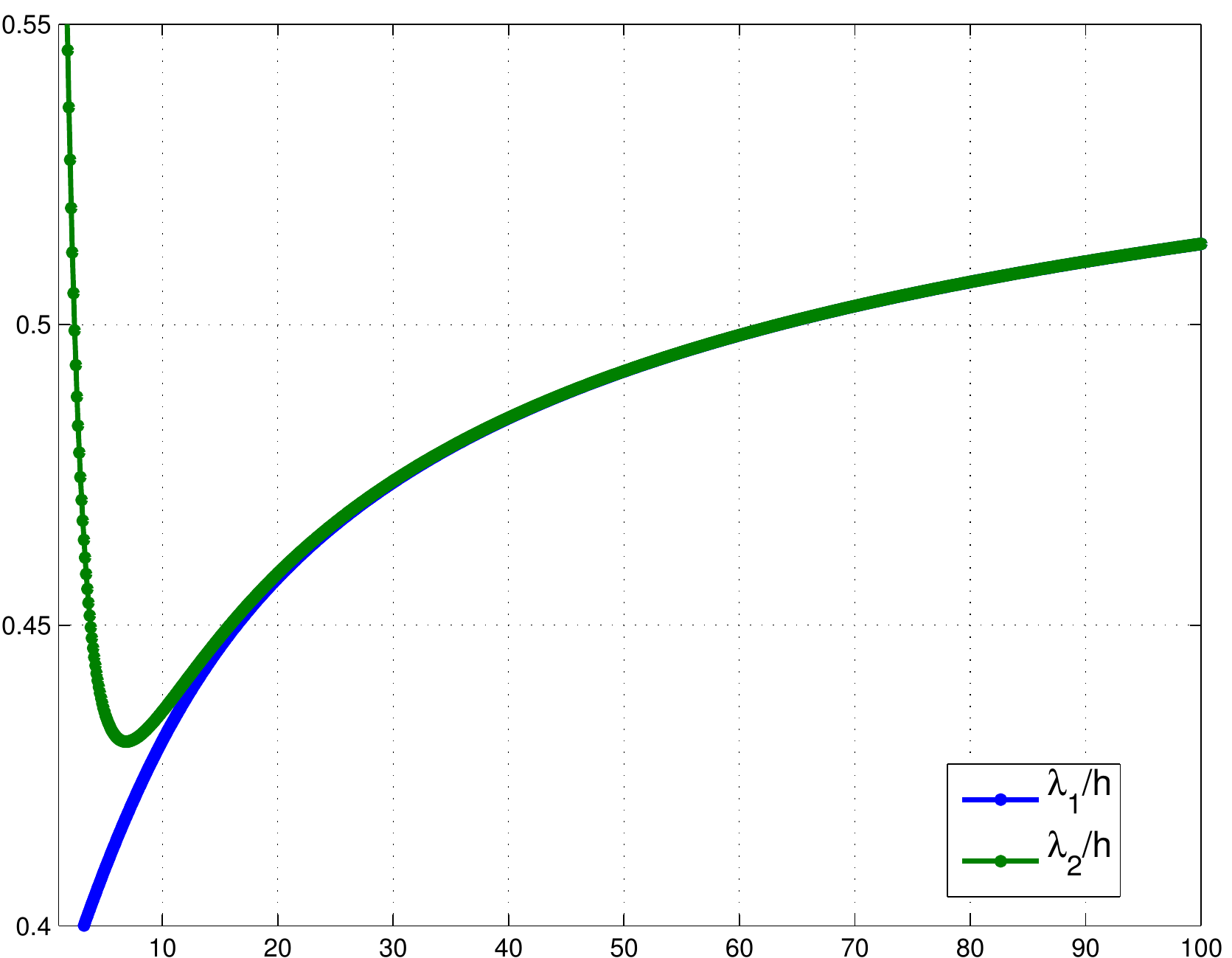}}
\subfigure[$\frac{\lambda_{2}^{\FH}(h)-\lambda_{1}^{\FH}(h)}h$ vs. $\frac1h$\label{fig.chameau2}]{\includegraphics[height=3.8cm]{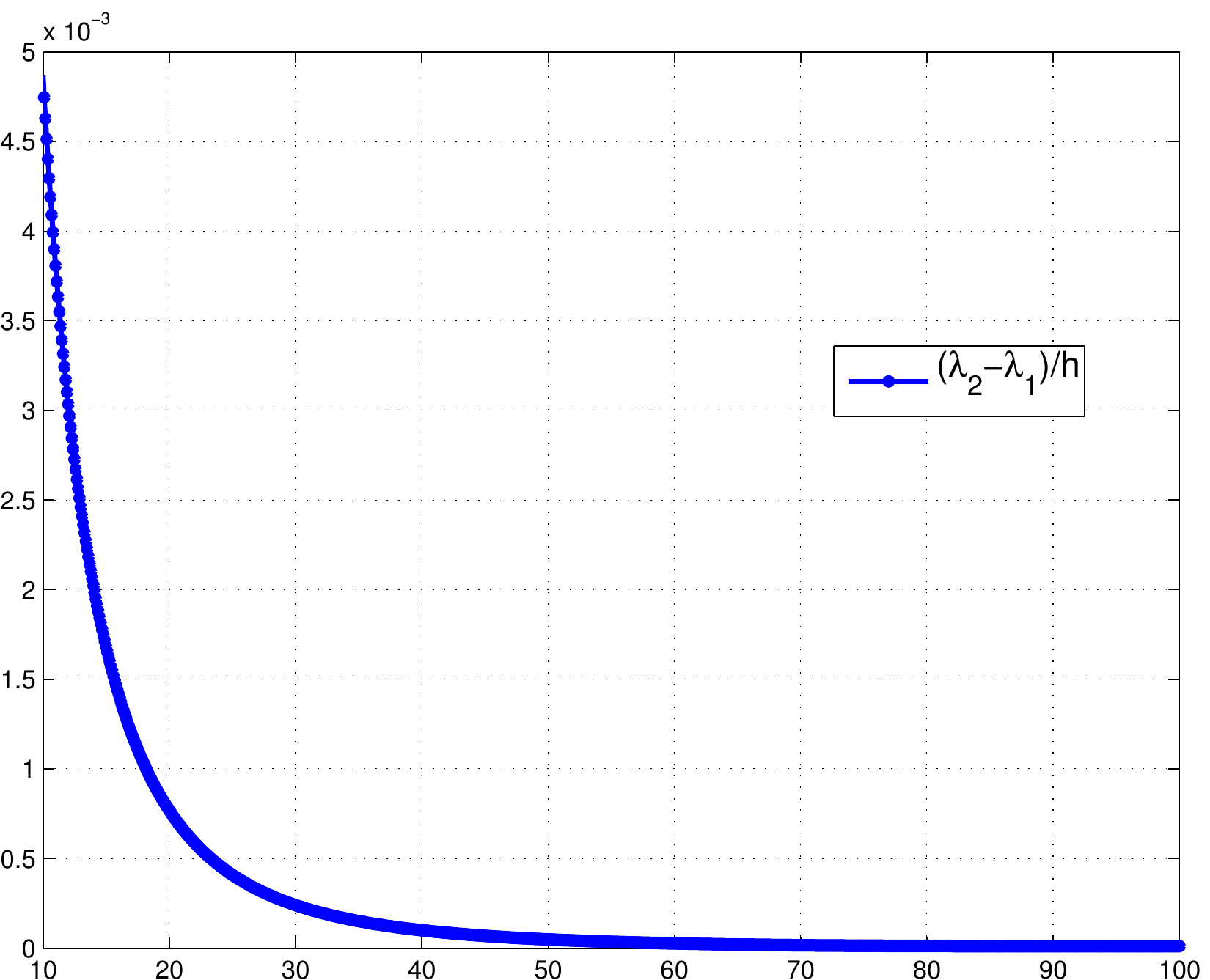}}
\subfigure[$-h^{1/4}\ln\frac{\lambda_{2}(h)-\lambda_{1}(h)}h$ vs. $\frac1h$\label{fig.chameau3}]{\includegraphics[height=3.8cm]{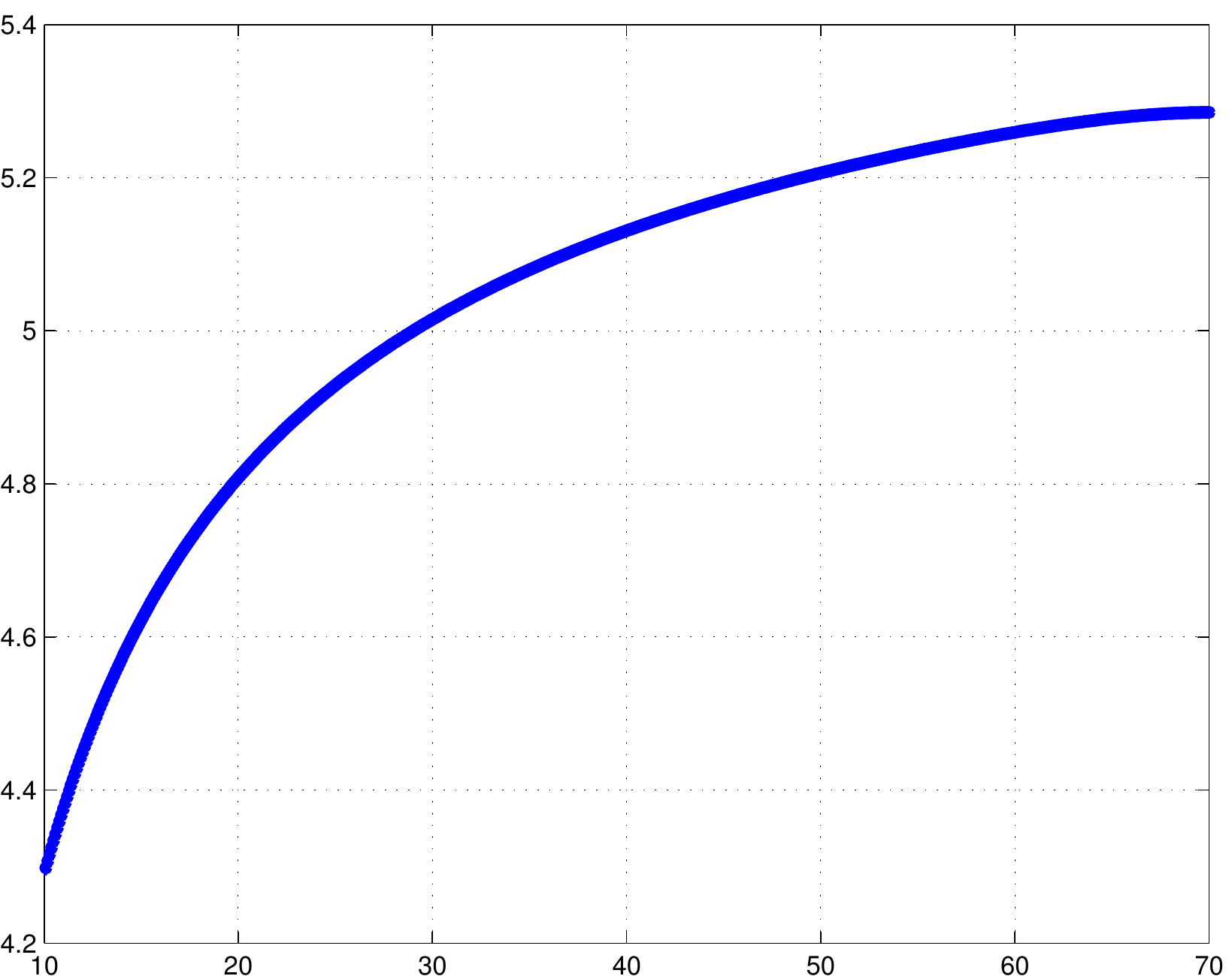}}
\caption{Convergence for $\lambda_{n}^{\FH}(h)$. \label{fig.VecPDoublepuitsVP}}
\end{center}
\end{figure}

Figure~\ref{fig.VecPDoublepuitsbis} gives the modulus, logarithm of the modulus and the phase of the first two eigenfunctions for $h=1/20$.
\begin{figure}[h!t]
\begin{center}
\subfigure[First eigenfunction]{\includegraphics[height=4cm]{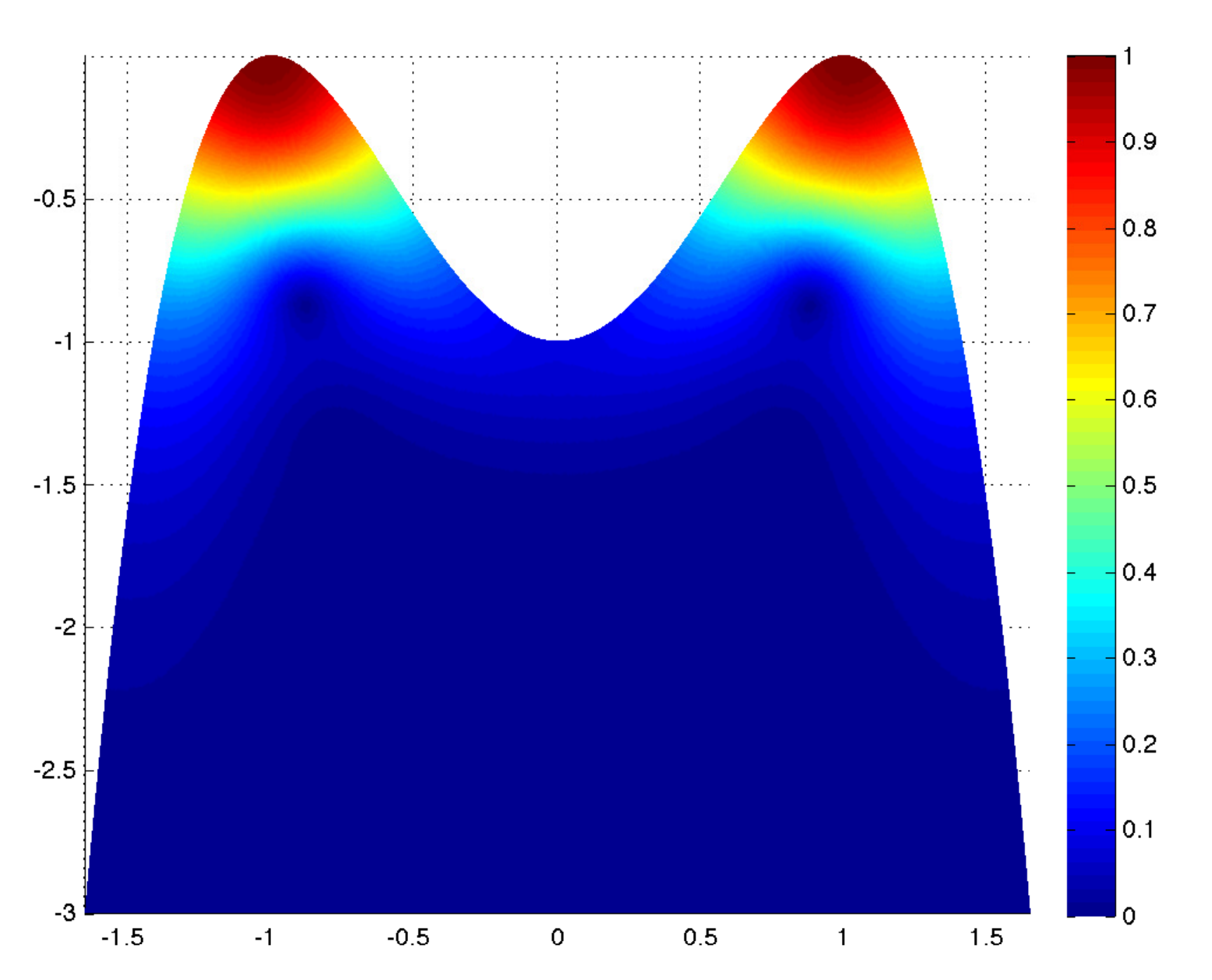}
\includegraphics[height=4cm]{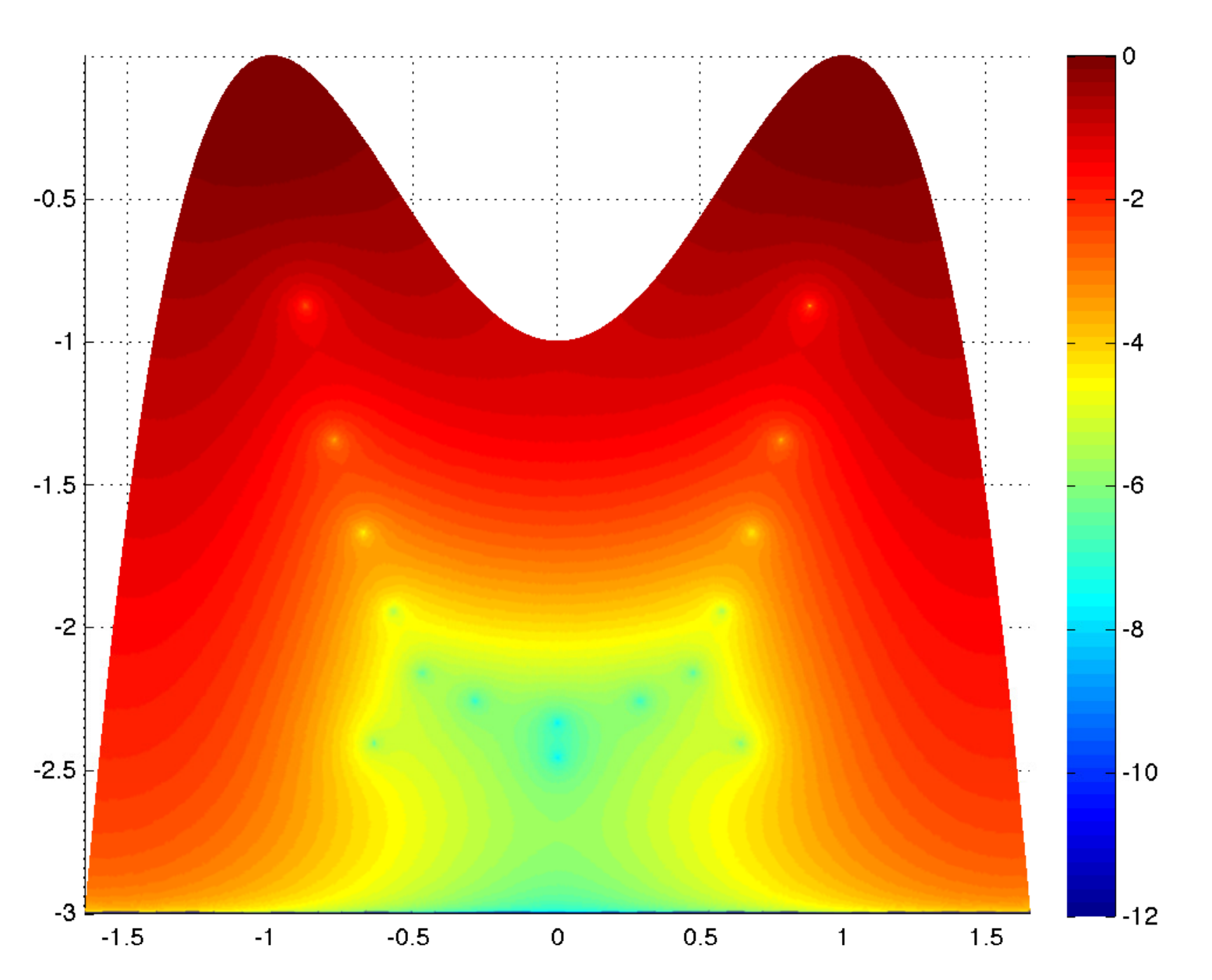}
\includegraphics[height=4cm]{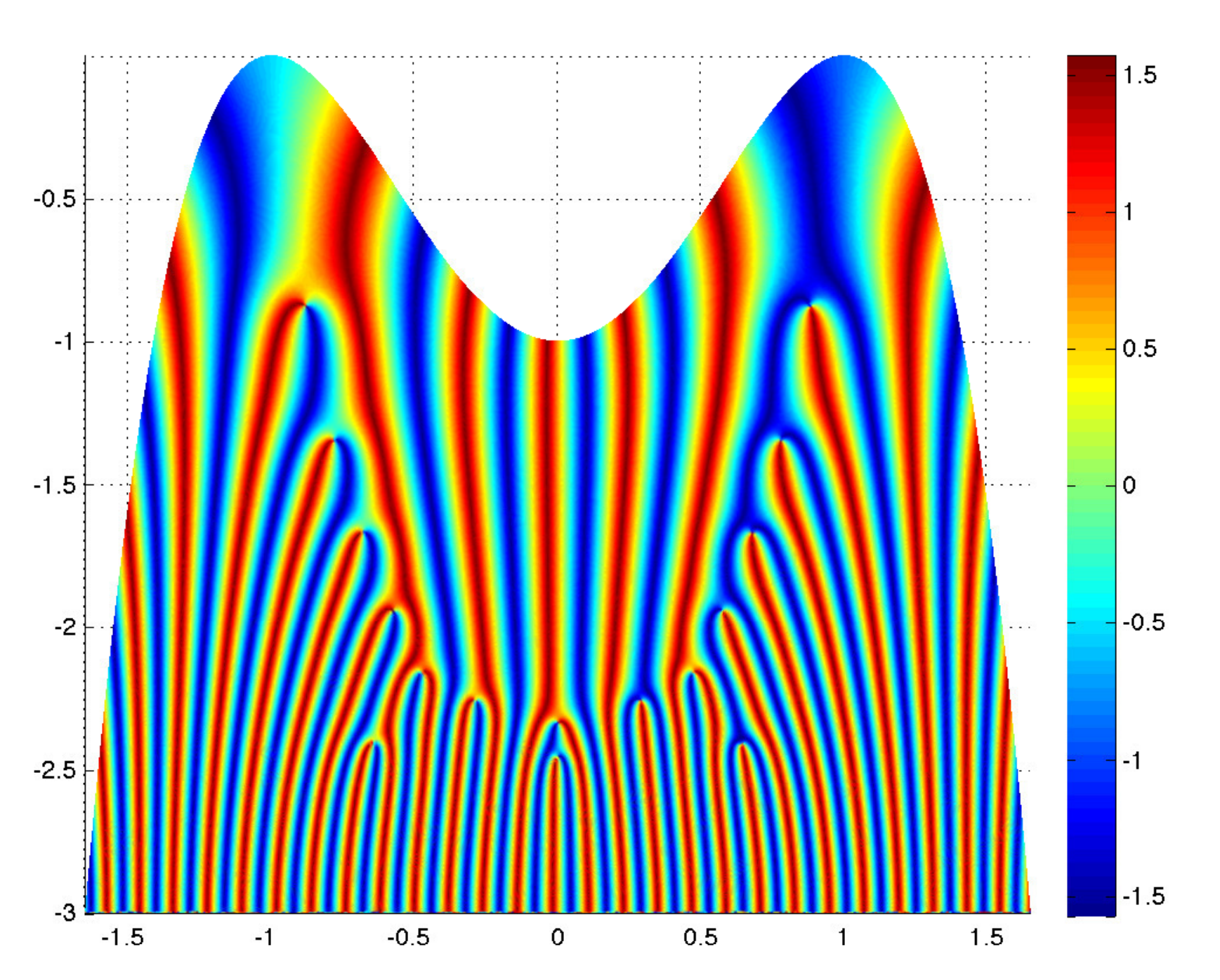}}
\subfigure[Second eigenfunction]{\includegraphics[height=4cm]{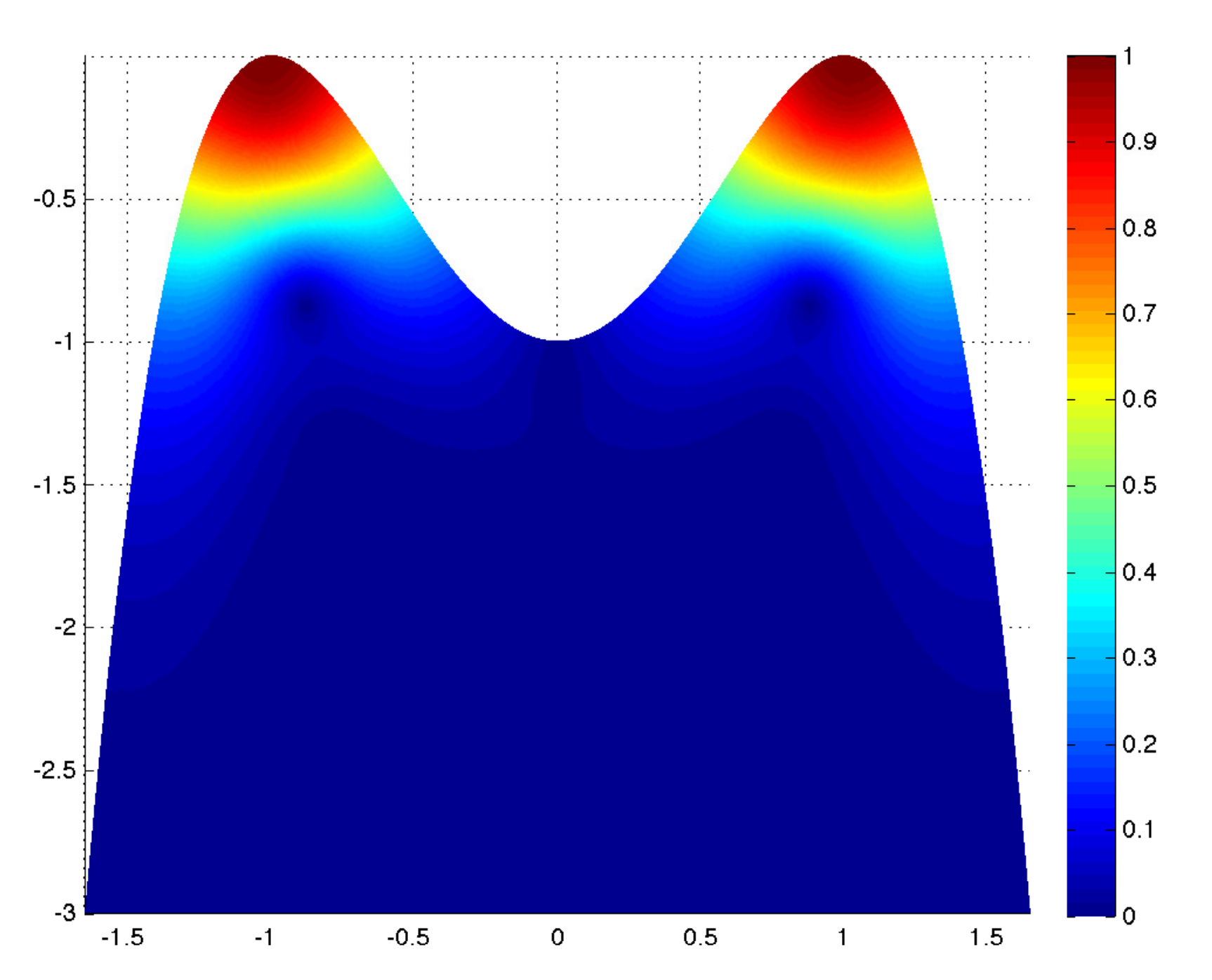}
\includegraphics[height=4cm]{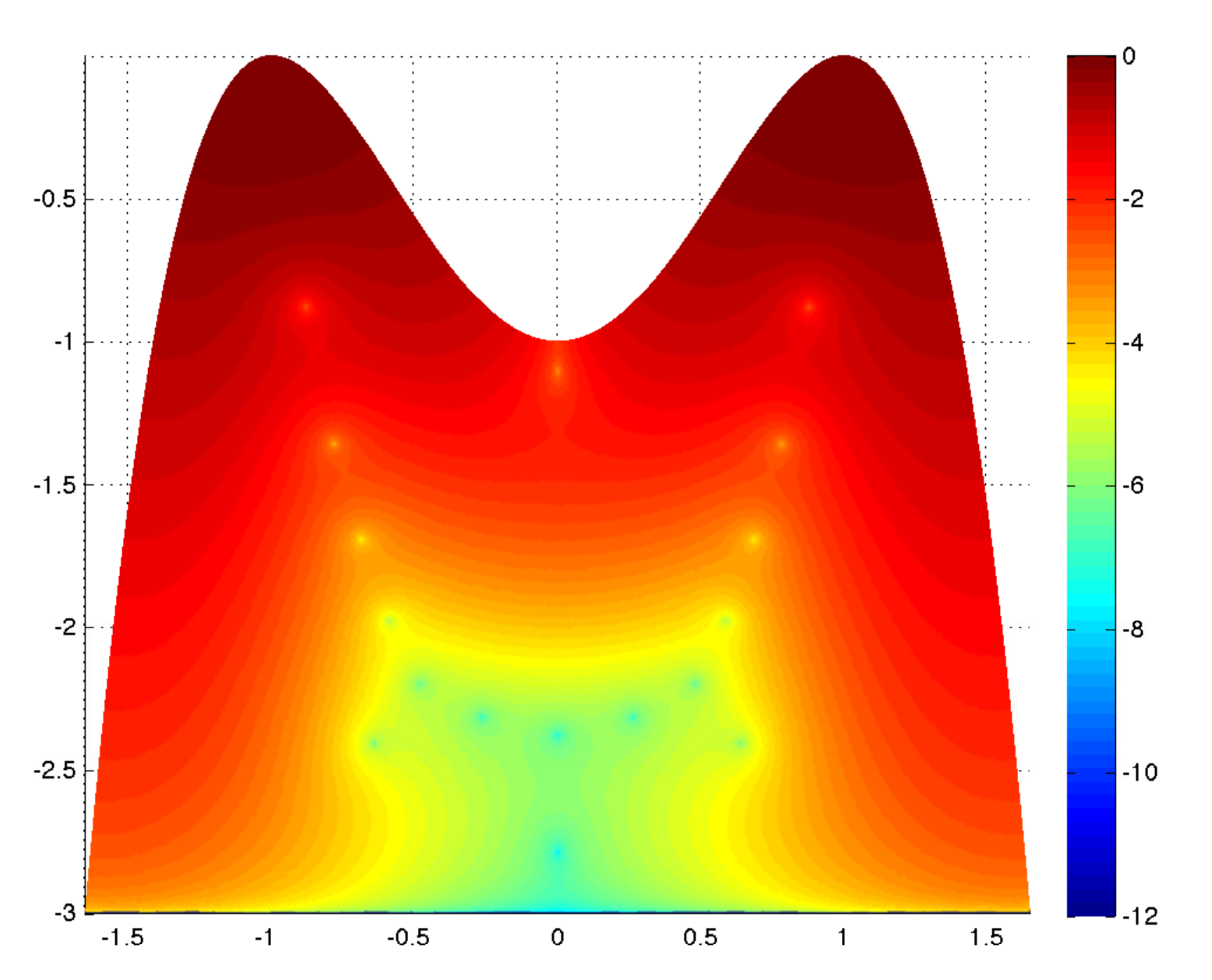}
\includegraphics[height=4cm]{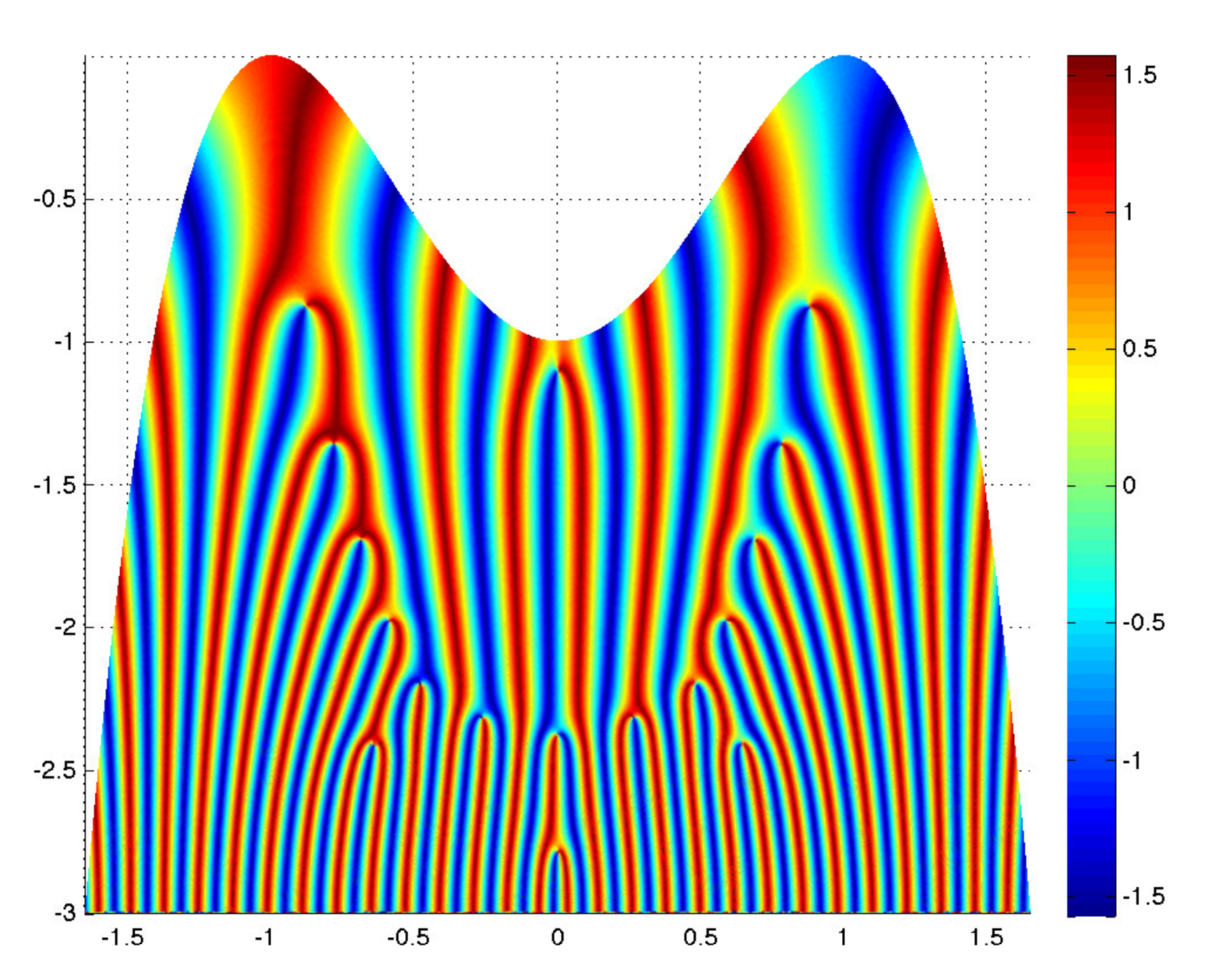}}
\caption{Moduli, log$_{10}$(moduli) and phases of the first two eigenfunctions, $h=\frac{1}{20}$. \label{fig.VecPDoublepuitsbis}}
\end{center}
\end{figure}

\section{Perspectives}
Let us finally provide some perspectives. As we have seen in Section \ref{SS:Agmon}, even in explicit situations, the optimal Agmon estimates are still an open problem. If these estimates are improved, one will obtain an accurate asymptotics of the splitting between the low-lying eigenvalues. Extended enough WKB constructions for computing the exponentially small splitting are related to the holomorphic extensions of the  model operators eigenpairs (for instance the generalized Montgomery operators). Furthermore in the case of curvature induced magnetic bound states, we have proved, at the WKB expansion level, that the effective operator is purely electric so that we can think that the optimal  Agmon estimates are accessible. Numerically, this paper was concerned with one symmetry (camel with two bumps) and we observed that the lowest eigenvalues seemed to be simple. With more symmetries, we expect multiplicity (see Figure \ref{fig.ellipse}). Moreover, in more singular geometrical situations (see \cite{BDMV07}), the WKB structure of the eigenfunctions is not clear at all since there is no obvious dimensional reduction (for example, the case of polygonal domains is based on models on angular sectors).

\begin{figure}[h!t]
\begin{center}
\subfigure[$\frac{\lambda_{n}(h)}{h}$ vs $\frac 1h$.]{\includegraphics[width=5cm]{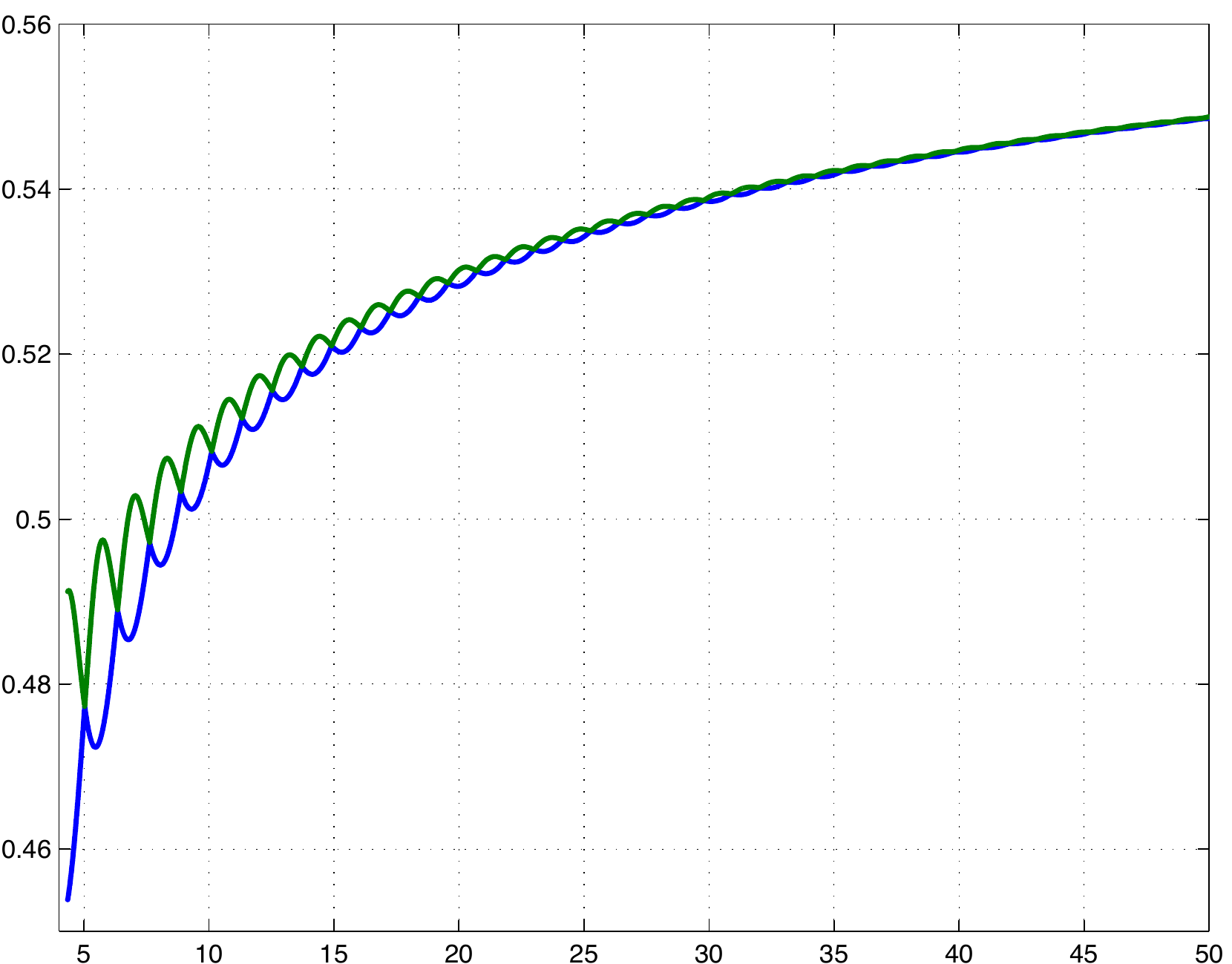}}
\subfigure[First two eigenfunctions, $h=\frac 1{50}$]
{\begin{tabular}{cc}
\ \vspace{-4cm}
\\
\includegraphics[width=4cm]{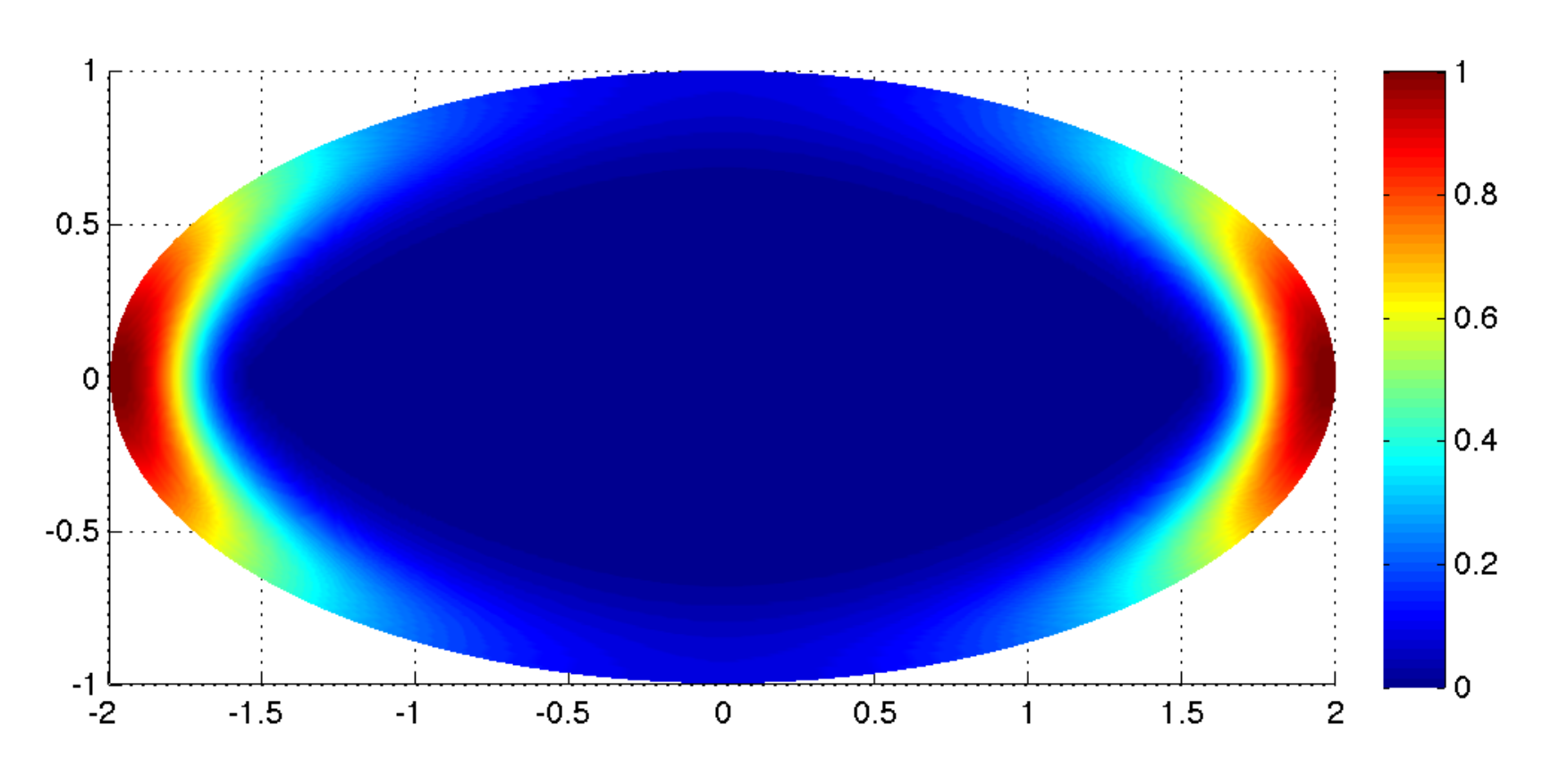}
&\includegraphics[width=4cm]{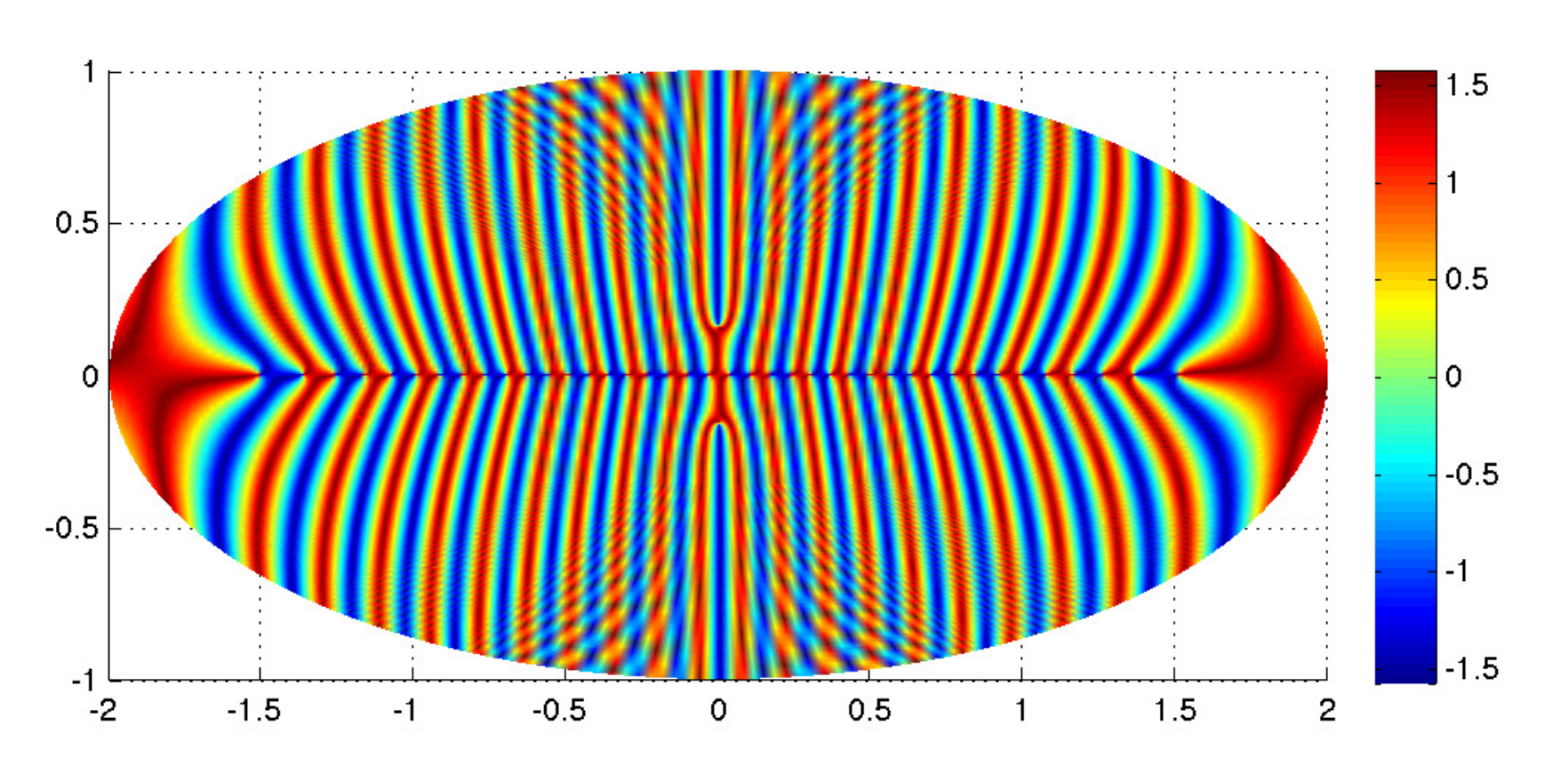}\\
\includegraphics[width=4cm]{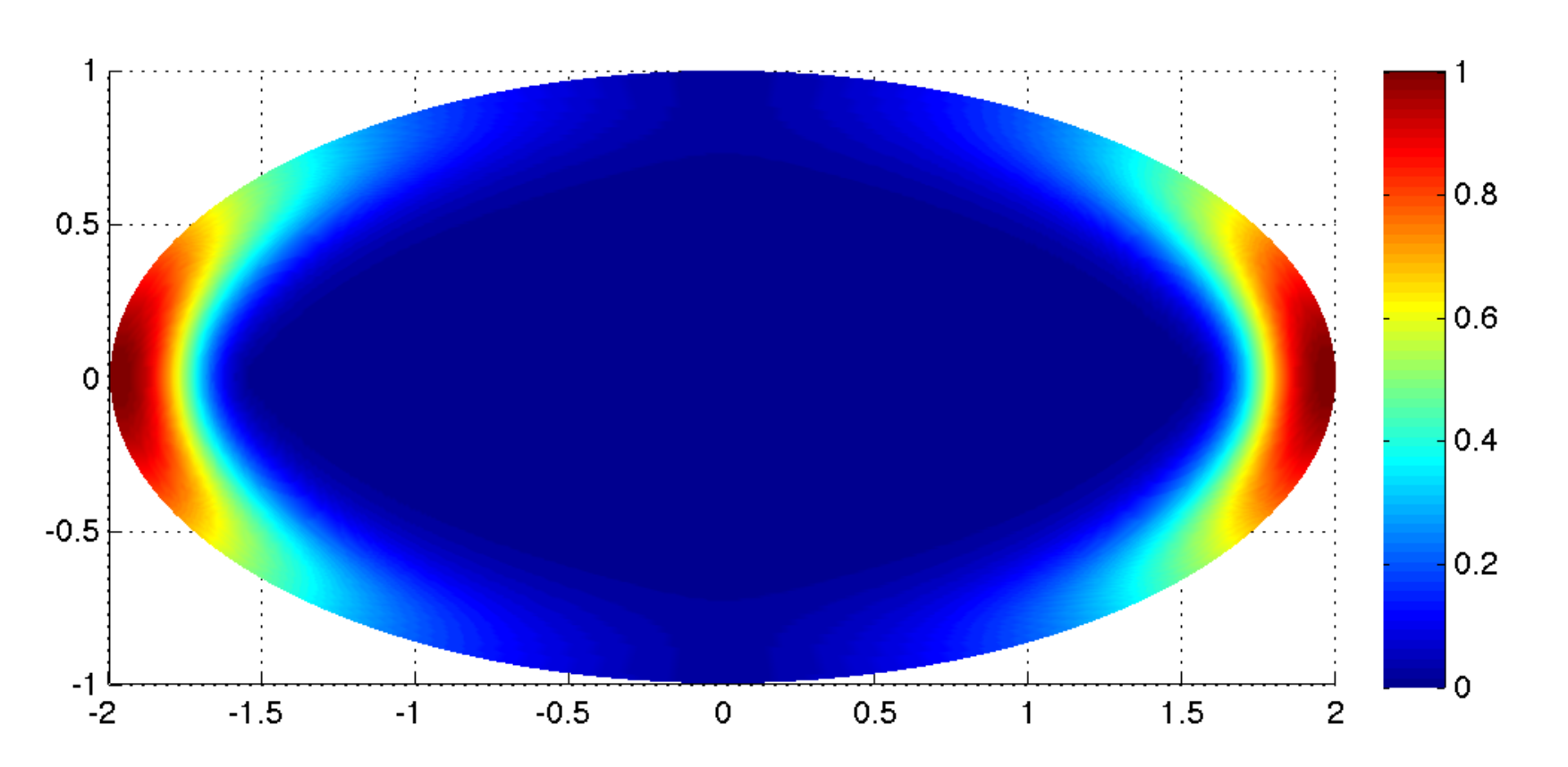}
&\includegraphics[width=4cm]{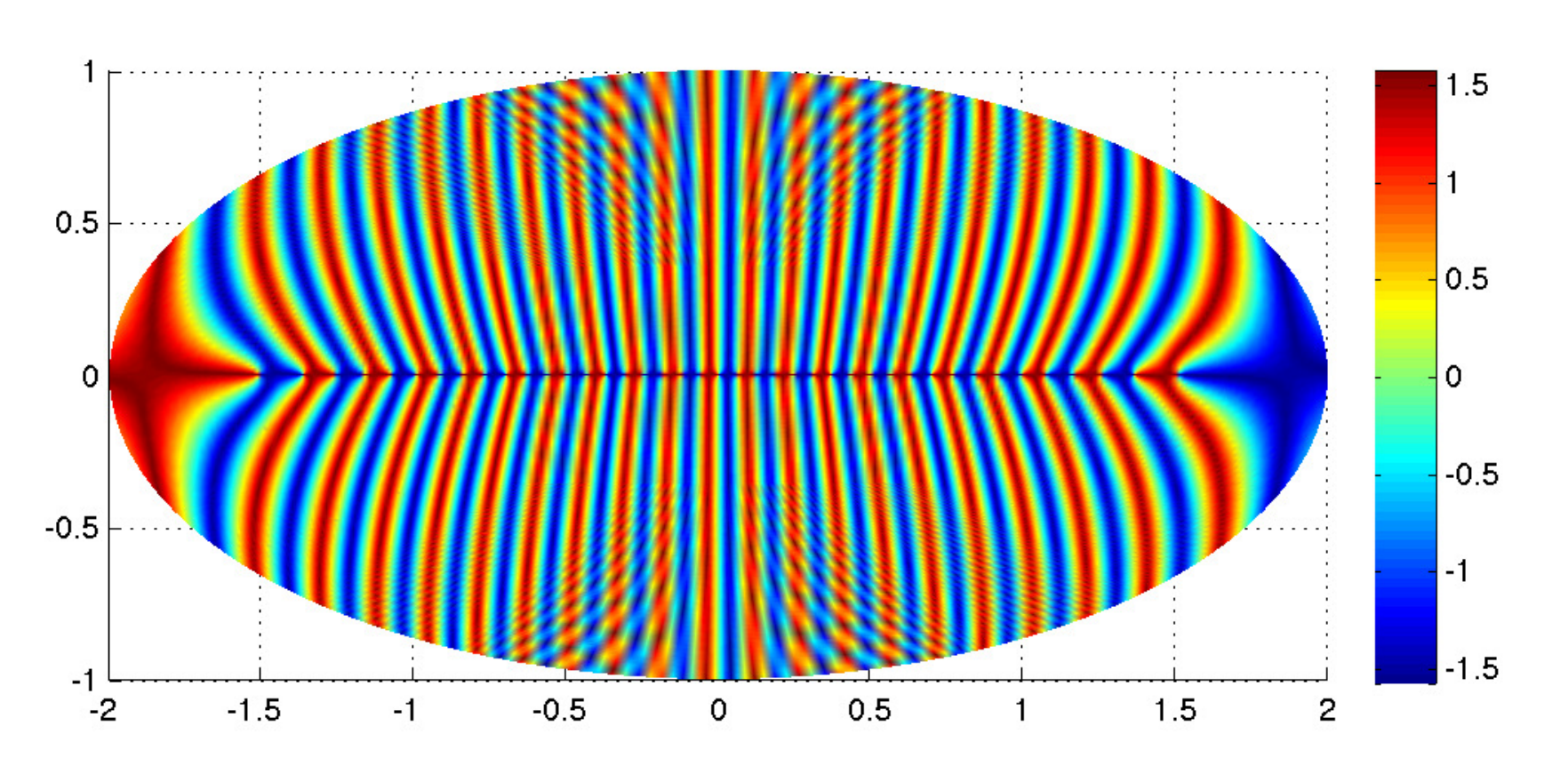}
\end{tabular}
}
\caption{Eigenpairs on the ellipse, $\bA=(-x_{2},0)$. \label{fig.ellipse}}
\end{center}
\end{figure}

\paragraph{Acknowledgments.} 
The authors would like to thank Bernard Helffer and Joseph Viola for fruitful discussions.
This work was partially supported by the ANR (Agence Nationale de la Recherche), project {\sc Nosevol} n$^{\rm o}$ ANR-11-BS01-0019 and by the Centre Henri Lebesgue (program \enquote{Investissements d'avenir} -- n$^{\rm o}$ ANR-11-LABX-0020-01).

\small

\def\cprime{$'$}

\end{document}